\documentclass[reqno,a4paper,11pt,titlepage]{amsart}
\usepackage[foot]{amsaddr}
\usepackage{amssymb}
\usepackage{amsthm}
\usepackage{amsmath}
\usepackage[utf8]{inputenc}
\usepackage[margin=1.0in]{geometry}
\usepackage{pst-node}
\usepackage{verbatim}
\usepackage{tikz-cd} 
\newtheorem{theorem}{Theorem}[section]
\newtheorem{lemma}[theorem]{Lemma}
\newtheorem{proposition}[theorem]{Proposition}
\newtheorem{corollary}[theorem]{Corollary}

\newtheorem{definition}[theorem]{Definition}

\theoremstyle{remark}
\newtheorem{remark}[theorem]{Remark}
\newtheorem{example}[theorem]{Example}

\usepackage[T1]{fontenc}
\catcode`@=11
\def\underarrow#1{\mathop{\vtop{\m@th\ialign{##\crcr
$\hfil\displaystyle{#1}\hfil$\crcr
\noalign{\kern3pt\nointerlineskip}
\hfil$\uparrow$\hfil\crcr\noalign{\kern3pt}}}}\limits}
\catcode`@=12

\usepackage[english]{babel}
\usepackage[utf8]{inputenc}

\usepackage{amsfonts}
\def\lim{\mathop{\rm lim}\nolimits}
\def\rank{\mathop{\rm rank}\nolimits}
\def\colim{\mathop{\rm colim}\nolimits}
\def\Spec{\mathop{\rm Spec}}
\def\spec{\mathop{\rm spec}}
\def\Spa{\mathop{\rm Spa}}
\def\Hom{\mathop{\rm Hom}\nolimits}

\def\Sh{\mathop{\textit{Sh}}\nolimits}

\def\B+{{B^+_{{\rm dR}}}}
\def\BdR{{B_{{\rm dR}}}}

\newcommand\cInd{\mathop{\mbox{$c$-$\mathrm{Ind}$}}}
\newcommand{\GU}{\mathrm{GU}}
\newcommand{\Gr}{\mathrm{Gr}}

\newcommand{\pr}{\mathrm{pr}}
\newcommand{\G}{\mathrm{G}}
\newcommand{\T}{\mathrm{T}}
\newcommand{\rP}{\mathrm{P}}
\newcommand{\rQ}{\mathrm{Q}}
\newcommand{\rN}{\mathrm{N}}
\newcommand{\rF}{\mathrm{F}}
\newcommand{\rR}{\mathrm{R}}
\newcommand{\M}{\mathrm{M}}

\newcommand{\U}{\mathrm{U}}
\newcommand{\Qp}{\mathbb{Q}_p}
\newcommand{\Z}{\mathbb{Z}}
\newcommand{\J}{\mathrm{J}}
\newcommand{\R}{\mathbb{R}}
\newcommand{\C}{\mathbb{C}}
\newcommand{\bL}{\mathbb{L}}
\newcommand{\Ws}{\mathcal{W}_{\mathfrak{s}_{\phi}}}
\newcommand{\Wschi}{\mathcal{W}_{\mathfrak{s}_{\phi}(\chi)}}
\newcommand{\Wchi}{\mathcal{W}(\chi)}

\newcommand{\Res}{\mathrm{Res}}
\newcommand{\Irr}{\mathrm{Irr}}

\newcommand{\ind}{\mathrm{ind}}

\newcommand{\GL}{\mathrm{GL}}

\newcommand{\id}{\mathrm{id}}
\newcommand{\Id}{\mathrm{Id}}
\newcommand{\End}{\mathrm{End}}

\newcommand{\Gm}{{\mathbb{G}_m}}

\newcommand{\E}{\mathcal{E}}
\newcommand{\OO}{\mathcal{O}}
\newcommand{\mc}{\mathcal}
\newcommand{\mf}{\mathfrak}
\newcommand{\Q}{\mathbb{Q}}

\newcommand{\bb}{\mathbb}

\newcommand{\Ig}{\mathrm{Ig}}
\newcommand{\diag}{\mathrm{diag}}
\newcommand{\Lie}{\mathrm{Lie}}

\newcommand{\LL}{{}^L}

\newcommand{\Red}{\mathrm{Red}}

\newcommand{\Sym}{\mathrm{Sym}}
\newcommand{\dom}{\mathrm{dom}}
\newcommand{\Std}{\mathrm{Std}}

\newcommand{\ov}{\overline}

\newcommand{\Cent}{\mathrm{Cent}}
\newcommand{\A}{\mathbb{A}}
\newcommand{\Gal}{\mathrm{Gal}}

\newcommand{\temp}{\mathrm{temp}}
\newcommand{\unit}{\mathrm{unit}}

\newcommand{\SL}{\mathrm{SL}}
\newcommand{\GSp}{\mathrm{GSp}}
\newcommand{\Sp}{\mathrm{Sp}}

\newcommand{\SO}{\mathrm{SO}}
\newcommand{\I}{\mathrm{I}}
\newcommand{\PGL}{\mathrm{PGL}}

\newcommand{\N}{\mathbb{N}}

\newcommand{\std}{\mathrm{std}}
\newcommand{\Ind}{\mathrm{Ind}}

\newcommand{\Sht}{\mathrm{Sht}}
\newcommand{\Spd}{\mathrm{Spd}}

\newcommand{\tri}{\mathrm{tri}}

\newcommand{\Bun}{\mathrm{Bun}}
\newcommand{\BunG}{\mathrm{Bun_{\G}}}
\newcommand{\Bunn}{\mathrm{Bun_{n}}}
\newcommand{\Rep}{\mathrm{Rep}}
\newcommand{\Aut}{\mathrm{Aut}}
\newcommand{\Dc}{\mathrm{D}}
\newcommand{\Dlis}{\mathrm{D}_{\mathrm{lis}}}

\newcommand{\Perf}{\mathrm{Perf}}
\newcommand{\Sing}{\mathrm{Sing}}

\newcommand{\IndPerf}{\mathrm{IndPerf}}

\newcommand{\ad}{\mathrm{ad}}

\newcommand{\rB}{\mathrm{B}}
\newcommand{\mP}{\mathrm{P}}

\mathchardef\mhyphen="2D

\def\ol{\overline}
\def\ra{\rightarrow}
\def\Div{\mathrm{Div}}

\newcommand\numberthis{\addtocounter{equation}{1}\tag{\theequation}}

\usepackage{hyperref}

\title{On Categorical local Langlands program for $\GL_n$}
\author{Kieu Hieu Nguyen}
\email{kieu-hieu.nguyen@uvsq.fr}
\address{University of Versailles Saint-Quentin, France}
\thanks{
\small{MSC class:	11S37}
}
\thanks{
\small{
The author acknowledges support by the European Union ERC in form of Consolidator Grants : NewtonStrat grant number 770936 and RELANTRA, project number 101044930 and by Deutsche Forschungsgemeinschaft  (DFG) through the Collaborative Research Centre TRR 326 "Geometry and Arithmetic of Uniformized Structures", project number 444845124 as well as under Germany’s Excellence Strategy EXC 2044 390685587, Mathematics Münster: Dynamics–Geometry–Structure. Views and opinions expressed are however those of the author only and do not necessarily reflect those of the European Union or the European Research Council. Neither the European Union nor the granting authority can be held responsible for them.
 }
}

\usepackage[english]{babel}
\usepackage[utf8]{inputenc}
\begin{document}

\begin{abstract}
    We study various moduli spaces of local Shtukas in the setting of Fargues' program for $\GL_n$. In certain cases, this gives us an explicit description of the spectral action which was recently introduced by Fargues and Scholze. This description sheds light to the categorical local Langlands program for $\GL_n$ and allows us to construct Hecke eigensheaves associated to certain $\ell$-adic Weil representations of rank $n$ and to prove some parts of Fargues' conjecture. Moreover, by using this description, we can prove new cases of the Harris-Viehmann conjecture for non-basic Rapoport-Zink spaces and compute some parts of the cohomology of the Igusa varieties associated to $\GL_n$.       
\end{abstract}

\maketitle
\tableofcontents
\section{Introduction}

Let $ \ell \neq p$ be primes and fix an algebraic closure $\overline{\Q}_p$ of $\Q_p$ with associated Weil group $W_{\Q_p}$. In \cite{Fa}, Fargues formulated his program to geometrize the local Langlands correspondence for $\Q_p$ via the stack $\Bun_{\G}$ of $\G$-bundles on the Fargues-Fontaine curve, with $\G$ a reductive group over $\Q_p$.

Fargues’ program has been put forward in the seminal recent work \cite{FS}. The aims of Fargues' program is to relate sheaves on the stack of Langlands parameters, "the spectral side", to sheaves on $\Bun_{\G}$, "the geometric side". On the spectral side, one considers $\Perf^{\mathrm{qc}}( [Z^1(W_{\Q_p}, \widehat{\G})/\widehat{\G}] )$ (resp. $\Dc^{b,\mathrm{qc}}_{\mathrm{coh}}( [Z^1(W_{\Q_p}, \widehat{\G})/\widehat{\G}]) $), the derived category of perfect complexes on the stack of $L$-parameters $ [Z^1(W_{\Q_p}, \widehat{\G})/\widehat{\G}] $ with quasi-compact support (resp. bounded derived category of sheaves with coherent cohomology with quasi-compact support), as in \cite{DH,Zhu1} and \cite[Section~VIII.I]{FS}. On the geometric side, one considers $\Dlis(\Bun_{\G},\ol{\mathbb{Q}}_{\ell})$, the category of lisse-\'etale $\ol{\mathbb{Q}}_{\ell}$-sheaves and $\Dlis(\Bun_{\G},\ol{\mathbb{Q}}_{\ell})^{\omega}$ the sub-category of compact objects, as defined in \cite[Section~VII.7]{FS}. Fargues and Scholze constructed an action (called the spectral action) of the category $\Perf^{\mathrm{qc}}( [Z^1(W_{\Q_p}, \widehat{\G})/\widehat{\G}] )$ on the category $\Dlis(\Bun_{\G},\ol{\mathbb{Q}}_{\ell})$ and expect this action relates these categories in a precise way. Fix a Borel subgroup $\rB \subset \G$ with unipotent radical $\U$ and a generic character $\psi : \U(\Q_p) \longrightarrow \overline{\Q}^{\times}_{\ell} $. Let $\mathcal{W}_{\varphi}$ be the Whittaker sheaf, which is the sheaf concentrated on the stratum $\Bun_{\G}^1$ corresponding to the representation $\cInd^{\G}_{\U} \psi $ of $\G(\Q_p)$. One can rephrase their conjecture by saying that the "non-abelian Fourier transform"
\begin{align*}
  \Perf^{\mathrm{qc}}( [Z^1(W_{\Q_p}, \widehat{\G})/\widehat{\G}] &\longrightarrow \Dlis(\Bun_{\G},\ol{\mathbb{Q}}_{\ell})  \\
  \M &\longmapsto \M \star \mathcal{W}_{\psi}
\end{align*}
is fully faithful and extends to an equivalence of $\overline{\Q}_{\ell}$-linear small stable $\infty$-categories
\[
\Dc^{b,\mathrm{qc}}_{\mathrm{coh}}( [Z^1(W_{\Q_p}, \widehat{\G})/\widehat{\G}] \longrightarrow \Dlis(\Bun_{\G},\ol{\mathbb{Q}}_{\ell})^{\omega},
\]
where $\star$ denotes the spectral action and $\Dlis(\Bun_{\G},\ol{\mathbb{Q}}_{\ell})^{\omega}$ denotes the full subcategory of compact objects.

This program is expected to capture in a geometric way everything we know about the local Langlands conjectures such as the internal structure of $L$-packets, Jacquet-Langlands correspondences, functoriality transfers and even some form of the local-global compatibility. We consider the case $\G = \GL_n$, the main goal of the manuscript is to compute explicitly the action of $\Perf^{\mathrm{qc}}([Z^1(W_{\Q_p}, \widehat{\GL}_n)/\widehat{\GL}_n])$ on $\pi_b$ where $b \in B(\GL_n)$ and $\pi_b$ is a smooth irreducible representation of $\G_b(\Q_p)$ whose corresponding $L$-parameter satisfies some extra conditions. Besides many geometric and cohomological applications, these computations also shed light to the relation between the categorical and the usual local Langlands conjectures.

We consider an ($\widehat{\GL}_n$-conjugacy class of) $L$-parameter $\phi$ of $\GL_n$ which has a decomposition 
\[
\phi = \phi_1 \oplus \dotsc \oplus \phi_r,
\]
where the $\phi_i$'s are irreducible and for $1 \le i \neq j \le r$, there does not exist unramified character $\chi$ such that $\phi_i \simeq \phi_j \otimes \chi$. Let $\pi$ be the irreducible representation of $\GL_n(\Q_p)$ corresponding to $\phi$ via the local Langlands correspondence. In this case we know that $\pi$ is isomorphic to the normalized parabolic induction $\Ind_{\rP}^{\GL_n} (\pi_1 \otimes \dotsc \otimes \pi_r) $ where $\rP = \M \rN$ is some standard parabolic subgroup whose Levi factor is given by $\M = \GL_{n_1} \times \dotsc \times \GL_{n_r}$ and the $\pi_i$'s are supercuspidal representations. 

We have a map $\theta : [Z^1(W_{\Q_p}, \widehat{\GL}_n)/\widehat{\GL}_n] \longrightarrow Z^1(W_{\Q_p}, \widehat{\GL}_n)//\widehat{\GL}_n$ where $Z^1(W_{\Q_p}, \widehat{\GL}_n)//\widehat{\GL}_n$ is the categorical quotient. Since $\phi$ is semi-simple, its orbit in $Z^1(W_{\Q_p}, \widehat{\GL}_n)$ by the action of $\GL_n$ is closed. Let $x$ be the $\ov \Q_{\ell}$-point in $Z^1(W_{\Q_p}, \widehat{\GL}_n)//\widehat{\GL}_n$ corresponding to this orbit. We can show that $\theta^{-1}(x)$ is closed and consists of exactly the point in $[Z^1(W_{\Q_p}, \widehat{\GL}_n)/\widehat{\GL}_n]$ corresponding to $\phi$ (proposition \ref{itm : inverse has 1 points}). We denote by $ [C_{\phi}] $ the connected component of $[Z^1(W_{\Q_p}, \widehat{\GL}_n)/\widehat{\GL}_n]$ containing $\phi$. Recall that by the computation in \cite[pg. 83]{KMSW}, we have an explicit description of the centralizer $ \displaystyle S_{\phi} := \Cent (\phi) = \prod_{i = 1}^r \Gm $ and by the results in \cite{DH}, we can show (proposition \ref{itm : simple connected components}) that $[C_{\phi}] \simeq [\bb G_m^r / \bb G_m^r ]$, the stack quotient of $\bb G_m^r$ with respect to the trivial action of $\bb G_m^r$.

Let $C \in \Perf^{\mathrm{qc}}([Z^1(W_{\Q_p}, \widehat{\GL}_n)/\widehat{\GL}_n])$ be a perfect complex such that its support does not intersect with $ \theta^{-1}(x) $. We then (by \cite[lemma 3.8]{Ham}) know that $ C \star \mathcal{F}_{\pi} = 0 $ where $\mathcal{F}_{\pi}$ is the sheaf over $\Bunn$ concentrated on $\Bun^1_n$ associated with $\pi$. More generally, if the support of $C$ does not intersect $[C_{\phi}]$ and $\pi'$ is any finitely generated representation of $\G_b(\Q_p)$ regarded as sheaf on $\Bunn$ whose (semi-simplified) $L$-parameter of any irreducible constituent is given by an $L$-parameter in $[C_{\phi}]$ then $ C \star \mathcal{F}_{\pi'} = 0 $. The study of the spectral action on $\mathcal{F}_{\pi}$ is therefore reduced to the description of the complex $ C \star \mathcal{F}_{\pi} $ for $C$ perfect complexes supported on the connected component $ [C_{\phi}]$. We denote by 
$\Dlis^{[C_{\phi}]}(\Bunn,\ol{\mathbb{Q}}_{\ell})^{\omega}$ the full sub-category of $\Dlis(\Bunn,\ol{\mathbb{Q}}_{\ell})^{\omega}$
such that the Schur-irreducible objects in this subcategory all have Fargues-Scholze parameter given by an $L$-parameter in $[C_{\phi}]$.

\subsection{Hecke operators and cohomology of local Shimura varieties} 
\subsubsection{Hecke operators and Hecke eigensheaves} \textbf{}

We fix a maximal split torus $\T$ and a Borel subgroup $\rB$ of $\GL_n$ and fix an $L$-parameter $\phi$ satisfying some conditions as above. In particular we have $\phi = \phi_1 \oplus \dotsc \oplus \phi_r $ and $ S_{\phi} = \displaystyle \prod_{i=1}^r \Gm $. Suppose that $\dim \phi_i = n_i$, thus $ \displaystyle \sum_{i=1}^r n_i = n $.

In this manuscript, we compute explicitly the action of a perfect complex $\bL \in \Perf([Z^1(W_{\Q_p}, \widehat{\GL}_n)/\widehat{\GL}_n])^{\rm qc}$ on $\mathcal{F}_{\pi'}$. Since the category $\Perf([Z^1(W_{\Q_p}, \widehat{\GL}_n)/\widehat{\GL}_n])^{\rm qc}$ is generated by shifts and cones, retracts of Hecke operators; the first step is to compute the Hecke actions acting on $\mathcal{F}_{\pi'}$ where $\pi'$ is irreducible. Since we have an explicit description of $[C_{\phi}]$, the category $\Perf([C_{\phi}])$ is not too hard to understand. In particular, we have a monoidal functor
\[
\Rep_{\ov \Q_{\ell}}(S_{\phi}) \longrightarrow \Perf([C_{\phi}]),
\]
where the image of an irreducible character $\chi$ is the vector bundle on $[C_{\phi}]$ corresponding to the structural sheaf on $\bb G_m^r$ together with the $\bb G_m^r$-action defined by $\chi$. We denote by $C_{\chi}$ this perfect sheaf and we want to concretely describe the $ C_{\chi} \star \mathcal{F}_{\pi} $ where $\pi$ is the irreducible representation of $\GL_n(\Q_p)$ with $L$-parameter given by $\phi$.

Remark that $S_{\phi}$ is commutative then the set $\Irr(S_{\phi})$ of its characters forms a group under the tensor product operator. In this specific case that group is isomorphic to $ 
\displaystyle \prod_{i=1}^r \Z $. Then a set of generators of $\Irr(S_{\phi})$ is given by the characters $\chi_i$ corresponding to $(0, \dotsc, 0, 1, 0, \dotsc, 0) $ ($1$ is in the $i^{th}$-position), $1 \le i \le r$. 

Let us explain the first main result describing the action of the $C_{\chi}$'s and the action of the Hecke operators. For each character $\chi = (d_1, \dotsc, d_r) \in \displaystyle \prod_{i=1}^r \Z $, we define an element $b_{\chi} \in B(\GL_n)$, an irreducible representation $\pi_{\chi}$ of $\G_{b_{\chi}}(\Q_p)$ and a sheaf $\mathcal{F}_{\chi}$ on $\Bunn$ as follow: 
\begin{enumerate}
    \item[$\bullet$] $b_{\chi}$ is the unique element in $B(\GL_n)$ such that its corresponding vector bundle is isormorphic to $\OO(\lambda_1)^{m_1} \oplus \dotsc \oplus \OO(\lambda_r)^{m_r} $ where $\OO(\lambda_i)$ is the stable vector bundle of slope $ \lambda_i = d_i / n_i $ and $ m_i = \gcd (d_i, n_i) $.  
    \item[$\bullet$] Consider the group $\G_{b_{\chi}}$, it is an inner form of a standard Levi subgroups of $\GL_n$. Let $\G^*_{b_{\chi}}$ be the split inner form of $\G_{b_{\chi}}$. For each $i$, let $\G_i := \GL_{m_i}(D_{-\lambda_i})$ where $D_{-\lambda_i}$ is the division algebra whose invariant is $-\lambda_i$, thus its split inner form $\G_i^*$ is isomorphic to $\GL_{n_i}$. We have a map $\phi_i : W_{\Q_p} \longrightarrow \widehat{\G^*_i}(\overline{\Q}_{\ell})$ and the direct sum $ \phi_1 \oplus \dotsc \oplus \phi_r $ gives us a map $W_{\Q_p} \longrightarrow \displaystyle \prod_{i=1}^r \widehat{\G^*_i}(\overline{\Q}_{\ell})$ whose post-composition with the natural embeddings $ \displaystyle \prod_{i=1}^r \widehat{\G^*_i}(\overline{\Q}_{\ell}) \hookrightarrow \widehat{\G^*_{b_{\chi}}}(\overline{\Q}_{\ell}) $ defines an $L$-parameter $\phi_{\chi}$ of $\G_{b_{\chi}}$. Moreover, the post-composition of $\phi_{\chi}$ with $ \widehat{\G^*_{b_{\chi}}}(\overline{\Q}_{\ell}) \hookrightarrow \widehat{\GL}_n(\overline{\Q}_{\ell}) $ is the ($\widehat{\GL}_n$-conjugacy class of) $\phi$. Finally, $\pi_{\chi}$ is the representation of $\G_{b_{\chi}}(\Q_p)$ whose $L$-parameter is given by $\phi_{\chi}$ via the local Langlands correspondence for general linear groups.
    \item[$\bullet$] We can suppose that $\G_{b_{\chi}}^*$ is standard. By (\ref{decomposition of Kottwitz' set}), there exists a unique standard parabolic subgroup $\rP$ of $\GL_n$ with Levi factor given by $\G_{b_{\chi}}^*$ such that $\nu_{b_{\chi}} \in X_*(\rP)^+ $. Let $\delta_{\rP}$ be the modulus character with respect to $\rP$ and denote $\delta_{\rP |  \G_{b_{\chi}}^*}$ its restriction to $\G_{b_{\chi}}^*$. Then we denote by $\delta_{b_{\chi}}$ the character of $ \G_{b_{\chi}} $ whose $L$-parameter is the same as that of $\delta_{\rP |  \G_{b_{\chi}}^*}$. We consider the embedding $ i_{b_{\chi}} : \Bun^{b_{\chi}}_n \longrightarrow \Bunn $ and define $\mathcal{F}_{\chi} \displaystyle := i_{b_{\chi} !}(\delta^{-1/2}_{b_{\chi}} \otimes \pi_{\chi}) [- d_{\chi}] $ where $d_{\chi} = \langle 2\rho, \nu_{b_{\chi}} \rangle$ and $\rho$ is the half  sum of all positive roots of $\GL_n$. In particular, if we denote by $\Id$ the identity element in $\Irr(S_{\phi})$ then $\mathcal{F}_{\Id} \simeq \mathcal{F}_{\pi}$. We can also identify $\mathcal{F}_{\chi}$ with $i_{b_{\chi}!}^{\mathrm{ren}}(\pi_{\chi})$ where $i_{b_{\chi}!}^{\mathrm{ren}}$ is the renormalized functor defined in \cite[Definition 1.1.1]{Hansen}. 

    For avoiding confusion, we recall that if $ \rP = \M \rN $ where $\M$ is the Levi subgroup and $\rN$ is the unipotent radical then the modulus character is defined by $\delta_{\rP}(mn) := |\det(\ad(m);\Lie \rN)| $ where $|\cdot|$ denotes the normalized absolute value of the field $\ov \Q_{\ell}$.
\end{enumerate}
\begin{remark}
\begin{enumerate}
    \item Our characters $\delta_{b}$ is the same as the character $\delta_b$ defined by Hamann and Imai in \cite{HI}. We recall that for $b \in B(\GL_n)$, we have $\nu_{b} = w_0(- \nu_{\E_b})$ where $w_0$ is the longest element in the Weyl group of $\GL_n$. Then the parabolic subgroup of $\GL_n$ corresponding to the Harder-Narasimhan reduction of $\E_{b_{\chi}}$ is conjugated to the opposite of $\rP$ above.
    \item We will sometimes use the notation $\delta_{\chi}$ to denote the character $\delta_{b_{\chi}}$ by abuse of notation. 
\end{enumerate}    
\end{remark}

The following result is theorem \ref{itm : main theorem}.

\begin{theorem}
Let $\phi$ be an $L$-parameter satisfying the above conditions and let $\chi = (d_1, \dotsc, d_r)$ be an element in $\displaystyle \Irr(S_{\phi}) \simeq \prod_{i=1}^r \Z $. Then we have
\[
C_{\chi} \star \mathcal{F}_{\Id} = \mathcal{F}_{\chi}
\]
where $\Id$ is the identity element in $ \Irr(S_{\phi}) $.
\end{theorem}

\begin{remark}
\begin{enumerate}
    \item Let $\G$ be an inner form of a Levi subgroup of $\GL_n$. Then $\G(\Q_p)$ has an irreducible representation with $L$-parameter given by $\phi$ if and only if $ \G(\Q_p) \simeq \G_{b_{\chi}}(\Q_p) $ for some $\chi \in \Irr(S_{\phi})$. Therefore the vector bundles $C_{\chi}$ realize all the possible transfers of $\pi$ to the inner forms of the Levi subgroups of $\GL_n$ by acting on $ \mathcal{F}_{\pi} \simeq \mathcal{F}_{\Id}$. It could be seen as geometric local Jacquet-Langlands and Langlands-Jacquet correspondences. Thus the categorical local Langlands program contains rich information about the classical Langlands correspondences as expected.  
    \item The modulus character $\delta^{-1/2}_{b_{\chi}}$ in the above formula is closely related to the twist appeared in \cite[Corollary IX.7.3]{FS}. More precisely, there are several Fargues-Scholze parameters associated to the $\G_{b_{\chi}}(\Q_p)$-representation $\pi_{\chi}$ : the one constructed by using excursion operators on $\Bun_n$ and the one constructed by using excursion operators on $\Bun_{\G_{b_{\chi}}}$. These parameters differ by a twist as explained in \cite[Corollary IX.7.3]{FS}. Thus, in fact $\delta^{-1/2}_{b_{\chi}} \otimes \pi_{\chi}$ has the same Fargues-Scholze parameter (with respect to $\Bun_n$) with $\pi_{\Id}$ and it is in line with \cite[Prop. 3.14]{Ham}.
    \item The case $r=1$ is treated in \cite[theorem 1.2]{AL}, \cite[theorem 1.5]{Han}. The above theorem could also be seen as 
 non-supercuspidal analogue of some of the results in \cite[\S 4]{BHN}. 
\end{enumerate}    
\end{remark}

Let us describe roughly the strategy of the proof. It is enough to prove the theorem for $\chi = (d_1, \dotsc, d_r)$ where the $d_i$'s are non-negative. We denote by $\Irr(S_{\phi})^+$ the set of all such characters and moreover denote by $\mu = (1, 0^{(n-1)})$ the simplest cocharacter of $\GL_n$. Then we have a Hecke operator $\T_{\mu}$ corresponding to $\mu$. Remark that we can compute the action of $C_{\chi_i}$ by using the Hecke operator $\T_{\mu}$ and the latter is closely related to the cohomology of Rapoport-Zink spaces \cite{RZ}, or more generally to the cohomology of Scholze's local Shimura varieties $\Sht(\GL_n, b_1, b_2, \mu)$ where $b_1, b_2$ are not necessarily basic \cite{SW}.

Indeed, we recall that $\Sht(\GL_n, b_1, b_2, \mu)$ denotes the moduli space of modifications of type $\mu$ from $\E_{b_1}$ to $\E_{b_2}$. As $ S_{\phi} \times W_{\Q_p} $-representations we have
\[
r_{\mu} \circ \phi = \bigoplus_{i = 1}^r  \chi_i \boxtimes \phi_i 
\]
where $r_{\mu}$ is the highest weight representation associated to $\mu$. In this case it is the standard representation. Thus by \cite[Corollary. VIII.4.3, Theorem. X.1.1]{FS}, we deduce that for any irreducible sheaf $\mathcal{F_{\pi'}}$ on $\Bunn$ whose (semi-simple) $L$-parameter is $\phi$ we have
\[
\T_{\mu} (\mathcal{F_{\pi'}}) = \bigoplus_{i = 1}^r C_{\chi_i} \star \mathcal{F_{\pi'}} \boxtimes \phi_i
\]
as sheaves on $\Bunn$ with $W_{\Q_p}$-action (see also proposition \ref{itm : fundamental decomposition of Hecke operator}). Thus by \cite[\S IX.3]{FS} the description of $R\Gamma_c(\Sht(\GL_n, b_1, b_2, \mu), \overline{\Q}_{\ell})$ as $W_{\Q_p} \times \G_{b_1}(\Q_p) \times \G_{b_2}(\Q_p) $-modules allows us to compute the action of $C_{\chi_i}$ by identifying the $\phi_i$-parts of the action of $W_{\Q_p}$ (see also lemma \ref{shimhecke}).

We proceed by induction on $r$. First we can deduce the case $r=1$ by \cite{Far20} \cite{Zou}, \cite[theorem 1.2]{AL}, \cite[theorem 1.5]{Han}. Suppose that the theorem is true for $1, \dotsc ,r-1$, we show that it is true for $r$. Now we proceed by induction on $D \displaystyle = \sum_{i = 1}^r d_i$. For $D = 1$, one could use known cases of the Harris-Viehmann's conjecture and \cite[theorem 1.2]{AL}, \cite[theorem 1.5]{Han} to compute $R\Gamma_c(\Sht(\GL_n, b_1, b_2, \mu), \overline{\Q}_{\ell})$ where $b_1 = 1$ and then get a description of $ C_{\chi_i} \star \mathcal{F}_{\Id} $ by identifying the $W_{\Q_p}$-action on both sides. 

In order to prove the induction step from $D = s-1$ to $D = s$, we will study the cohomology of the local Shimura varieties $\Sht(\GL_n, b_1, b_2, \mu)$ for various $b_1$ and $b_2$ such that $\kappa(b_1) = 1-s$ and $\kappa(b_2) = -s$. An important ingredient of this step is an analogue of Boyer's trick for $\Sht(\GL_n, b_1, b_2, \mu)$ when $b_1, b_2$ are non-basic (proposition \ref{generalized Boyer's trick}, corollary \ref{coro of generelized Boyer's trick} and lemmas \ref{itm : first computation}, \ref{itm : second computation}), generalizing classical Boyer's trick \cite{Boy}, \cite{HT}, \cite{Man08}, \cite{Shen14}, \cite{Han1}, \cite{Ho2}, \cite{GI}. This result allows us to relate $\Sht(\GL_n, b_1, b_2, \mu)$ with local Shimura varieties associated to $\GL_{m}$ where $m < n$ and we can then apply the induction hypothesis on $r$. Let us give more details on the strategy of this induction step.

Let $\chi = (d_1, \dotsc, d_r)$ be a character of $\Irr(S_{\phi})^+$ such that $\displaystyle \sum_{i=1}^r d_i = s$. We would like to compute $ C_{\chi} \star \mathcal{F}_{\Id} $. First we want to compute the restriction of the complex $ C_{\chi} \star \mathcal{F}_{\Id} $ to the stratum corresponding to $ b_{\chi} $. We can choose a character $\chi' = \chi \otimes \chi^{-1}_i$ for some well chosen $i$ such that $ \chi' $ belongs to $\Irr(S_{\phi})^+$ and the Newton polygons of $ b_{\chi} $ and $b_{\chi'}$ are sufficiently close so that we can use the analogue of Boyer's trick to the triple $(b_{\chi'}, b_{\chi}, \mu)$. In this case, by monoidal property of spectral action we have
\[
i^*_{b_{\chi}}C_{\chi} \star \mathcal{F}_{\Id} =  i^*_{b_{\chi}} C_{\chi_i} \star (C_{\chi'} \star \mathcal{F}_{\Id} ).
\]

However, by the induction hypothesis on $D = s-1$, we see that $C_{\chi'} \star \mathcal{F}_{\Id} = \mathcal{F}_{\chi'}$. Since the complex $\mathcal{F}_{\chi'}$ is supported only on the stratum corresponding to $b_{\chi'}$, we can use $ R\Gamma_c(\Sht(\GL_n, b_{\chi'}, b_{\chi}, \mu))[\pi_{\chi'}] $ to compute $i^*_{b_{\chi}} C_{\chi_i} \star \mathcal{F}_{\chi'}$ as above (by identifying the appropriate $W_{\Q_p}$-action). The analogue of Boyer's trick (proposition \ref{generalized Boyer's trick}, lemma \ref{itm : first computation}) applying to $ R\Gamma_c(\Sht(\GL_n, b_{\chi'}, b_{\chi}, \mu))[\pi_{\chi'}] $ allows us to use the induction hypothesis on $t < r$ to compute $ R\Gamma_c(\Sht(\GL_n, b_{\chi'}, b_{\chi}, \mu))[\pi_{\chi'}] $ and deduce that $i^*_{b_{\chi}}C_{\chi} \star \mathcal{F}_{\Id} = \mathcal{F}_{\chi}$. 

We remark that more generally for an arbitrary $b$, we can use $ R\Gamma_c(\Sht(\GL_n, b_{\chi'}, b, \mu))[\pi_{\chi'}] $ to compute the restriction $i^*_b C_{\chi_{i}} \star \mathcal{F}_{\chi'} $ to any stratum corresponding to $b$. In particular, we deduce that $ C_{\chi} \star \mathcal{F}_{\Id} = C_{\chi_i} \star \mathcal{F}_{\chi'} $ is only supported on the strata $b$ such that there exists a modification of type $\mu$ from $\E_{b_{\chi'}}$ to $\E_b$.  

The next step is to show that the restrictions of $C_{\chi} \star \mathcal{F}_{\Id}$ to the strata different from $b_{\chi}$ vanish. It is the most technical part of the manuscript. We know that the irreducible constituents appearing in the cohomology of $i^*_{b}C_{\chi} \star \mathcal{F}_{\Id}$ are irreducible representations of $ \G_{b}(\Q_p) $ whose (semi-simplified) $L$-parameters are given by $\phi$ after post-composing with the natural $L$-embedding $ \widehat{\G^*_b} \hookrightarrow \widehat{\GL}_n $. By proposition \ref{itm : shape of strata}, we deduce that if $i^*_bC_{\chi} \star \mathcal{F}_{\Id}$ does not vanish then $b$ is of the form $b_{\xi}$ for some $\xi \in \Irr(S_{\phi})$. It is a generalization of the more familiar fact that the support of an irreducible sheaf $\mathcal{F}$ on $\Bunn$ with cuspidal $L$-parameter is contained in the semi-stable locus. Some simpler forms of this property were already used in \cite{CS, CS1, Ko1} to study cohomology of Shimura varieties with torsion coefficients. This property relies ultimately on the classification of irreducible representations by $L$-parameters and the compatibility up to semi-simplification between Fargues-Scholze $L$-parameters and the usual $L$-parameters for inner forms of $\GL_n(\Q_p)$ \cite[IX.7.3]{FS} \cite[Theorem. 6.6.1]{HKW}.

Thus we need to describe the elements of the form $b_{\xi}$ that can be obtained by some modifications of type $\mu$ of $b_{\chi'}$ and then study the cohomology of the associated moduli spaces $ \Sht(\GL_n, b_{\chi'}, b_{\xi}, \mu) $. Remark that in general, for a non semi-stable vector bundle $\E$ of rank $n$, the set of elements $b$ in $B(\GL_n)$ such that there exists a modification $ \E \longrightarrow \E_b $ of type $\mu$ is unknown. In order to get some information on the $b_{\xi}$'s that could appear, we notice that the modification of type $\mu$
\[
f : \E_{b_{\chi'}| X \setminus \infty} \longrightarrow \E_{b_{\xi}| X \setminus \infty}
\]
 induces a bijection from the set of $\rP$-reductions of $\E_{b_{\chi'}}$ to the set of $\rP$-reductions of $\E_{b_{\xi}}$ where $\rP$ is a standard parabolic subgroup of $\GL_n$ (where $X$ denotes the Fargues-Fontaine curve associated to the algebraically closed perfectoid field $ C^{\flat} $ and $\infty$ is some Cartier divisor associated to some untilt $C$ of $ C^{\flat} $). Therefore, by choosing an appropriate $\rP$-reduction $\E^{\rP}_{b_{\chi'}}$ of $\E_{b_{\chi'}}$ corresponding to a $\rP$-reduction of $b_{\chi'}$, we get a $\rP$-reduction $\E^{\rP}_{b_{\xi}}$ of $\E_{b_{\xi}}$. In our case, a $\rP$-reduction is no other than a filtration of vector bundle. Then by the theory of extensions and of Harder-Narasimhan filtration of vector bundles on the Fargues-Fontaine curve, we see that the information on $\E^{\rP}_{b_{\xi}}$ yields some controls on $b_{\xi}$. One important fact is that we can use \cite[lemma 2.6]{CFS} and \cite[lemma 3.11]{Vie} to compute $\E^{\rP}_{b_{\xi}}$. More precisely, the modification $f$ above is encoded by an element $x$ in the Schubert cell $\Gr^{\mu}_n(C)$ associated to $\mu$ inside the $\B+$-Grassmannian $\Gr_n(C)$. If $\M$ denotes the Levi factor of $\rP$ then the restriction to $\M$ of $\E^{\rP}_{b_{\xi}}$ is the image of a modification $f^{\M}$ of the restriction to $\M$ of $\E^{\rP}_{b_{\chi'}}$. That modification $f^{\M}$ is encoded by the element $ y := \Pr_{\M}(x)$ in $\Gr_{\M}(C)$ by using the Iwasawa decomposition. In particular, the type of the modification $f^{\M}$ is closely related to $\mu$ and can be computed explicitly. Since $\mu = (1, 0^{(n-1)})$, the above information is traceable. We study the above process case by case depending on the combinatoric description of $\E^{\rP}_{b_{\chi'}}$. 

The upshot is that if $b_{\xi} \neq b_{\chi}$ then we can show that in all cases $b_{\xi}$ and $b_{\chi}$ are close enough so that we can apply the analogue of Boyer's trick to reduce the computation of $ R\Gamma_c(\Sht(\GL_n, b_{\chi'}, b_{\xi}, \mu))[\pi_{\chi'}] $ to some computations involving only local Shimura varieties associated to $\GL_m$ with $m < n$ and then use the induction hypothesis. We emphasis that we do not need to compute the whole cohomology $ R\Gamma_c(\Sht(\GL_n, b_{\chi'}, b_{\xi}, \mu))[\pi_{\chi'}]$; we only need to compute its $[\phi_i]$-isotypical part since we only need to determine the action of $C_{\chi_i}$. This makes the computation process simpler since $C_{\chi_i} \star ( C_{\chi_i^{-1}} \star (-)) = C_{\Id} \star (-)$ is the identity functor of $\Dlis^{[C_{\phi}]}(\Bunn,\ol{\mathbb{Q}}_{\ell})^{\omega}$. More precisely, it allows us to use the topological properties of $\Bun_n$ via the semi-orthogonal decomposition of $\Dlis(\Bun_n, \ov \Q_{\ell})$ into the $\Dlis(\Bun^b_n, \ov \Q_{\ell})$'s where $b$ varies in $B(\GL_n)$ (lemma \ref{itm : morphism between strata}) and to use various adjunction formulas.

The above theorem could lead to some interesting applications. First of all, since the irreducible representations of $S_{\phi}$ are all of dimension $1$, we deduce the following description of the regular representation $ \displaystyle V_{\text{reg}} = \bigoplus_{\chi \in \Irr(S_{\phi})} \chi $. Thus the sheaf $ \displaystyle \mathcal{G}_{\phi} := \bigoplus_{\chi \in \Irr(S_{\phi})} \mathcal{F}_{\chi} $ is non trivial and by using the above theorem, we can compute $\displaystyle \T_{V} (\mathcal{G}_{\phi})$ for every algebraic representation $V$ of $\GL_n$ and then prove the following result (theorem \ref{itm : Hecke eigensheaf}).
\begin{theorem}
The sheaf $ \displaystyle \mathcal{G}_{\phi} $ is a non trivial Hecke eigensheaf corresponding to the $L$-parameter $\phi$.
\end{theorem}

\subsubsection{Cohomology of local Shimura varieties and Harris-Viehmann's conjecture} \textbf{}

As we know, the theory of Rapoport-Zink spaces and local Shimura varieties play a crucial role in the Langlands program. There are two important conjectures in the study of the cohomology of these spaces.  
\begin{enumerate}
    \item The Kottwitz conjecture which concerns the discrete part of the cohomology of the basic local Shimura varieties.
    \item The Harris-Viehmann conjecture which gives an inductive formula of the cohomology of the non-basic local Shimura varieties.
\end{enumerate}

By combining these two conjectures, one should be able to \textit{describe} the cohomology of these spaces.

The Kottwitz conjecture was proven in many cases. The Lubin-Tate case was proven by \cite{Boy, HT}, the case of basic EL unramified spaces was proven in \cite{SWS, Far04}, the case of basic PEL odd-unitary groups was treated in \cite{KH1, BMN} and the case of basic PEL type associated with $\GSp_4$ was proven in \cite{Ham}. However, the Harris-Viehmann conjecture was only known in some cases involving the Hodge-Newton reducible conditions and classical Boyer's trick. 

In \cite[\S 4.1.2]{BHN} we prove the Kottwitz conjecture for non-minuscule cocharacters for unitary similitude groups by exploiting the explicit description of the spectral action on the supercuspidal $L$-packets. Now we know how the Hecke operatos act on the representations of $\GL_n(\Q_p)$ whose $L$-parameter satisfies our starting hypothesis. Thus by the same argument, we can \textit{compute} some parts of the cohomology of non-basic local Shtukas spaces and then deduce new cases of the Harris-Viehmann's conjecture for $\GL_n$. Let $b \in B(\GL_n)$ and suppose that $\E_b \simeq \E(\lambda_1) \oplus \dotsc \oplus \E(\lambda_k)$ for $\lambda_1 > \dotsc > \lambda_k$ and where $\E(\lambda_j)$ is a semi-stable vector bundle of slope $\lambda_j$. Let $\pi_b$ be a representation of $\G_b(\Q_p)$ with the corresponding $L$-parameter $\phi^b$. We denote by $\M$ the standard Levi subgroup of $\GL_n$ that is the split inner form of $\G_b$ and denote by $\rP$ the standard parabolic subgroup of $\GL_n$ with Levi factor $\M$. Then $\M = \GL_{m_1} \times \dotsc \times \GL_{m_k}$ where $m_j := \rank\E(\lambda_j)$ and the natural embedding $ ^{L}\G_b(\overline{\mathbb{Q}}_{\ell}) \longrightarrow \ ^{L}\GL_n(\overline{\mathbb{Q}}_{\ell}) $ induces a morphism of $L$-groups $ \eta : \ ^{L}\M(\overline{\mathbb{Q}}_{\ell}) \longrightarrow \ ^{L}\GL_n(\overline{\mathbb{Q}}_{\ell}) $. We define an $L$-parameter of $\GL_n$ by $ \phi := \eta \circ \phi^b $ and suppose that $\phi$ satisfies the conditions in theorem \ref{itm : main theorem}.

Let $\mu$ be an arbitrary cocharacter of $\GL_n$ and suppose that  as $S_{\phi} \times W_{\Q_p} $-modules we have
\[
 \displaystyle r_{\mu} \circ \phi = \bigoplus_{\chi \in \Irr(S_{\phi})} \chi \boxtimes \sigma_{\chi}. 
\]
Then we have (up to some shifts and twists)
\[
\T_{\mu} (i_{b!}\pi_b) \simeq \bigoplus_{\chi \in \Irr(S_{\phi})} C_{\chi} \star i_{b!}\pi_b  \boxtimes \sigma_{\chi}.
\]

Therefore theorem \ref{itm : main theorem} allows us to describe the Hecke action $\T_{\mu} (i_{b!}\pi_b)$ and compute the cohomology of moduli of local Shtukas. The Levi subgroup $\M$ is a product of general linear groups then the same arguments allow us to compute the cohomology of moduli of local Shtukas associated to $\M$. Therefore, a generalized form of the Harris-Viehmann conjecture follows from this description of the cohomologies. For simplicity, let us state the result up to some shifts and twists and only in the minuscule case where $\mu = (1^{(a)}, 0^{(n-a)})$ for some integer $1 \le a \le n$ (see theorem \ref{itm : Harris-Viehmann conjecture} for a more general version of the result).
\begin{theorem}
With the above notations, if $\mu$ is minuscule then we have the following identity (up to some shifts and some twists of $\pi_b$ by unramified characters) of $ \GL_n(\Q_p) \times W_{\Q_p} $-representations
\[
R\Gamma_{c}(\GL_n, b,\mu) [\pi_{b}] \simeq \Ind_{\rP}^{\GL_n} R\Gamma_{c}(\M, b_{\M},\mu_{\M})  [\pi_{b}].
\]
where $ b_{\M} $ is the reduction of $b$ to $\M$ and $\mu_{\M} = \mu_1 \times \dotsc \times \mu_k$ and where $\mu_j = (1^{(\deg\E(\lambda_j))}, 0^{(\rank\E(\lambda_j) - \deg\E(\lambda_j))})$. 
\end{theorem}

\subsection{Spectral action} \textbf{}

We now want to compute the action of a perfect complex $\bL \in \Perf([C_{\phi}])$ on $\mathcal{F}_{\pi}$. The connected component $[C_{\phi}]$ containing $\phi$ gives rise to a direct summand
\[
\Perf([C_{\phi}]) \hookrightarrow \Perf([Z^1(W_{\Q_p}, \GL_n)/\GL_n]).
\]

Therefore the spectral action gives rise to a corresponding direct summand 
\[
\Dlis^{[C_{\phi}]}(\Bunn,\ol{\mathbb{Q}}_{\ell})^{\omega} \subset \Dlis(\Bunn,\ol{\mathbb{Q}}_{\ell})^{\omega}.
\]

We want to understand the structure of the category $\Dlis^{[C_{\phi}]}(\Bunn,\ol{\mathbb{Q}}_{\ell})^{\omega}$ together with the action of $\Perf([C_{\phi}])$. However we have
\[
\Perf([C_{\phi}]) \simeq \bigoplus_{\chi \in \Irr(S_{\phi})} \Perf(\mc O(\bb G_m^r)),
\]
where $\Irr(S_{\phi}) \simeq \displaystyle \prod_{i=1}^r \Z $ is the set of characters of $S_{\phi}$. From the theory of Bernstein decomposition, we know that the category of $\mathcal{O}(\bb G_m^r)$-modules is equivalent to the Bernstein block of $ \Rep_{\overline{\Q}_{\ell}}(\G(\Q_p))$ containing a representation $\pi$ whose $L$-parameter is given by $\phi$ where $\G$ is any inner form of a Levi sub-group of $\GL_n$. Thus $\Perf(\mc O(\bb G_m^r))$ is equivalent to the full sub-category of compact objects of the derived category of these Bernstein blocks. Therefore, it is expected that $\Dlis^{[C_{\phi}]}(\Bunn,\ol{\mathbb{Q}}_{\ell})^{\omega}$ is the direct sum of all of these categories, indexed by the elements in $\Irr(S_{\phi})$. 

 For each $\chi \in \Irr(S_{\phi})$, we define a full subcategory $\mathrm{D}(\Rep(\mathfrak{s}_{\phi}(\chi)))^{\omega}$ of $\Dlis^{[C_{\phi}]}(\Bunn,\ol{\mathbb{Q}}_{\ell})^{\omega}$ as follow. Let $\mathrm{D}(\Rep(\mathfrak{s}_{\phi}(\chi)))$ be the derived category of the Bernstein block of the category $\Rep_{\overline{\Q}_{\ell}}(\G_{b_{\chi}}(\Q_p))$ containing the irreducible representation $\pi_{\chi}$. We denote by $\mathrm{D}(\Rep(\mathfrak{s}_{\phi}(\chi)))^{\omega}$ its full sub-category of compact objects. Denote by $\Ws$ the projection of $ \mathcal{W}_{\psi} $ on $\Rep(\mathfrak{s}_{\phi}(\Id)))^{\omega}$. The following results are theorems \ref{itm : spectral action - basic}, \ref{itm : spectral action - general} and \ref{itm : orthogonal decomposition}.

\begin{theorem}
    Let $\bL$ be a perfect complex in $\Perf([C_{\phi}])$. Then
\begin{enumerate}
    \item We can compute explicitly $\bL \star \Ws $,
    \item We have a decomposition of categories
    \[
    \Dlis^{[C_{\phi}]}(\Bun_n, \overline{\Q}_{\ell})^{\omega} \simeq \bigoplus_{\chi \in \Irr(S_{\chi})}\mathrm{D}(\Rep(\mathfrak{s}_{\phi}(\chi)))^{\omega}.
    \]
\end{enumerate}
\end{theorem}

Let us describe briefly the strategy of the proof. The first step is to understand $C_{\chi} \star \Ws$ and we proceed 
by induction on $r$ as in the proof of theorem \ref{itm : main theorem}. We deduce the case $r=1$ by using \cite[theorem 1.2]{AL}, \cite[theorem 1.5]{Han} and Bernstein's result on the equivalence between $\Rep(\mathfrak{s}_{\phi}(\Id)))$ and the category of $\ov \Q_{\ell}[X, X^{-1}]$-modules. Then we argue as in the proof of theorem \ref{itm : main theorem} to deduce the general case. Note that it is less complicated since we can use theorem \ref{itm : main theorem} to simplify the arguments.

By the construction of the spectral action, we know that if we have morphisms, cones or retracts of perfect complexes in $\Perf([C_{\phi}])$, then by acting on an object $A \in \Dlis(\Bun_n, \overline{\Q}_{\ell})^{\omega} $, it induces corresponding morphisms, cones or retracts in $\Dlis(\Bun_n, \overline{\Q}_{\ell})^{\omega}$. Moreover, the category $\Perf([C_{\phi}])$ is generated under shifts, cones and retracts by $\Rep(S_{\phi})$ then by tracing back the construction, we are able to compute $\bL \star \Ws $. One important ingredient is the explicit description of the morphisms  
\[
\Psi_{\GL_n} : \mathcal{O}([C_{\phi}]) \longrightarrow \mathcal{Z}(\Rep(\mathfrak{s}_{\phi}(\chi)))
\]
between the factor $\mathcal{O}([C_{\phi}])$ of the spectral Bernstein center and the Bernstein center of the blocks $\Rep(\mathfrak{s}_{\phi}(\chi))$'s. This description relies ultimately on the compatibility between Fargues-Scholze $L$-parameters and the usual $L$-parameters of inner forms of $\GL_n$. \\

In the final part, we would like to verify some cases of the local-global compatibility conjecture. The rough idea is that by using the comparison between the fibers of the Hodge-Tate period map with the Igusa varieties \cite[Thm. 1.15]{CS}, one reduces the local-global compatibility to the computation of the cohomology of the Igusa varieties. Then we can use Mantovan's formula and the description of the Hecke operators to understand Igusa varieties.  

\begin{remark}
\begin{enumerate}
    \item In \cite{Hell, Zhu1}, Eugen Hellmann and Xinwen Zhu stated some conjectures relating the derived category of smooth representations of a $p$-adic split reductive group with the derived category of (quasi-)coherent sheaves on a stack of $L$-parameters and Hellmann also proved his conjectures for $\GL_2$. However it is not clear how to incorporate the spectral actions into their conjectures.
    \item By studying geometric Eisenstein series \cite{Ham1}, Linus Hamann also obtained the description of Hecke eigensheaves associated with $L$-parameters that factor through some maximal torus of an arbitrary reductive group. A construction of Hecke eigensheaves associated with “generous” $L$-parameters for arbitrary reductive groups, together with connections between the Harris–Viehmann conjecture and other conjectures inspired by the geometric Langlands program, will appear in forthcoming work by Hamann, Hansen, and Scholze. \cite{HHS}. 
    \item In \cite{Zou25}, by combining an observation from geometric Langlands program and the techniques in this paper, Konrad Zou obtained a version of the main results in this paper for integral coefficients. The case $r=1$ was also obtained by Chenji Fu.  
\end{enumerate}  
\end{remark}  

\subsubsection*{Organization of the paper} \textbf{} \\

In section \ref{itm : Generalities on Bun}, we recollect some well-known results on vector bundles over the Fargues-Fontaine curve and the stack of these bundles as well as the construction of Fargues-Scholze $L$-parameters. In section \ref{itm : stack of $L$-parameters}, we recall the definition of the stack of $L$-parameters and study some of its basic geometric properties. We also study perfect complexes on this stack and deduce some consequences on the Hecke operators. In section \ref{itm : combinatoric description of Hecke operators}, we state theorem \ref{itm : main theorem}, our first main results. In section \ref{itm : Boyer's trick}, we prove a generalization of Boyer's trick for moduli spaces of local Shtukas $\Sht(\GL_n, b_1, b_2, \mu)$ where we allow both $b_1, b_2$ to be non basic. Then we also deduce some computations that are important for the proof of theorem \ref{itm : main theorem} that occupies the whole section \ref{itm : proof of the first main theorem}. Then in the following sections, we give some applications of our theorem \ref{itm : main theorem}. In section \ref{itm : Hecke-eigensheaves}, we construct Hecke eigensheaves associated to some $L$-parameters $\phi$ and in sections \ref{itm : Harris-Viehmann}, we prove new cases of the Harris-Viehmann's conjecture for non-basic Rapoport-Zink spaces associated to $\GL_n$. In section \ref{itm : categorical Langlands}, we describe some parts of the map $\Psi_{\GL_n}$ between spectral Bernstein center and the usual Bernstein center constructed recently by Fargues and Scholze. Then we use all these results to compute the spectral action of $\Perf([C_{\phi}])$ on $\Dlis^{[C_{\phi}]}(\Bunn, \ov \Q_{\ell})^{\omega}$. In the final section \ref{itm : Local-Global compatibility}, we compute parts of the cohomology of the Igusa varieties and deduce some weak form of the local-global compatibility. 

\subsection*{Acknowledgments}

I would like to thank Alexander Bertoloni-Meli, Linus Hamann, Bao Le Hung, Tuan Ngo Dac and Eva Viehmann for very helpful discussions, feedback and for pointing out several inaccuracies in an earlier draft. I thank Laurent Fargues for feedback in an early stage of this project. I would also like to thank Robin Bartlett, Claudius Heyer, Damien Junger, Lucas Mann, Viet Cuong Pham, Zhixiang Wu, Yifei Zhao for fruitful discussions. Special thanks go to Alexander Bertoloni-Meli, Linus Hamann, Claudius Heyer for constant support and for patiently answering all my numerous questions.  \\

\section*{Notation}
We use the following notation: 
\begin{itemize}
\item $\ell$ and $p$ are distinct prime numbers. 
\item $ \breve \Q_p $ is the completion of the maximal unramified extension of $\Q_p$ with Frobenius $ \sigma $. Let $\overline{\Q}_p$ be an algebraic closure of $\Q_p$ and denote $\Gamma := \Gal(\overline{\Q}_p / \Q_p)$.
\item $ \G $ is a connected reductive group over $\Q_p$. Let $\mathrm{H}$ be a quasi-split inner form of $\G$ and fix an inner twisting $ \G_{\breve \Q_p} \overset{\sim}{\longrightarrow} \mathrm{H}_{\breve \Q_p} $. 
\item  $ A \subseteq T \subseteq B $ where $ A $ is a maximal split torus, $ T = Z_H(A) $ is the centralizer of $ A $ in $T$ and $ B $ is a Borel subgroup in $\mathrm{H}$. Let $ \U $ be its unipotent radical.
\item $ (X^*(T), \Phi, X_*(T), \Phi^{\vee}) $ is the absolute root datum of $\G$ with positive roots $ \Phi^+ $ and simple roots $ \Delta $ with respect to the choice of $B$.
\item $\rho$ is the half sum of the positive roots.
\item  Further, $ (X^*(A), \Phi_0, X_*(A), \Phi^{\vee}_0) $ denotes the relative root datum, with positive roots $ \Phi^+_0 $ and simple roots $ \Delta_0 $.
\item On $ X_*(A)_{\Q} $ resp.~$ X_*(T)_{\Q} $ we consider the partial order given by $ \nu \leq \nu' $ if $ \nu' - \nu $ is a non-negative rational linear combination of positive resp.~relative coroots. 
\item Let $\rP$ be a parabolic subgroup of $\G$, then $\Ind_{\rP}^{\G}$, resp. $\ind_{\rP}^{\G}$, denotes the normalized, resp. un-normalized parabolic induction.
\item Let $ C | \overline{\Q}_p $ be an algebraically closed complete field. Let $C^{\circ}$ resp. $C^{\flat,\circ}$ be the subring of power-bounded elements of $C$ resp. $C^{\flat}$ and let $\xi$ be a generator of the kernel of the surjective map $W(C^{\flat,\circ})\rightarrow C^{\circ}$. Let $\B+:=\B+(C)$ be the $\xi$-adic completion of $W(C^{\flat,\circ})[1/p]$ and $\BdR=\BdR(C)=\B+[\xi^{-1}]$. Then $\B+\cong C  [[\xi]]$ and $\BdR\cong C((\xi))$.
\item Let $X$ be the schematic Fargues-Fontaine curve over $C^{\flat}$. The untilt $C$ of $C^{\flat}$ corresponds to a point $ \infty \in |X|$ with residue field $C$ and $\widehat{\mathcal{O}}_{X, \infty}\cong \B+$. 
\item By $\Gr_{\G}$ we denote the $\B+$-Grassmannian as in \cite[Definition 19.1.1]{SW} and we also write $\Gr_n$ for $\Gr_{\GL_n}$.
\item For $ \lambda \in \Q $, we denote the stable vector bundle on $X$ whose slope is $\lambda$ by $\OO(\lambda)$. If $\lambda = 0$, we simply write $\OO$ for $\OO(0)$.
\item Let $B(\G)$ be the set of $\G(\breve \Q_p)$-$\sigma$-conjugacy classes of elements of $\G(\breve \Q_p)$. For each elements $[b] \in B(\G)$, we denote by $\G_b$ the $\sigma$-centralizer of $b$.  By work of Kottwitz, elements $[b]$ are classified by their Kottwitz point $\kappa_{\G}(b)\in \pi_1(\G)_{\Gamma}$ and their Newton point $\nu_b \in X_*(A)_{\mathbb Q,\dom}$.
\item For a $\G$-bundle $\E$ on $X$ let $\nu_{\E} \in X_*(A)_{\mathbb Q,\dom}$ be the corresponding Newton polygon.
\end{itemize} 	
\section{Generalities on $\Bunn$} \phantomsection \label{itm : Generalities on Bun}

\subsection{Vector bundles on the Fargues-Fontaine curve}
Let $ \Bun(X) $ denote the category of vector bundles on the Fargues-Fontaine curve and recall from \cite{FF} that the curve $X$ is complete in the sense that if $f \in k(X)$ is any nonzero rational function on $X$, then the divisor of $f$ has degree zero. Thus if $\mathcal{E}$ is a rank $n$ vector bundle on the curve $X$ then we can define its degree by $ \deg (\mathcal{E}) := \deg \Lambda^{n} \mathcal{E} $ and its slope by $\mu(\mathcal{E}):= \frac{ \deg \E }{\rank \E}$.

Recall that a vector bundle $\E$ on $X$ is stable if it has no proper, non-zero subbundles $\mathcal{F} \longrightarrow \E$ with $\mu(\mathcal{F}) \ge \mu(\E)$. We say that $\E$ is semi-stable if it has no proper, non-zero subbundles $\mathcal{F} \longrightarrow \E$ with $\mu(\mathcal{F}) > \mu(\E)$.
\begin{definition}
    A Harder-Narasimhan (HN) filtration of a vector bundle $\E$ is a filtration of $\E$ by subbundles $ 0 = \E_0 \subset \E_1 \subset \dotsc \E_m = \E $ such that the quotients $\E_{i+1}/ \E_{i}$ are semi-stable with strictly decreasing slopes $ \mu_1 > \mu_2 > \dotsc > \mu_m $. The Harder-Narasimhan (HN) polygon of $\E$ is the upper convex hull of the points  ($\rank \E_i, \deg \E_i $).
\end{definition}
We have the following results about vector bundles on the Fargues-Fontaine curve.
\begin{theorem}(Fargues-Fontaine, Kedlaya) \cite{FF, Ked}
    Vector bundles on $X$ satisfy the following properties:
    \begin{enumerate}
    \item The set of isomorphism classes of rank $n$ vector bundles over $X$ is classified by the Kottwitz set $B(\GL_n)$.
        \item Every vector bundle $\E$ admits a canonical Harder-Narasimhan filtration.
        \item For every $\lambda$ in $\Q$, there is a unique stable bundle of slope $\lambda$ on $X$, which is denoted by $\mathcal{O}(\lambda)$. Writing $\lambda = r/s$ in lowest terms, then the bundle $\mathcal{O}(\lambda)$ has rank $s$ and degree $r$.
        \item Any semistable bundle of slope $\lambda$ is a finite direct sum $\mathcal{O}^d(\lambda)$, and tensor products of semi-stable bundles are semi-stable.
        \item For any  $\lambda \in \Q $, we have
            \[
            H^0(\mathcal{O}^d(\lambda)) = 0 \textrm{ \ if and only if \ } \lambda < 0
            \]
            and
            \[
            H^1(\mathcal{O}^d(\lambda)) = 0 \textrm{ \ if and only if \ } \lambda \ge 0
            \]
            In particular, any vector bundle $\E$ admits a splitting $ \displaystyle \E = \bigoplus_i \mathcal{O}(\lambda_i) $ of its Harder-Narasimhan filtration.
    \end{enumerate}
\end{theorem}

Recall that $X$ is the Fargues-Fontaine curve over $C^{\flat}$, which comes equipped with a point $\infty$ corresponding to $C$. By Beauville-Laszlo's gluing theorem \cite{BL} we have a bijective correspondence between vector bundles $\E$ on $X$ and triples $(\E^e,\E_{\B+},\iota)$ where $\E^e$ is a vector bundle over $X\setminus \{\infty\}$, where $\E_{\B+}$ is a vector bundle on $\Spec(\B+)$ and where $\iota:\mathcal{E}^e \otimes_{B_e} \BdR\rightarrow \E_{\B+}\otimes_{\B+}\BdR$ is an isomorphism. Here, the triple corresponding to some $\E$ is given by the respective base changes of $\E$ together with the induced isomorphism.

The pullback of a vector bundle $\E_b$ via $ \Spec(\B+) \longrightarrow X $ is trivial. Indeed, the inclusion $\Q_p \hookrightarrow C\hookrightarrow \B+$ extends to an embedding of an algebraic closure $\overline{\Q}_p$ into $\B+$. By Lang's theorem there is a $g\in \GL_n(\overline{\Q}_p)$ with $gb\sigma(g^{-1})=1$, which induces the desired trivialization. It is well-defined up to the action of $\GL_n(\Q_p)$. The Beauville-Laszlo uniformization depends on the choice of such a trivialization. From now on we consider $\E_b$ together with a trivialization of $\E_{b,\B+}$, without explicitly mentioning it. If $b=1$, we choose the natural trivialization. In all cases, the trivialization of $\E_{b,\B+}$ induces a trivialization of $\E_{b,\B+}\otimes_{\B+}\BdR$ (i.e., an identification with $\BdR^n$ where $n$ is the rank) identifying $\E_{b,\B+}$ with the standard lattice $(\B+)^n$ in $\BdR^n$.

In this context, a modification $f$ from $ \E_1 = (\E_1^e,\E_{1, \B+},\iota_1) $ to $\E_2 = (\E_2^e,\E_{2, \B+},\iota_2)$ is an isomorphism $f : \E_1^e \longrightarrow \E_2^e $. It induces an isomorphism
\[
\overline{f}:= \iota^{-1}_2 \circ f \circ \iota_1^{-1} : \E_{1, \B+} \otimes_{\B+}\BdR \longrightarrow \E_{2, \B+} \otimes_{\B+}\BdR.
\]

The type of this modification is the relative position of $(\B+)^n$ with respect to $ \overline{f}\Big( (\B+)^n\Big) $. In particular, if the type is given by the tuple $ (k_1, \dotsc, k_n)$ where $k_1 \ge \dotsc k_n \ge 0$ then $ \overline{f}\Big( (\B+)^n\Big) \subset (\B+)^n $. Hence the couple $(f, \overline{f}_{| (\B+)^n})$ gives us an injective map from $\E_1$ to $\E_2$.

For each $ x \in \Gr_n (C) $ one can construct a modification $ \mathcal{E}_{b,x} $ of $ \mathcal{E}_b $ as follows. Using the trivialization of $\E_{b,\B+}$, we can write the triple corresponding to $\E_b$ as $(\mathcal{E}_{b | X \setminus \infty}, \E^{n,\tri}_{\B+},\iota)$ where $\E^{n,\tri}_{\B+}$ is the trivial bundle of rank $n$ on $ \Spec(\B+) $. Then $\E_{b,x}$ is given as the vector bundle corresponding to the triple $(\mathcal{E}_{b | X \setminus \infty}, \E^{n,\tri}_{\B+},\iota_x )$ where the isomorphism $\iota_x$ is given by the commutative diagram

\begin{center}
	\begin{tikzpicture}[scale = 1]
	\draw (0, 0) node { $\mathcal{E}_b^e \otimes_{B_e} \BdR $ };
	\draw (0, -2) node { $\mathcal{E}^{n, \tri}_{\B+} \otimes_{\B+} \BdR $. };
	\draw (4, 0) node { $\mathcal{E}^{n, \tri}_{\B+} \otimes_{\B+} \BdR $ };	
	\draw [->] (1.25, 0) -- (2.25, 0)node[midway, above]{$ \iota $};
	\draw [->] (0, -0.5) -- (0, -1.5)node[midway, right]{$ \iota_x $};	
	\draw [->] (0.5, -1.5) -- (3.5, -0.5)node[midway, right]{$~x $};	
	\end{tikzpicture}
\end{center}
Here, $B_e = H^0(X \setminus \infty, \mathcal{O}_X)$ and the map $x$ in the diagram is multiplication by a representative of $x$ on $\BdR^n$. The isomorphism class of the triple only depends on the lattice $x(\E^{n,\tri}_{\B+})\subset \BdR^n$ and is in particular independent of the choice of the representative. 

Write $ \Lambda_{x} := x(\mathcal{E}^{n, \tri}_{\B+}) $ and $ \mathcal{E}^{n, \tri}_{\BdR} := \mathcal{E}^{n, \tri}_{\B+} \otimes_{\B+} \BdR $. The type of the induced modification $ f_x : \E_b \longrightarrow \E_{b, x} $ is the relative position of $\Lambda_x$ with respect to $(\B+)^n$. By the Cartan's decomposition
\[
\Gr_{\GL_n}(C) = \coprod_{\mu} \GL_n(\B+) \mu^{-1}(\xi)\GL_n(\B+) / \GL_n(\B+)
\]
where the union is over all conjugacy classes of cocharacters of $\GL_n$. We have the following decomposition of locally spatial diamonds (\cite[19.2]{SW})
\[
\Gr_{\GL_n} = \coprod_{\mu \in X_*(T)_{\dom}} \Gr_{\GL_n, \mu} 
\]
The element $x$ is type $\mu$ if and only if $x \in \GL_n(\B+) \mu^{-1}(\xi)\GL_n(\B+) / \GL_n(\B+) $. Note that $x$ is of type $\mu$ if and only if the modification $f_x$ is of type $(-\mu)_{\dom}$. Recall that the Iwasawa's decomposition induces a decomposition

\[
\Gr_{\GL_n}(C) = \coprod_{\lambda \in X_*(T)} \U(\BdR) \lambda(\xi)\GL_n(\B+) / \GL_n(\B+),
\]
and we can define the semi-infinite orbits $S_{\lambda}$ associated to $\lambda$ as in \cite[Definition 2.7]{Vie}.

Let $\E, \E'$ be vector bundles on $X$ such that $\E' \simeq \E_x$ for some $x \in \Gr_{\GL_n}(C)$. Recall from \cite[lemma 2.4]{CFS} that the isomorphism between $\E_{| X \setminus \infty}$ and $\E'_{| X \setminus \infty}$ induces for every parabolic subgroup $\rP$ a bijection between
\[
\{ 
\text{reduction of $\E$ to $P$} \} \longrightarrow \{ \text{reduction of $\E'$ to $P$} \}.
\] 

By \cite[lemma 2.6]{CFS} and \cite[lemma 3.11]{Vie}, we can compute explicit this bijection in certain cases by using the information from the intersections $ S_{\lambda} \cap \Gr_{\GL_n, -\mu}(C) $. Moreover, in some cases it also allows us to study modifications of vector bundles.
\begin{example} \phantomsection \label{itm : image modif}
    Let $\E$ be the rank $n$ vector bundle $\OO(1/n') \oplus \OO^{n-n'}$. Then there exists a modification
    \[
    f : \E' \longrightarrow \E
    \]
    of type $\mu = (1, 0, \dotsc, 0)$ if and only if $\E' \in S:= \{ \OO^n, \OO(1/n') \oplus \OO^m \oplus \OO(-1/m') \ | \ n' + m + m' = n \}$.

    It is clear that if $\E' \in S$ then there exists a modification of type $\mu$ from $\E'$ to $\E$. We are going to show the inverse inclusion. Indeed, we consider the canonical Harder-Narasimhan filtration $ \E_1 = \OO(1/n') \subset \E $. This filtration corresponds to a reduction of $\E$ to the standard parabolic subgroup $\rP$ of $\GL_n$ whose Levi factor is $ \M := \GL_{n'} \times \GL_{n-n'} $. The induced $\M$-bundle is $ \E_1 \times \E/ \E_1 $. Then the modification $f$ induces a filtration $\E_1' \subset \E'$ whose corresponding $\M$-bundle is given by $ \E_1' \times \E' / \E_1' $. Moreover, the modification $f$ induces two modifications
    \[
    f_1 : \E'_1 \longrightarrow \E_1 \quad \quad f_2 :   \E' / \E_1' \longrightarrow \OO^{n-n'}
    \]
    of type $\mu_1$ and $\mu_2$ respectively. We want to compute $\E_1'$ and $ \E' / \E_1' $. Suppose that $f$ corresponds to the point $x \in \Gr_{\GL_n, (- \mu)_{\dom}}(C)$. We can further suppose that $ x \in S_{\lambda} \cap \Gr_{\GL_n, (- \mu)_{\dom}}(C) $ where $\lambda \leq (- \mu)_{\dom}$ but $\lambda$ is not necessarily dominant. By \cite[lemma 2.6]{CFS} and \cite[lemma 3.11]{Vie}, the modifications $f_1 \times f_2$ corresponds to the point $\pr_{\M}(x)$ inside $\Gr_{\M}(C)$. We can check that $\pr_{\M}(x)$ belongs to $\Gr_{\M, (- \mu_1 \times \mu_2)_{\dom}}(C)$ where $\mu_1 \times \mu_2 \in \{ (1, 0^{(n'-1)}) \times (0^{(n-n')}) ; (0^{(n')}) \times (1, 0^{(n-n'-1)}) \}$. Thus there are 2 cases as follow:

    \textit{Case 1:} $\mu_1 = (1, 0^{(n'-1)})$ and $\mu_2 = (0^{(n-n')})$. We see that $\E' / \E_1' \simeq \OO^{n-n'} $ and $ 
\E'_1 \simeq \OO^{n'} $ and hence $ \E' \simeq \OO^n $.

    \textit{Case 2:} $\mu_1 = (0^{(n')})$ and $\mu_2 = (1, 0^{(n-n'-1)})$. We see that $\E_1' \simeq \OO(1/n')$ and $\E' / \E_1' \in \{\OO^m \oplus \OO(-1/m') \ | \ m + m' = n - n' \} $. Since $H^1(\mathcal{O}^d(\lambda)) = 0 \textrm{ \ if \ } \lambda \ge 0$, we deduce that $ \E' \simeq \OO(1/n') \oplus \OO^m \oplus \OO(-1/m') $. Therefore $\E'$ belongs to the set $S$.
    
\end{example}
\subsection{Stack of vector bundles on the Fargues-Fontaine curve}
Let $\G/\mathbb{Q}_{p}$ be a connected reductive group. We let $\Perf$ denote the category of perfectoid spaces over $\ol{\mathbb{F}}_{p}$. We write $\ast := \Spd(\ol{\mathbb{F}}_{p})$ for the natural base. The key object of study is the moduli pre-stack $\BunG$ sending $S \in \Perf$ to the groupoid of $\G$-bundles on the relative Fargues--Fontaine curve $X_{S}$. The following theorem gives a geometric description of $\BunG$.
\begin{theorem} \cite[Theorem III.0.2]{FS}
    The pre-stack $\Bun_{\G}$ satisfies the following properties:
    \begin{enumerate}
        \item The prestack $\Bun_{\G}$ is a small $v$-stack.
        \item The points $ | \Bun_{\G} | $ are naturally in bijection with Kottwitz’ set $B(\G)$ of $\G$-isocrystals.
        \item The Newton map 
        \[
        \nu :  | \BunG | \longrightarrow B(\G) \longrightarrow  (X_{*}(\T)_{\Q, \mathrm{dom}})^{\Gamma}
        \]
        is semi-continuous and the Kottwitz map
        \[
        \kappa : | \BunG | \longrightarrow B(\G) \longrightarrow \pi_1(\G_{\ol{\Q}_p})_{\Gamma}
        \]
        is locally constant. Moreover, the map $| \BunG | \longrightarrow B(\G)$ is a homeomorphism when $B(\G)$ is equipped
with the order topology \cite{Vie, Han2}.
        \item The semistable locus $\Bun^{ss}_{\G} \subset \BunG $ is open, and given by
        \[
        \Bun^{ss}_{\G} \simeq \coprod_{b \in B(\G)_{\mathrm{basic}}} [\bullet / \underline{\G_b(\Q_p)}].
        \]
        \item For any $b \in B(\G)$, the corresponding subfunctor
        \[
        i^b : \Bun^b_{\G} \subset \BunG
        \]
        is locally closed, and isomorphic to $[\bullet / \widetilde{\G}_b ]$  where $\widetilde{\G}_b$ is a $v$-sheaf of groups such that $\widetilde{\G}_b \longrightarrow \ast $  is representable in locally spatial diamonds with $ \pi_0\widetilde{\G}_b = \G_b(\Q_p) $. The connected component $\widetilde{\G}_b^{0} \subset \widetilde{\G}_b$ of the identity is cohomologically smooth of dimension $\langle 2\rho, \nu_b \rangle$. 
    \end{enumerate}
\end{theorem}

For $\G = \GL_n$, we have $X_*(\T) \simeq \Z^n $ and the target of the map $\nu$ is the set of nonincreasing sequences of rational numbers, which are the slopes of the Newton polygon of the corresponding isocrystal. Moreover, $\pi_1(\G_{\ol{\Q}_p})_{\Gamma} = X_*(T)/(\text{coroot lattice}) $ is naturally isomorphic to $\Z$, and in this case $\kappa(b)$ is the endpoint of the Newton polygon. We can make the root data of $\GL_n$ explicit. The positive roots of $\GL_n$ (corresponding to the Borel subgroup given by the upper triangular matrices) are 
\[
 \Phi^+ = \{ e_k - e_h \ | \ k,h \in \{ 1, 2, \dotsc, n \}, k < h \}, 
\]
and the simple roots are
\[
 \Delta = \{ e_i - e_{i+1} \ | \ i \in \{ 1, 2, \dotsc, n-1 \} \}.
\]
Thus if $\nu_b = ( \underbrace{\lambda_1, \dotsc, \lambda_1}_{m_1}, \dotsc, \underbrace{\lambda_r, \dotsc, \lambda_r}_{m_r})$ then $ \langle 2\rho, \nu_b \rangle = \displaystyle \sum_{i < j} m_im_j(\lambda_i - \lambda_j) $.
\subsection{Overview of the Fargues-Scholze local Langlands correspondence}

 We recall that, for any Artin $v$-stack $X$, Fargues--Scholze define a triangulated category $\Dc_{\blacksquare}(X,\overline{\mathbb{Q}}_{\ell})$ of solid $\overline{\mathbb{Q}}_{\ell}$-sheaves \cite[Section~VII.1]{FS} and isolate a nice full subcategory $\Dlis(X,\overline{\mathbb{Q}}_{\ell}) \subset \Dc_{\blacksquare}(X,\overline{\mathbb{Q}}_{\ell})$ of lisse-\'etale $\overline{\mathbb{Q}}_{\ell}$-sheaves \cite[Section~VII.6.]{FS}. We will be interested in the derived category $\Dlis(\Bun_{\G},\ol{\mathbb{Q}}_{\ell})$. The key point is that objects in this category are manifestly related to smooth admissible representations of $\G(\mathbb{Q}_{p})$. The strata of $\BunG$ admit a natural map $\Bun_{\G}^{b} \ra [\ast/\G_{b}(\mathbb{Q}_{p})]$ to the classifying stack defined by the $\mathbb{Q}_{p}$-points of the $\sigma$-centralizer $\G_{b}$, and, by \cite[Proposition~VII.7.1]{FS}, pullback along this map induces an equivalence
\[ \Dc(\G_{b}(\mathbb{Q}_{p}),\ol{\mathbb{Q}}_{\ell}) \simeq \Dlis([\ast/\G_{b}(\mathbb{Q}_{p})],\ol{\mathbb{Q}}_{\ell}) \xrightarrow{\simeq} \Dlis(\Bun^{b}_{\G},\ol{\mathbb{Q}}_{\ell}), \]
where $\Dc(\G_{b}(\mathbb{Q}_{p}),\ol{\mathbb{Q}}_{\ell})$ is the unbounded derived category of smooth $\ol{\mathbb{Q}}_{\ell}$-representations of  $\G_{b}(\mathbb{Q}_{p})$.

We denote by $\Dlis(\Bun_n, \overline{\Q}_{\ell})^{\omega}$ the stable $\infty$-category of compact objects of $\Dlis(\Bun_{\GL_n},\ol{\mathbb{Q}}_{\ell})$. For any $b \in B(\GL_n)$, there is a local chart
\[
\mathrm{Ch}_b : \mathcal{M}_b \longrightarrow \Bun_n
\]
that is representable in locally spatial diamonds, partially proper and cohomologically smooth \cite[Theo. V.3.7]{FS}. The image of $\mathrm{Ch}_b$ is the open sub-stack $\Bun_n^{\le b}$ consisting of the strata smaller than $b$ with respect to the usual partial order in $B(\GL_n)$ \cite[Theo. 1.1]{Han2} \cite[Theo. 6.7]{Vie}.

\begin{remark}
For $b, b' \in B(\GL_n)$, we have $ \nu_b = (-\nu_{\E_b})_{\rm dom} $ and $ \nu_{b'} = (-\nu_{\E_{b'}})_{\rm dom} $. Thus $b$ is smaller than $b'$ with respect to the usual partial order in $B(\GL_n)$ if and only if $\nu_{\E_b}$ is smaller than $ \nu_{\E_{b'}} $ with respect to the usual partial order in $X_{*}(\T)_{\Q}$ where $\T$ is a maximal split torus of $\GL_n$.   
\end{remark}

Recall that via excision triangles, there is an infinite semi-orthogonal decomposition of $\Dlis(\Bun_n,\ol{\mathbb{Q}}_{\ell})$ into the various $\Dlis(\Bun^b_n, \overline{\Q}_{\ell})$ for $b \in B(\GL_n)$ \cite[Chap. VII]{FS}. 
\begin{lemma} \phantomsection \label{itm : morphism between strata}
    Let $b$ be an element in $B(\GL_n)$ and let $\mathcal{F}, \mathcal{G}$ be in $\Dlis(\Bun_{\GL_n},\ol{\mathbb{Q}}_{\ell})$ such that $\mathcal{F}$ is supported on $\Bun_n^{\le b}$ and the intersection of the support of $\mathcal{G}$ with $\Bun_n^{\le b}$ is empty. Then there is no non-trivial map in $\Dlis(\Bun_{\GL_n},\ol{\mathbb{Q}}_{\ell})$ from $\mathcal{F}$ to $\mathcal{G}$.
\end{lemma}
\begin{proof}
    The local chart $ \mathrm{Ch}_b : \mathcal{M}_b \longrightarrow \Bun_n $ is cohomologically smooth by \cite[Theo. V.3.7]{FS} and moreover the proof of this theorem actually show that the map $ \mathrm{Ch}'_b : \mathcal{M}_b \longrightarrow \Bun^{\le b}_n $ is also cohomologically smooth. Therefore by \cite[Def. IV.1.11]{FS}, the open immersion $ i : \Bun_n^{\le b} \hookrightarrow \Bun_n $ is cohomologically smooth. Denote by $\mathcal{C}$ (resp. $\mathcal{C}'$) the category $\Dlis(\Bun_{\GL_n},\ol{\mathbb{Q}}_{\ell})$ (resp. $\Dlis(\Bun^{\le b}_{\GL_n},\ol{\mathbb{Q}}_{\ell})$) for short. We have
    \begin{align*}
        \Hom_{\mathcal{C}}(\mathcal{F}, \mathcal{G}) & = \Hom_{\mathcal{C}}(i_{!}i^*\mathcal{F}, \mathcal{G}) \\
        & = \Hom_{\mathcal{C}'} (i^*\mathcal{F}, i^{!}\mathcal{G}) \quad (i^! \text{\ is the right adjoint of \ } i_{!}).
    \end{align*}

    However, since $i$ is cohomologically smooth, the map $i^!$ is given by taking derived tensor product of the dualizing object with $i^*$. We see that $i^*\mathcal{G} \simeq 0 $ because the intersection of the support of $\mathcal{G}$ with $\Bun_n^{\le b}$ is empty. Hence $ i^! \mathcal{G} \simeq 0 $ and 
    \[
    \Hom_{\mathcal{C}'} (i^*\mathcal{F}, i^{!}\mathcal{G}) = 0.
    \]
\end{proof}

For $\G$ a reductive group over $\Q_p$, we denote by $\Pi(\G(\Q_p))$ the set of smooth $\overline{\Q}_{\ell}$-irreducible representations of $\G(\Q_p)$. For $\pi \in \Pi(\G_b(\Q_p))$, we get an object $\rho \in \Dlis(\Bun_{\G}^{b},\ol{\mathbb{Q}}_{\ell}) \subset \Dc(\Bun_{\G},\ol{\mathbb{Q}}_{\ell})$ by extension by zero along the locally closed embedding $i_{b}: \Bun_{\G}^{b} \hookrightarrow \Bun_{\G}$, and the Fargues--Scholze parameter comes from acting on this representation by endofunctors of $\Dlis(\Bun_{\G},\ol{\mathbb{Q}}_{\ell})$ called Hecke operators. To introduce this, for a finite index set $I$, we let $\Rep_{\ol{\mathbb{Q}}_{\ell}}(\phantom{}^{L}\G^{I})$ denote the category of algebraic representations of $I$-copies of the Langlands dual group, and we let $\Div^{I}$ be the product of $I$-copies of the diamond $\Div^{1} = \Spd(\Breve{\mathbb{Q}}_{p})/\mathrm{Frob}^{\mathbb{Z}}$. The diamond $\Div^{1}$ parametrizes, for $S \in \Perf$, characteristic $0$ untilts of $S$, which in particular give rise to Cartier divisors in $X_{S}$. We then have the Hecke stack
\[ \begin{tikzcd}
& & \arrow[dl, swap, "h^{\leftarrow}"] \mathrm{Hck} \arrow[dr,"h^{\rightarrow} \times supp"] & & \\
& \Bun_{\G} & & \Bun_{\G} \times \Div^{I}  & 
\end{tikzcd} \]
defined as the functor parametrizing, for $S \in \Perf$ together with a map $S \rightarrow \Div^{I}$ corresponding to characteristic $0$ untilts $S_{i}^{\sharp}$ defining Cartier divisors in $X_{S}$ for $i \in I$, a pair of $G$-torsors $\mathcal{E}_{1}$, $\mathcal{E}_{2}$ together with an isomorphism 
\[ \beta:\mathcal{E}_{1}|_{X_{S} \setminus \bigcup_{i \in I} S_{i}^{\sharp}} \xrightarrow{\simeq} \mathcal{E}_{2}|_{X_{S} \setminus \bigcup_{i \in I} S_{i}^{\sharp}},\]
where $h^{\leftarrow}((\mathcal{E}_{1},\mathcal{E}_{2},i,(S_{i}^{\sharp})_{i \in I})) = \mathcal{E}_{1}$ and $h^{\rightarrow} \times supp((\mathcal{E}_{1},\mathcal{E}_{2},\beta,(S_{i}^{\sharp})_{i \in I})) = (\mathcal{E}_{2},(S_{i}^{\sharp})_{i \in I})$. For each element $W \in \Rep_{\overline{\mathbb{Q}}_{\ell}}(^{L}\G^{I})$, the geometric Satake correspondence of Fargues--Scholze \cite[Chapter~VI]{FS} furnishes a solid $\overline{\mathbb{Q}}_{\ell}$-sheaf $\mathcal{S}_{W}$ on $\mathrm{Hck}$. This allows us to define Hecke operators.
\begin{definition}
For each $W \in \Rep_{\overline{\mathbb{Q}}_{\ell}}(\phantom{}^{L}\G^{I})$, we define the Hecke operator
\[ \T_{W}: \Dlis(\Bun_{\G},\overline{\mathbb{Q}}_{\ell}) \rightarrow \Dc_{\blacksquare}(\Bun_{\G} \times X^{I},\ol{\mathbb{Q}}_{\ell}) \]
\[ A \mapsto R(h^{\rightarrow} \times supp)_{\natural}(h^{\leftarrow *}(A) \otimes^{\mathbb{L}} \mathcal{S}_{W}),\]
where $\mathcal{S}_{W}$ is a solid $\overline{\mathbb{Q}}_{\ell}$-sheaf and the functor $R(h^{\rightarrow} \times supp)_{\natural}$ is the natural push-forward (i.e the left adjoint to the restriction functor in the category of solid $\overline{\mathbb{Q}}_{\ell}$-sheaves \cite[Proposition~VII.3.1]{FS}). 
\end{definition}
It follows by 
\cite[Theorem~I.7.2, Proposition~IX.2.1, Corollary~IX.2.3]{FS} that this induces a functor 
\[ \Dlis(\Bun_{\G},\overline{\mathbb{Q}}_{\ell}) \rightarrow \Dlis(\Bun_{\G},\overline{\mathbb{Q}}_{\ell})^{BW_{\mathbb{Q}_{p}}^{I}} \]
which we will also denote by $\T_{W}$. The Hecke operators are natural in $I$ and $W$ and compatible with exterior tensor products. For a finite set $I$, a representation $W \in \Rep_{\overline{\mathbb{Q}}_{\ell}}(\phantom{}^L\G^{I})$, maps $\alpha: \overline{\mathbb{Q}}_{\ell} \rightarrow \Delta^{*}W$ and $\beta: \Delta^{*}W \rightarrow \overline{\mathbb{Q}}_{\ell}$, and elements $(\gamma_{i})_{i \in I} \in W_{\mathbb{Q}_{p}}^{I}$ for $i \in  I$, one defines the excursion operator on $\Dlis(\Bun_{\G},\overline{\mathbb{Q}}_{\ell})$ to be the natural transformation of the identity functor given by the composition:
\[ \id = \T_{\overline{\mathbb{Q}}_{\ell}} \xrightarrow{\alpha} \T_{\Delta^{*}W} = \T_{W} \xrightarrow{(\gamma_{i})_{i \in I}} \T_{W} = \T_{\Delta^{*}W}  \xrightarrow{\beta} \T_{\overline{\mathbb{Q}}_{\ell}} = \id. \]
In particular, for all such data, we get an endomorphism of a smooth irreducible $\pi \in \Dc(G(\mathbb{Q}_{p}),\ol{\mathbb{Q}}_{\ell}) \simeq \Dlis(\Bun_{\G}^{1},\ol{\mathbb{Q}}_{\ell}) \subset \Dlis(\Bun_{\G},\ol{\mathbb{Q}}_{\ell})$ which, by Schur's lemma will give us a scalar in $\ol{\mathbb{Q}}_{\ell}$. In other words, to the datum $(I,W,(\gamma_{i})_{i \in I},\alpha,\beta)$ we assign a scalar. The natural compatibilities between Hecke operators will give rise to natural relationships between these scalars. These scalars and the relations they satisfy can be used, via Lafforgue's reconstruction theorem \cite[Proposition~11.7]{VL}, to construct a unique continuous semisimple map
\[ \phi^{\mathrm{FS}}_{\pi}: W_{\mathbb{Q}_{p}} \rightarrow \phantom{}^{L}\G(\overline{\mathbb{Q}}_{\ell}), \]
which is the Fargues--Scholze parameter of $\pi$. It is characterized by the property that for all $I,W,\alpha,\beta$ and $(\gamma_{i})_{i \in I} \in W_{\mathbb{Q}_{p}}^{I}$, the corresponding endomorphism of $\pi$ defined above is given by multiplication by the scalar that results from the composite
\[ \overline{\mathbb{Q}}_{\ell} \xrightarrow{\alpha} \Delta^{*}W = W \xrightarrow{(\phi_{\pi}(\gamma_{i}))_{i \in I}} W = \Delta^{*}W \xrightarrow{\beta} \overline{\mathbb{Q}}_{\ell}. \]
Fargues and Scholze show that their correspondence has various good properties which we will invoke throughout.
\begin{theorem}{\cite[Theorem~I.9.6]{FS}}{\label{FSproperties}}
The mapping defined above 
\[ \pi \mapsto \phi^{\mathrm{FS}}_{\pi} \]
enjoys the following properties:
\begin{enumerate}
    \item (Compatibility with Local Class Field Theory) If $\G = T$ is a torus, then $\pi \mapsto \phi_{\pi}$ is the usual local Langlands correspondence 
    \item The correspondence is compatible with character twists, passage to contragredients, and central characters.
    \item (Compatibility with products) Given two irreducible representations $\pi_{1}$ and $\pi_{2}$ of two connected reductive groups $\G_{1}$ and $\G_{2}$ over $\mathbb{Q}_{p}$, respectively, we have
    \[ \pi_{1} \boxtimes \pi_{2} \mapsto \phi^{\mathrm{FS}}_{\pi_{1}} \times \phi^{\mathrm{FS}}_{\pi_{2}}\]
    under the Fargues--Scholze local Langlands correspondence for $\G_{1} \times \G_{2}$. 
    \item (Compatibility with parabolic induction) Given a parabolic subgroup $P \subset \G$ with Levi factor $M$ and a representation $\pi_{M}$ of $M$, then the Weil parameter corresponding to any sub-quotient of $I_{P}^{\G}(\pi_{M})$, the (normalized) parabolic induction, is the composition
    \[ W_{\mathbb{Q}_{p}}\xrightarrow{\phi^{\mathrm{FS}}_{\pi_{M}}} \\  ^{L}M(\overline{\mathbb{Q}}_{\ell}) \rightarrow ^{L}\G(\overline{\mathbb{Q}}_{\ell}) \]
    where the map $\phantom{}^{L}M(\overline{\mathbb{Q}}_{\ell}) \rightarrow \phantom{}^{L}\G(\overline{\mathbb{Q}}_{\ell})$ is the natural embedding. 
    \item (Compatibility with Harris--Taylor/Henniart LLC)
    For $\G = \GL_{n}$ or an inner form of $\G$ the Weil parameter associated to $\pi$ is the (semi-simplified) parameter $\phi_{\pi}$ associated to $\pi$ by Harris--Taylor/Henniart \cite[Theorem 6.6.1]{HKW}. 
\end{enumerate}
\end{theorem}


Let $\G$ be a reductive group and let $b, b'$ be elements in $ B(\G)$. Given a geometric dominant cocharacter $\mu$ of $\G$ with reflex field $E$, we call the quadruple $(\G,b,b',\mu)$ a local shtuka datum. Attached to it, we define the shtuka space
\[ \Sht(\G,b,b',\mu) \longrightarrow \Spd(\Breve{E}), \]
as in \cite{SW} where $\Breve{E} := \Breve{\mathbb{Q}}_{p}E$, to be the space parametrizing modifications 
\[  \mathcal{E}_{b} \longrightarrow \mathcal{E}_{b'} \]
of $\G$-bundles on the Fargues--Fontaine curve $X$ with type bounded by $\mu$.
\begin{remark}
    When $b=1$, we also use the notation $\Sht(\G, b', \mu)$ for $\Sht(\G, 1, b', \mu)$. We note that our definition of $\Sht(\G, b', \mu)$ coincides with $\Sht(\G, b', -\mu)$ in the notation of \cite{SW}, where $ - \mu$ is the dominant inverse of $\mu$. This convention limits the appearance of duals when studying the cohomology of these spaces. 
\end{remark}

This has commuting actions of $\G_{b}(\mathbb{Q}_{p})$ and $\G_{b'}(\mathbb{Q}_{p})$ coming from automorphisms of $\mathcal{E}_{b}$ and $\mathcal{E}_{b'}$, respectively. Moreover the equlities $ b b^{\sigma} (b^{-1})^{\sigma} = b $ and $ (b')^{-1} b' (b')^{\sigma} = (b')^{\sigma} $ induces the isomorphisms $ t_b : \E_{b^{\sigma}} \simeq \E_{b} $ and $ t_{b'} : \E_{b'} \simeq \E_{(b')^{\sigma}}$. Thus we have a Weil descent datum of $\Sht(\G, b, b', \mu)$ defined by
\begin{align*}
    \Sht(\G, b, b', \mu) &\longrightarrow \Sht(\G, b^{\sigma}, (b')^{\sigma}, \mu) \\
    f \quad &\longmapsto t_{b'} \circ f \circ t_b.
\end{align*}

We define the tower 
\[ \Sht(\G,b,b',\mu)_{K} := \Sht(\G,b,b',\mu)/\underline{K} \longrightarrow \Spd(\Breve{E}) \] of locally spatial diamonds \cite[Theorem~23.1.4]{SW}
for varying open compact subgroups $K \subset \G_{b'}(\mathbb{Q}_{p})$. We can define the cohomology $R\Gamma_c(\Sht(\G,b,b',\mu)_{K}, \mathcal{S}_{\mu})$ as in \cite[section 3]{Nao}.

When $b$ is trivial, we denote the moduli space $\Sht(\G,b,b',\mu)$ by $\Sht(\G,b',\mu)$. There is a natural map 
\[ \mathsf{p}: \Sht(\G,b,b',\mu)_{\infty} \longrightarrow \mathrm{Hck} \] 
mapping to the locus of modifications with type bounded by $\mu$. Attached to the geometric cocharacter $\mu$, consider the highest weight representation $V_{\mu} \in \Rep_{\ol{\mathbb{Q}}_{\ell}}(\phantom{}^{L}\G)$. The associated $\ol{\mathbb{Q}}_{\ell}$-sheaf  $\mathcal{S}_{\mu}$ on $\mathrm{Hck}$ (by the geometric Satake isomorphism) considered above will be supported on this locus, and we abusively denote $\mathcal{S}_{\mu}$ for the pullback of this sheaf along $\mathsf{p}$. Since $\mathsf{p}$ factors through the quotient of $\Sht(\G,b,b',\mu)_{\infty}$ by the simultaneous group action of $\G_{b}(\mathbb{Q}_{p}) \times \G_{b'}(\mathbb{Q}_{p})$, this sheaf will be equivariant with respect to these actions. This allows us to define the complex
\[ R\Gamma_{c}(\G,b,b',\mu) := \colim_{K \rightarrow \{1\}} R\Gamma_{c}(\Sht(\G,b,b',\mu)_{K,\mathbb{C}_{p}},\mathcal{S}_{\mu}) \]
which will be a complex of smooth admissible $\G_{b}(\mathbb{Q}_{p}) \times \G_{b'}(\mathbb{Q}_{p}) \times W_{E_{\mu}}$-modules, where $\Sht(\G,b,b',\mu)_{K,\mathbb{C}_{p}}$ is the base change of $\Sht(\G,b,b',\mu)_{K}$ to $\mathbb{C}_{p}$. Remark that if $\mu$ is minuscule then $ \mathcal{S}_{\mu} \simeq \overline{\Q}_{\ell}[d_{\mu}](\tfrac{d_{\mu}}{2})$ where $d_{\mu} = \langle 2 \rho, \mu \rangle$. For $\pi_{b} \in \Pi(\G_{b}(\Q_p))$, this allows us to define the following complexes 
\begin{equation}{\label{eq: RGammaflat}}
    R\Gamma^{\flat}_{c}(\G,b,b',\mu)[\pi_{b}] := R\mathcal{H}om_{\G_{b}(\mathbb{Q}_{p})}(R\Gamma_{c}(\G,b,b',\mu),\pi_{b})
\end{equation}
and
\begin{equation}{\label{eq: RGamma}}
    R\Gamma_{c}(\G,b,b',\mu)[\pi_{b}] := R\Gamma_{c}(\G,b,b',\mu) \otimes^{\mathbb{L}}_{\mathcal{H}(\G_{b}(\Q_p))} \pi_{b}
\end{equation}
where $\mathcal{H}(\G_{b}(\Q_p))$ is the smooth Hecke algebra. Analogously, for $\pi_{b'} \in \Pi(\G_{b'}(\Q_p))$, we can define $R\Gamma_{c}(\G,b,b',\mu)[\pi_{b'}]$ and $R\Gamma^{\flat}_{c}(\G,b,b',\mu)[\pi_{b'}]$. It follows by \cite[Theorem~I.7.2]{FS} and \cite[Pages~324, 325]{FS} that these will be valued in smooth admissible representations of finite length. More precisely, the Hecke operators $ \T_{\mu} $ and the functors $i_{b!}, i^*_{b'}$ preserve compact and ULA objects for all $b, b'$. Then lemma $\ref{shimhecke}$ below implies that the cohomology groups of $R\Gamma_{c}(\G,b,b',\mu)[\pi_{b}]$ are compact and ULA in the category of smooth $\G_{b'}(\Q_p)$-representations with $\overline{\Q}_{\ell}$-coefficients. Now, by pulling through Verdier duality as in \cite[Pages~325]{FS}, we deduce the same conclusion for $R\Gamma^{\flat}_{c}(\G,b,b',\mu)[\pi_{b}]$.

Recall that for each $b \in B(\GL_n)$, we can define a character $\kappa_b : \G_b(\Q_p) \longrightarrow \overline{\Q}^{\times}_{\ell}$ and we have
\[
R\Gamma_c(\widetilde{\G}_b^{0}, \overline{\Q}_{\ell}) = \kappa_b[-2d_b],
\]
as in \cite{GI}, lemma $4.18$ and as in \cite{Ham1}, page $67$, before lemma $10.1$. We now relate the above complexes to Hecke operators on $\Bun_{\G}$. In particular, we have the following result.

\begin{lemma}{\label{shimhecke}}{\cite[Section~IX.3]{FS}}
Given a local shtuka datum $(\G,b,b',\mu)$ as above and $\pi_{b'}$ (resp. $\pi_{b}$) an admissible smooth representation of $\G_{b'}(\mathbb{Q}_{p})$ (resp. of $\G_{b}(\mathbb{Q}_{p})$), we can consider the associated sheaves $\pi_{b} \in \Dc(\G_{b}(\mathbb{Q}_{p}), \overline{\Q}_{\ell}) \simeq \Dlis(\Bun_{\G}^{b})$ and $\pi_{b'} \in \Dlis(\Bun_{\G}^{b'}) \simeq \Dc(\G_{b'}(\mathbb{Q}_{p}), \overline{\Q}_{\ell})$ on the HN-strata $i_{b}: \Bun_{\G}^{b} \hookrightarrow \Bun_{\G}$ and $i_{b'}: \Bun_{\G}^{b'} \hookrightarrow \Bun_{\G}$. There then exists an isomorphism
\[ 
R\Gamma_{c}(\G,b,b',\mu)[\pi_{b} \otimes \kappa_b^{-1}][2d_{b}] \simeq  i_{b'}^{*}\T_{\mu}i_{b!}(\pi_{b})
\]
\[
R\Gamma^{\flat}_{c}(\G,b,b',\mu)[\pi_{b}][-2d_{b'}] \simeq  \big( i_{b'}^{!}\T_{\mu}Ri_{b*}(\pi_{b}) \big) \otimes\kappa_{b'}^{-1},
\]
of complexes of $\G_{b'}(\mathbb{Q}_{p}) \times W_{E_{\mu}}$-modules and an isomorphism 
\[ 
R\Gamma_{c}(\G,b,b',\mu)[\pi_{b'}\otimes \kappa_{b'}^{-1}][2d_{b'}] \simeq i_{b}^{*}\T_{-\mu}i_{b'!}(\pi_{b'}), 
\]
\[
R\Gamma^{\flat}_{c}(\G,b,b',\mu)[\pi_{b'}][-2d_{b}] \simeq \big( i_{b}^{!}\T_{-\mu}Ri_{b'*}(\pi_{b'}) \big) \otimes \kappa_{b}^{-1}, 
\]
of complexes of $\G_{b}(\mathbb{Q}_{p}) \times W_{E_{\mu}}$-modules where $-\mu$ is the dominant cocharacter conjugate to the inverse of $\mu$ and where $d_b = \langle 2\rho, \nu_b \rangle$ and $d_{b'} = \langle 2\rho, \nu_{b'} \rangle$. 
\end{lemma}
\begin{proof}
In fact, we can argue as in \cite{Fa}, pages $67$-$68$ or as in \cite[section IX.3]{FS}, pages $324$-$325$ to deduce that 
\[
R\Gamma_{c}(\G,b,b',\mu)[\pi_{b} \otimes \kappa_b^{-1} ][2d_{b}] \simeq  i_{b'}^{*}\T_{\mu}i_{b!}(\pi_{b}).
\]
 
Note that the Hodge-Tate period map $\pi_{\rm HT}$ in \cite{Fa} is a $\widetilde{\G}_b$-torsor over the Newton stratum corresponding to $[b']$. Thus we need to take into account the cohomology of the locally spatial diamond $ R\Gamma_c( \widetilde{\G}_b^{0}, \overline{\Q}_{\ell} ) \simeq \kappa_b[-2d_b] $, hence the shift $[2d_{b}]$ and the character $\kappa_b^{-1}$ appear in the formulas. 




\end{proof}

\section{The stack of $L$-parameters} \phantomsection \label{itm : stack of $L$-parameters}
 We fix a prime $\ell \neq p$. Let $\G$ be a reductive group over $\Q_p$ and then we get the dual group $\widehat{G}/ \Z_{\ell}$, which we endow with its usual algebraic action of the Weil group $W_{\Q_p}$. We will be mainly interested in the case $\G = \GL_n$ in this section.
\subsection{The geometry of the stack of $L$-parameters}

In this paragraph we recall the definition of the stack of $L$-parameters and recollect some of its geometric properties such as characterization of some smooth points and connected components.

\subsubsection{Smoothness}

Let $A$ be any $\Z_{\ell}$-algebra, regarded as a condensed $\Z_{\ell}$-algebra. 
\begin{definition} \phantomsection \label{itm : new definition of L-parameters}
    An $L$-parameter for $\G$ with coefficients in $A$ is a section
    \[
    \phi : W_{\Q_p} \longrightarrow \widehat{G}(A) \rtimes W_{\Q_p}
    \]
    of the natural map of condensed groups $ \widehat{G}(A) \rtimes W_{\Q_p} \longrightarrow W_{\Q_p} $. Equivalently, an $L$-parameter for $\G$ with coefficients in $A$ is a condensed $1$-cocycle $\phi : W_{\Q_p} \longrightarrow \widehat{\G}(A)$ for the given $W_{\Q_p}$ action. More concretely, if $A$ is endowed with the topology given by writing $A = \textrm{colim}_{A' \subset A} A'$ where $A'$ is finitely generated $\Z_{\ell}$-algebra with its $\ell$-adic topology then an $L$-parameter with values in $A$ is a $1$-cocycle $ \phi : W_{\Q_p} \longrightarrow \widehat{\G}(A)$ such that if $\widehat{\G} \hookrightarrow \GL_n $ the associated map $ \phi : W_{\Q_p} \longrightarrow \GL_n(A) $ is continuous.
\end{definition}
\begin{remark}
    In constrast to definition \ref{itm : classical L parameter}, we do not require that $\phi$ sends semi-simple elements to semi-simple elements in the above definition.
\end{remark}
With this definition of $L$-parameter over any $\Z_{\ell}$-algebra $A$, we can define a moduli space, denoted by $ Z^1(W_{\Q_p}, \widehat{\G})_{\Z_{\ell}} $, over $\Z_{\ell}$, whose $A$-points are the continuous $1$-cocycles $ \phi : W_{\Q_p} \longrightarrow \widehat{\G}(A) $ with respect to the natural action of $W_{\Q_p}$ on $\widehat{\G}(A)$. This defines the scheme considered in, \cite{FS, DH, Zhu1}. 

Any condensed $1$-cocycle $ \phi : W_{\Q_p} \longrightarrow \widehat{\G}(A)$ is trivial on an open subgroup of the wild
inertia subgroup $\rP_{\Q_p}$; note also that $\rP_{\Q_p}$ acts on $\widehat{\G}$ through a finite quotient. Thus by using discretization process \cite{Helm, DH}, Fargues and Scholze show \cite[ Theorem I.8.1 ]{FS} that $ Z^1(W_{\Q_p}, \widehat{\G})_{\Z_{\ell}} $ can be written as a union of open and closed affine subschemes $ Z^1(W_{\Q_p}/\rP, \widehat{\G})_{\Z_{\ell}} $ as $\rP$ runs through subgroups of the wild inertia $\rP_{\Q_p}$ of $W_{\Q_p}$, where $ Z^1(W_{\Q_p}/\rP, \widehat{\G})_{\Z_{\ell}} $ parametrizes condensed $1$-cocycles that are trivial on $\rP$. Each $ Z^1(W_{\Q_p}/\rP, \widehat{\G})_{\Z_{\ell}} $ is a flat local complete intersection over $\Z_{\ell}$ of dimension $\dim(\G)$. 

This allows us to consider the Artin stack quotient $[Z^1(W_{\Q_p}, \widehat{\G})/\widehat{\G} ]_{\Z_{\ell}}$, where $\widehat{\G}$ acts via conjugation. We then consider the base change to $\overline{\Q}_{\ell}$, denoted by $[Z^1(W_{\Q_p}, \widehat{\G})/\widehat{\G} ]$ and referred to it as the stack of Langlands parameters, as well as the category $\Perf([Z^1(W_{\Q_p}, \widehat{\G})/\widehat{\G}])$ of perfect complexes of coherent sheaves on this space.

To study the geometric properties of the stack of $L$-parameters, we also need to consider its coarse moduli quotient in the category of scheme
\[
Z^1(W_{\Q_p}, \widehat{\G})//\widehat{\G}
\]
of $Z^1(W_{\Q_p}, \widehat{\G})$ by the action of $\widehat{\G}$ via conjugation. Concretely, for every connected component $\Spec R \subset Z^1(W_{\Q_p}, \widehat{\G}) $, we get a corresponding connected component $\Spec R^{\widehat{\G}} \subset Z^1(W_{\Q_p}, \widehat{\G})//\widehat{\G} $. We recall the following definition of semi-simple $L$-parameter.
\begin{definition}
    Let $K$ be an algebraically closed field over $\Z_{\ell}$. An $L$-parameter $\phi : W_{\Q_p} \longrightarrow \widehat{\G}(K) \rtimes W_{\Q_p} $ is semi-simple if whenever the image of $\phi$ is contained in a parabolic subgroup of $\widehat{\G} \rtimes W_{\Q_p}$ then it is contained in the Levi subgroup of this parabolic subgroup. 
\end{definition}

Given an $L$-parameter $\phi : W_{\Q_p} \longrightarrow \widehat{\G} \rtimes W_{\Q_p} $, by \cite[Proposition VIII.3.2]{FS} and \cite[Proposition 4.13]{DH} the $\widehat{\G}$-orbit of $\phi$ defines a closed $\overline{\Q}_{\ell}$-point of $Z^1(W_{\Q_p}, \widehat{\G})//\widehat{\G}$ if and only if $\phi$ is a semi-simple parameter. The
natural map 
\[
\theta : [Z^1(W_{\Q_p}, \widehat{\G})/\widehat{\G}] \longrightarrow Z^1(W_{\Q_p}, \widehat{\G})//\widehat{\G}
\]
evaluated on a $\overline{\Q}_{\ell}$-point in the stack quotient defined by an $L$-parameter $\phi$ defines a $\overline{\Q}_{\ell}$-point in the coarse moduli space given by its semi-simplification $\phi^{ss}$.

From now on we suppose $\G = \GL_n$. In particular the action of $W_{\Q_p}$ on $\widehat{\G}$ is trivial. Let $\phi : W_{\Q_p} \longrightarrow \GL_n(\overline{\Q}_{\ell})$ be a semi-simple and Frobenius semi-simple $L$-parameter. Thus it defines a geometric point $x$ in $Z^1(W_{\Q_p}, \GL_n)//\GL_n$ and a point $\overline{x}$ in the stack $[Z^1(W_{\Q_p}, \GL_n)/\GL_n]$. The geometric properties of $\theta^{-1}(x)$ will be essentially important since they reveal interesting information on the perfect complexes over $[Z^1(W_{\Q_p}, \GL_n)/\GL_n]$. Since $\phi$ is semi-simple, there is a decomposition $\phi = \phi_1 \oplus \dotsc \oplus \phi_r$ of $\phi$ into irreducible representations. 
\begin{proposition} \label{itm : inverse has 1 points}
    Let $|\cdot|_p$ be the norm character $W_{\Q_p} \longrightarrow W_{\Q_p}^{\text{ab}} \simeq \Q_p^{\times} \longrightarrow \C \simeq \overline{\Q}_{\ell} $. Suppose that there do not exist $ 1 \le i \neq j \le r$ such that $ \phi_i \simeq \phi_j $ or $ \phi_i \simeq \phi_j \otimes |\cdot|_p $ then $\theta^{-1}(x)$ has only $1$ geometric point $\overline{x}$. In particular, $\overline{x}$ is a closed point in $[Z^1(W_{\Q_p}, \GL_n)/\GL_n]$.
\end{proposition}
\begin{proof}

In fact, the semisimple $L$-parameter $\phi$ is generous following the terminology in \cite{Hansen}. Therefore the conclusion of the proposition follows from \cite[remark 6.5]{HL} and \cite[Example 2.1.6]{Hansen}.

\end{proof}

The above proposition implies that $ \displaystyle \theta^{-1}(x) \simeq [\bullet/S_{\phi}] \simeq [\bullet/\prod^r_{i = 1} \Gm] $ and the immersion $ [\bullet/S_{\phi}] \hookrightarrow [Z^1(W_{\Q_p}, \GL_n)/\GL_n] $ is closed. The following proposition tells us that the immersion is in fact closed and regular.
\begin{proposition} \label{itm : smooth locus}
    Let $|\cdot|_p$ be the norm character $W_{\Q_p} \longrightarrow W_{\Q_p}^{\text{ab}} \simeq \Q_p^{\times} \longrightarrow \C \simeq \overline{\Q}_{\ell} $. Suppose that there do not exist $ 1 \le i \neq j \le r$ such that $ \phi_i \simeq \phi_j $ or $ \phi_i \simeq \phi_j \otimes |\cdot|_p $ then $\phi$ defines a closed and smooth geometric point of $[Z^1(W_{\Q_p}, \GL_n)/\GL_n]$.
\end{proposition}
\begin{proof}
    Denote by $x$ the point determined by $\phi$ and denote by $\widehat{g}$ the Lie algebra of $\widehat{\G}$. Since $\phi$ is semi-simple, $x$ is a closed point. Then, using the notation in \cite[section VIII.2]{FS}, we need to show that $ x^*\Sing_{[Z^1(W_{\Q_p}, \GL_n)/\GL_n]_{\Z_{\ell}}} = H^0(W_{\Q_p}, \widehat{g}^{*} \otimes_{\Z_{\ell}} \overline{\Q}_{\ell}(1) ) $ is trivial where the action of $W_{\Q_p}$ on $\widehat{g}^{*} \otimes_{\Z_{\ell}} \overline{\Q}_{\ell}(1)$ is the (adjoint) action defining the $L$-group twisted by $\phi$ and the cyclotomic character. Let $v \in x^*\Sing_{[Z^1(W_{\Q_p}, \GL_n)/\GL_n]_{\Z_{\ell}}} \subset \widehat{g} \simeq \widehat{g}^{*}$. By \cite[Proposition VIII.2.11]{FS}, we know that $v$ belongs to the nilpotent cone. However, our group is split then we can deduce that $\phi(g)v\phi(g)^{-1} = v$ for all $g$ in the inertia subgroup $\I$ of $W_{\Q_p}$ and $\phi(\sigma)v\phi(\sigma)^{-1} = p\cdot v$. Since there do not exist $ 1 \le i \neq j \le r$ such that $ \phi_i \simeq \phi_j \otimes |\cdot|_p $, the latter commutative condition implies that $v$ is a diagonal matrix. Therefore $v$ is trivial. 
\end{proof}

\subsubsection{Connected components}

 We recall the description of the connected components of $[Z^1(W_{\Q_p}, \GL_n)/\GL_n]$, following \cite{DH}. We fix a lift $\sigma$ of the arithmetic Frobenius to $W_{\Q_p}/\rP_{\Q_p}$ and $\tau$ a topological generator of $\I_{\Q_p}/\rP_{\Q_p}$. Let $(W_{\Q_p}/\rP_{\Q_p})^0$ be the subgroup of $W_{\Q_p}/\rP_{\Q_p}$ generated by $\tau , \sigma$ and $\rP_{\Q_p}$, regarded as discrete group and let $W^0_{\Q_p}$ be the inverse image of $(W_{\Q_p}/\rP_{\Q_p})^0$ in $W_{\Q_p}$. Similarly we can define $(\I_{\Q_p}/\rP_{\Q_p})^0$ and $\I^0_{\Q_p}$. The functor that sends $R$ to $Z^1(W^0_{\Q_p}, \GL_n(R))$ is representable by an affine scheme denoted by $Z^1(W^0_{\Q_p}, \GL_n)$. We study the connected components of the stack $[Z^1(W_{\Q_p}, \GL_n)/\GL_n]$ by a discretization process. Thus we choose a decreasing sequence $(\rP^e)_{e \in \N}$ of open normal subgroups of the wild inertia $\rP_{\Q_p}$ whose intersection is $\{ 1 \}$. We know that $ Z^1(W_{\Q_p}, \GL_n)_{\Z_{\ell}} $ can be written as a union of open and closed affine sub-schemes $ Z^1(W^0_{\Q_p}/\rP^e, \GL_n) $ for $e \in \N$. 

For each $e \in \N$, by a reduction to tame parameters argument \cite[Theorem 3.1, equation 4.2]{DH}, there exists a finite set $\Phi_e^{\mathrm{adm}}$ consisting of $1$-cocycle $ \varphi : \rP_{\Q_p} / \rP^e \longrightarrow \GL_n(\overline{\Q}_{\ell})$ such that we have a decomposition
\[
[Z^1(W^0_{\Q_p}/\rP^e, \GL_n)/\GL_n] = \coprod_{\varphi \in \Phi_e^{\mathrm{adm}}} [Z^1(W^0_{\Q_p}, \GL_n)_{\varphi}/C(\varphi)]
\]
where $C(\varphi)$ is the centralizer of $\varphi$ and $Z^1(W^0_{\Q_p}, \GL_n)_{\varphi}$ is the (non empty) sub-scheme of $Z^1(W^0_{\Q_p}, \GL_n)$ parametrizing the $1$-cocycles that extend $\varphi$. There is also a similar decomposition in the situation where we consider the categorical quotient.


For each $e \in \N$, let $\I^e$ be the open sub-group of the inertia $\I_{\Q_p}$ as in \cite[Corollary 4.16]{DH}. In particular, every semi-simple map $ f : W^0/\rP^e \longrightarrow \GL_n(L) $ is trivial on $ \I^e / \rP^e $ and then extends canonically to $W_{\Q_p}/\rP^e$, where $L$ is any algebraically closed field of characteristic $0$ or $ \ell \neq p $. With the group $\I^e$, we can consider the $\GL_n$-stable closed sub-scheme $Z^1(W^0_{\Q_p}/\I^e, \GL_n)$ of $ Z^1(W^0_{\Q_p}/ \rP^e, \GL_n) $  parametrizing the maps that are trivial on $\I^e$.

We only consider the stack of $L$-parameters defined over $\overline{\Q}_{\ell}$ for some prime $\ell \neq p$. 
By \cite[Proposition 6.4]{DH}, there exists a finite set $\widetilde{\Phi}_e^{\mathrm{adm}}$ consisting of $1$-cocycle $ \widetilde{\varphi} : \I_{\Q_p} / \I^e \longrightarrow \GL_n(\overline{\Q}_{\ell})$ such that we have a decomposition similar to the above tamely discretization process
\[
[Z^1(W^0_{\Q_p}/\I^e, \GL_n)/\GL_n] = \coprod_{\widetilde{\varphi} \in \widetilde{\Phi}_e^{\mathrm{adm}}} [Z^1(W^0_{\Q_p}, \GL_n)_{\widetilde{\varphi}}/C(\widetilde{\varphi})]
\]
where $C(\widetilde{\varphi})$ is the centralizer of $\widetilde{\varphi}$ and $Z^1(W^0_{\Q_p}, \GL_n)_{\widetilde{\varphi}}$ is the (non empty) sub-scheme of $Z^1(W^0_{\Q_p}, \GL_n)$ parametrizing the $1$-cocycles $\phi$ such that $ \phi_{| \I_{\Q_p}} = \widetilde{\varphi} $.

Our goal is to describe the connected component of $\phi$ in $[Z^1(W_{\Q_p}, \GL_n)/\GL_n]$ where $\phi_{| \I_{\Q_p}}$ is multiplicity free (as representation of $\I_{\Q_p}$) and $\phi(\sigma)$ is semi-simple. Hence it is enough to consider the case $ \widetilde{\varphi} $ is multiplicity free (as representation of $\I_{\Q_p}$) and we can write $\widetilde{\varphi} = \widetilde{\varphi}_1 \oplus \dotsc \oplus \widetilde{\varphi}_k $ for some $k$ and the centralizer group $C(\widetilde{\varphi})$ is isomorphic to $\mathbb{G}^k_m$.

\begin{proposition} \phantomsection \label{itm : simple connected components}
    Suppose that $\widetilde{\varphi}$ is multiplicity free then $[Z^1(W^0_{\Q_p}, \GL_n)_{\widetilde{\varphi}}/C(\widetilde{\varphi})]$ is a connected component of $[Z^1(W_{\Q_p}, \GL_n)/\GL_n]$ and we have an isomorphism 
\[
[Z^1(W^0_{\Q_p}, \GL_n)_{\widetilde{\varphi}}/C(\widetilde{\varphi})] \simeq [\mathbb{G}^r_m / \mathbb{G}^r_m],
\]
the quotient stack of $\mathbb{G}^r_m$ by the trivial action of $\mathbb{G}^r_m$ for some natural number $r$.
\end{proposition}

\begin{proof}
Fix an element $\phi \in Z^1(W^0_{\Q_p}, \GL_n)_{\widetilde{\varphi}}$ then the conjugation action of $W^0_{\Q_p}$ on $\GL_n(\overline{\Q}_{\ell})$ by $\phi$ stabilizes $C(\widetilde{\varphi})$ and the restricted action on this group factors over $W^0_{\Q_p} / \I^0_{\Q_p}$. Denoting by $\mathrm{Ad}_{\phi}$ this action and denoting by $Z^1_{\mathrm{Ad}_{\phi}}(W^0_{\Q_p}/\I^0_{\Q_p}, C(\widetilde{\varphi}) )_{\overline{\Q}_{\ell}}$ the affine scheme of $1$-cocycles $f : W^0_{\Q_p}/\I^0_{\Q_p} \longrightarrow C(\widetilde{\varphi}) $. Note that for any $w \in W_{\Q_p}^0$ and any element $\phi' \in Z^1(W^0_{\Q_p}, \GL_n(\overline{\Q}_{\ell}))_{\widetilde{\varphi}}$ we can write $ \phi'(w) = \eta(w) \phi(w) $ and by Schur's lemma, we can deduce that $\eta(w)$ belongs to $C(\widetilde{\varphi})(\overline{\Q}_{\ell})$. Hence, as in \cite[pages 19, 20]{DH}, the map $ \eta \longmapsto \eta \cdot \phi $ sets up an isomorphism of $\overline{\Q}_{\ell}$-schemes 
\[
Z^1_{\mathrm{Ad}_{\phi}}(W^0_{\Q_p}/\I^0_{\Q_p}, C(\widetilde{\varphi}) )_{\overline{\Q}_{\ell}} \xrightarrow{ \ \simeq \ }   Z^1(W^0_{\Q_p}, \GL_n)_{\widetilde{\varphi}}.
\]

The scheme $ Z^1_{\mathrm{Ad}_{\phi}}(W^0_{\Q_p}/\I^0_{\Q_p}, C(\widetilde{\varphi}) )_{\overline{\Q}_{\ell}} $ is in fact isomorphic to $C(\widetilde{\varphi})$. We analyze the action of $C(\widetilde{\varphi}) \simeq \mathbb{G}^k_{m}$ on $Z^1(W^0_{\Q_p}, \GL_n)_{\widetilde{\varphi}} \simeq \mathbb{G}^k_m $.    

Suppose first that $\phi$ is irreducible and $\phi_{| I_{\Q_p}} \simeq \widetilde{\varphi} \simeq \widetilde{\varphi}_1 \oplus \dotsc \oplus \widetilde{\varphi}_k$. If $V$ is a vector space together with a morphism $ f : \I_{\Q_p} \longrightarrow \Aut(V) $ then we denote by $V^{\sigma}$ the representation of $\I_{\Q_p}$ whose underlying vector space is $V$ and where $ i \in \I_{\Q_p}$ acts by $ f (\sigma i \sigma^{-1} ) $. Denote by $V_i$ the underlying vector space of $\widetilde{\varphi}_i$ and $ \displaystyle V = \bigoplus_{i = 1}^k V_i$ . Since the action of the Frobenius gives an isomorphism between the $\I_{\Q_p}$-representations $ V $ and $ V^{\sigma} $, we deduce that up to permutation, there are isomorphisms of $\I_{\Q_p}$-representations $ V_1 \simeq V_2^{\sigma} \simeq V_3^{\sigma^2} \simeq \dotsc \simeq V_{k}^{\sigma^{k-1}} $ and $ V_1 \simeq V^{\sigma^k}_1 $. Thus, we can suppose that in the basis $ V_1 \oplus V_2 \oplus \dotsc \oplus V_k $, the map $\phi(\sigma)$ is given by 

\[
\begin{pmatrix}
    0 & \cdots & 0 & 0 & A \\
    1 & 0 & \cdots & 0 & 0 \\
    0 & 1 & 0 & \cdots & 0  \\
    0 & 0 & \ddots & 0 & 0 \\
    0 & \cdots & 0 & 1 & 0
\end{pmatrix}
\]
for some invertible square matrix $A$. Let $h$ be the smallest integer such that $ 1 < h \leq k $ and $V_1 \simeq V^{\sigma^h}_1 $ then $h$ is a divisor of $k$. Note that in the case we consider, $\phi(\sigma)$ is semisimple, therefore $A$ is semisimple and there exists $B$ such that $B^{k/h} = A$. In particular, we see that the sub-vector space
\[
\Big\{ \sum_{j = 0}^{\tfrac{k}{h}-1} \phi(\sigma^{jh})(B^{-j}v) \ \ | \ \ v \in \bigoplus_{i = 1}^h V_i \Big\}
\]
is in fact a sub $W_{\Q_p}$-representation. Since $\phi$ is irreducible, we deduce that $h = k$.

With these explicit computations, we see that the action of $ t = (t_1, \dotsc, t_k) \in C(\widetilde{\varphi}) \simeq \mathbb{G}^r_m$ acting on $Z^1(W^0_{\Q_p}, \GL_n)_{\widetilde{\varphi}} \simeq \mathbb{G}^k_m $ is given by $t \cdot (x_1, \dotsc, x_k) = ( \tfrac{t_1}{t_2}x_1, \dotsc, \tfrac{t_k}{t_1}x_k )$. By using the map $ \mathbb{G}^r_m \longrightarrow \mathbb{G}_m $ that sends $ (t_1, \dotsc, t_k) $ to $ \displaystyle \prod_{i=1}^k t_i $, we see that
\[
[Z^1(W^0_{\Q_p}, \GL_n)_{\widetilde{\varphi}}/C(\widetilde{\varphi})] \simeq [\Gm / \Gm]
\]
where $\Gm$ acts trivially on $\Gm$. Moreover, for $t \in \overline{\Q}^{\times}_{\ell}$ if we denote by $\chi_t$ the unramified character of $W_{\Q_p}$ such that $\chi_t(\sigma) = t$ then the conjugacy class of the $L$-parameter $\phi \otimes \chi_t$ corresponds to the closed immersion $[\bullet/ \Gm] \hookrightarrow [\Gm / \Gm] $ whose image is the closed point $t^k$ in $\Gm$. We remark that $k$ is the torsion number of the supercuspidal representation $\pi_{\phi}$ of $\GL_n(\Q_p)$ (\cite[Definition 4.2]{Secherre}).

In general, if $\phi = \phi_1 \oplus \dotsc \oplus \phi_r$ where the $\phi_i$'s are irreducible and $\phi_{| \I_{\Q_p}}$ is multiplicity free then the same arguments implies that 
\[
[Z^1(W^0_{\Q_p}, \GL_n)_{\widetilde{\varphi}}/C(\widetilde{\varphi})] \simeq [\mathbb{G}^r_m / \mathbb{G}^r_m]
\]
where $\mathbb{G}^r_m$ acts trivially on $\mathbb{G}^r_m$ and the closed point $t = (t^{k_1}_1, \dotsc, t^{k_r}_r) \in \mathbb{G}^r_m$ corresponds to the conjugacy class of the $L$-parameter $ \phi_1 \otimes \chi_{t_1} \oplus \dotsc \oplus \phi_r \otimes \chi_{t_r} $. \\

Now we show that $[Z^1(W^0_{\Q_p}, \GL_n)_{\widetilde{\varphi}}/C(\widetilde{\varphi})]$ is a connected component of $[Z^1(W_{\Q_p}, \GL_n)/\GL_n]$. Indeed, by \cite{DH} we see that $[Z^1(W^0_{\Q_p}, \GL_n)_{\widetilde{\varphi}}/C(\widetilde{\varphi})]$ is contained in the connected component $[\phi]$ of $[Z^1(W_{\Q_p}, \GL_n)/\GL_n]$ containing $\phi$. We can describe all the closed points of $[\phi]$ by \cite[Theorem 6.8, Corollary 4.22]{DH} and lemmas \ref{itm : inverse has 1 points}, \ref{itm : multiplicity free of Weil action}. Thus we deduce that  $[\phi]$ and $[Z^1(W^0_{\Q_p}, \GL_n)_{\widetilde{\varphi}}/C(\widetilde{\varphi})]$ have the same sets of $\overline{\Q}_{\ell}$-points. Since the schemes $ Z^1(W^0_{\Q_p}, \GL_n)_{\widetilde{\varphi}} $ and $ Z^1(W^0_{\Q_p}/ \rP^e, \GL_n) $ are reduced \cite[Theorem 4.1]{DH}, we see that $[Z^1(W^0_{\Q_p}, \GL_n)_{\widetilde{\varphi}}/C(\widetilde{\varphi})]$ is equal to $[\phi]$.
\end{proof}

 Finally we have the following lemma characterizing the conjugacy classes of $L$-parameters $\phi$ such that $\phi_{| \I_{\Q_p}}$ is multiplicity free. 

\begin{lemma} \phantomsection \label{itm : multiplicity free of Weil action}
    Let $\phi$ be an $L$-parameter of $\GL_n$ and $\phi(\sigma)$ is semi-simple. Then $\phi_{| \I_{\Q_p}}$ is multiplicity free if and only if $\phi = \phi_1 \oplus \dotsc \oplus \phi_r$ where the $\phi_i$'s are irreducible and if $i \neq j$ then there does not exists unramified character $\chi$ such that $\phi_i \otimes \chi \simeq \phi_j$.
\end{lemma}
\begin{proof}
It is clear that if $\phi_{| \I_{\Q_p}}$ is multiplicity free then $\phi = \phi_1 \oplus \dotsc \oplus \phi_r$ where the $\phi_i$'s are irreducible and if $\dim \phi_i = \dim \phi_j $ for $i \neq j$ then there does not exists unramified character $\chi$ such that $\phi_i \otimes \chi \simeq \phi_j$.

We prove the inverse direction. By considering the action of the Frobenius $\sigma$ as before we see that for $1 \le i \le r$, $ \phi_{i | \I_{\Q_p}} \simeq V_i \oplus V_i^{\sigma} \oplus \dotsc \oplus V_i^{\sigma^{k_i-1}} $ where $k_i$ is the smallest integer such that $ V_i^{\sigma^{k_i}} \simeq V_i $. Therefore $\phi_{i | \I_{\Q_p}}$ is multiplicity free. Moreover if $ V_i \simeq V_j^{\sigma^h} $ then $V_i \oplus V_i^{\sigma} \oplus \dotsc \oplus V_i^{\sigma^{k_i-1}} \simeq V^{\sigma^h}_j \oplus V_j^{\sigma^{h+1}} \oplus \dotsc \oplus V_j^{\sigma^{h + k_j-1}}$. Since $\phi_i$ and $\phi_j$ are irreducible, we deduce that $ 
 k_i = k_j $. By choosing suitable sequences $(t^m_i)^{k_i}_{m=1}$ and $(t^m_j)^{k_j}_{m=1}$ and by multiplying a basis of $V^{\sigma^{m}}_i$ (respectively a basis of $V^{\sigma^{m}}_j$) by a factor $t^m_i$ (respectively $t^m_j$), we see that there exists $t \in \overline{\Q}^{\times}_{\ell}$ such that the matrices of $\phi_i(\sigma)$ and $ t \phi_j(\sigma)$ given in the corresponding bases are equal. Thus $\phi_i$ is an unramified twist of $\phi_j$ and then we have $i = j$. Therefore $\phi_{| \I_{\Q_p}}$ is multiplicity free.
\end{proof}


\subsection{Perfect complexes on the stack of $L$-parameters}
In this subsection we talk about the spectral action defined in \cite{FS}. Denote by $\Perf([Z^1(W_{\Q_p}, \GL_n)/\GL_n])$ the derived category of perfect complexes on $[Z^1(W_{\Q_p}, \GL_n)/\GL_n]$. We write $\Perf([Z^1(W_{\Q_p}, \GL_n)/\GL_n])^{BW_{\mathbb{Q}_{p}}^{\I}}$ for the derived category of objects with a continuous $W_{\mathbb{Q}_{p}}^{\I}$ action for a finite index set $\I$, and $\Dlis(\Bunn,\ol{\mathbb{Q}}_{\ell})^{\omega}$ for the triangulated sub-category of compact objects in $\Dlis(\Bunn,\ol{\mathbb{Q}}_{\ell})$. By \cite[Corollary~X.1.3]{FS}, there exists a $\ol{\mathbb{Q}}_{\ell}$-linear action 
\[ \Perf([Z^1(W_{\Q_p}, \GL_n)/\GL_n])^{BW_{\mathbb{Q}_{p}}^{I}} \ra \mathrm{End}(\Dlis(\Bunn,\ol{\mathbb{Q}}_{\ell})^{\omega})^{BW_{\mathbb{Q}_{p}}^{\I}} \]
\[ C \mapsto \{A \mapsto C \star A\}\]
which, extending by colimits, gives rise to an action 
\[ \IndPerf([Z^1(W_{\Q_p}, \GL_n)/\GL_n])^{BW_{\mathbb{Q}_{p}}^{I}} \ra \mathrm{End}(\Dlis(\Bunn,\ol{\mathbb{Q}}_{\ell})^{\omega})^{BW_{\mathbb{Q}_{p}}^{\I}} \]
where $\IndPerf([Z^1(W_{\Q_p}, \GL_n)/\GL_n])$ is the triangulated category of Ind-Perfect complexes, and this action is uniquely characterized by some complicated properties. For our purposes, we will need the following:
\begin{enumerate}
    \item For $V = \boxtimes_{i \in I} V_{i}\in \Rep_{\ol{\mathbb{Q}}_{\ell}}(\phantom{}^{L}\GL_n^{I})$, there is an attached vector bundle $C_{V} \in \Perf([Z^1(W_{\Q_p}, \GL_n)/\GL_n])^{BW_{\mathbb{Q}_{p}}^{I}}$ whose evaluation at a $\ol{\mathbb{Q}}_{\ell}$-point of $[Z^1(W_{\Q_p}, \GL_n)/\GL_n]$ corresponding to a (not necessarily semi-simple) $L$-parameter $\tilde{\phi}: W_{\mathbb{Q}_{p}} \ra \phantom{}^{L}G(\ol{\mathbb{Q}}_{\ell})$ is the vector space $V$ with $W_{\mathbb{Q}_{p}}^{I}$-action given by $\boxtimes_{i \in I} V_{i} \circ \tilde{\phi}$. Then the Hecke operator $\T_{V}$ defined in the previous sections, is given by the endomorphism
    \[ C_{V} \star (-): \Dlis(\Bunn,\ol{\mathbb{Q}}_{\ell}) \ra \Dlis(\Bunn,\ol{\mathbb{Q}}_{\ell})^{BW_{\mathbb{Q}_{p}}^{I}}. \] 
    by compatibility between Hecke operators and spectral action \cite[Theorem X.1.1]{FS}.
    \item The action is symmetric monoidal in the sense that given $C_{1},C_{2} \in \IndPerf([Z^1(W_{\Q_p}, \GL_n)/\GL_n])$, we have a natural equivalence of endofunctors: 
    \[ (C_{1} \otimes^{\mathbb{L}} C_{2}) \star (-) \simeq C_{1} \star (C_{2} \star (-)). \]
\end{enumerate}

 Let $\phi = \phi_1 \oplus \dotsc \oplus \phi_r$ be a semi-simple and Frobenius semi-simple $L$-parameter as in the previous sub-section, more precisely the $\phi_i$'s are irreducible and if $i \neq j$ then there does not exist unramified character $\xi$ such that $\phi_i \otimes \xi \simeq \phi_j$. Let $S_{\phi}$ be the centralizer of our $\phi$ and let $\Rep_{\overline{\mathbb{Q}}_{\ell}}(S_{\phi})$ be the category of finite-dimensional algebraic $\overline{\mathbb{Q}}_{\ell}$-representations of the group $S_{\phi}$. By the computation in \cite[page 83]{KMSW}, $S_{\phi} \simeq \bb G_m^r$. The parameter $\phi$ defines a closed (by the semi-simplicity of $\phi$) $\ol{\mathbb{Q}}_{\ell}$-point inside the moduli stack $[Z^1(W_{\Q_p}, \GL_n)/\GL_n]$, giving rise to a closed embedding
\[ [\bullet/S_{\phi}] \hookrightarrow [Z^1(W_{\Q_p}, \GL_n)/\GL_n] \] 
which is regularly immersed (by \ref{itm : smooth locus}) inside the connected component containing $\phi$ which we denote by $[C_{\phi}]$. By proposition \ref{itm : simple connected components}, $[C_{\phi}]$ consists of the $L$-parameters of the form $\phi_1 \otimes \xi_1 \oplus \dotsc \oplus \phi_r \otimes \xi_r$ where the $\xi_i$'s are unramified characters. This connected component gives rise to a direct summand
\[
\Perf([C_{\phi}]) \hookrightarrow \Perf([Z^1(W_{\Q_p}, \GL_n)/\GL_n]).
\]

Therefore the spectral action gives rise to a corresponding direct summand 
\[
\Dlis^{[C_{\phi}]}(\Bunn,\ol{\mathbb{Q}}_{\ell})^{\omega} \subset \Dlis(\Bunn,\ol{\mathbb{Q}}_{\ell})^{\omega},
\]
explicitly given as those objects on which the excursion operator corresponding to the function that
is $1$ on $[C_{\phi}]$ and $0$ elsewhere acts via the identity. In particular, the Schur-irreducible objects in this subcategory all have Fargues-Scholze parameter given by an $L$-parameter in $[C_{\phi}]$. 

Now we consider the vector bundles on the connected component $[C_{\phi}]$. By proposition $\ref{itm : simple connected components}$ we have $[C_{\phi}] \simeq [\bb G_m^r/ \bb G_m^r] $ where $ \bb G^r_m $ acts trivially. Let $\Irr(S_{\phi})$ be the set of irreducible algebraic representations of $S_{\phi}$. For each $\chi \in \Irr(S_{\phi}) = \Irr(\bb G^r_m)$ we have an associated vector bundle $C_{\chi}$ on $[\bb G_m^r/ \bb G_m^r]$. More precisely the trivial line bundle on $\bb G_m^r$ together with the action of $\bb G_m^r$ defined by $\chi$ at every fibers gives rise to a $\bb G_m^r$-equivariant vector bundle on $\bb G_m^r$ and this vector bundle corresponds to $C_{\chi}$ on $[\bb G_m^r/ \bb G_m^r]$. Since $\bb G_m^r$ is affine, we see that $C_{\chi}$ is projective. Therefore it is clear that for $\chi, \chi' \in \Irr(S_{\phi})$ we have $ C_{\chi \otimes \chi'} \simeq C_{\chi} \otimes C_{\chi'} \simeq C_{\chi} \otimes^{\bb L} C_{\chi'} $. In particular, $ C_{\chi} \otimes C_{\chi^{-1}} \simeq C_{\chi^{-1}} \otimes C_{\chi} \simeq C_{\Id} $ is the identity functor of $\Dlis^{[C_{\phi}]}(\Bunn, \ov \Q_{\ell})^{\omega}$. Hence for an arbitrary $\chi \in \Irr(S_{\phi})$, $ C_{\chi} $ is an auto-equivalence of $\Dlis^{[C_{\phi}]}(\Bunn, \ov \Q_{\ell})^{\omega}$.


Let $V \in \Rep_{\overline{\Q}_{\ell}}(\GL_n)$, by the arguments in pages $339$--$340$ and theorem $X.1.1$ in \cite{FS}, we can express the action of the Hecke operator $\T_V$ on $\Dlis^{[C_{\phi}]}(\Bunn,\ol{\mathbb{Q}}_{\ell})^{\omega}$ in terms of the spectral action of the $C_{\chi}$'s for $\chi \in \Irr(S_{\phi})$. Recall that the action of $\T_V$ is given by the vector bundle $C_V$ defined above. By \cite[lemma 3.8]{Ham}, it is enough to consider the restriction of $C_V$ to $[\bb G_m^r/ \bb G_m^r]$ and by simplicity we also use $C_V$ to denote the restriction of $ C_V $ to $[\bb G_m^r/ \bb G_m^r]$.

The restriction of $V$ to $S_{\phi}$ admits a commuting $W_{\Q_p}$-action given by $\phi$. This defines a monoidal functor
\[
\Rep_{\overline{\Q}_{\ell}}(\GL_n) \longrightarrow \Rep_{\overline{\Q}_{\ell}}(S_{\phi})^{BW_{\Q_p}}.
\]

Note that as a $S_{\phi} \times W_{\Q_p}$-representation, we can decompose $V$ into a direct sum
\[
V \simeq \bigoplus_{\chi \in \Irr(S_{\phi})} \chi \boxtimes \sigma_{\chi},
\]
where $\sigma_{\chi}$ is the $W_{\Q_p}$-representation $\Hom_{S_{\phi}}(\chi, V)$.  If we forget the $W_{\Q_p}$-action then on $\Dlis^{[C_{\phi}]}(\Bunn,\ol{\mathbb{Q}}_{\ell})^{\omega}$, we have
\[
\T_V = \bigoplus_{\chi \in \Irr(S_{\phi})} C^{\dim \sigma_{\chi}}_{\chi}.
\]

However, $\T_V$ has an action of $W_{\Q_p}$ and we want to understand this action. Recall that $W_{\Q_p}$ acts on $C_V$ and this action gives rise to the action of $W_{\Q_p}$ on $\T_V(\mathcal{F})$ for any sheaf $\mathcal{F}$. If we apply $\T_V$ to a Schur-irreducible sheaf $\mathcal{F}$ on $\Bun_n$ whose $L$-parameter is given by $\phi$ then we can get precise information of the action of $W_{\Q_p}$.

\begin{proposition} \phantomsection \label{itm : fundamental decomposition of Hecke operator}
Suppose that as a $S_{\phi} \times W_{\Q_p}$-representation, we can decompose $V$ into a direct sum
\[
V \simeq \bigoplus_{\chi \in \Irr(S_{\phi})} \chi \boxtimes \sigma_{\chi},
\]
where $\sigma_{\chi}$ is the $W_{\Q_p}$-representation $\Hom_{S_{\phi}}(\chi, V)$. Then for $\mathcal{F}$ a Schur-irreducible sheaf on $\Bun_n$ whose $L$-parameter is given by $\phi$, we have 
\[
\T_V(\mathcal{F}) = \bigoplus_{\chi \in \Irr(S_{\phi})} C_{\chi} \star \mathcal{F} \boxtimes \sigma_{\chi}.
\]   
as object in $\Dlis^{[C_{\phi}]}((\Bunn,\ol{\mathbb{Q}}_{\ell})^{\omega})^{BW_{\Q_p}}$.
\end{proposition}
\begin{proof}
Let $\alpha$ be an element in $W_{\Q_p}$. Thus it gives rise to an isomorphism of $C_V = \bigoplus_{\chi \in \Irr(S_{\phi})} C^{\dim \sigma_{\chi}}_{\chi} $. Hence it is given by a square matrix $(a_{\chi, \chi'})$ of dimension $\dim_V$ where $a_{\chi, \chi'}$ is the matrix recording the morphism from $C^{\dim \sigma_{\chi}}_{\chi}$ to $C^{\dim \sigma_{\chi'}}_{\chi'}$. 
By the concrete description of $W_{\Q_p}$ on the fibers of $C_V$ we see that $(a_{\chi, \chi'})$ is trivial when $\chi \neq \chi'$. Remark that $ \phi'_{| \I_{\Q_p}} \simeq \phi_{| \I_{\Q_p}} $ for every $\phi'$ in $[C_{\phi}]$, we deduce that if $\alpha \in \I_{\Q_p} $ then the isomorphism given by $\alpha$ is just given by  $\bigoplus_{\chi \in \Irr(S_{\phi})} \sigma_{\chi | \I_{\Q_p}}$, namely the square matrix $a_{\chi, \chi'}$ is given by $\sigma_{\chi | \I_{\Q_p}} (\alpha) $. Thus the action of $\I_{\Q_p}$ on $\T_V$ can be expressed as $ \displaystyle \T_V = \bigoplus_{\chi \in \Irr(S_{\phi})} C_{\chi} \boxtimes \sigma_{\chi | \I_{\Q_p}} $. In other words, the following diagram of monoidal functors commutes
\begin{center}
     \begin{tikzpicture}
     \draw (0,0) node [above] {} node{$\Rep_{\overline{\Q}_{\ell}}(\GL_n)$};
     \draw (7,0) node [above] {} node{$\End_{\overline{\Q}_{\ell}}( \Dlis^{[C_{\phi}]}(\Bunn,\ol{\mathbb{Q}}_{\ell})^{\omega} )^{B\I_{\Q_p}}$};
     \draw (0,-2.5) node [above] {} node{$\Rep_{\overline{\Q}_{\ell}}(S_{\phi})^{B\I_{\Q_p}}$};
     \draw [->] (0,-0.5) -- (0,-2);
     \draw [->] (1.3,0) -- (3.9,0);
     \draw [->] (1.3,-2) -- (5,-0.5);
     \end{tikzpicture}
 \end{center}

Now we want to define the square matrix $ (a_{\chi, \chi}(\sigma)) $ corresponding to the Frobenius $\sigma$. Each entry of $ (a_{\chi, \chi}(\sigma)) $ records an endomorphism of $C_{\chi}$ thus it is given by an element in $\mathcal{O}([C_{\phi}])$. We can suppose that $\phi = \phi_1 \oplus \dotsc \oplus \phi_r$ and then $\mathcal{O}([C_{\phi}]) \simeq \ov \Q_{\ell} [X_1^{\pm}, \dotsc, X_r^{\pm}] $. We claim that if we identify $\phi$ with the closed point $(1, \dotsc, 1)$ then the matrix  $(a_{\chi, \chi}(\sigma))$ is given by $\sigma_{\chi}(\sigma) \cdot X^{\chi} $ where $X^{\chi} := X_1^{d_1} \dotsc X_r^{d_r} $ if $\chi = (d_1, \dotsc, d_r)$. Remark that $\mathcal{O}([C_{\phi}])$ is reduce, then to prove the claim, it enough to compare the actions of $\sigma$ and $\sigma_{\chi}(\sigma) \cdot X^{\chi} $ on fibers. At the closed point given by $\phi$, it is given by $\sigma_{\chi}(\sigma)$ by the description of $C_V$. Consider $t = (t_1, \dotsc, t_r)$ we want to consider the action of $\sigma$ on the fiber at the closed point $\phi_t := \phi_1 \otimes \xi_{t_1} \oplus \dotsc \oplus \phi_r \otimes \xi_{t_r} $ where $\xi_{t_i}$ is the unramified character of $W_{\Q_p}$ corresponding to the point $t_i \in \ov \Q_{\ell}^{\times}$. Note that there is a canonique identification of $S_{\phi}$ and $S_{\phi_t}$. When we have an algebraic representation $V$ of $\widehat{\GL_n}$ and we have $2$ decompositions as $S_{\phi} \times W_{\Q_p} $ representations corresponding to $\phi$ and $\phi_t$.
\[
V = \bigoplus_{\chi \in \Irr(S_{\phi})} \chi \boxtimes \sigma_{\chi}
\]
and 
\[
V = \bigoplus_{\chi \in \Irr(S_{\phi})} \chi \boxtimes \sigma^{\phi_t}_{\chi}.
\]

Note that the parameter $\phi : W_{\Q_p} \longrightarrow \widehat{\GL_n}$ give rise to an action of $S_{\phi} \times W_{\Q_p}$ and we can identify the action of $\phi_t(\sigma)$ with the action of $t \times \phi(\sigma)$. Hence we see that $\sigma^{\phi_t}_{\chi}(\sigma) = t^{\chi} \cdot \sigma_{\chi}(\sigma)$. Thus the claim is proved.

Since $\mathcal{F}$ is Schur-irreducible and $C_{\chi^{-1}} \star C_{\chi} = C_{\Id}  $ is the identity functor, we deduce that $ C_{\chi} \star \mathcal{F} $ is also Schur-irreducible. Thus we have a decomposition as objects in $\Dlis^{[C_{\phi}]}((\Bunn,\ol{\mathbb{Q}}_{\ell})^{\omega})^{BW_{\Q_p}}$
\[
\T_V(\mathcal{F}) = \bigoplus_{\chi \in \Irr(S_{\phi})} C_{\chi} \star \mathcal{F} \boxtimes \sigma'_{\chi}.
\] 
where $ \sigma'_{\chi | \I_{\Q_p}} \simeq \sigma_{\chi | \I_{\Q_p}} $. Since $C_{\chi} \star \mathcal{F}$ is also a Schur irreducible sheaf whose Fargues-Scholze parameter is given by $\phi$, the spectral action of $X^{\chi}$ act by $1$ on $C_{\chi} \star \mathcal{F}$. Thus we deduce that $\sigma_{\chi}' (\sigma) = \sigma_{\chi}(\sigma)$ and the conclusion of the proposition follows.



\end{proof}
  
\section{Hecke operators on $\Bunn$} \phantomsection \label{itm : combinatoric description of Hecke operators}
\subsection{Combinatoric description of the Hecke operators} \label{itm : shape of the supports}
\subsubsection{Local Langlands correspondence for inner forms of $\GL_n$} \textbf{}

Let $F$ be a field of characteristic zero, $\overline{F}$ be an algebraic closure, and $\Gamma$ the Galois group of $ \overline{F} / F $. Let $\G$ be a connected reductive group defined over $F$. An inner form of $\G$ is a connected reductive group $\G_1$ defined over $F$ for which there exists an isomorphism $\xi : \G \otimes_F \overline{F} \longrightarrow \G_1 \otimes_F \overline{F} $  such that for all $ \sigma \in \Gamma $, the automorphism $ \xi^{-1} \sigma (\xi) = \xi^{-1} \circ \sigma \circ \xi \circ \sigma^{-1} $ is an inner automorphism of $\G$. If $ \xi_1 : \G \longrightarrow \G_1 $ and $ \xi_2 : \G \longrightarrow \G_2 $ are two inner twists, then an isomorphism $\xi_1 \longrightarrow \xi_2$ consists of an isomorphism $f : \G_1 \longrightarrow \G_2$ defined over $F$ and having the property that $ \xi_2^{-1} \circ f \circ \xi_1 $ is an inner automorphism of $\G$. The map $ \xi \longmapsto \xi^{-1} \sigma(\xi) $ sets up a bijection from the set of isomorphism classes of inner twists of $\G$ to the set set $H^1(\Gamma, \G_{ad})$. 

When $\G = \GL_n$ and $F$ is a local field then $H^1(\Gamma, \G_{ad}) = H^1(\Gamma, \PGL_n) \simeq \Z / n\Z $ and all the inner forms of $\GL_n$ are the groups $\Res_{D/F} \GL_m$, where $D$ is any division algebra over $F$ of degree $g^2$, where $g$ is a natural number such that $ gm = n $.\\

We consider a local field  ${\Q_v}$ for $v$ any place of $\Q$. The local Langlands group is defined by $\mc{L}_{\Q_v} : = W_{\R} $ (the non trivial extension of $\C^{\times}$ by $\Z/2\Z$) if $v = \infty$ and by $ W_{\Q_p} \times \SL_2(\C)$ if $v = p$ is a prime. For a connected reductive group $\G$, we also set $ \LL\G = \widehat{\G}(\C) \rtimes W_{\Q_v} $ as a topological group where $ \widehat{\G} $ is the Langlands dual group of $\G$. In our case we see that $\LL\GL_n = \GL_n(\C) \times W_{\Q_v} $ as $\GL_n$ is a split group and then $W_{\Q_v}$ acts trivially on $\GL_n(\C)$.
\begin{definition} \phantomsection \label{itm : classical L parameter}
A local $ L $-parameter of $\SL_2$-type for a connected reductive group $ \G $ defined over $ \Q_v $ is a continuous morphism $ \phi: \mc{L}_{\Q_v} \longrightarrow \LL\G$ which commutes with the canonical projections of $\mc{L}_{\Q_v} $ and $ \LL\G $ to $ W_{\Q_v} $ and such that $ \phi $ sends semisimple elements to semisimple elements. We say that $\phi$ is discrete if the image does not lie in any proper parabolic subgroup of $\LL\G$.
\end{definition}

We denote by $\Phi(\GL_n)$ the set of local $ L $-parameters of $\SL_2$-type for $\GL_n$. We denote by $\Phi_2(\GL_n)$ the subset of discrete parameters of $\Phi(\GL_n)$. \\


We denote the set of isomorphism classes of irreducible admissible representations of a connected reductive group $\G$ by $\Pi(\G)$. We denote the set of tempered, essentially square integrable, and unitary representations by $\Pi_{\temp}(\G)$, $\Pi_2(\G)$, and $\Pi_{\unit}(\G)$ respectively. We denote $\Pi_{\temp}(\G) \cap \Pi_{2}(\G)$ by $\Pi_{2, \temp}(\G)$.

For each $L$-parameter $\phi$ of $\SL_2$-type, we define the centralizer group as below, which plays an important role in the theory
\[
S_{\phi} := \text{Cent}(\text{Im}\phi, \widehat{\G}).
\]

Every $\phi \in \Phi(\GL_n)$ can be written as $ \phi = \phi_1^{k_1} \oplus \phi_2^{k_2} \oplus \dotsc \oplus \phi_r^{k_r} $ with $\phi_i \in \Phi_2(\GL_{n_i})$ and $k_i$ is the multiplicity of $\phi_i$. In particular we have $ k_1n_1 + k_2n_2 + \dotsc + k_rn_r = n $. Hence 
\[
S_{\phi} \simeq \GL_{k_1} \times \dotsc \times \GL_{k_r}.
\]

Each $\phi_i \in \Phi_2(\GL_{n_i})$ is of the form $ \nu_i \otimes \Sym^{m_i-1} $ where $\nu_i$ is an irreducible representation of $W_{\Q_v}$ of dimension $t_i$ and $\Sym^{m_i-1}$ is the representation of dimension $m_i$ of $\SL_2(\C)$ such that $ t_i \cdot m_i = n_i $. Moreover if for each $i$, the restriction $ \phi_{i | \SL_2(\C)} $ is trivial then $\phi$ is semi-simple.

Now let $\G = \Res_{D/F} \GL_m$ be an inner form of $\GL_n$, where $D$ is a division algebra over $\Q_v$ of degree $g^2$. Then the parameter $\phi$ is relevant for $\G$ if and only if $g$ divides $n_i$ for all $1 \leq i \leq r$. We can now recall the results of \cite[sections 2.1, 2.3]{Badu-Alex} \cite{Badu-Renard} on the local Langlands correspondence for $\G$.

\begin{theorem}  \phantomsection \label{itm : relevant}
    Let $\phi = \phi_1^{k_1} \oplus \phi_2^{k_2} \oplus \dotsc \oplus \phi_r^{k_r} \in \Phi(\GL_n)$ where $\phi_i \in \Phi_2(\GL_{n_i})$ be a decomposition of an $L$-parameter of $\SL_2$-type into simple constituents. If $\phi$ is relevant for $\G$ then there exists a unique irreducible unitary representation $\pi_{\phi}$ of $\G(\Q_v)$ corresponding to $\phi$ and characterised by traces identities. Moreover, as $\phi$ runs over $\Phi(\GL_n)$, the representations $\pi_{\phi}$ are different and exhaust $\Pi(\G)$.  
\end{theorem}

Remark that we can identify (the $\widehat{\G}$-conjugacy class of) a local $\SL_2$-type $L$-parameter with (the $\widehat{\G}$-conjugacy classes of) an $L$-parameter over $\overline{\Q}_{\ell}$ in definition \ref{itm : new definition of L-parameters} by \cite[Propositions 1.13, 1.17]{Nao1}. For example, if an $L$-parameter of $\SL_2$-type $\phi = \phi_1 \oplus \dotsc \oplus \phi_r$ where the $\phi_i$'s are discrete, pairwise non-isomorphic and the restrictions $\phi_{i | \SL_2}$ is trivial then the corresponding $L$-parameter is given by $ \iota_{\ell}(\phi_1) \oplus \dotsc \oplus \iota_{\ell} (\phi_r)$ where  $\iota_{\ell} : \C \xrightarrow{ \ \sim \ } \overline{\Q}_{\ell} $ is an isomorphism of fields.

\subsubsection{Combinatoric description of the spectral action} \textbf{}

Next we prove that there is a bijection between the irreducible representations of the centralizer $S_{\phi}$ of some $L$-parameters $\phi$ and the irreducible sheaves over $\Bunn$ whose $L$-parameter is $\phi$. The bijection is similar to the one described in the introduction. It is of combinatoric nature and could be proved independently with the theory of spectral action. Thus we do not need to impose too much assumption on the $L$-parameter $\phi$.

We fix a maximal split torus $\T$ and a Borel subgroup $\rB$ of $\GL_n$. Let $\overline{C}$ denote the closed Weyl chamber in $X_*(\T)_{\R}$ associated to $\rB$ and let $\overline{C}_{\Q}$ denote its intersection with $X_*(\T)_{\Q}$. For any standard parabolic subgroup $\rP$ with Levi decomposition $\rP = \M\rN$ such that $ \T \subset \M $ (i.e., $\M$ is a standard Levi subgroup), we put
\[
X_*(\rP)^+ := \{ \mu \in X_*(\M) \ | \ \langle \mu, \alpha \rangle > 0 \ \text{for any root of $\T$ in $\rN$}  \}.
\]

Then we have a decomposition 
\[
\overline{C} = \coprod_{\rP}  X_*(\rP)^+
\]
where the index is the set of standard parabolic subgroups of $\GL_n$. We define the subset $B(\GL_n)_{\rP}$ of $B(\G)$ to be the pre-image of $X_*(\rP)^+$ under the Newton map. This gives a decomposition
\begin{equation} \phantomsection \label{decomposition of Kottwitz' set}
 B(\GL_n) = \coprod_{\rP} B(\GL_n)_{\rP}   
\end{equation}
where the index is the set of standard parabolic subgroups of $\GL_n$. For a general standard parabolic $\rP = \M\rN$, $B(\GL_n)_{\rP}$ has the following description. By noting that the image of the Newton map $\nu_{\M}$ for $\M$ lies in $X_*(\M)$, we define $B(\M)^+_{\text{bas}}$ by 
\[
B(\M)^+_{\text{bas}} := \{ b \in B(\M)_{\text{bas}} \ | \ \nu_{\M}(b) \in X_*(\rP)^+ \}. 
\]
Then the canonical map $B(\M) \longrightarrow B(\G)$ induces a bijection $ B(\M)^+_{\text{bas}} \simeq B(\GL_n)_{\rP} $ (see \cite[\S 5.1]{KottwitzIsocrystals2}).

Let us recall the formula given in the introduction. We could consider the case the $L$-parameter $\phi$ has a decomposition 
\[
\phi = \phi_1 \oplus \dotsc \oplus \phi_r,
\]
where the $\phi_i$'s are irreducible of dimension $n_i$ and if $n_i = n_j$ then there does not exists unramified character $\xi$ such that $\phi_i \simeq \phi_j \otimes \xi$. We see that this $L$-parameter corresponds to the $L$-parameter of $\SL_2$-type $\iota^{-1}_{\ell}(\phi_1) \oplus \dotsc \oplus \iota^{-1}_{\ell}(\phi_r) $. In this case we know that the Langlands correspondence $\pi$ of $\phi$ is of the form $\Ind_{\rP}^{\GL_n} (\pi_1 \otimes \dotsc \otimes \pi_r) $ for some standard parabolic subgroup $\rP$ and supercuspidal representations $\pi_i$. We have that $ \displaystyle S_{\phi} := \Cent (\phi) = \prod_{i = 1}^r \Gm $ and the set $\Irr(S_{\phi})$ of irreducible representations of $S_{\phi}$ is isomorphic to the abelian group $ \displaystyle \prod_{i = 1}^r \Z$. For each character $\chi = (d_1, \dotsc, d_r) \in \displaystyle \prod_{i=1}^r \Z $, we define an element $b_{\chi} \in B(\GL_n)$, an irreducible representation $\pi_{\chi}$ of $ \G_{\chi}(\Q_p) := \G_{b_{\chi}}(\Q_p)$ and a sheaf $\mathcal{F}_{\chi}$ as follow: 
\begin{enumerate}
    \item[$\bullet$] $b_{\chi}$ is the unique element in $B(\GL_n)$ such that $ \E_{b_{\chi}} \simeq \OO(\lambda_1)^{m_1} \oplus \dotsc \oplus \OO(\lambda_r)^{m_r} $ where $\OO(\lambda_i)$ is the stable vector bundle of slope $ \lambda_i = d_i / n_i $ and $m_i = (d_i, n_i)$.  
    \item[$\bullet$] Consider the group $\G_{b_{\chi}}$, it is an inner form of a standard Levi subgroups of $\GL_n$. Let $\G^*_{b_{\chi}}$ be the split inner form of $\G_{b_{\chi}}$. For each $i$, denote $\G_i := \GL_{m_i}(D_{-\lambda_i})$ where $D_{-\lambda_i}$ is the division algebra whose invariant is $-\lambda_i$, thus $\G_i^* = \GL_{n_i}$. Thus we have a map $\phi_i : W_{\Q_p} \longrightarrow \widehat{\G^*_i}(\overline{\Q}_{\ell})$ and the direct sum $ \phi_1 \oplus \dotsc \oplus \phi_r $ gives us a map $W_{\Q_p} \longrightarrow \displaystyle \prod_{i=1}^r \widehat{\G^*_i}(\overline{\Q}_{\ell})$ whose post-composition with the natural embeddings $ \displaystyle \prod_{i=1}^r \widehat{\G^*_i}(\overline{\Q}_{\ell}) \hookrightarrow \widehat{\G^*_{\chi}}(\overline{\Q}_{\ell}) $ defines an $L$-parameter $\phi_{\chi}$ of $\G_{b_{\chi}}$. Moreover, the post-composition of $\phi_{\chi}$ with $ \widehat{\G^*_{b_{\chi}}}(\overline{\Q}_{\ell}) \hookrightarrow \widehat{\GL}_n(\overline{\Q}_{\ell}) $ is the ($\widehat{\GL}_n$-conjugacy class of) $\phi$. Finally, $\pi_{\chi}$ is the representation of $\G_{b_{\chi}}(\Q_p)$ whose $L$-parameter is given by $\phi_{\chi}$ via the local Langlands correspondence for general linear groups.
    \item[$\bullet$] We can suppose that $\G_{b_{\chi}}^*$ is standard. By (\ref{decomposition of Kottwitz' set}), there exists a unique standard parabolic subgroup $\rP$ of $\GL_n$ with Levi factor given by $\G_{b_{\chi}}^*$ such that $\nu_{b_{\chi}} \in X_*(\rP)^+ $. Let $\delta_{\rP}$ be the modulus character with respect to $\rP$. Then we denote by $\delta_{b_{\chi}}$ (or $\delta_{\chi}$) the character of $ \G_{b_{\chi}} $ whose $L$-parameter is the same as that of $\delta_{\rP| \G_{b_{\chi}}^*}$. We consider the embedding $ i_{b_{\chi}} : \Bun^{b_{\chi}}_n \longrightarrow \Bunn $ and define $\mathcal{F}_{\chi} \displaystyle := i_{b_{\chi} !}(\delta^{-1/2}_{b_{\chi}} \otimes \pi_{\chi}) [- d_{\chi}] $ where $d_{\chi} = \langle 2\rho, \nu_{b_{\chi}} \rangle$.

     We recall that if $ \rP = \M \rN $ where $\M$ is the Levi subgroup and $\rN$ is the unipotent radical then the modulus character is defined by $\delta_{\rP}(mn) := |\det(\ad(m);\Lie \rN)| $ where $|\cdot|$ denotes the normalized absolute value of the field $\ov \Q_{\ell}$.

    Our characters $\delta_{b}$ is the same as the character $\delta_b$ defined by Hamann and Imai in \cite{HI}. We recall also that for $b \in B(\GL_n)$, we have $\nu_{b} = (- \nu_{\E_b})_{\text{dom}}$ then the standard parabolic subgroup corresponding to the Harder-Narasimhan reduction of $\E_{b_{\chi}}$ (\cite[Theorem 1.7]{CFS}) is conjugated to the opposite of $\rP$ above.

\end{enumerate}

\textbf{} 

We show that the map $ \chi \longmapsto (b_{\chi}, \pi_{\chi}) $ is a bijection between the set of characters of $S_{\phi}$ and the set of pairs $(b, \pi_b)$ where $b \in B(\GL_n)$ and $\pi_b$ is an irreducible representation of $\G_b(\Q_p)$ whose $L$-parameter is given by $\phi$. First suppose that there exists $ \chi_1 = (d_{1,1}, \dotsc, d_{1,r})$ and $\chi_2 = (d_{2,1}, \dotsc, d_{2,r}) $ such that $ (b_{\chi_1}, \pi_{\chi_1}) = (b_{\chi_2}, \pi_{\chi_2}) $.

Thus we have $ \E_{b_{\chi_1}} \simeq \E_{b_{\chi_2}} \simeq \OO(\lambda_1)^{m_1} \oplus \dotsc \oplus \OO(\lambda_k)^{m_k} $ where $ \lambda_1 > \lambda_2 > \dotsc > \lambda_k $. By the construction of $b_{\chi_1}$ and $b_{\chi_2}$, for each $\lambda_i$ there exists subsets $\I(\lambda_i)$ and $\J(\lambda_i)$ of $\{ 1, 2, \dotsc, r \}$ such that $d_{1, x}/n_x = \lambda_i$ if and only if $x \in \I(\lambda_i)$ and $d_{2, y}/n_y = \lambda_i$ if and only if $y \in \J(\lambda_i)$. More over $ \G_{\chi_1} (\Q_p) = \G_{\chi_2} (\Q_p) = \displaystyle \prod_{i=1}^k \GL_{m_i}(D_{-\lambda_i}) $ where $\GL_{m_i}(D_{-\lambda_i})$ is isomorphic to the group of automorphisms of the semi-stable vector bundle $\OO(\lambda_i)^{m_i}$. By the construction of $\pi_{\chi_1}$ and $\pi_{\chi_2}$, for each $1 \le i \le k$ one can construct representations $\pi^1_i$ and $\pi^2_i$ of $ \GL_{m_i}(D_{-\lambda_i}) $ such that $ \pi_{\chi_1} = \displaystyle \boxtimes_{i=1}^k \pi^1_i $ and $ \displaystyle \pi_{\chi_2} = \boxtimes_{i=1}^k \pi^2_i $. Moreover the $L$-parameter of 
$\pi^1_i$ is given by $ \displaystyle \bigoplus_{j \in \I(i)} \phi_j $ and the $L$-parameter of $\pi^2_i$ is given by $ \displaystyle \bigoplus_{j \in \J(i)} \phi_j $. Thus $ \pi_{\chi_1} \simeq \pi_{\chi_2} $ if and only if $\pi^1_i \simeq \pi^2_i $ as representations of $ \GL_{m_i}(D_{-\lambda_i}) $ for all $ 1 \le i \le k $. 

Thus for each $ 1 \le i \le k $ we have $ \displaystyle \bigoplus_{j \in \I(i)} \phi_j = \bigoplus_{j \in \J(i)} \phi_j $. We deduce that $ \I(i) = \J(i) $ and more over $ \lambda_i = d_{2,t}/n_t = d_{1, t}/n_t $ for every $t \in \I(i)$. We deduce that $ d_{2,t} = d_{1,t} $ for all $ t \in \I(i) $ and for all $ 1 \le i \le k $. In other words, $\chi_1 = \chi_2$.

We now show that the map $\chi \longmapsto (b_{\chi}, \pi_{\chi})$ is surjective. Consider a pair $(b, \pi_b)$ consisting of an element $b \in B(\GL_n)$ and an irreducible representation $\pi_b$ of $ \G_b(\Q_p) $ whose $L$-parameter, after post-composing with $ \widehat{\G^*_b} \hookrightarrow \widehat{\GL}_n $, is given by $ \phi $.

We can decompose $\E_b \simeq \OO(\lambda_1)^{m_1} \oplus \dotsc \oplus \OO(\lambda_k)^{m_k} $ where $ \lambda_1 > \lambda_2 > \dotsc > \lambda_k $. As before $\G_b(\Q_p) = \displaystyle \prod_{i=1}^d \GL_{m_i}(D_{-\lambda_i}) $. Thus for each $ 1 \le i \le k $, there exists a subset $\I(i)$ of $ \{ 1, 2, \dotsc, r \} $ such that $ \displaystyle \bigoplus_{j \in \I(i)} \phi_j $ is an $L$-parameter of $\GL_{m_i}(D_{-\lambda_i})$ and the direct sum $ \displaystyle \bigoplus_{1 \le i \le k} \bigoplus_{j \in \I(i)} \phi_j $ gives us $ \phi $.

Now for each $t \in \I(i)$, by the explicit description of Levi subgroups of $\GL_{m_i}(D_{-\lambda_i})$, there exists a unique integer $d_t$ such that $\lambda_i = d_t/n_t$. If we take $\chi = (d_1, \dotsc, d_r)$ then we can see that $ (b_{\chi}, \pi_{\chi}) = (b, \pi_b) $. Thus the map $\chi \longmapsto (b_{\chi}, \pi_{\chi})$ is surjective. \\

\begin{proposition} \phantomsection \label{itm : shape of strata}
   If $\mathcal{F}$ is a non-zero irreducible sheaf on $\Bun_n$ supported on a stratum corresponding to $b \in B(\GL_n)$ whose Fargues-Scholze parameter is given by $\phi$ then there exists $\chi \in \displaystyle \Irr(S_{\phi}) \simeq \prod_{i=1}^r \Z $ such that $b = b_{\chi}$.    
\end{proposition} 
\begin{proof}
We note that the Fargues-Scholze parameter of the sheaf $i_{b_{\chi}!}( \delta^{-1/2}_{b_{\chi}} \otimes \pi_{\chi} )$ is given by $\phi$ (see remark 1.3(2)). Thus we have a bijection between the set of irreducible character $\chi$ of $S_{\phi}$ and the set of pairs $(b, \pi_b)$ where $b \in B(\GL_n)$ and $\pi_b$ is an irreducible representation of $\G_b(\Q_p)$ such that the Fargues-Scholze parameter of $i_{b!}(\pi_b)$ is given by $\phi$. Therefore the conclusion of the proposition follows.
\end{proof}

\begin{remark}
    In \cite[Remark I.10.3]{FS}, Fargues and Scholze note that Fargues' conjecture should also include a comparison of $t$-structures. Then the equivalence would also yield a bijection between irreducible objects in the abelian hearts. On one side, these irreducible objects would then be enumerated by pairs $(b, \pi_b)$ of an element $b \in B(\G)$ and an irreducible smooth representation $\pi_b$, by using intermediate extensions. On the other side, they would likely correspond to a Frobenius-semisimple $L$-parameter $ \varphi : W_{\Q_p} \longrightarrow \widehat{\G}(\overline{\Q}_{\ell}) $ together with an irreducible representation of the centralizer $S_{\varphi}$ of $\varphi$. The above combinatoric bijection $ \chi \longmapsto (b_{\chi}, \pi_{\chi}) $ could be a good candidate for Fargues-Scholze's predicted bijection for $\GL_n$. In \cite{BMO}, Bertoloni Meli and Oi proposed a good candidate for the bijection between irreducible objects in the abelian hearts for an arbitrary reductive group $\G$ and $L$-parameters $\varphi$ satisfying some conditions similar to ours. 
\end{remark}

\subsection{Statement of the first main theorem} \label{conditions on L-parameters}

Let $ [Z^1(W_{\Q_p}, \widehat{\GL}_n)/\widehat{\GL}_n] $ be the stack of $L$-parameters of $\GL_{n, \Q_p}$ and let $Z^1(W_{\Q_p}, \widehat{\GL}_n)//\widehat{\GL}_n$ be the categorical quotient. We have a map $\theta : [Z^1(W_{\Q_p}, \widehat{\GL}_n)/\widehat{\GL}_n] \longrightarrow Z^1(W_{\Q_p}, \widehat{\GL}_n)//\widehat{\GL}_n$ given by the semi-simplification. Let $\phi$ be a semi-simple $L$-parameter of $\GL_n$ and let $|\cdot|_p$ be the norm character $W_{\Q_p} \longrightarrow W_{\Q_p}^{\text{ab}} \simeq \Q_p^{\times} \longrightarrow \C \simeq \overline{\Q}_{\ell} $. We suppose that $\phi$ satisfies the following condition:
\begin{enumerate} \label{itm : condition A1}
    \item[(A1)] We have a decomposition $ \phi = \phi_1 \oplus \dotsc \oplus \phi_r $ where the $\phi_i$'s are pairwise disjoint irreducible representations of dimension $n_i$ and if $i \neq j$ then there does not exist unramified character $\chi$ such that $ \phi_i \simeq \phi_j \otimes \chi $. 
\end{enumerate}

We have that $ \displaystyle S_{\phi} := \Cent (\phi) = \prod_{i = 1}^r \Gm $. 

By proposition \ref{itm : inverse has 1 points}, we see that $ \phi $ defines a closed point $x$ in $Z^1(W_{\Q_p}, \widehat{\GL}_n)//\widehat{\GL}_n$ and $\theta^{-1}(x)$ is the image of the closed embedding $ \iota : [\bullet /S_{\phi}] \longrightarrow [Z^1(W_{\Q_p}, \widehat{\GL}_n)/\widehat{\GL}_n] $. Denote by $ [C_{\phi}] $ the connected component of $[Z^1(W_{\Q_p}, \widehat{\GL}_n)/\widehat{\GL}_n]$ containing the image of $\iota$. By proposition \ref{itm : simple connected components}, we know that $[C_{\phi}] \simeq [\bb G_m^r / \bb G_m^r ]$ where the quotient is taken with respect to the trivial action of $\bb G_m^r$. Let $C \in \Perf^{\mathrm{qc}}([Z^1(W_{\Q_p}, \widehat{\GL}_n)/\widehat{\GL}_n])$ be a sheaf such that its support does not intersect with $[C_{\phi}]$. Then by \cite[lemma 3.8]{Ham}, we know that $ C \star \mathcal{F} = 0 $ where $\mathcal{F}$ is any Schur irreducible sheaf on $\Bunn$  whose $L$-parameter belongs to $[C_{\phi}]$. The study of the spectral action on those sheaves $\mathcal{F}$ is therefore reduced to the description of the $ C \star \mathcal{F}$ for $C$ the vector bundles on $ [C_{\phi}]$.

Denote by $\Perf([C_{\phi}])$ the category of perfect complex supported on $[C_{\phi}]$. We have monoidal functors
\[
\Rep_{\ov \Q_{\ell}}(S_{\phi}) \longrightarrow \Perf([C_{\phi}]) \longrightarrow \Perf([Z^1(W_{\Q_p}, \widehat{\GL}_n)/\widehat{\GL}_n])
\]
where the image of an irreducible character $\chi$ is the vector bundle on $[C_{\phi}]$ corresponding to the structural sheaf on $\bb G_m^r$ together with the $\bb G_m^r$-action defined by $\chi$.


Remark that $S_{\phi}$ is commutative then the set $\Irr( S_{\phi} )$ of its algebraic characters forms a group under the tensor product operator. In this specific case that group is isomorphic to $ \displaystyle \prod_{i=1}^r \Z $. Let $ \displaystyle \chi = (d_1, \dotsc, d_r) \in \prod_{i=1}^r \Z$ be the character of $S_{\phi}$ such that $ \chi (t_1, \dotsc, t_r) \displaystyle = \prod_{i = 1}^r t_i^{d_i} $. Then we denote by $C_{\chi}$ the corresponding vector bundle on $[C_{\phi}]$. We can also define a triple $( b_{\chi}, \pi_{\chi}, \mathcal{F}_{\chi} )$ as in the previous section. Recall that $ \mathcal{F}_{\chi} := i_{b_{\chi} !} ( \delta^{-1/2}_{b_{\chi}} \otimes \pi_{\chi}) [-d_{\chi}]$ where $d_{\chi} = \langle 2\rho, \nu_{b_{\chi}} \rangle$.

\begin{theorem} \phantomsection \label{itm : main theorem}
Let $\phi = \phi_1 \oplus \dotsc \oplus \phi_r$ be an $L$-parameter satisfying the conditions in the beginning of subsection \ref{conditions on L-parameters}. Let $\chi = (d_1, \dotsc, d_r)$ be an element in $\displaystyle \prod_{i=1}^r \Z $, then we have
\[
C_{\chi} \star \mathcal{F}_{\Id} = \mathcal{F}_{\chi}.
\]
where $\Id$ is the identity of $ \displaystyle \prod_{i=1}^r \Z$.
\end{theorem}

\begin{remark}
    While preparing this manuscript, the author was informed that Linus Hamann and David Hansen formulated a conjecture describing the Hecke operators for an arbitrary reductive group $\G$ and $L$-parameters satisfying similar conditions as ours \cite{Hansen}. Hamann and Hansen call the $L$-parameters satisfying their specific conditions "generous $L$-parameters" since it is closely related to the vanishing of the cohomology of Shimura varieties with mod $\ell$-coefficient.    
\end{remark} 

We record here a useful corollary of the the theorem.
\begin{corollary} \phantomsection \label{itm : useful corollary}
    Let $\phi = \phi_1 \oplus \dotsc \oplus \phi_r$ be an $L$-parameter satisfying the conditions in the beginning of subsection \ref{conditions on L-parameters} and let $\chi \in \displaystyle \prod_{i=1}^r \Z $. Then for every stratum $b \in B(\GL_n)$ not equal to $b_{\chi}$, we have
    \[
    i_b^{!}( \mathcal{F}_{\chi} ) \simeq 0.
    \]
\end{corollary}
\begin{proof}
    The sheaf $\mathcal{F}_{\lambda}$ is compact and ULA then $i_b^{!}( \mathcal{F}_{\chi} )$ is also compact and ULA by \cite[Theorem 1.3.1]{HHS1}. Thus this sheaf is cohomologically bounded and its cohomology groups are admissible representations of $\G_{b_{\lambda}}(\Q_p)$. 
    
    Suppose that $i_b^{!}( \mathcal{F}_{\chi} )$ is not trivial. Then there exists an integer $k \in \Z$ and a smooth irreducible representation $\pi$ of $\G_{b_{\lambda}}(\Q_p)$ such that $\Hom_{\G_{b_{\lambda}}(\Q_p)}(\pi[k], i_b^{!}( \mathcal{F}_{\chi} )) \neq 0 $. Since $i_b^{!}$ is right adjoint of $i_{b!}$, we deduce that
    \[
    \Hom_{\Dlis(\Bunn, \ov\Q_{\ell})}(i_{b!} \pi[k],  \mathcal{F}_{\chi} ) \neq 0.
    \]

    If the Fargues-Scholze parameter of $i_{b!} \pi[k]$ with respect to $\Bunn$ is not equal to $\phi$ then we can chose an element $x \in \mathcal{Z}^{\spec}(\GL_n, \overline{\Q}_{\ell}) $ such that the excursion operator corresponding to $x$ acts by $0$ on $\mathcal{F}_{\chi}$ and acts by an isomorphism on $i_{b!} \pi[k]$. Since an excursion operator gives rise to an element of the geometric Bernstein center of $\Dlis(\Bunn, \overline{\Q}_{\ell})$ \cite[Theorem X.5.2]{FS}, we see that $\Hom_{\Dlis(\Bunn, \ov\Q_{\ell})}(i_{b!} \pi[k],  \mathcal{F}_{\chi} ) = 0$. Therefore, the Fargues-Scholze parameter of $i_{b!} \pi[k]$ is equal to $\phi$ and by proposition $\ref{itm : shape of strata}$, there exists $\chi' \in \displaystyle \prod_{i=1}^r \Z  $ such that $b = b_{\chi'}$ and moreover $i_{b!} \pi[k] \simeq \mathcal{F}_{\chi'}[k']$. Thus we have
    \begin{align*}
        \Hom_{\Dlis(\Bunn, \ov\Q_{\ell})}(i_{b!} \pi[k],  \mathcal{F}_{\chi} ) &= \Hom_{\Dlis(\Bunn, \ov\Q_{\ell})}(\mathcal{F}_{\chi'}[k'],  \mathcal{F}_{\chi} ) \\
        &= \Hom_{\Dlis(\Bunn, \ov\Q_{\ell})}(\mathcal{F}_{\Id}[k'],  \mathcal{F}_{\chi \otimes (\chi')^{-1} } ) \\
        &= 0
    \end{align*}
    where we applied the spectral action $C_{(\chi')^{-1}}$ and use theorem \ref{itm : main theorem} to obtain the second equality and the last equality follows from the fact that $\mathcal{F}_{\chi \otimes (\chi')^{-1} }$ does not supported on the trivial stratum of $\Bunn$. It yeilds a contradiction and hence we conclude that $i_b^{!}( \mathcal{F}_{\chi} )$ is trivial.
\end{proof}

\section{An analogue of Boyer's trick} \phantomsection \label{itm : Boyer's trick}
In this section, we prove an analogue of theorems $1.7$ and $4.13$ in \cite{Han1} based on the results developed in loc.cit. and derive some applications.
\subsection{An analogue of Boyer's trick}

Fix a maximal torus and a Borel subgroup $ \T \subset \rB \subset \GL_n $. For each $\lambda \in \Q$, let $\OO(\lambda)$ be the stable vector bundle of slope $\lambda$. Consider $b, b' \in B(\GL_n)$ and suppose that we have $\mathcal{E}_{b} = \OO(\lambda_1) \oplus \dotsc \oplus \OO(\lambda_s) $ and $\mathcal{E}_{b'} = \OO(\lambda'_1) \oplus \dotsc \oplus \OO(\lambda'_t) $ where $\lambda_i \ge \lambda_{i+1}$ and $\lambda'_i \ge \lambda'_{i+1}$. Let $\mu = (k_1, k_2, \dotsc, k_n)$ be a minuscule dominant cochatacter of $\GL_n$ and let $\deg(\mu) \displaystyle := \sum_{i = 1}^n k_i$ and suppose $ \deg(\mu) = \deg(\E_{b'}) - \deg(\E_{b})$.

Recall that for each $b \in B(\GL_n)$, we can define a character $\kappa_b : \G_b(\Q_p) \longrightarrow \overline{\Q}^{\times}_{\ell}$ as in \cite{GI}, lemma $4.18$ and as in \cite{Ham1}, page $91$, before lemma $11.1$.

\begin{proposition} \phantomsection \label{generalized Boyer's trick}

Suppose that we have decompositions $\E_b = \E_{b_1} \oplus \E_{b_2} $ and $\E_{b'} = \E_{b'_1} \oplus \E_{b'_2} $ where $ \E_{b_1} = \OO(\lambda_1) \oplus \dotsc \oplus \OO(\lambda_{s'}) $; $ \E_{b_2} = \OO(\lambda_{s'+1}) \oplus \dotsc \oplus \OO(\lambda_{s})$ and $ \E_{b'_1} = \OO(\lambda'_1) \oplus \dotsc \oplus \OO(\lambda'_{t'}) $; $ \E_{b'_2} = \OO(\lambda'_{t'+1}) \oplus \dotsc \oplus \OO(\lambda'_{t})$ such that $\rank(\E_{b_1}) = \rank(\E_{b'_1}) = m < n $ and $\deg(\E_{b'_1}) \displaystyle = \deg(\E_{b_1}) + \sum_{i=1}^m k_i $ as well as $\lambda'_{t'} > \lambda'_{t'+1} $. Suppose that we have a filtration $ \mathcal{F}_1 \subset \dotsc \subset \mathcal{F}_{\ell} \subseteq \E_{b_1} \subset \mathcal{F}_{\ell + 1} \subset \dotsc \subset \E_{b} $ where $ \mathcal{F}_1 \subset \dotsc \subset \mathcal{F}_{\ell} \subset \mathcal{F}_{\ell + 1} \subset \dotsc \subset \E_{b} $ is the canonical Harder-Narasimhan filtration of $\E_{b}$. Let $\rP_b$ be the parabolic subgroup of $\G_b$ fixing this filtration. We define $\mu_1 := (k_1, k_2, \dotsc, k_m)$ and $ \mu_2 := (k_{m+1}, k_{m+2}, \dotsc, k_n)$ then we have the following equality
\[
\Sht(\GL_n, b, b', \mu) = \Big( \Sht(\GL_m, b_1, b'_1, \mu_1) \times_{\Spd \breve{\Q}_p} \Sht(\GL_{n-m}, b_2, b'_2, \mu_2) \times_{\Spd \breve{\Q}_p} \mathcal{J}^{U} \Big) \times^{\underline{\rP_b}} \underline{\G}_b,
\]
where $\mathcal{J}^{U}$ is the unipotent diamond in group $ \widetilde{\G}_{b'}^0 / \widetilde{\G}_{b'_1}^0 \times \widetilde{\G}_{b'_2}^0 $. Moreover there is a natural isomorphism of $\G_b(\Q_p) \times \G_{b'}(\Q_p) \times W_{\Q_p} $-modules
\[
H^*_c(\Sht(\GL_n, b, b', \mu), \overline{\Q}_{\ell}) \simeq \ind^{\G_b}_{\rP_b} \Big( H^{*-2d}_c(\Sht(\GL_m, b_1, b'_1, \mu_1) \times \Sht(\GL_{n-m}, b_2, b'_2, \mu_2), \overline{\Q}_{\ell}) \otimes \kappa \Big)
\]
where $\ind^{\G_b}_{\rP_b}$ denotes the un-normalized parabolic induction; $d = \langle 2\rho, \nu_{b'} \rangle - \langle 2\rho_{m}, \nu_{b'_1} \rangle - \langle 2\rho_{n-m}, \nu_{b'_2} \rangle$ is the dimension of $\mathcal{J}^{U}$ and $\rho, \rho_{m}, \rho_{n-m}$, respectively, are half of the sum of positive roots of $\GL_n$, $\GL_m$, $\GL_{n-m}$ respectively and where $\kappa = \kappa_{b'} \otimes ( \kappa_{b'_1}^{-1} \times \kappa_{b'_2}^{-1} )$.
\end{proposition}

\begin{corollary} \label{coro of generelized Boyer's trick}
  Suppose that $\mu = (k_1, \dotsc, k_n)$ with $ k_{n-m+1} = \dotsc = k_n = 0 $ and we have a decomposition $\E_b = \E_{b_1} \oplus \E_{b_2} $ and $\E_{b'} = \E_{b'_1} \oplus \E_{b'_2} $ where $ \E_{b'_1} \simeq \E_{b_1} = \OO(\lambda_1) \oplus \dotsc \oplus \OO(\lambda_{s'}) $; $ \E_{b_2} = \OO(\lambda_{s'+1}) \oplus \dotsc \oplus \OO(\lambda_{s})$ and $ \E_{b'_2} = \OO(\lambda'_{s'+1}) \oplus \dotsc \oplus \OO(\lambda'_{t})$ such that $\rank(\E_{b_1}) = \rank(\E_{b'_1}) = m < n $ as well as $\lambda_{s'} > \lambda_{s'+1} $. Suppose that we have a filtration $ \mathcal{F}_1 \subset \dotsc \subset \mathcal{F}_{\ell} \subseteq \E_{b'_1} \subset \mathcal{F}_{\ell + 1} \subset \dotsc \subset \E_{b'} $ where $ \mathcal{F}_1 \subset \dotsc \subset \mathcal{F}_{\ell} \subset \mathcal{F}_{\ell + 1} \subset \dotsc \subset \E_{b'} $ is the canonical Harder-Narasimhan filtration of $\E_{b'}$. Let $\rP_{b'}$ be the parabolic subgroup of $\G_{b'}$ fixing this filtration. Denote $\mu_1 = (0, 0, \dotsc, 0)$ and $  \mu_2 = (k_1, k_{2}, \dotsc, k_{n-m})$ then we have the following equality

\[
\Sht(\GL_n, b, b', \mu) = \Big( \Sht(\GL_m, b_1, b'_1, \mu_1) \times_{\Spd \breve{\Q}_p} \Sht(\GL_{n-m}, b_2, b'_2, \mu_2) \times_{\Spd \breve{\Q}_p} \mathcal{J}^{U} \Big) \times^{\underline{\rP_{b'}}} \underline{\G}_{b'},
\]
where $\mathcal{J}^{U}$ is the unipotent diamond in group $ \widetilde{\G}_{b}^0 / \widetilde{\G}_{b_1}^0 \times \widetilde{\G}_{b_2}^0 $. Moreover there is a natural isomorphism of $\G_b(\Q_p) \times \G_{b'}(\Q_p) \times W_{\Q_p} $-modules 
\[
H^*_c(\Sht(\GL_n, b, b', \mu), \overline{\Q}_{\ell}) \simeq \ind^{\G_{b'}}_{\rP_{b'}} \Big( H^{*-2d}_c(\Sht(\GL_m, b_1, b'_1, \mu_1) \times \Sht(\GL_{n-m}, b_2, b'_2, \mu_2), \overline{\Q}_{\ell}) \otimes \kappa \Big)
\]
where $d = \langle 2\rho, \nu_{b} \rangle - \langle 2\rho_{m}, \nu_{b_1} \rangle - \langle 2\rho_{n-m}, \nu_{b_2} \rangle$ is the dimension of $\mathcal{J}^{U}$ and $\rho, \rho_{m}, \rho_{n-m}$, respectively, are half of the sum of positive roots of $\GL_n$, $\GL_m$, $\GL_{n-m}$ respectively and $\kappa = \kappa_b \otimes ( \kappa_{b_1}^{-1} \times \kappa_{b_2}^{-1}) $.
\end{corollary}

For a proof of the corollary, we just need to apply the theorem to the moduli space of shtukas $\Sht(\GL_n, b', b, -\mu)$ where $-\mu = (-k_n, -k_{n-1}, \dotsc, -k_1)$ is the dominant cocharacter corresponding to the inverse of $\mu$.

Below are some illustrating pictures.
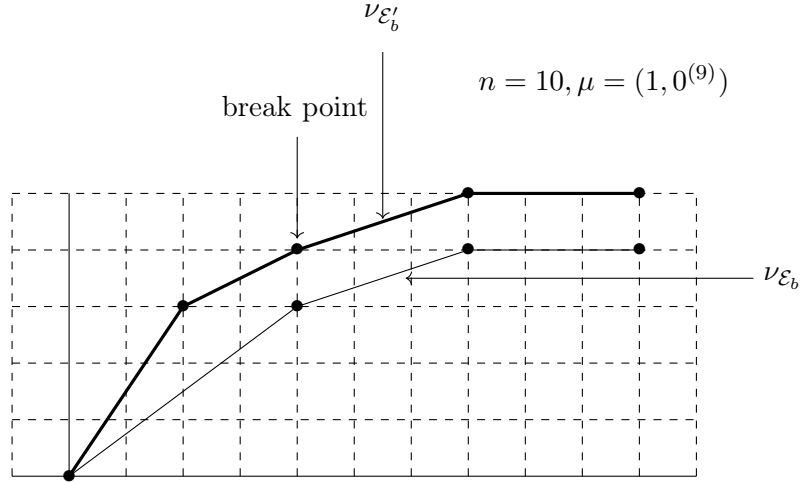
\begin{center}
			\begin{tikzpicture}[scale = 0.75]
			\draw (0,0) node [above] {} node{$\bullet$};	
			\draw (2,3) node [below] {} node{$\bullet$};
			\draw (4,4) node [below] {} node{$\bullet$};
			\draw (7,5) node [below] {} node{$\bullet$};
			\draw (10,4) node [below] {} node{$\bullet$};
                \draw (4,3) node [below] {} node{$\bullet$};
			\draw (7,4) node [below] {} node{$\bullet$};
			\draw (10,5) node [below] {} node{$\bullet$};
			\draw [dashed] (-1,1) -- (11,1);
			\draw [dashed] (-1,5) -- (11,5);
			\draw [dashed] (-1,2) -- (11,2);
			\draw [dashed] (-1,3) -- (11,3);
			\draw [dashed] (-1,4) -- (11,4);
			\draw [dashed] (1,0) -- (1,5);
			\draw [dashed] (2,0) -- (2,5);
			\draw [dashed] (3,0) -- (3,5);
			\draw [dashed] (4,0) -- (4,5);
			\draw [dashed] (5,0) -- (5,5);
			\draw [dashed] (6,0) -- (6,5);
			\draw [dashed] (7,0) -- (7,5);
			\draw [dashed] (8,0) -- (8,5);
			\draw [dashed] (9,0) -- (9,5);
			\draw [dashed] (10,0) -- (10,5);
			\draw [dashed] (11,0) -- (11,5);
			\draw [dashed] (-1,0) -- (-1,5);
			\draw [-] (-1,0) -- (11,0);
			\draw [-] (0,0) -- (0,5);
			\draw [-] (0,0) -- (4,3);
			\draw [-] (4,3) -- (7,4);
                \draw [-] (7,4) -- (10,4);
			\draw [very thick] (0,0) -- (2,3);
			\draw [very thick] (2,3) -- (4,4);
			\draw [very thick] (4,4) -- (7,5);
                \draw [very thick] (7,5) -- (10,5);
			\draw [->] (12,3.5) -- (5.9,3.5);
			\draw (12,3.5) node [above, right] {$\nu_{\E_b}$};
			\draw [->] (4,6) -- (4,4.2);
			\draw (2.5,6.5) node [above, right] {\text{break point}};
                \draw (7,7) node [above, right] {\text{$n = 10, \mu = (1, 0^{(9)})$}};
			\draw [->] (5.5,7.5) -- (5.5,4.55);
			\draw (5.5,7.7) node [above] {$\nu_{\E_b'}$};
			\end{tikzpicture}
\end{center}

In this case we have $ \mu = (1, 0^{(9)}) $ and
\[
\E_{b} = \OO(3/4) \oplus \OO(1/3) \oplus \OO^3 \quad \quad \E_{b'} = \OO(3/2) \oplus \OO(1/2) \oplus \OO(1/3) \oplus \OO^3.
\]
Thus $\E_{b_1} = \OO(3/4)$; $\E_{b'_1} = \OO(3/2) \oplus \OO(1/2)$ and $\E_{b_2} \simeq \E_{b'_2} \simeq \OO(1/3) \oplus \OO^3 $ and $\mu_1 = (1, 0^{(3)}) $; $\mu_2 = (0^{(6)}) $  as well as 
\[
P_b(\Q_p) = \G_b(\Q_p) \simeq D^{\times}_{-3/4}(\Q_p) \times D^{\times}_{-1/3}(\Q_p) \times \GL_3(\Q_p)
\]
Moreover we have $ \langle 2\rho_4, \nu_{b'_1} \rangle = 4 $; $ \langle 2\rho_6, \nu_{b'_2} \rangle = 3 $ and $ \langle 2\rho_{10}, \nu_{b'} \rangle = 27 $, hence $d = 20$.

\begin{center}
			\begin{tikzpicture}[scale = 0.75]
                \draw (-2,-3) node [above] {} node{$\bullet$};
			\draw (0,0) node [above] {} node{$\bullet$};	
			\draw (2,3) node [below] {} node{$\bullet$};
			\draw (7,5) node [below] {} node{$\bullet$};
			\draw (10,5) node [below] {} node{$\bullet$};
                \draw (4,2) node [below] {} node{$\bullet$};
			\draw (10,3) node [below] {} node{$\bullet$};
			\draw [dashed] (-2,1) -- (11,1);
			\draw [dashed] (-2,5) -- (11,5);
			\draw [dashed] (-2,2) -- (11,2);
			\draw [dashed] (-2,3) -- (11,3);
			\draw [dashed] (-2,4) -- (11,4);
                \draw [dashed] (-2,-1) -- (11,-1);
                \draw [dashed] (-2,-2) -- (11,-2);
                \draw [dashed] (-2,-3) -- (11,-3);
			\draw [dashed] (1,-3) -- (1,5);
			\draw [dashed] (2,-3) -- (2,5);
			\draw [dashed] (3,-3) -- (3,5);
			\draw [dashed] (4,-3) -- (4,5);
			\draw [dashed] (5,-3) -- (5,5);
			\draw [dashed] (6,-3) -- (6,5);
			\draw [dashed] (7,-3) -- (7,5);
			\draw [dashed] (8,-3) -- (8,5);
			\draw [dashed] (9,-3) -- (9,5);
			\draw [dashed] (10,-3) -- (10,5);
			\draw [dashed] (11,-3) -- (11,5);
			\draw [dashed] (-1,-3) -- (-1,5);
			\draw [dashed] (-2,0) -- (11,0);
			\draw [dashed] (0,-3) -- (0,5);
                \draw [dashed] (-2,-3) -- (-2,5);
			\draw [-] (0,0) -- (4,2);
			\draw [-] (4,2) -- (10,3);
			\draw [very thick] (-2,-3) -- (2,3);
			\draw [very thick] (2,3) -- (4,4);
			\draw [very thick] (4,4) -- (7,5);
                \draw [very thick] (7,5) -- (10,5);
			\draw [->] (12,2) -- (5.9,2);
			\draw (12,2) node [above, right] {$\nu_{\E_b}$};
			\draw [->] (0,6) -- (0,0.2);
			\draw (-1.5,6.5) node [above, right] {\text{break point}};
                \draw (7,7) node [above, right] {\text{$n = 12, \mu = (1^{(2)}, 0^{(10)})$}};
			\draw [->] (5.5,7.5) -- (5.5,4.55);
			\draw (5.5,7.7) node [above] {$\nu_{\E_b'}$};
			\end{tikzpicture}
\end{center}

In this case we see that $ \mu = (1^{(2)}, 0^{(10)}) $ and
\[
\E_b \simeq \OO(3/2) \oplus \OO(1/2)^2 \oplus (1/6)  \quad \quad \E_{b'} \simeq \OO(3/2)^2 \oplus \OO(1/2) \oplus \OO(1/3) \oplus \OO^3.
\]
Thus $ \E_{b'_1} \simeq \E_{b_1} \simeq \OO(3/2) $ and $ \E_{b_2} \simeq \OO(1/2)^2 \oplus (1/6) $; $ \E_{b'_2} \simeq \OO(3/2) \oplus \OO(1/2) \oplus \OO(1/3) \oplus \OO^3 $ and $\mu_1 = (0^{(2)})$; $ \mu_2 = (1^{(2)}, 0^{(8)})$. Moreover
\[
P_{b'}(\Q_p) = D^{\times}_{-3/2}(\Q_p) \times D^{\times}_{-3/2}(\Q_p) \times D^{\times}_{-1/2}(\Q_p) \times D^{\times}_{-1/3}(\Q_p) \times \GL_3(\Q_p)
\]
and 
\[
G_{b'}(\Q_p) = \GL_2 (D_{-3/2}(\Q_p)) \times D^{\times}_{-1/2}(\Q_p) \times D^{\times}_{-1/3}(\Q_p) \times \GL_3(\Q_p).
\]

\begin{proof}

Notice that if $ f : \OO(x_1) \oplus \dotsc \oplus \OO(x_r) \longrightarrow \E $ is a modification of type $ \mu_{\text{cent}} = (1, \dotsc, 1) $ then $ \E \simeq \OO(x_1 + 1) \oplus \dotsc \oplus \OO(x_r + 1) $. Thus by the same arguments in \cite[Theorem 6.3]{KH20}, we can replace the moduli space of Shtukas $ \Sht(\GL_n, b, b', \mu) $ by $ \Sht(\GL_n, b, b' \cdot b_{\text{cent}} , \mu \cdot \mu_{\text{cent}} ) $ where $ b_{\text{cent}} $ is the unique basic element in $B(\GL_n, \mu_{\text{cent}})$. Therefore we can suppose that $ \mu = (k_1, \dotsc, k_n) $ with $k_n \ge 0$.

We use the same strategy as in \cite{Han1}. First we prove that any modification $f : \E_b \longrightarrow \E_{b'}$ gives rise to a couple of modifications $ f_1 : \E_{b_1} \longrightarrow \E_{b'_1} $ and $ f_2 : \E_{b_2} \longrightarrow \E_{b'_2} $ of type $\mu_1$ and $\mu_2$ respectively. Then we prove the first statement of the proposition by examining the action of $ \widetilde{\Aut}(\E_{b'_1} \oplus \E_{b'_2}) $ on the moduli spaces of Shtukas. Finally we deduce the cohomological consequence by using the results concerning $\mathcal{J}^{U}$ proved in section 4.3 of \cite{Han1}. \\

 \textbf{Step 1} \ First we prove that any modification of type $\mu$ from $\E_{b}$ to $\E_{b'}$ gives rise to modifications from $\E_{b_1}$ to $\E_{b'_1}$ of type $\mu_1$ and from $\E_{b_2}$ to $\E_{b'_2}$ of type $\mu_2$. \\

\textit{ The case of a point} \\

Let $C$ be a complete algebraically closed field over $\Q_p$ and let $X$ be the Fargues-Fontaine  curve associated to $(C, C^{\flat})$. We denote by $\infty$ the Cartier divisor in $X$ corresponding to the untilt $C$ of $C^{\flat}$. Consider a modification 
\[
f : \E_{b | X \setminus \infty} \longrightarrow \E_{b' | X \setminus \infty}
\]
of type $\mu$ between vector bundles on $X$.

We choose a Borel subgroup $B$ of $\GL_n$. The decomposition $\E_{b'} = \E_{b'_1} \oplus \E_{b'_2} $ is compatible with the canonical reduction of $\E_{b'}$. Thus there exists a standard parabolic subgroup $\rP$ whose Levi component $\M$ is isormophic to $\GL_m \times \GL_{n-m}$ and a reduction $\E_{b', \rP}$ of $\E_{b'}$ to $\rP$ corresponding to the natural filtration coming from the above decomposition. By \cite[lemma 2.4]{CFS}, the modification $f$ and the reduction $\E_{b', \rP}$ induces a reduction $\E_{b, \rP}$ of $\E_b$ to $\rP$. We denote by $\E_{b, \M} = \E_1 \times \E_2 $ the reduction to $\M$ of $\E_{b, \rP}$. Notice that $ \mu = (k_1, \dotsc, k_n) $ with $k_n \ge 0$ then $ f $ can be extended to an injection from $ \E_b $ to $\E_{b'}$. Indeed, we view a vector bundle $\E$ as a triple $(\mathcal{E}_{ | X \setminus \infty}, \E^{\tri}_{\B+},\iota)$. Thus we have an isomorphism 
\[
f : \mathcal{E}_{b | X \setminus \infty} \longrightarrow \mathcal{E}_{b' | X \setminus \infty},
\]
and by using $\iota_b, \iota_{b'}$, it can be extended to an isomorphism 
\[
\overline{f} : \E^{\tri}_{b, \B+} \otimes_{\B+} \BdR  \longrightarrow \E^{ \tri}_{b' \B+} \otimes_{\B+} \BdR 
\]
Since $k_n \ge 0$, we see that $\overline{f}( \E^{\tri}_{b, \B+} ) \subset \E^{\tri}_{b', \B+} $. Thus it induces a map $ f' : \E^{\tri}_{b, \B+} \longrightarrow \E^{\tri}_{b', \B+} $. The couple $(f, f')$ is an injective map from $ \E_b $ to $\E_{b'}$.

In particular we see that $\E_1$ is exactly the intersection $ \E_b \cap \E_{b'_1} $. 

We will show that $\E_1 \simeq \E_{b_1}$ and $\E_2 \simeq \E_{b_2}$.  

Indeed, remark first that $\deg (\E_1) \le \deg (\E_{b_1}) $ by the property of the Harder-Narasimhan filtration of $\E_b$ \cite[Theorem A.5]{NV}. By \cite[lemma 2.6]{CFS} \cite[lemma 3.11]{Vie}, the modification $f$ induces the modifications
\[
f_1 : \E_{1 | X \setminus \infty} \longrightarrow \E_{b'_1 | X \setminus \infty} \quad \text{and} \quad f_2 : \E_{2 | X \setminus \infty} \longrightarrow \E_{b'_2 | X \setminus \infty}
\]
of type $\mu_1'$ resp. $\mu_2'$ such that $(\mu_1' \times \mu_2')_{\text{dom}} \le \mu = (k_1, \dotsc, k_n) $. In particular we have the equality
\[
\deg(\E_1) + \deg(\mu_1') = \deg(\E_{b'_1}).
\]
Moreover the condition $(\mu_1' \times \mu_2')_{\text{dom}} \le \mu$ implies that $\deg (\mu_1') \displaystyle \le \sum_{i = 1}^m k_i $. Thus we have
\[
\deg(\E_{b'_1}) = \deg(\E_1) + \deg(\mu_1') \le \deg (\E_{b_1}) + \sum_{i = 1}^m k_i = \deg (\E_{b'_1}).
\]
Therefore we have $ \deg(\E_1) = \deg (\E_{b_1}) $ and $\deg (\mu_1') \displaystyle = \sum_{i = 1}^m k_i$.

The Harder-Narasimhan filtration $ \mathcal{G}_1 \subset \dotsc \subset \mathcal{G}_{\ell} = \E_1 $ of $\E_1$ induces a filtration $ \mathcal{G}_1 \subset \dotsc \subset \mathcal{G}_{\ell} = \E_1 \subset \E_b $. Therefore, by the property of Herder-Narasimhan filtration, we see that $\nu_{\E_1}(\rank (\mathcal{G}_i)) \le \nu_{\E_b}(\rank (\mathcal{G}_i)) = \nu_{\E_{b_1}}(\rank (\mathcal{G}_i)) $ for $1 \le i \le \ell$ where $\nu_{\E_1}$, resp $\nu_{\E_b}$ are the Harder-Narasimhan polygons of $\E_1$, resp. $\E_{b}$. Thus $\nu_{\E_1}$ lies below the convex hull of the points $(\rank (\mathcal{G}_i), \nu_{\E_{b_1}}( \rank (\mathcal{F}_i))$, $ 1 \leq i \leq \ell $. Thus we deduce that $ \nu_{\E_1} $ lies below $ \nu_{\E_{b_1}} $. The same argument together with the equality $ \deg(\E_1) = \deg (\E_{b_1}) $ implies that $ \nu_{\E_2} $ lies below $\nu_{\E_{b_2}}$. However, by \cite[corollary 2.9]{MC}, the polygon $\nu$ forming from the slopes of $\nu_{\E_1}$ and $\nu_{\E_2}$ lies above the Newton polygon $\nu_{b}$. Therefore, by the uniqueness of Harder-Narasimhan filtration, we deduce that $ \E_1 \simeq \E_{b_1} $ and $ \E_2 \simeq \E_{b_2} $. Now we view a vector bundle $\E$ as a triple $(\mathcal{E}_{ | X \setminus \infty}, \E^{n,\tri}_{\B+},\iota)$ and notice that in this case $\B+$ is a discrete valuation ring. Then by  applying \cite[lemma 3.2]{Han1} to the modifications $f, f_1, f_2$; we deduce that $ \mu $ is a direct summand of $\mu_1$ and $\mu_2$. It implies that $ f_1 $, respectively, $f_2$ are of type $ \mu_1 $ and $\mu_2$ respectively.

\textit{ The general case} 

Let $S = \Spa (A, A^+)$ be an affinoid perfectoid space over $\Q_p$ and consider the vector bundles $ \E_{b'} = \E_{b'_1} \oplus \E_{b'_2} $ over the Fargues-Fontaine curve $X_{S}$ associated to $S$. Then $\E_{b'_1}$ is a saturated sub-bundle of $\E_{b'}$. Let $ f : \E_b \longrightarrow \E_{b'} $ be a modification of type $\mu$. Since $ \mu = (k_1, \dotsc, k_n) $ with $k_n \ge 0$, we deduce that that $f$ can be extended to an injection from $\E_b$ to $\E_{b'}$. Consider the intersection $\E_1 = \E_{b'_1} \cap \E_b $. By proceeding as in \cite[\S 3.2]{Han1}, by using the equivalence between the category of $B$-pairs over $A$ and that of vector bunbles on the Fargues-Fontaine curve $X_{S}$ (\cite{KL}) and also proposition $2.4$ in \cite{Han1}, we can show that $ \E_1 \simeq \E_{b_1} $ and the quotient $ \E_{b} / \E_{b_1} $ is isormorphic to $\E_{b_2}$. Moreover the modification $f$ induces modifications 
\[
f_1 : \E_{b_1} \longrightarrow \E_{b'_1}, \quad \quad f_2 : \E_{b_2} \longrightarrow \E_{b'_2},
\]
of type $\mu_1$ respectively $\mu_2$. \\

\textbf{Step 2} Analyzing the effect of the actions of $ \widetilde{\G}^0_{b'} / \widetilde{\G}^0_{b'_1} \times \widetilde{\G}^0_{b'_2} $ and of $ \underline{\G_b(\Q_p)} $. \\

 We consider an intermediate space $\Sht(\rP_b, b, b', \mu)$ of $\rP_b$-shtukas as in \cite[section 1.5]{Han1}. Let $f$ be a type $\mu$ modification
\[
f : \E_{b} \longrightarrow \E_{b'},
\]
then the decomposition $ \E_{b'} = \E_{b'_1} \oplus \E_{b'_2} $ induces a filtration $ 0 \subset \E_{ f} \subset \E_b $. We can describe the points of $\Sht(\rP_b, b, b' \mu)$ as the modifications $f$ such that the filtration $ 0 \subset \E_{f} \subset \E_b $ is exactly the filtration $ 0 \subset \E_{b_1} \subset \E_{b} $. If $\lambda_{s'} > \lambda_{s' + 1}$ then $\rP_b(\Q_p) \simeq \G_b(\Q_p)$ and then by the uniqueness of Harder-Narasimhan filtration, $\Sht(\rP_b, b, b', \mu) \simeq \Sht(\GL_n, b, b' \mu) $. However, if $\lambda_{s'} = \lambda_{s' + 1}$ then $\rP_b(\Q_p)$ is a proper parabolic subgroup of $\G_b(\Q_p)$.

The same analyse in \cite{Han1} goes through in our situation and we have an isomorphism commuting with all additional structures
\[
\Sht(\GL_n, b, b', \mu ) = \Sht(\rP_b, b, b', \mu) \times^{\underline{\rP_b}} \underline{\G}_b.
\]

Now we show that 
\[
\displaystyle \Sht(\rP_b, b, b', \mu) \simeq \Big( \Sht(\GL_m, b_1, b'_1, \mu_1) \times_{\Spd \breve{\Q}_p} \Sht(\GL_{n-m}, b_2, b'_2, \mu_2) \times_{\Spd \breve{\Q}_p} \mathcal{J}^{U} \Big).
\]   

Let $ \displaystyle f : \E_b \longrightarrow \E_{b'} $ be a modification of type $\mu$ in $\Sht(\rP_b, b, b', \mu)$. Thus by step $1$, it gives rise to modifications
\[
f_1 : \E_{b_1} \longrightarrow \E_{b'_1}, \quad \quad f_2 : \E_{b_2} \longrightarrow \E_{b'_2},
\]
of type $\mu_1$ respectively $\mu_2$. Moreover, the direct sum $ f' = f_1 \oplus f_2 $ is a modification from $\E_b$ to $\E_{b'}$ of type $\mu$. We show that there exists $ u \in \mathcal{J}^{U} \simeq \widetilde{\G}^0_{b'} / \widetilde{\G}^0_{b'_1} \times \widetilde{\G}^0_{b'_2} $ such that $ f = u \circ f' $.

We know that $ f, f' \in \Hom( \E_b, \E_{b'} ) $. If we write $ \E_b \simeq \E_1 \oplus \E_2 $ and $ \E_{b'} = \E_{b'_1} \oplus \E_{b'_2} $ then $ f $ can be written by
\[
\begin{pmatrix}
    f_1 & t \\
    0 & f_2
\end{pmatrix},
\]
where $t = (t_e, t')$ is a map from $ \E_2 $ to $\E_{b'_1}$. We show that there exists a map $ h = (h_e, h') : \E_{b'_2} \longrightarrow \E_{b'_1} $ such that $ t = h \circ f_2 $. We use again the description of a vector bundle $\E$ by a triple $(\mathcal{E}_{ | X \setminus \infty}, \E^{\tri}_{\B+},\iota)$. Since $ f_2 : \mathcal{E}_{2 | X \setminus \infty} \simeq \mathcal{E}_{b'_2 | X \setminus \infty} $ is an isomorphism, it induces a map $ h_e :  \mathcal{E}_{b'_2 | X \setminus \infty} \longrightarrow \mathcal{E}_{b'_1 | X \setminus \infty} $ such that $ t_e = h_e \circ f_2 $. By using the map $ \iota_{b'_2} $ and $\iota_{b'_1}$, we have an isomorphism
\[
\overline{h}_e := \iota_{b'_1} \circ h_e \circ \iota^{-1}_{b'_2} : \E^{\tri}_{b'_2, \B+} \otimes_{\B+} \BdR  \longrightarrow \E^{ \tri}_{b'_1 \B+} \otimes_{\B+} \BdR.
\]
Similarly, we can consider the maps $ \overline{t}_e $ and $ \overline{f}_2 $ and we have the identity $ \overline{t}_e = \overline{h}_e \circ \overline{f}_2 $ such that the restriction to $\E^{\tri}_{2, \B+}$ gives rise to an identity:
\[
t' : \E^{\tri}_{2, \B+} \xrightarrow{f'_2} f'_2(\E^{\tri}_{2, \B+}) \xrightarrow{\overline{h}^*_e} \E^{\tri}_{b'_1, \B+},
\]
where $\overline{h}^*_e$ is the restriction of $\overline{h}_e$ to $f'_2(\E^{\tri}_{2, \B+})$. Note that $f_2$ is a modification of type $ (k_{m+1}, \dotsc, k_n) $ with $ k_n \ge 0 $ then $ \xi^{k_{m+1}} \E^{\tri}_{b'_2, \B+} \subset f'_2(\E^{\tri}_{2, \B+}) \subset \E^{\tri}_{b'_2, \B+} $. Note that $ \mu_1 = (k_1, \dotsc, k_m) $ then we have $\overline{h}^*_e ( f'_2(\E^{\tri}_{2, \B+}) ) \subset \xi^{k_m} \E^{\tri}_{b'_1, \B+}$. Since $ k_{m} \ge k_{m+1} $, we deduce that $ \overline{h}^*_e (\E^{\tri}_{b'_2, \B+}) \subset \E^{\tri}_{b'_1, \B+} $. Therefore the couple $ (h_e, \overline{h}_e) $ gives rise to a map $ h : \E_{b'_2} \longrightarrow \E_{b'_1} $ such that $t = h \circ f_2$. Now if we choose the map $u$ of the form 
\[
\begin{pmatrix}
    \Id_{\E_{b'_1}} & h \\
    0 & \Id_{\E_{b'_2}}
\end{pmatrix},
\]
with respect to $ \E_{b'} = \E_{b'_1} \oplus \E_{b'_2} $ then we have $ f = u \circ f' $. We conclude that there is an isomorphism commuting with all additional structures
\[
\displaystyle \Sht(\rP_b, b, b', \mu) \simeq \Big( \Sht(\GL_m, b_1, b'_1, \mu_1) \times_{\Spd \breve{\Q}_p} \Sht(\GL_{n-m}, b_2, b'_2, \mu_2) \times_{\Spd \breve{\Q}_p} \mathcal{J}^{U} \Big).
\]  

Finally, we deduce the cohomological consequences by arguing as in the proof of \cite[Proposition 4.11]{HI} and by using the computation of the cohomology of $\mathcal{J}^{U}$ (page $91$, before lemma 11.1 in \cite{Ham1}).
\end{proof}   

\begin{remark} \phantomsection \label{itm : opposite parabolic group}
By (\ref{decomposition of Kottwitz' set}), there exists a unique standard parabolic subgroup $\rP$ of $\GL_n$ such that the reduction of $\E_b$ to the Levi factor $\M$ of $\rP$ is semi-stable and the Newton point $\nu_{\E_b}$ of $\E_b$ is in $X_*(\rP)^+_{\Q}$. Moreover there is a twisting $\xi_b$ between $\G_b$ and $\M$ realizing $\G_b$ as an inner form of $\M$. There also exists a unique 
standard parabolic subgroup $\rQ$ whose Levi factor is $\rN$ such that $\nu_b \in X_*(\rQ)^+_{\Q}$. Since $\nu_b = - (\nu_{\E_b})_{\text{dom}}$, we see that $\rQ$ is conjugated to the opposite of $\rP$. Similarly, there exists the objects $\M', \rP', \rQ'$ and $\xi_{b'}$ for $\E_{b'}$. Moreover, the filtration $ 0 \subset \E_{b'_1} \subset \E_{b'} $ determines a reduction of $\E_{b'}$ to a standard parabolic subgroup $\rR'$ containing $\rQ'$. Note that $ \rP_b $ is a parabolic subgroup of $\G_b$. Since we can suppose that the filtration $ 0 \subset \E_{b_1} \subset \E_b $ corresponds to the filtration $ 0 \subset \E_{b'_1} \subset \E_{b'} $ (by the modifications in $\Sht(\rP_b, b, b' \mu)$), if we denote by $\widetilde{\rP}_b$ the parabolic subgroup of $\M$ corresponding to $\rP_b$ by the twisting $\xi_b$ then $\widetilde{\rP}_b$ is standard in $\M$, thus $\widetilde{\rP}_b$ is conjugated to an opposite parabolic subgroup in $\rN$.    
\end{remark}

\begin{lemma} \phantomsection \label{itm : kappa and modulus}
    For $b \in B(\GL_n)$, we have $ \kappa_b = \delta^{-1}_{b} $.
\end{lemma}
\begin{proof}

It is a special case of \cite[Corollary 1.9(2)]{HI}.

\end{proof}

\subsection{Some preliminary computations} \textbf{}

Let $\phi = \phi_1 \oplus \dotsc \oplus \phi_r$ be an $L$-parameter satisfying the conditions of theorem \ref{itm : main theorem}. We know that $\Irr(S_{\phi}) \displaystyle \simeq \prod_{i=1}^r \Z $. Let $\chi = (d_1, \dotsc, d_r)$ be a character in $ \Irr(S_{\phi})^+ := \{ (t_1, \dotsc, t_r) \in \Irr(S_{\phi}) \ | \ t_j \ge 0 \ \forall \ 1 \le j \le r \}$.  Suppose that $ \E_{b_{\chi}} \simeq \OO(\lambda_1)^{m_1} \oplus \dotsc \oplus \OO(\lambda_k)^{m_k} $ where $ \lambda_1 > \dotsc > \lambda_k $. For each $ 1 \le \ell \le k $ we denote by $\I(\ell) \subset \{ 1, 2, \dotsc, r \}$ the set of indexes $j$ such that $d_j/n_j = \lambda_{\ell}$ where $ n_j := \dim \phi_j $ and $\displaystyle \sum^r_{j=1} n_j = n $.

 Let $ b \in B(\GL_n) $ be an element such that $ \E_b \simeq \mathcal{G} \oplus \OO(\lambda_i)^{m'_i} \oplus \OO(\lambda_{i+1})^{m_{i+1}} \oplus \dotsc \oplus \OO(\lambda_k)^{m_k} $ where $1 \le i \le k$ and $0 < m'_i \le m_i$ are some fixed integers and $\mathcal{G}$ is a vector bundle such that the slopes of its Newton polygon are strictly bigger than the smallest slope of $\OO(\lambda_1)^{m_1} \oplus \dotsc \oplus \OO(\lambda_{i-1})^{m_{i-1}} \oplus \OO(\lambda_i)^{m_i - m'_i}$.

Let $ b' \in B(\GL_n) $ be an element such that $ \E_{b'} \simeq \OO(\lambda_1)^{m_1} \oplus \dotsc \oplus \OO(\lambda_{i-1})^{m_{i-1}} \oplus \OO(\lambda_i)^{m_i} \oplus \mathcal{H} $ where $1 \le i \le k$ is a fixed integer as above and $\mathcal{H}$ is a vector bundle such that the biggest slope of its Newton polygon is smaller than $\lambda_i$ and strictly bigger than $\lambda_{i+1}$.

We define $D := \deg(\E_{b_{\chi}})$ and $D_1:= \deg(\E_b)$ and $D_2:= \deg(\E_{b'})$. 
Let $\chi_j \in \Irr(S_{\phi})$ be the element $(d_1, \dotsc, d_r)$ where $d_j = 1$ and $d_a = 0$ for $ 1 \le a \neq j \le r $. In the following lemmas, we analyse the complexes $i_b^* C_{\chi_j} \star \mathcal{F}_{\chi} $ and $i_{b'}^* C_{\chi_j} \star \mathcal{F}_{\chi} $ where $\mathcal{F}_{\chi}$ is defined as in \S \ref{itm : shape of the supports}.

\begin{lemma} \phantomsection \label{itm : first computation}
    Suppose that $D_1 - D = 1$ and let $\chi' = (a_1, \dotsc, a_r) \in \Irr(S_{\phi})^+$ be a character such that $\displaystyle \sum^r_{j=1} a_j = 1$. We define $S_+ := \{ 1 \le j \le r \ | \ a_j > 0 \} $ and $S_0 := \{ 1, 2, \dotsc, k \} \setminus S_+ $, in particular $|S_+| = 1$. Suppose that theorem $\ref{itm : main theorem}$ is true for every $L$-parameters $\displaystyle \phi' = \bigoplus_{j \in J \subsetneq \{ 1, \dotsc, r \}} \phi_j$, then:
    \begin{enumerate}
        \item If $ \displaystyle S_+ \subset \big( \bigcup_{i < j \le k} \I(j) \big) $ then
        \[
        i^*_bC_{\chi'} \star (i_{b_{\chi}!} \delta_{\chi}^{-1/2} \otimes \pi_{\chi} ) \simeq 0
        \]
        where $\pi_{\chi}$ is the irreducible representation of $\G_{b_{\chi}}(\Q_p)$ whose $L$-parameter is given by $\phi_{\chi}$ as in \S \ref{itm : shape of the supports} (whose post-composition with $ \widehat{\G^*_{b_{\chi}}} \hookrightarrow \widehat{\GL}_n $ is given by $\phi$).
        \item For $ \displaystyle S_+ = \{ c \} \subset \big( \bigcup_{1 \le j \le i} \I(j) \big) $ we have
        \[
        i^*_bC_{\chi'} \star \mathcal{F}_{\chi} \simeq
        \begin{cases}
         i_b^* \mathcal{F}_{\chi' \otimes \chi} \text{ \ if  $b = b_{\chi' \otimes \chi}$} \\
         0 \quad \quad \text{ \ if  $b \neq b_{\chi' \otimes \chi}$.}
        \end{cases}
         \]
    \end{enumerate}
\end{lemma}

Similarly we have 

\begin{lemma} \phantomsection \label{itm : second computation}
    Suppose that $D_2 - D = 1$ and let $\chi' = (a_1, \dotsc, a_r) \in \Irr(S_{\phi})^+$ be a character such that $\displaystyle \sum^r_{j=1} a_j = 1$. We define $S_+ := \{ 1 \le j \le r \ | \ a_j > 0 \} $ and $S_0 := \{ 1, 2, \dotsc, k \} \setminus S_+ $. Suppose that theorem $\ref{itm : main theorem}$ is true for every $L$-parameters $ \displaystyle \phi' = \bigoplus_{j \in J \subsetneq \{1, \dotsc, r\} } \phi_j $, then:
    \begin{enumerate}
        \item If $ \displaystyle S_+ \subset \big( \bigcup_{1 \le j < i} \I(j) \big) $ then
        \[
        i^*_bC_{\chi'} \star (i_{b_{\chi}!} \delta_{\chi}^{-1/2} \otimes \pi_{\chi} ) \simeq 0
        \]
        where $\pi_{\chi}$ is the irreducible representation of $\G_{b_{\chi}}(\Q_p)$ whose $L$-parameter is given by $\phi_{\chi}$ as in \S \ref{itm : shape of the supports} (whose post-composition with $ \widehat{\G^*_{b_{\chi}}} \hookrightarrow \widehat{\GL}_n $ is given by $\phi$).
        \item For $ \displaystyle S_+ = \{ c  \} \subset \big( \bigcup_{i \le j \le k} \I(j) \big) $ we have
        \[
        i^*_{b'}C_{\chi'} \star \mathcal{F}_{\chi} \simeq
        \begin{cases}
         i_{b'}^* \mathcal{F}_{\chi' \otimes \chi} \text{ \ if $b' = b_{\chi' \otimes \chi}$} \\
         0 \quad \quad \text{ \ if $b' \neq b_{\chi' \otimes \chi}$.}
        \end{cases}
         \]
    \end{enumerate}
\end{lemma}

Remark that part $(2)$ of the lemmas does not completely describe $C_{\chi'} \star \mathcal{F}_{\chi}$ since we assume that $b$ satisfies some constraints related to $b_{\chi}$ in our running hypothesis. For example, if $\OO(\lambda_k)$ is not a direct factor of $\E_b$ then part $(2)$ of lemma \ref{itm : first computation} does not tell us anything about $i^*_{b}C_{\chi'} \star \mathcal{F}_{\chi}$. If we suppose $D - D_1 = 1$ and $D - D_2 = 1$ as well as all the $a_j$'s are non-positive then we also get similar results. 
\begin{proof}
    We only prove lemma \ref{itm : first computation}, the proof for lemma \ref{itm : second computation} is similar and less complicated. 

    We suppose first that $m'_i = m_i$ and denote by $\mu$ the cocharacter $(1, 0^{(n-1)})$. Thus we have $r_{\mu} = \Std \GL_n$ and since all the $a_j$'s are non negative, we see that $ \Hom_{S_{\phi}} (\chi', r_{\mu} \circ \phi) $ is non trivial. Note that in this case $i > 1$ and the triple $ (b_{\chi}, b, \mu) $ satisfies the condition in proposition \ref{generalized Boyer's trick}. Therefore we have
\[
\Sht(\GL_n, b_{\chi}, b, \mu) = \Sht(\GL_m, \widetilde{b}, b_1, \mu_1) \times_{\Spd \breve{\Q}_p} \Sht(\GL_{n-m}, b_2, b_2, \mu_2) \times_{\Spd \breve{\Q}_p} \mathcal{J}^{U} 
\]
where
\begin{enumerate}
    \item[$\bullet$] $\E_{\widetilde{b}} \simeq \OO(\lambda_1)^{m_1} \oplus \dotsc \oplus \OO(\lambda_{i-1})^{m_{i-1}} $; $\E_{b_2} \simeq \OO(\lambda_i)^{m_i} \oplus \dotsc \oplus \OO(\lambda_{k})^{m_{k}} $ and $\E_{b_1} \simeq \mathcal{G}$. In particular $\E_{b_{\chi}} \simeq \E_{\widetilde{b}} \oplus \E_{b_2}$ and $\E_b \simeq \E_{b_1} \oplus \E_{b_2}$.
    \item[$\bullet$] $\mu_1 = (1, 0^{(m-1)}) $ and $\mu_2 = (0^{(n-m)})$ where $m = \rank (\E_{b_1})$.
    \item[$\bullet$] $\mathcal{J}^{U}$ is the unipotent diamond in group $ \widetilde{\G}_{b}^0 / \big( \widetilde{\G}_{b_1}^0 \times \widetilde{\G}_{b_2}^0 $ \big).
\end{enumerate}

Recall that for $\chi$, we have an associated representation $\pi_{\chi}$ of $\G_{b_{\chi}}(\Q_p) = \displaystyle \prod_{1\le j \le k} \Aut_{\OO^{m_j}(\lambda_j)}(\Q_p)$ whose $L$-parameter is given by $\displaystyle \prod_{1\le j \le k} \bigoplus_{h \in \I(j)} \phi_h $. If we write $\pi_{\chi} = \widetilde{\pi} \boxtimes \pi_2$ as representation of $ \G_{b_{\chi}}(\Q_p) = \G_{\widetilde{b}}(\Q_p) \times \G_{b_2}(\Q_p) $, then we have an isomorphism of $\G_b(\Q_p) \times W_{\Q_p}$-modules
\begin{align*} 
   & \quad \ R\Gamma_c(\GL_n, b_{\chi}, b, \mu)[\delta^{1/2}_{\chi} \otimes \pi_{\chi} ] \\
   & \simeq R\Gamma_c(\GL_n, b_{\chi}, b, \mu) \otimes^{\mathbb{L}}_{\mathcal{H}(\G_{b_{\chi}})} \delta^{1/2}_{\chi} \otimes \pi_{\chi} \\
    & \simeq \big( R\Gamma_c( \Sht(\GL_m, \widetilde{b}, b_1, \mu_1) \times_{\Spd \breve{\Q}_p} \Sht(\GL_{n-m}, b_2, b_2, \mu_2)) \otimes \kappa \big)\otimes^{\mathbb{L}}_{\mathcal{H}(\G_{b_{\chi}})} \delta^{1/2}_{\chi} \otimes \pi_{\chi} [h-2d]\otimes|\cdot|^{\tfrac{h}{2}} \\
    & \simeq \kappa \otimes \big( R\Gamma_c(\GL_m, \widetilde{b}, b_1, \mu_1) \otimes^{\mathbb{L}}_{\mathcal{H}(\G_{\widetilde{b}})} \delta^{1/2}_{\chi | \G_{\widetilde{b}}} \otimes {\widetilde{\pi}} \boxtimes R\Gamma_c(\GL_{n-m}, b_2, b_2, \mu_2) \otimes^{\mathbb{L}}_{\mathcal{H}(\G_{b_2})} \delta^{1/2}_{\chi | \G_{b_2}} \otimes \pi_2 \big) [h-2d] \otimes|\cdot|^{\tfrac{h}{2}},
\end{align*}
where the first isomorphism follows from the definition and where
\begin{enumerate}
    \item [$\bullet$] $\delta_{\chi} = \delta_{\chi | \G_{\widetilde{b}}} \boxtimes \delta_{\chi | \G_{b_2}} $ as characters of $ \G_{b_{\chi}}(\Q_p) = \G_{\widetilde{b}}(\Q_p) \times \G_{b_2}(\Q_p) $.
    \item [$\bullet$] $|\cdot|$ denotes the norm character of $W_{\Q_p}$ under the geometric normalization of local class field theory.
    \item [$\bullet$] $h = \langle \mu, 2\rho \rangle - \langle \mu_1, 2\rho_m \rangle$ and the twist $|\cdot|^{\tfrac{h}{2}}$ as well as the shift $[h]$ come from the difference in the normalization of the sheaves $\mathcal{S}_{\mu}$ and $\mathcal{S}_{\mu_1}$ in the definition of $R\Gamma_c(\GL_n, b_{\chi}, b, \mu)$ and $R\Gamma_c(\GL_m, \widetilde{b}, b_1, \mu_1)$.
    \item [$\bullet$] $d = d_b - d_{b_1} - d_{b_2}$ and the shift $[-2d]$ comes from the cohomology of the automorphism group $\mathcal{J}^U$. Here $ d_b = \langle \nu_b, 2\rho \rangle $, $ d_{b_1} = \langle \nu_{b_1}, 2\rho_m \rangle$ and $ d_{b_2} = \langle \nu_{b_2}, 2\rho_{n-m} \rangle$ and $\kappa = \kappa_b \otimes (\kappa^{-1}_{b_1} \times \kappa^{-1}_{b_2}) = \delta_b^{-1} \otimes (\delta_{b_1} \times \delta_{b_2}) $ is a character of $\G_b(\Q_p)$ (lemma \ref{itm : kappa and modulus}).
\end{enumerate} 

Therefore, we have an isomorphism of $\G_b(\Q_p) \times W_{\Q_p}$-modules
\begin{align*}
    & \quad \ R\Gamma_c(\GL_n, b_{\chi}, b, \mu)[\delta^{1/2}_{\chi} \otimes \pi_{\chi} ] \\
    & \simeq \kappa \otimes \big( R\Gamma_c(\GL_m, \widetilde{b}, b_1, \mu_1)[\delta^{1/2}_{\chi | \G_{\widetilde{b}}} \otimes {\widetilde{\pi}}] \boxtimes R\Gamma_c(\GL_{n-m}, b_2, b_2, \mu_2)[\delta^{1/2}_{\chi | \G_{b_2}} \otimes \pi_2 ]\big)[h-2d]\otimes|\cdot|^{\tfrac{h}{2}}. \numberthis \phantomsection \label{itm : Boyertrick-cohomology}
\end{align*}

Besides, by lemmas \ref{shimhecke} and \ref{itm : kappa and modulus}, we have 
\begin{equation} \phantomsection \label{itm : equa}
    i^*_b \T_{\mu}(i_{b_{\chi} !} \delta^{-1/2}_{\chi} \otimes \pi_{\chi}) \simeq R\Gamma_c(\GL_n, b_{\chi}, b, \mu)[\delta^{1/2}_{\chi} \otimes \pi_{\chi}][2d_{b_{\chi}}].
\end{equation}

Since $\T_{\Id}$ is the identity functor, lemma \ref{shimhecke} implies that
\begin{equation} \phantomsection \label{itm : equa3}
    R\Gamma_c(\GL_{n-m}, b_2, b_2, \mu_2)[\delta^{1/2}_{\chi | \G_{b_2}} \otimes \pi_2 ] \simeq \delta^{1/2}_{\chi | \G_{b_2}} \otimes \pi_2 \otimes \kappa_{b_2}[-2d_{b_2}],
\end{equation}
and in particular, $W_{\Qp}$ acts trivially on $R\Gamma_c(\GL_{n-m}, b_2, b_2, \mu_2)[\delta^{1/2}_{\chi | \G_{b_2}} \otimes \pi_2]$. \\

Remark that as $S_{\phi} \times W_{\Q_p}$-representations, we have an identification
\[
r_{\mu} \circ \phi = \bigoplus_{\xi \in \Irr(S_{\phi})} \xi \boxtimes \sigma_{\xi} 
\]
where $\sigma_{\xi}$ is the $W_{\Q_p}$-representations $\Hom_{S_{\phi}}(\xi, r_{\mu} \circ \phi)$. Since the Fargues-Scholze's parameter (with respect to $\Bun_n$) of $i_{b_{\chi} !} \delta^{-1/2}_{\chi} \otimes \pi_{\chi}$ is given by $\phi$, we have the following decomposition of Hecke operator by proposition \ref{itm : fundamental decomposition of Hecke operator}
\begin{equation*}
    \T_{\mu}(i_{b_{\chi} !} \delta^{-1/2}_{\chi} \otimes \pi_{\chi}) = \bigoplus_{\xi \in \Irr(S_{\phi})} C_{\xi} \star (i_{b_{\chi} !} \delta^{-1/2}_{\chi} \otimes \pi_{\chi}) \boxtimes \sigma_{\xi}. 
\end{equation*}

Remark that the condition on $\phi$ allows us to understand $C_{\xi} \star (i_{b !} \delta^{-1/2}_{\chi} \otimes \pi_{\chi}) $ by identifying the $W_{\Q_p}$-action of $\T_{\mu}(i_{b_{\chi} !} \delta^{-1/2}_{\chi} \otimes \pi_{\chi})$. More precisely, for any $d \in B(\GL_n)$, we have an identification of the cohomology groups as $\G_d(\Q_p) \times W_{\Q_p}$-representations:
\[
\mathrm{H}^*\big(i_d^*\T_{\mu}(i_{b_{\chi} !} \delta^{-1/2}_{\chi} \otimes \pi_{\chi})\big) \simeq \bigoplus_{\xi \in \Irr(S_{\phi})} \mathrm{H}^* \big( i_d^* C_{\xi} \star (i_{b_{\chi} !} \delta^{-1/2}_{\chi} \otimes \pi_{\chi}) \big) \boxtimes \sigma_{\xi}.
\]

Hence, by taking the $[\sigma_{\xi}]$-isotypical part of each cohomology group $\mathrm{H}^i\big(i_d^*\T_{\mu}(i_{b_{\chi} !} \delta^{-1/2}_{\chi} \otimes \pi_{\chi})\big)$, one can compute $\mathrm{H}^i \big( i_d^* C_{\xi} \star (i_{b_{\chi} !} \delta^{-1/2}_{\chi} \otimes \pi_{\chi}) \big)$. We will see that the complex $i_d^* C_{\xi} \star (i_{b_{\chi} !} \delta^{-1/2}_{\chi} \otimes \pi_{\chi})$ is concentrated in $1$ degree and then we can recover $i_d^* C_{\xi} \star (i_{b_{\chi} !} \delta^{-1/2}_{\chi} \otimes \pi_{\chi})$ from its cohomology groups.

Note that the (usual) $L$-parameter of $\pi_2$ is given by $\displaystyle \bigoplus_{ t \in \bigcup_{i \le j \le k} \I(j)} \phi_t $ and the $L$-parameter of $\widetilde{\pi}$ is given by $ \displaystyle \bigoplus_{ t \in \bigcup_{1 \le j < i} \I(j)} \phi_t $. The observation is that if we regard the cohomology groups $R\Gamma_c(\GL_n, b_{\chi}, b, \mu)[\delta^{1/2}_{\chi} \otimes \pi_{\chi}]$ as $W_{\Q_p}$-module then by (\ref{itm : Boyertrick-cohomology}), its irreducible sub-quotients are isomorphic to an unramified twist of $\phi_t$ for $ \displaystyle t \in \bigcup_{1 \le j < i} \I(j) $ since $W_{\Q_p}$ acts trivially on $ R\Gamma_c(\GL_{n-m}, b_2, b_2, \mu_2)[\delta^{1/2}_{\chi | \G_{b_2}} \otimes \pi_2] $ by equation (\ref{itm : equa3}). 

However if $ \displaystyle S_+ = \{ c \} \subset \big( \bigcup_{i \le j \le k} \I(j) \big) $ then $\sigma_{\chi} \simeq \phi_c$.  
Thus the condition on the irreducible factors of $\phi$ in the beginning of \S \ref{conditions on L-parameters} allows us to conclude that the $[\sigma_{\chi}]$-isotypical part of $R\Gamma_c(\GL_n, b_{\chi}, b, \mu)[\delta^{1/2}_{\chi} \otimes \pi_{\chi}]$ vanishes. Therefore $i^*_bC_{\chi'} \star (i_{b_{\chi}!} \delta_{\chi}^{-1/2} \otimes \pi_{\chi} ) \simeq 0$, hence the first point of the lemma. \\

We are going to prove the second point. We see that $ r_{\mu} \circ \phi = \displaystyle \bigoplus_{ 1 \le j \le r} \phi_j $. Hence 
\[
\T_{\mu}(i_{b_{\chi} !} \delta^{-1/2}_{\chi} \otimes \pi_{\chi}) = \bigoplus^r_{j = 1} C_{\chi_j} \star (i_{b_{\chi} !} \delta^{-1/2}_{\chi} \otimes \pi_{\chi}) \boxtimes \phi_j,
\]
and we only need to consider the case $\chi' = \chi_{c}$ for some $c \in \big( \bigcup_{1 \le j < i} \I(j) \big)$. We need to compute the $[\phi_{c}]$-isotypical part of the cohomology groups $\mathrm{H}^*(R\Gamma_c(\GL_m, \widetilde{b}, b_1, \mu_1)[\delta^{1/2}_{\chi | \G_{\widetilde{b}}} \otimes {\widetilde{\pi}}] \otimes |\cdot|^{\tfrac{h}{2}} )$.

Let $\widetilde{\phi}$ denote the $L$-parameter $ \displaystyle \bigoplus_{ t \in \bigcup_{1 \le j < i} \I(j)} \phi_t $ and let $\widetilde{\chi}$, respectively $ \widetilde{\chi}_c $ be the element in $ \displaystyle \Irr(S_{\widetilde{\phi}}) \simeq \prod_{t \in \bigcup_{1 \le j < i} \I(j)} \Z $ obtained from $\chi$, respectively $\chi_c$ by removing all the indexes corresponding to $\phi_t$, where ${t \in \bigcup_{i \le j < k} \I(j)}$. Thus we can check that $ \widetilde{b} = b_{\widetilde{\chi}} $ and $\widetilde{\pi} = \pi_{\widetilde{\chi}}$. Hence by lemmas \ref{shimhecke} and \ref{itm : kappa and modulus}, we have 
\begin{equation} \phantomsection \label{itm : equa2}
    i^*_{b_1} \T_{\mu_1}(i_{{\widetilde{b}} !} \delta^{1/2}_{\chi | \G_{\widetilde{b}}} \otimes \pi_{\widetilde{\chi}} \otimes \delta_{\widetilde{b}}^{-1}) \simeq R\Gamma_c(\GL_m, \widetilde{b}, b_1, \mu_1)[ {\delta^{1/2}_{\chi | \G_{\widetilde{b}}} \otimes \pi_{\widetilde{\chi}}}][2d_{\widetilde{b}}],
\end{equation}
where $ d_{\widetilde{b}} = \langle \nu_{\widetilde{b}}, 2\rho_m \rangle$.
 
We have $\delta^{1/2}_{\chi | \G_{\widetilde{b}}} \otimes \pi_{\widetilde{\chi}} \otimes \delta_{\widetilde{b}}^{-1} = \delta^{1/2}_{\chi | \G_{\widetilde{b}}} \delta_{\widetilde{b}}^{-1/2} \otimes \pi_{\widetilde{\chi}} \otimes \delta_{\widetilde{b}}^{-1/2}$ and the twist by the character $\delta^{1/2}_{\chi | \G_{\widetilde{b}}} \delta_{\widetilde{b}}^{-1/2}$ corresponding to the twist by $|\cdot|^{\tfrac{m-n}{2}}$ for the $L$-parameter of $\pi_{\widetilde{\chi}} \otimes \delta_{\widetilde{b}}^{-1/2}$. Thus the (Fargues-Scholze) $L$-parameter of $\delta^{1/2}_{\chi | \G_{\widetilde{b}}} \otimes \pi_{\widetilde{\chi}} \otimes \delta_{\widetilde{b}}^{-1}$ with respect to $\Bun_m$ is given by $\widetilde{\phi} \otimes |\cdot|^{\tfrac{m-n}{2}}$. 

By theorem $\ref{itm : main theorem}$, applying to the $L$-parameter $\widetilde{\phi} \otimes |\cdot|^{\tfrac{m-n}{2}} $ we deduce that 
\[
i^*_{b_1} C_{\widetilde{\chi}_c} \star (i_{b_{\widetilde{b}} !} \delta^{1/2}_{\chi | \G_{\widetilde{b}}} \otimes \pi_{\widetilde{\chi}} \otimes \delta^{-1}_{\widetilde{b}} ) [-d_{\widetilde{b}}] \simeq
\begin{cases}
       i^*_{b_1} (\delta^{1/2}_{b | \G_{b_1}} \delta^{-1/2}_{b_1}) \otimes \mathcal{F}_{\widetilde{\chi}_c \otimes \widetilde{\chi}}  \text{ \ if $b_1 = b_{ \widetilde{\chi}_c \otimes \widetilde{\chi} }$} \\
         0 \quad \quad \text{ \  otherwise.}
        \end{cases}
\]

Note that the twist by $(\delta^{1/2}_{b | \G_{b_1}} \delta^{-1/2}_{b_1})$ comes from the twist by $|\cdot|^{\tfrac{m-n}{2}}$ in the $L$-parameter side. 

We can check from the constructions that $ \E_{b_{ \widetilde{\chi}_c \otimes \widetilde{\chi} }} \oplus \E_{b_2} \simeq \E_{b_{\chi_c \otimes \chi}} $. We see that if $ b \neq b_{\chi_c \otimes \chi} $ then $b_1 \neq b_{ \widetilde{\chi}_c \otimes \widetilde{\chi} } $ and the $[\phi_c]$-isotypical part of $ \mathrm{H}^* (R\Gamma_c(\GL_m, \widetilde{b}, b_1, \mu_1)[\delta^{1/2}_{\chi | \G_{\widetilde{b}}} \otimes {\widetilde{\pi}}] \otimes |\cdot|^{\tfrac{h}{2}})$ is empty. Thus by equality (\ref{itm : Boyertrick-cohomology}), the $[\phi_c]$-isotypical part of $\mathrm{H}^* (R\Gamma_c(\GL_n, b_{\chi}, b, \mu)[\delta^{1/2}_{\chi} \otimes \pi])$ is empty and then $i^*_bC_{\chi_c} \star \mathcal{F}_{\chi} \simeq 0$.

Note that in this case by using the fact that all the slopes of the Newton polygon of $\E_{b_{\widetilde{\chi}}}$ are bigger than that of $\E_{b_2}$ we can see that the filtration $ \E_{b_{ \widetilde{\chi}_c \otimes \widetilde{\chi} }} \subset \E_{b_{ \widetilde{\chi}_c \otimes \widetilde{\chi} }} \oplus \E_{b_2} \simeq \E_{b_{\chi_c \otimes \chi}} $ can be refined to the canonical Harder-Narasimhan filtration of $\E_{b_{\chi_c \otimes \chi}}$. However in the context of lemma $\ref{itm : second computation}$, it is not always true. Nevertheless, it is still true under the assumption that $ b' = b_{ \widetilde{\chi}_c \otimes \widetilde{\chi} } $.

Now we suppose that $ b = b_{\chi_c \otimes \chi } $ then $b_1 = b_{ \widetilde{\chi}_c \otimes \widetilde{\chi} } $. Note that by the assumption, all the slopes of the Newton polygon of $\E_{b_1} \simeq \mathcal{G}$ are strictly bigger than $\lambda_{i-1}$, we deduce that $\I(i-1) = \{ c \}$. Note that $\mathcal{F}_{\widetilde{\chi}_c \otimes \widetilde{\chi}} \simeq \delta_{b_1}^{-1/2} \otimes \pi_{\widetilde{\chi}_{c'} \otimes \widetilde{\chi}} [-d_{b_1}] $ so from (\ref{itm : equa2}) we see that
\[
R\Gamma_c(\GL_m, \widetilde{b}, b_1, \mu_1)[\delta^{1/2}_{\chi | \G_{\widetilde{b}}} \otimes {\widetilde{\pi}} ] [2d_{\widetilde{b}}] \otimes |\cdot|^{\tfrac{h}{2}} \simeq \bigoplus_{b_{ \widetilde{\chi}_{c'} \otimes \widetilde{\chi} } = b_1 } (\delta^{1/2}_{b | \G_{b_1}} \delta^{-1}_{b_1}) \otimes \pi_{\widetilde{\chi}_{c'} \otimes \widetilde{\chi}} \boxtimes \phi_{c'} [d_{\widetilde{b}}-d_{b_1}]
\]
where the sum is over all $c' \in \big( \bigcup_{1 \le j < i} \I(j) \big) $ such that $b_{\widetilde{\chi}_{c'} \otimes \widetilde{\chi}} = b_1$. Note that $h = m-n$ and the twist $|\cdot|^{\tfrac{m-n}{2}}$ appearing on the left hand side of (\ref{itm : equa2}) is canceled by the twist $|\cdot|^{\tfrac{h}{2}}$ on the left hand side of the above equation. 

By combining with equations (\ref{itm : Boyertrick-cohomology}), (\ref{itm : equa}), (\ref{itm : equa3}) we have: 
\begin{align*}
 &\quad \ i^*_b \T_{\mu}(i_{b_{\chi} !} \delta^{-1/2}_{\chi} \otimes \pi_{\chi}) \\
 &\simeq \kappa \otimes \Big( R\Gamma_c(\GL_m, \widetilde{b}, b_1, \mu_1)[\delta^{1/2}_{\chi | \G_{\widetilde{b}}} \otimes {\widetilde{\pi}}] \times \delta^{1/2}_{\chi | \G_{b_2}} \otimes \pi_2 \otimes \kappa_{b_2} \Big) [h-2d+2d_{b_{\chi}}-2d_{b_2}] \otimes |\cdot|^{\tfrac{h}{2}} \\
 &\simeq \bigoplus_{b_{ \widetilde{\chi}_{c'} \otimes \widetilde{\chi} } = b_1 } \kappa \otimes \big( \delta^{1/2}_{b | \G_{b_1}} \pi_{\widetilde{\chi}_{c'} \otimes \widetilde{\chi}} \otimes \kappa_{b_1} \boxtimes \delta^{1/2}_{b | \G_{b_2}} \otimes \pi_2 \otimes \kappa_{b_2} \big) \boxtimes \phi_{c'} [h-2d+2d_{b_{\chi}}-2d_{b_2}-d_{\widetilde{b}}-d_{b_1}] \\
 &\simeq \bigoplus_{b_{ \widetilde{\chi}_{c'} \otimes \widetilde{\chi} } = b_1 } \delta^{-1/2}_{b} \otimes \big( \pi_{\widetilde{\chi}_{c'} \otimes \widetilde{\chi}} \boxtimes \pi_2 \big) \boxtimes \phi_{c'} [h-2d+2d_{b_{\chi}}-2d_{b_2}-d_{\widetilde{b}}-d_{b_1}]
\end{align*}
where in the second isomorphism we use that $\delta_{b | \G_{b_2}} = \delta_{\chi | \G_{b_2}}$ and in the last isomorphism we use the fact that $\kappa = \kappa_b \otimes (\kappa_{b_1}^{-1} \times \kappa_{b_2}^{-1})$.

Since $\mathcal{F}_{\chi} = (i_{b_{\chi} !} \delta^{-1/2}_{\chi} \otimes \pi_{\chi})[-d_{b_{\chi}}] $ and by the condition $\I(i-1) = \{c\}$, we can verify that
\[
d_b = d_{b_{\chi}} - d_{\widetilde{b}} + d_{b_1} + h,
\]
moreover, we have the identity $d = d_b - d_{b_1} - d_{b_2}$. 
Thus we see that
\begin{align*}
    i^*_b \T_{\mu}( i_{b_{\chi} !} \delta^{-1/2}_{\chi} \otimes \pi_{\chi}) \simeq \bigoplus_{b_{ \widetilde{\chi}_{c'} \otimes \widetilde{\chi} } = b_1 } \big( \delta^{-1/2}_b \otimes \pi_{\chi \otimes \chi_{c'}} \big) \boxtimes \phi_{c'}[-h+d_{\widetilde{b}}-d_{b_1}].
\end{align*}

By identifying the $W_{\Q_p}$-action, we see that
\[
i^*_b C_{\chi_c} \star (i_{b_{\chi} !} \delta^{-1/2}_{\chi} \otimes \pi_{\chi})[-d_{b_{\chi}}] \simeq \delta^{-1/2}_b \otimes \pi_{\chi \otimes \chi_{c}}[-h+d_{\widetilde{b}}-d_{b_{\chi}}-d_{b_1}] \simeq \delta^{-1/2}_b \otimes \pi_{\chi \otimes \chi_{c}}[-d_b]
\]
and we deduce that
\[
i^*_bC_{\chi'} \star \mathcal{F}_{\chi} \simeq i^*_b \mathcal{F}_{\chi \otimes \chi_{c}}. \\
\]
\\
Now we treat the case $m'_i < m_i$. Thus again by proposition \ref{generalized Boyer's trick} we have
\[
\Sht(\GL_n, b_{\chi}, b, \mu) = \Big( \Sht(\GL_m, \widetilde{b}, b_1, \mu_1) \times_{\Spd \breve{\Q}_p} \Sht(\GL_{n-m}, b_2, b_2, \mu_2) \times_{\Spd \breve{\Q}_p} \mathcal{J}^{U} \Big) \times^{\underline{\rP}} \underline{\G}
\]
with similar notations as before. The only difference is that now $\rP$ is a proper parabolic subgroup of $\G := \G_{b_{\chi}}$. Unravelling the definition of these groups, we see that
\[
\G_{b_{\chi}} = \Aut(\OO(\lambda_1)^{m_1}) \times \dotsc \times \Aut(\OO(\lambda_i)^{m_i}) \times \dotsc \times \Aut(\OO(\lambda_k)^{m_k})
\]
and $\rP$ is a parabolic subgroup whose Levi factor is
\[
\M = \Aut(\OO(\lambda_1)^{m_1}) \times \dotsc \times \Aut(\OO(\lambda_i)^{m_i-m'_i}) \times \Aut(\OO(\lambda_i)^{m'_i}) \times \dotsc \times \Aut(\OO(\lambda_k)^{m_k}).
\]

Thus 
\[
R\Gamma_c(\GL_n, b_{\chi}, b, \mu) = \ind^{\G}_{\rP} \Big( R\Gamma_c(\GL_m, \widetilde{b}, b_1, \mu_1) \times R\Gamma_c(\GL_{n-m}, b_2, b_2, \mu_2) \times R\Gamma_c(\mathcal{J}^{U}) \Big)
\]
and we want to compute $ R\Gamma_c(\GL_n, b_{\chi}, b, \mu)[\delta^{1/2}_{b_{\chi}} \otimes \pi_{\chi}] $ where $\pi_{\chi}$ is a representation of $\G_{b_{\chi}}(\Q_p)$ associated to the character $\chi$ (and the $L$-parameter $\phi$). Thus we need to find a way to get rid of the (unnormalize) parabolic induction $\ind^{\G}_{\rP}$. Note that we do not need to deal with this issue in the context of lemma \ref{itm : second computation} since the parabolic induction in that case is from some parabolic subgroup of $\G_{b}(\Q_p)$. 

Let $\pi$ be an irreducible representation of $\G'(\Q_p)$ whose $L$-parameter is given by $\displaystyle \bigoplus_{j \in \I(i)} \theta_j \otimes \phi_j $, where $\G' = \Aut(\OO^{m_i}(\lambda_i))$ and the $\theta_j$'s are unramified characters. Then by second adjointness theorem we see that
\[
 \Hom_{\G'}(\Ind_{\rP'}^{\G'}(\rho), \pi^{\vee}) = \Hom_{\M'}(\rho, r_{\overline{\rP'}}(\pi^{\vee})),
\]
where $r_{\overline{\rP'}}$ denotes the normalized Jacquet functor, $\pi^{\vee}$ is the contragredient of $\pi$; $\M' = \Aut(\OO(\lambda_i)^{m_i-m'_i}) \times \Aut(\OO(\lambda_i)^{m'_i}) $ and $\rP'$ is the parabolic subgroup of $\G'$ and which is the direct factor of $\rP$ and whose Levi factor is $\M'$ ; finally $\overline{\rP'}$ is the opposite parabolic of $\rP'$. Note that $\Ind_{\rP'}^{\G'}$ and $r_{\overline{\rP'}}$ are exact functors and $\Ind_{\rP'}^{\G'}$ preserves projective objects, we deduce that 

\[
 R\Hom_{\G'}(\Ind_{\rP'}^{\G'}(\rho), \pi^{\vee}) = R\Hom_{\M'}(\rho, r_{\overline{\rP'}}(\pi^{\vee})),
\]
by using the Hom-Tensor adjunction we see that
\[
 \Ind_{\rP'}^{\G'}(\rho) \otimes^{\mathbb{L}}_{\G'} \pi = \rho \otimes^{\mathbb{L}}_{\M'} (r_{\overline{\rP'}}(\pi^{\vee}))^{\vee},
\]
The constituents of $r_{\overline{\rP'}}(\pi^{\vee})$ are some representations of $\M'(\Q_p)$ whose $L$-parameter post-composed with the natural embedding $\widehat{\M'} \longrightarrow \widehat{\G'}$ is given by $\displaystyle \bigoplus_{j \in \I(i)} \theta_j \otimes \phi_j^{\vee} $. Note that the $\theta_j \otimes \phi_j$'s are pairwise distinct then the $L$-parameters of these representations of $\M'(\Q_p)$ are pairwise distinct. Now since there is no non-trivial extension between tempered representations with different cuspidal supports, we deduce that
\[
r_{\overline{\rP'}}(\pi^{\vee}) = \bigoplus_{A \in \A} \pi_{A}^{\vee}
\]
where $\A$ is the set of subsets $A$ of $\I(i)$ such that $ \displaystyle \sum_{j \in A} \dim\phi_j = \rank \OO(\lambda_i)^{m_i - m'_i}  $ and $\pi_A$ is the representation $\pi_{1,A} \times \pi_{2,A}$ of $\M'(\Q_p)$ whose $L$-parameter is given by $\displaystyle \bigoplus_{j \in A} \theta_j \otimes \phi_j \times \bigoplus_{j' \in \I(i) \setminus A} \theta_{j'} \otimes \phi_{j'}$. Therefore, we have
\[
\ind_{\rP'}^{\G'}(\rho) \otimes^{\mathbb{L}}_{\G'} \pi = \Ind_{\rP'}^{\G'}(\rho \otimes \delta^{-1/2}_{\mP'} ) \otimes^{\mathbb{L}}_{\G'} \pi = \bigoplus_{A \in \A} (\rho \otimes \delta^{-1/2}_{\mP'}) \otimes^{\mathbb{L}}_{\M'} \pi_A.
\]

Note that by abuse of notation, we will also use $\delta_{\mP'}$ the character of $\M$ obtained from $\delta_{\mP'}$ by extending trivially from $\M'$. Now by choosing $\rho$ to be the $\M'(\Q_p)$-part of $ \big( R\Gamma_c(\GL_m, \widetilde{b}, b_1, \mu_1) \times R\Gamma_c(\GL_{n-m}, b_2, b_2, \mu_2) \big) \times R\Gamma_c(\mathcal{J}^{U})$ and $\pi$ to be the $\G'(\Q_p)$-part of $ \delta^{1/2}_{b_{\chi}} \otimes \pi_{\chi} $, we deduce the following equality, similar to equation (\ref{itm : Boyertrick-cohomology}) with the same notations
\begin{align*}
& \quad \kappa^{-1} \otimes R\Gamma_c(\GL_n, b_{\chi}, b, \mu)[\delta^{1/2}_{b_{\chi}} \otimes \pi_{\chi}] [2d-h]\otimes|\cdot|^{\tfrac{-h}{2}} \\
& \simeq \bigoplus_{A \in \A} R\Gamma_c(\GL_m, \widetilde{b}, b_1, \mu_1)[\delta^{1/2}_{\chi | \G_{\widetilde{b}}} \delta^{-1/2}_{\mP'| \G_{\widetilde{b}}} \otimes \pi^A_{1,\chi}] \times  R\Gamma_c(\GL_{n-m}, b_2, b_2, \mu_2)[\delta^{1/2}_{\chi | \G_{b_2}} \delta^{-1/2}_{\mP'| \G_{b_2}}  \otimes \pi^A_{2,\chi}],
\numberthis \phantomsection \label{itm : second equation for the computation}
\end{align*} 
where for each $A \in \A$, $ \pi^A_{1, \chi} \times \pi^A_{2, \chi} $ is the representation of $\M(\Q_p)$ such that the $\Aut(\OO^{m_j}{\lambda_j})$-factor is given by $\pi_A$ if $j = i$ and is given by the $\Aut(\OO(\lambda_j)^{m_j})$-factor of $\pi_{\chi}$ if $i \neq j$.

This equation is an analogue of $(\ref{itm : Boyertrick-cohomology})$, now we can use the same arguments as in the case $m_i = m'_i$. We can show the vanishing results in exactly the same manner. More precisely, if $c \in \displaystyle \big(\I(i) \setminus A \big) \cup \big( \bigcup_{i < j \le k} \I(j) \big) $ then the $[\phi_c]$-isotypical part of $ \mathrm{H}^* (R\Gamma_c(\GL_n, b_{\chi}, b, \mu)[\delta^{1/2}_{b_{\chi}} \otimes \pi_{\chi}])$ is empty since $W_{\Q_p}$ acts trivially on $R\Gamma_c(\GL_{n-m}, b_2, b_2, \mu_2)$. In particular if $ c \in \big( \bigcup_{i < j \le k} \I(j) \big) $ then the $[\phi_c]$-isotypical part of $ \mathrm{H}^* \big( R\Gamma_c(\GL_n, b_{\chi}, b, \mu)[\delta^{1/2}_{b_{\chi}} \otimes \pi_{\chi}] \big) $ is empty.

If $b \neq b_{\chi \otimes \chi_c } $ then $i^*_{b'}C_{\chi'} \star \mathcal{F}_{\chi} \simeq 0$ as before. Now we suppose that $c \in \big( \bigcup_{1 < j \le i} \I(j) \big) $ and $ b = b_{\chi \otimes \chi_c } $ then we can show that $c \in \I(i)$ and $ \dim \phi_c = \rank \OO(\lambda_i)^{m_i - m'_i} $. In particular, if $ c \in A$ then $ A = \{ c \} $. Thus when taking the $[\phi_c]$-isotypical part we have
\begin{align*}
    & \quad \mathrm{H}^* \big( \kappa^{-1} \otimes R\Gamma_c(\GL_n, b_{\chi}, b, \mu)[\delta^{1/2}_{b_{\chi}} \otimes \pi_{\chi}][2d-h]\otimes|\cdot|^{\tfrac{-h}{2}} \big) [\phi_c] \\ 
   &= \mathrm{H}^* \big( R\Gamma_c(\GL_m, \widetilde{b}, b_1, \mu_1)[\delta^{1/2}_{\chi | \G_{\widetilde{b}}} \delta^{-1/2}_{\mP'| \G_{\widetilde{b}}} \otimes \pi^{\{ c \}}_{1,\chi}]\big)[\phi_c] \times \mathrm{H}^* \big( R\Gamma_c(\GL_{n-m}, b_2, b_2, \mu_2)[\delta^{1/2}_{\chi | \G_{b_2}} \delta^{-1/2}_{\mP'| \G_{b_2}}  \otimes \pi^{\{c\}}_{2,\chi}] \big).    
\end{align*}

We note that $\delta^{1/2}_{\chi | \G_{\widetilde{b}}} \delta^{-1/2}_{\mP'| \G_{\widetilde{b}}} \otimes \pi^A_{1,\chi} \otimes \delta_{\widetilde{b}}^{-1} = \delta^{1/2}_{\chi | \G_{\widetilde{b}}} \delta^{-1/2}_{\mP'| \G_{\widetilde{b}}} \delta_{\widetilde{b}}^{-1/2} \otimes \pi^A_{1,\chi} \otimes \delta_{\widetilde{b}}^{-1/2}$ and the twist by the character $\delta^{1/2}_{\chi | \G_{\widetilde{b}}} \delta^{-1/2}_{\mP'| \G_{\widetilde{b}}} \delta_{\widetilde{b}}^{-1/2}$ corresponding to the twist by $|\cdot|^{\tfrac{m-n}{2}}$ for the $L$-parameter of $\pi^A_{1,\chi} \otimes \delta_{\widetilde{b}}^{-1/2}$. This twist compensates the twist $|\cdot|^{\tfrac{-h}{2}}$ on the left hand side of the equation (\ref{itm : second equation for the computation}).

Now, using the same computations as in the case $m_i = m_{i'}$ we can show that
\[
i^*_bC_{\chi'} \star \mathcal{F}_{\chi} \simeq i^*_b \mathcal{F}_{\chi \otimes \chi_{c}},
\]
and we are done.
\end{proof}

\section{Proof of the first main theorem} \textbf{} \phantomsection \label{itm : proof of the first main theorem}
 
 This section is devoted to the proof of theorem \ref{itm : main theorem}. By \cite[theorem 6.3]{KH20}, it is enough to prove the theorem for $\chi \in \Irr(S_{\phi})^+ := \{ (d_1, \dotsc, d_r) \in \Irr(S_{\phi}) \ | \ d_i \ge 0 \ \forall \ 1 \le i \le r \}$. We will proceed by induction on $r$. One can deduce the case $r=1$ by \cite[theorem 1.5]{Han} and lemma \ref{shimhecke}. Thus we mainly focus on proving the case $r \ge 2$. The arguments are rather technical and complicated. However it is greatly simpler if we suppose further that $\phi$ is a direct sum of characters. In particular, under this assumption, the slopes of the Newton point $\nu_{b_{\chi}}$ are all integer for $\chi \in \Irr(S_{\phi}) $. The readers who only want to grab the main ideas could impose that condition in the rest of this section.

Suppose that the theorem is true for $1, \dotsc ,r-1$ then we show that it is true for $r$. Now we proceed by induction on $\displaystyle D =  \sum_{i = 1}^r d_i =: |\chi| $. For $D = 1$, one could use lemmas \ref{itm : first computation}, \ref{itm : second computation} to verify the theorem. However, we will use the known cases of Harris-Viehmann's conjecture to show the base case. Then we suppose that the theorem is true for $ D < s $ and prove the case $D = s$. By a direct application of proposition \ref{generalized Boyer's trick} and lemmas \ref{itm : first computation}, \ref{itm : second computation}, we can show that $i^*_{b_{\chi}}C_{\chi} \star \mathcal{F}_{\Id} = \mathcal{F}_{\chi}$. The most ad hoc part of the proof is to show that the restrictions of $C_{\chi} \star \mathcal{F}_{\Id}$ to other strata vanish. 

In this section, we fix a maximal split torus and a Borel subgroup $\T \subset \rB \subset \GL_n$. For $b, b' \in B(\GL_n)$, we have $ \nu_b = (-\nu_{\E_b})_{\rm dom} $ and $ \nu_{b'} = (-\nu_{\E_{b'}})_{\rm dom} $. Thus $b$ is smaller than $b'$ with respect to the usual partial order in $B(\GL_n)$ if and only if $\nu_{\E_b}$ is smaller than $ \nu_{\E_{b'}} $ with respect to the usual partial order in $X_{*}(\T)_{\Q}$. We always use the partial order when we compare elements $b, b' \in B(\GL_n)$ and elements $\nu_{\E_b}$, $\nu_{\E_{b'}}$ in $X_{*}(\T)_{\Q}$.

\subsection{The base case} \textbf{} \label{itm : the base case} \textbf{}  \\

     We prove the case $D = 1$. For $1 \le i \le r$, let $\chi_i = (d_1, \dotsc, d_r)$ be the character such that $d_i = 1$ and $d_j = 0$ for $ 1 \le j \neq i \le r $. We will show that $ C_{\chi_i^{-1}} \star \mathcal{F}_{\chi_i} = \mathcal{F}_{\Id} $. Then by the monoidal property and the fact that $C_{\Id} \star \mathcal{F}_{\chi} = \mathcal{F}_{\chi} $ for all $\chi$, we deduce that $ C_{\chi_i} \star \mathcal{F}_{\Id} = \mathcal{F}_{\chi_i} $. 
    
    Let $ \mu = (1, 0^{(n-1)}) $. Fix $ 1 \le i \le r $, we have $ \E_{b_{\chi_i}} \simeq \OO(1/n_i) \oplus \OO^{n-n_i} $.  By the Harris-Viehmann conjecture (or proposition \ref{generalized Boyer's trick}) for $\Sht(\GL_n, b_{\chi_i}, \mu)$ and \cite{HT}, \cite[theorem 1.2]{AL}, \cite[theorem 1.5]{Han}, we see that
    \[
     \Sht(\GL_n, b_{\chi_i}, \mu) = \Big( \Sht(\GL_{n-n_i}, 1_{\GL_{n-n_i}}, \Id) \times_{\Spd \breve{\Q}_p} \Sht(\GL_{n_i}, b_1, \mu_1) \times_{\Spd \breve{\Q}_p} \widetilde{\G}^0_{b_{\chi_i}} \Big) \times^{\underline{\rP}} \underline{\GL_n}.
     \]
    where $1_{\GL_{n-n_i}}$ is the trivial element in $\GL_{n-n_i}(\breve{\Q}_p)$, $\E_{b_1} \simeq \OO(1/n_i)$, $ \mu_1 = (1, 0^{(n_i - 1)}) $ is a cocharacter of $\GL_{n_i}$ and $\rP$ is the standard parabolic subgroup of $\GL_n$ whose Levi factor is $\GL_{n_i} \times \GL_{n - n_i}$. By taking the cohomology and lemma \ref{itm : kappa and modulus}, we have  
    \begin{align*} 
       & \ \quad R\Gamma_c(\GL_n, b_{\chi_i}, \mu)  [\delta^{1/2}_{b_{\chi_i}} \otimes \pi_{\chi_i} ] \\
       &\simeq \ind^{\GL_n}_{\rP}\Big( R\Gamma_c(\GL_{n-n_i}, b_1, \Id)[\delta^{-1/2}_{b_{\chi_i} | \GL_{n-n_i}} \otimes \pi] \boxtimes R\Gamma_c(\GL_{n_i}, b_1, \mu_1)[ \delta^{-1/2}_{b_{\chi_i}| \G_{b_1}} \otimes \pi^{b_1}_1] \Big) [n_i-n] \otimes |\cdot|^{\tfrac{n-n_i}{2}} \numberthis \label{itm : auxillaire} \\ 
       &\simeq \ind^{\GL_n}_{\rP}\Big( \big( \delta^{1/2}_{\rP | \GL_{n-n_i}} \otimes \pi \big) \boxtimes \big( \delta^{1/2}_{\rP | \GL_{n_i}} \otimes \pi_1 \big) \Big) \boxtimes \phi_i^{\vee} [n_i-n] \quad   \\
       &\simeq \pi_{\Id} \boxtimes \phi_i^{\vee} [n_i-n] \numberthis \label{itm : first equation}
    \end{align*}
    as complexes of $ \GL_n(\Q_p) \times W_{\Q_p} $-modules where $\pi$ is the representation of $\GL_{n-n_i}(\Q_p)$ whose $L$-parameter is given by $ \displaystyle \bigoplus_{1 \le j \neq i \le r} \phi_j$. The representations $\pi^{b_1}_1$ resp. $\pi_1$ are supercuspidal representations of $\G_{b_1}(\Q_p)$, resp. $\GL_{n_i}(\Q_p)$ whose $L$-parameter is given by $\phi_i$ and $\vee$ denotes the contragredient representation. In particular, we see that $ \pi_{\chi_i} \simeq \pi \boxtimes \pi^{b_1}_1 $. Moreover, by remark (\ref{itm : opposite parabolic group}), we see that we use the opposite parabolic subgroup $\rP^{op}$ to define the modulus character $\delta_{b_{\chi_i}}$, thus $\delta_{b_{\chi_i}}$ and $\delta^{-1}_{\rP}$ have the same $L$-parameters. 
    
      Note that in equation (\ref{itm : auxillaire}) the cohomology of the connected component of the automorphisms group $\widetilde{G}_{b_{\chi_i}}^0$ contributes a shift $[-2(n-n_i)]$ and a twist $\delta^{-1}_{b_{\chi_i}}$ (as right $\G_{b_{\chi_i}}(\Q_p)$ module), the normalization in the sheaves $\mathcal{S}_{\mu}$ and $\mathcal{S}_{\mu_1}$ contributes a twist $|\cdot|^{\tfrac{n-n_i}{2}}$ and a shift $[n-n_i]$, hence a shift $[n_i-n]$ and a twist $|\cdot|^{\tfrac{n-n_i}{2}}$. There is no Tate twist in the third line since the modulus character $\delta^{-1/2}_{b_{\chi_i}| \G_{b_1}}$ contributes a twist $|\cdot|^{\tfrac{n_i-n}{2}}$ when taking the cohomology of local Shimura variety, more precisely by the case $r=1$ we have $\Sht(\GL_{n_i}, b_1, \mu_1)[ \delta^{-1/2}_{b_{\chi_i}| \G_{b_1}} \otimes \pi^{b_1}_1] \simeq \delta^{1/2}_{\rP | \GL_{n_i}} \otimes \pi_1 \boxtimes \phi_i^{\vee}(-\tfrac{n-n_i}{2})$.

    However $ \displaystyle r_{\mu^{-1}} \circ \phi = \bigoplus_{j=1}^r \chi^{-1}_j \boxtimes \phi_j^{\vee} $ and $ \mathcal{F}_{\chi_i} := i_{b_{\chi_i} !} (\delta^{-1/2}_{b_{\chi_i}} \otimes \pi_{\chi_i}) [n_i-n] $ then by lemma \ref{shimhecke} we deduce that
    \begin{align*}
        R\Gamma_c(\GL_n, b_{\Id}, b_{\chi_i}, \mu) [ \delta^{1/2}_{b_{\chi_i}} \otimes \pi_{\chi_i}][n-n_i] \simeq i_1^*\T_{\mu^{-1}}(\mathcal{F}_{\chi_i}) \simeq \bigoplus_{j=1}^r i^*_1 C_{\chi^{-1}_j} \star \mathcal{F}_{\chi_i} \boxtimes \phi_j^{\vee}.
    \end{align*}
    
     By using equation (\ref{itm : first equation}), we see that the left hand side is isomorphic to $ \mathcal{F}_{\Id} \boxtimes \phi_i^{\vee}$. By identifying the $W_{\Q_p}$-action, we deduce that
     \[
     i_1^* C_{\chi_i^{-1}} \star \mathcal{F}_{\chi_i} \simeq \mathcal{F}_{\Id},
     \]
and for $ j \neq i$, we have 
\begin{align} \phantomsection \label{itm : restriction to 1}
  i_1^* C_{\chi_j^{-1}} \star \mathcal{F}_{\chi_i} \simeq 0.  
\end{align}

     To show that $C_{\chi_i^{-1}} \star \mathcal{F}_{\chi_i} \simeq \mathcal{F}_{\Id}$, we need to show that the restrictions of $C_{\chi_i^{-1}} \star \mathcal{F}_{\chi_i}$ to other $HN$-strata are zero. Thus we need to compute the image of $b_{\chi_i}$ by some modifications of type $\mu^{-1}$. The set of images is given by $ \Big\{ \OO^n, \ \E_{b_{m, m'}} := \OO(1/n_i) \oplus \OO^m \oplus \OO(-1/m') \ | \ n_i + m + m' = n \Big\}$ by example \ref{itm : image modif}. 

     Thus we need to compute $R\Gamma_c(\GL_n, b_{\chi_i},b_{m, m'} , \mu)$ for $m' \neq 0$. To do this, we use the same arguments as in lemmas \ref{itm : first computation}, \ref{itm : second computation}. Hence we have
    \[
    i_{b_{m,m'}}^* C_{\chi_i^{-1}} \star \mathcal{F}_{\chi_i} = 0.
    \]

\subsection{Induction step I : Restriction of $C_{\chi} \star \mathcal{F}_{\Id}$ to the stratum $b_{\chi}$} \textbf{} \\
     
\textit{Suppose that the theorem is true for $ D = 1, \dotsc, s-1   $. Now consider the case $D = s$.} Let $ \chi = (d_1, \dotsc, d_r) $ be an element in $\Irr(S_{\phi})$ such that the $d_i$'s are non-negative and $\displaystyle \sum_{i = 1}^r d_i = s$. 

Suppose that we have $ \E_{b_{\chi}} = \OO(\lambda_1)^{m_1} \oplus \dotsc \oplus \OO(\lambda_k)^{m_k} $ with $ \lambda_1 > \lambda_2 > \dotsc > \lambda_k \ge 0 $ and $ m_i > 0 $ for $ 1 \le i \le k $. By the construction of $b_{\chi}$, there exists for each $ 1 \le i \le k $, a non-empty subset $\I(i)$ of $\{ 1, 2, \dotsc, r \}$ such that for $j \in \I(i)$, we have $ d_j/n_j = \lambda_i $. We denote by $r(\chi)$ the number $ \mathrm{min} \{ \dim \phi_j \ | \ j \in \I(k) \}$.  \\

    We show that $ i^*_{b_{\chi}} C_{\chi} \star \mathcal{F}_{\Id} = \mathcal{F}_{\chi} $. \\

     We suppose that $\lambda_k = 0$. The case $ \lambda_k > 0 $ is simpler and can be treated similarly. \\
     
     In this case $\lambda_{k-1} > 0$. We consider a character $\chi' = (d_1', \dotsc, d_r')$ where for one index $j^* \in \I(k-1)$ we have $ d_{j^*}' = d_{j^*} - 1 $ and $ d_m' = d_m $ for all other $r-1$ indexes. In particular the $d_j$'s are non-negative and $ \displaystyle \sum_{i=1}^k d'_i = s-1 $. By construction $ \E_{b_{\chi'}} = \OO(\lambda_1)^{m_1} \oplus \dotsc \oplus \OO(\lambda_{k-1})^{m_{k-1} - t} \oplus \OO(\lambda'_k)^{t'} \oplus \OO(\lambda_k)^{m_k} $ where $ \lambda'_{k} = \tfrac{d_{j^*}'}{n_{j^*}}$; $ t = \gcd(d_{j^*}, n_{j^*}) $ and $t' = \gcd( d_{j^*}', n_{j^*} )$. Moreover we have $\chi = \chi_{j^*} \otimes \chi'$. Thus by monoidal property we have
     \[
     C_{\chi} \star \mathcal{F}_{\Id} = C_{\chi_{j^*}} \star ( C_{\chi'} \star \mathcal{F}_{\Id} ).
     \]

     By the induction hypothesis (on D) for $\chi'$ we see that $ C_{\chi'} \star \mathcal{F}_{\Id} = \mathcal{F}_{\chi'} $. Thus we need to compute $ i^*_{b_{\chi}} C_{\chi_{j^*}} \star \mathcal{F}_{\chi'} $. We can do it by computing $ R\Gamma_c(\GL_n, b_{\chi'}, b_{\chi}, \mu )[\delta^{1/2}_{b_{\chi'}} \otimes \pi_{\chi'}][2d_{b_{\chi'}}] $ where $\mu = (1, 0^{(n-1)})$. The triple $(b_{\chi'}, b_{\chi}, \mu)$ satisfies the hypothesis of corollary \ref{coro of generelized Boyer's trick} and lemma \ref{itm : second computation} then we deduce that
     \[
    i^*_{b_{\chi}} C_{\chi_{j^*}} \star \mathcal{F}_{\chi'} = \mathcal{F}_{\chi}.
    \]

\textit{Now we need to show that the restrictions of $C_{\chi} \star \mathcal{F}_{\Id}$ to other strata vanish.} \textbf{} \\

\subsection{Induction step II : Vanishing of the restrictions of $C_{\chi} \star \mathcal{F}_{\Id}$ to other strata}  \phantomsection \label{itm : positive slopes} \textbf{} \\

 From now on we suppose that $C_{\chi} \star \mathcal{F}_{\Id} \not\simeq \mathcal{F}_{\chi} $ and try to find a contradiction. 

\subsubsection{The case where $\lambda_k > 0$} \textbf{}  \\

We suppose $\lambda_k > 0$ in this paragraph. 

In the last subsection, we associated a natural number $r(\chi)$ to $\chi$. Let $C$ denotes the set of characters $\xi \in \Irr(S_{\phi})^+$ such that $\E_{b_{\xi}}$ is a direct sum of semi-stable vector bundles of strictly positive slopes, $C_{\xi} \star \mathcal{F}_{\Id} \not\simeq \mathcal{F}_{\xi} $ and $|\xi| = s$. Since $C$ is finite, we can assume that $r(\chi) = \mathrm{min}\{ r(\xi) \ | \ \xi \in C \} $.
      
Consider  a character $\chi' = (d_1', \dotsc, d_r')$ where $ d_{j^*}' = d_{j^*} - 1 $ for one index $j^* \in \I(k)$ such that $\dim \phi_{j^*} = r(\chi)$ and $ d_i' = d_i $ for all other $r-1$ indexes. Hence $ \E_{b_{\chi'}} = \OO(\lambda_1)^{m_1} \oplus \dotsc \oplus \OO(\lambda_{k-1})^{m_{k-1}} \oplus \OO(\lambda_k)^{m_k - t} \oplus \OO(\lambda'_k)^{t'} $ where $ t = \gcd(d_j, n_j)$; $t' = \gcd(d_{j^*}', n_{j^*})$ and $ \lambda'_k = \tfrac{d_{j^*}'}{n_{j^*}}$. 
     
     We have $ \chi = \chi_{j^*} \otimes \chi' $ and then $ C_{\chi} \star \mathcal{F}_{\Id} = C_{\chi_{j^*}} \star \mathcal{F}_{\chi'} $. Thus the set of strata where the restriction of $C_{\chi} \star \mathcal{F}_{\Id}$ is non null is a subset of the strata indexed by images of the modification of $ \E_{b_{\chi'}} $ of type $\mu = (1, 0^{(n-1)}) $. Moreover, by proposition \ref{itm : shape of strata} and the fact that the formalism of Fargues-Scholze parameter and the restrictions to Harder-Narasimhan strata commute \cite[Proppsition 3.14, Corollary 3.15]{Ham}, we deduce that these strata are of the form $b_{\xi}$ for some $\xi \in \Irr(S_{\phi})$.
   
     Since $C_{\chi} \star \mathcal{F}_{\Id} \neq \mathcal{F}_{\chi} $, we can choose an element $ b \neq b_{\chi} \in B(\GL_n) $ such that $i^*_bC_{\chi} \star \mathcal{F}_{\Id}$ is non trivial. Thus there is a modification of type $\mu$
     \[
     f : \E_{b_{\chi'}} \longrightarrow \E_b.
     \]

    Remark that if the biggest slope of $\nu_{\E_b}$ equals to $\lambda_1$ then we can apply lemmas \ref{itm : first computation}, \ref{itm : second computation} to the triple $\E_{b_{\chi'}}$, $\E_b$ and $\mu = (1, 0^{(n-1)})$ to compute $i^*_bC_{\chi} \star \mathcal{F}_{\Id}$ and check that if $ b \neq b_{\chi} \in B(\GL_n) $ then the restriction $i^*_bC_{\chi} \star \mathcal{F}_{\Id}$ is trivial, which is a contradiction. Hence we deduce that the biggest slope of $\nu_{\E_b}$ is not equal to $\lambda_1$.

    Consider the filtration $\OO(\lambda_1)^{m_1} \oplus \dotsc \oplus \OO(\lambda_{k-1})^{m_{k-1}} \oplus \OO(\lambda_k)^{m_k - t} \subset \E_{b_{\chi'}} $. Denote by $\rP$ the standard parabolic subgroup corresponding to this filtration and denote by $\M$ its Levi factor. Hence $ \M \simeq \GL_{n - n_{j^*}} \times \GL_{n_{j^*}} $. The modification $\mu$ induces a filtration $ \E \subset \E_b $ of $ \E_b $. The knowledge of this filtration and that of extension between vector bundles allows us to understand $b$.

     The $\M$-bundle corresponding to the filtration of $\E_{\chi'}$ is $ \OO(\lambda_1)^{m_1} \oplus \dotsc \oplus \OO(\lambda_{k-1})^{m_{k-1}} \oplus \OO(\lambda_k)^{m_k - t} \times \OO(\lambda'_k)^{t'} $. The modification $f$ induces modifications
     \[
     f_1 : \OO(\lambda_1)^{m_1} \oplus \dotsc \oplus \OO(\lambda_{k-1})^{m_{k-1}} \oplus \OO(\lambda_k)^{m_k - t} \longrightarrow \E
     \]
     and 
     \[
     f_2 : \OO(\lambda'_k)^{t'} \longrightarrow \E_{b} / \E =: \E'.
     \]
     of type $\mu_1$ and $\mu_2$ respectively. By \cite[lemma 2.6]{CFS}, \cite[lemma 3.11]{Vie}, there are $2$ possibilities: 

     \textbf{Case 1}: $\mu_1 = (1, 0^{(n-n_{j*}-1)}) $ and $\mu_2 = (0^{(n_{j*})})$.
     
     Thus $\E' \simeq \OO(\lambda'_k)^{t'}$ and in particular, all the slopes of $\nu_{\E}$ are strictly bigger than $\lambda'_k$ (even strictly bigger than $\lambda_{k}$ if $|\I(k)| = 1$). Since $H^1(\mathcal{O}(\lambda)) = 0$ if $\lambda \ge 0$ we deduce that $ \E_b = \E \oplus \OO(\lambda'_k)^{t'} $. 

     Thus we can apply lemma \ref{itm : first computation} to the triple $\E_{b_{\chi'}}$, $\E_b$ and $\chi_{j^*}$ to deduce that if $b \neq b_{\chi} \in B(\GL_n)$ then
     \[
     i^*_b C_{\chi_{j^*}} \star \mathcal{F}_{\chi'} = i^*_b C_{\chi} \star \mathcal{F}_{\Id} \simeq 0,
     \]
     which is a contradiction as above.


     \textbf{Case 2}: $\mu_1 = (0^{(n-n_{j*})}) $ and $\mu_2 = (1, 0^{(n_{j*} - 1)})$. 
     
     In particular $\E \simeq \OO(\lambda_1)^{m_1} \oplus \dotsc \oplus \OO(\lambda_{k-1})^{m_{k-1}} \oplus \OO(\lambda_k)^{m_k - t} $. Since $\OO(\lambda'_k)^{t'}$ is semi-stable, we can compute $\E'$ explicitly. If $\E'$ is a semi-stable vector bundle then it is semi-stable of slope $\lambda_k$ and therefore we find out that $ b = b_{\chi} $ and we are done. If it is not the case we still see that $ \E' = \E_x \oplus \E'' $ where $\E''$ is a semi-stable vector bundle of slope $ \lambda $ such that $\lambda_k > \lambda \ge 0$. Moreover if $h$ is a slope of $\nu_{\E_x}$ then $\lambda < h \le \lambda + 1$. 
     \begin{lemma}
         We have $\lambda > 0$.
     \end{lemma}
     \begin{proof}
         Indeed, if $\lambda = 0$ then $\lambda'_k = 0$ and therefore $\E_b \simeq \E_y \oplus \OO^{\ell_1}$ for some vector bundle $\E_y$ such that the slopes of $\nu_{\E_y}$ are strictly positive and $\ell_1 < n_{j*} $. We also see that $\OO^{n_{j*}}$ is a direct factor of $\E_{\chi'}$. Therefore we can apply lemma \ref{itm : first computation} to triple $\E_{\chi'}$, $\E_b$ and $\chi_{j^*}$ to deduce that
         \[
         i^*_bC_{\chi_{j*}} \star \mathcal{F}_{\chi'} \simeq 0,
         \]
     a contradiction. Hence we have $\lambda > 0$.    
     \end{proof}
     By the property of the Harder-Narasimhan filtration \cite[Theorem A.5]{NV}, we see that $\nu_{\E_b}$ is bounded below by the concatenation $ \nu_{\E} \oplus \nu_{\E'} $ of $\nu_{\E}$ and $\nu_{\E'}$. By \cite[lemma 2.9]{MC} we also have $ \nu_{\E_b} \le \nu_{\E \oplus \E'} $. Therefore $\E_b \simeq \E_y \oplus \E''$ where $\E_y$ is a vector bundle such that if $h$ is a slope of $\nu_{\E_y}$ then $ h > \lambda$. Moreover we already know that if $h_{\rm max}$ is the biggest slope of $\nu_{\E_y}$ then $h_{\rm max} \neq \lambda_1$. Therefore $ h_{\rm max} > \lambda_1 $ and we deduce that $b$ is not smaller than $b_{\chi}$ with respect to the usual order in $B(\GL_n)$.
     
     Now we see that by proposition \ref{itm : shape of strata}, there exists $\xi = (a_1,\dotsc, a_r) \in \Irr(S_{\phi})^+$ such that $ b = b_{\xi} $. For any such $\xi$, all the slopes of $\nu_{\E_{b_{\xi}}} = \nu_{\E_b}$ are strictly positive and $r(\xi) < r(\chi) $ (since $\rank \E'' < \rank \E' = r(\chi)$). Therefore
     \[
     C_{\xi} \star \mathcal{F}_{\Id} \simeq \mathcal{F}_{\xi}.
     \]

     We know that the spectral action preserves compact objects and ULA objects then $ i^*_b C_{\chi} \star \mathcal{F}_{\Id} $ is a ULA object and moreover all of its Schur-irreducible constituents have (Fargues-Scholze) $L$-parameter given by $\phi$. Therefore, up to replacing $\xi$ by another $\xi'$ such that $b_{\xi} = b_{\xi'} = b$ and up to some shift, we can suppose that there is a non-trivial morphism 
     \[
     g_1 : i_{b !}i^*_b C_{\chi} \star \mathcal{F}_{\Id} \longrightarrow \mathcal{F}_{\xi},
     \]
     more precisely, we can construct $g_1$ from a quotient of the highest degree non-trivial cohomology group of $i_{b !}i^*_b C_{\chi} \star \mathcal{F}_{\Id}$.

     Since the set $S_{\rm supp} \subset B(\GL_n) $ of strata where the restriction of $C_{\chi} \star \mathcal{F}_{\Id}$ is non trivial is finite, there exists maximal elements in $S_{\rm supp}$ with respect to the restriction of the usual order in $B(\GL_n)$. Since $b$ is not smaller than $b_{\chi}$ with respect to that partial order, we can suppose that $b$ is a maximal element in $S_{\rm supp}$. Hence, by applying the excision exact triangle to the closed embedding $  \Bun_n^{ \ge b}  \longrightarrow \Bun_n $, we deduce that there is a non trivial morphism
     \[
     g_2 : C_{\chi} \star \mathcal{F}_{\Id} \longrightarrow i_{b !}i^*_b C_{\chi} \star \mathcal{F}_{\Id} \xrightarrow{\ g_1 \ } \mathcal{F}_{\xi} \simeq C_{\xi} \star \mathcal{F}_{\Id}.
     \]

     Recall that $\lambda$ is the smallest slope of $\E_{b_{\xi}}$ and there exists an index $\widetilde{j}$ such that $ \tfrac{a_{\widetilde{j}}}{n_{\widetilde{j}}} = \lambda $. Since $\lambda > 0$ we see that $ \chi, \xi \in \Irr(S_{\phi})^{> 0} := \{ (c_1, \dotsc, c_r) \ | \ c_i > 0 \}$. However for an arbitrary character $\theta \in \Irr(S_{\phi})$, $C_{\theta} \star( C_{\theta^{-1}} \star (-)) = C_{\theta^{-1}} \star( C_{\theta} \star (-)) = C_{\Id} \star (-) $ is the identity functor of $ \Dlis^{[C_{\phi}]}(\Bun_n, \overline{\Q}_{\ell})^{\omega} $, we deduce that $C_{\theta^{-1}}(-)$ is an auto-equivalence of $ \Dlis^{[C_{\phi}]}(\Bun_n, \overline{\Q}_{\ell})^{\omega} $. Thus by applying the auto-equivalence $ C_{\chi^{-1}_{\widetilde{j}}} $ to $g_2$ we get a non trivial morphism  
     \[
     g : C_{\chi \otimes \chi^{-1}_{\widetilde{j}}} \star \mathcal{F}_{\Id} \longrightarrow C_{\xi \otimes \chi^{-1}_{\widetilde{j}}} \star \mathcal{F}_{\Id}.
     \]

     The characters $\chi \otimes \chi^{-1}_{\widetilde{j}}$ and $\xi \otimes \chi^{-1}_{\widetilde{j}}$ belong to $\Irr(S_{\phi})^+$ then by the induction hypothesis (on $D$), we have $C_{\chi \otimes \chi^{-1}_{\widetilde{j}}} \star \mathcal{F}_{\Id} \simeq \mathcal{F}_{\chi \otimes \chi^{-1}_{\widetilde{j}}} $ and $C_{\xi \otimes \chi^{-1}_{\widetilde{j}}} \star \mathcal{F}_{\Id} \simeq \mathcal{F}_{\xi \otimes \chi^{-1}_{\widetilde{j}}} $. By the choice of $\widetilde{j}$, the biggest slope of $\E_{b_{\xi \otimes \chi^{-1}_{\widetilde{j}}}}$ is still $h_{\rm max}$ and the biggest slope of $\E_{b_{\chi \otimes \chi^{-1}_{\widetilde{j}}}}$ is not bigger than $\lambda_1$. In particular, $b_{\xi \otimes \chi^{-1}_{\widetilde{j}}}$ is not smaller than $b_{\chi \otimes \chi^{-1}_{\widetilde{j}}}$ with respect to the usual partial order of $B(\GL_n)$. Thus by lemma \ref{itm : morphism between strata} there is no non-trivial morphism from $\mathcal{F}_{\chi \otimes \chi^{-1}_{\widetilde{j}}}$ to $\mathcal{F}_{\xi \otimes \chi^{-1}_{\widetilde{j}}}$. Hence a contradiction. In other words, the restriction $i^*_b C_{\chi} \star \mathcal{F}_{\Id}$ must be trivial.

\subsubsection{The case where $\lambda_k = 0$ and $r > 2$} \textbf{} \\

 We suppose that $\lambda_k = 0$ in this paragraph.

  Consider  a character $\chi' = (d_1', \dotsc, d_r')$ where $ d_{j^*}' = d_{j^*} - 1 $ for one index $j^* \in \I(k-1)$ and $ d_i' = d_i $ for all other $r-1$ indexes. Thus $ \E_{b_{\chi'}} = \OO(\lambda_1)^{m_1} \oplus \dotsc \oplus \OO(\lambda_{k-1})^{m_{k-1} - t} \oplus \OO(\lambda'_k)^{t'} \oplus \OO(\lambda_k)^{m_k} $ where $ \lambda'_{k} = \tfrac{d_{j^*}'}{n_{j^*}}$; $ t = \gcd(d_{j^*}, n_{j^*}) $ and $t' = \gcd( d_{j^*}', n_{j^*} )$. In particular $\chi' \in \Irr(S_{\phi})^+$ and $ \chi = \chi_{j^*} \otimes \chi' $. Then by induction hypothesis we have $ C_{\chi} \star \mathcal{F}_{\Id} = C_{\chi_{j^*}} \star \mathcal{F}_{\chi'} $.

 As before we can choose an element $ b \neq b_{\chi} \in B(\GL_n) $ such that $i^*_bC_{\chi} \star \mathcal{F}_{\Id}$ is non trivial. Hence, there is a modification of type $\mu = (1, 0^{(n-1)})$
\[
f : \E_{b_{\chi'}} \longrightarrow \E_b,
\]
in particular $\E_b$ is a direct sum of semi-stable vector bundles with non-negative slopes. If the trivial line bundle $\OO$ is a direct factor of $\E_b$ then we can apply lemma \ref{itm : first computation} to the triple $\E_{b_{\chi'}}$, $\E_b$ and $\mu$ to deduce that $ \E_b \simeq \E_{b_{\chi}} $ and we are done. 
  
Thus we suppose that $\E_b$ is a direct sum of semi-stable vector bundles whose slopes are strictly positive. 

Consider the filtration $ \OO(\lambda_1)^{m_1} \oplus \dotsc \oplus \OO(\lambda_{k-1})^{m_{k-1}-t} \oplus \OO(\lambda'_k)^{t'}  \subset \E_{b_{\chi'}} $. Denote by $\rP$ the standard parabolic subgroup of $\GL_n$ corresponding to the filtration as well as $\M$ the Levi factor. Hence $ \M \simeq \GL_{n - m} \times \GL_{m} $ where $ m = \rank (\OO(\lambda_k)^{m_k}) $. The modification $f$ induces a filtration $ \E \subset \E_b $ of $ \E_b $ and $ \E_b / \E \simeq \E' $.

  The $\M$-bundle corresponding to the filtration of $\E_{b_{\chi'}}$ is  $ \OO(\lambda_1)^{m_1} \oplus \dotsc \oplus \OO(\lambda'_k)^{t'} \times \OO(\lambda_k)^{m_k} $. The modification $f$ induces modifications
\[
f_1 : \OO(\lambda_1)^{m_1} \oplus \dotsc \oplus \OO(\lambda'_k)^{t'} \longrightarrow \E
\]
and 
\[
f_2 : \OO(\lambda_k)^{m_k} \longrightarrow \E'
\]
of type $\mu_1$ and $\mu_2$ respectively. As before we can compute the $\M$-bundle $ \mathcal{E} \times \mathcal{E}' $ by using \cite[lemma 2.6]{CFS}. There are $2$ possibilities:

  \textbf{Case 1}: $\mu_1 = (1, 0^{(n-m-1)})$ and $\mu_2 = (0^{(m)})$.
  
  Thus $\E' \simeq  \OO(\lambda_k)^{m_k}$ and all the slopes of $\nu_{\E}$ are strictly positive. Since $H^1(\mathcal{O}(\lambda)) = 0$ if $\lambda \ge 0$ we deduce that $ \E_b = \E \oplus \OO(\lambda_k)^{m_k} $. It can not happen since $\lambda_k = 0$ and we suppose that the trivial line bundle $\OO$ is not a direct factor of $\E_b$. 

  \textbf{Case 2}: $\mu_1 = (0^{(n-m)})$ and $\mu_2 = (1, 0^{(m-1)})$ and in particular we have $\E \simeq \OO(\lambda_1)^{m_1} \oplus \dotsc \oplus \OO(\lambda'_k)^{t'} $. 
 
 Since $ \OO(\lambda_k)^{m_k} \simeq \OO^m $, we can compute $\E'$ explicitly. 

  If $\E'$ is not semi-stable then it is of the form $ \E_z \oplus \mathcal{O}^a $ for some $a > 0$ and it can not happen as in case $1$. Therefore we deduce that $\E'$ is semi-stable of slope $ \lambda = 1/m$. Note that since $s > 1$ we see that $\deg \E \ge 1$.

\begin{lemma} \phantomsection \label{itm : first auxiliary lemma}
    We have 
    \[
    \sum_{j=1}^{k-1} |\I(j)| > 1
    \]
\end{lemma}  
\begin{proof}
    Suppose the contrary that $ \displaystyle \sum_{j=1}^{k-1} |\I(j)| = 1 $. Thus we see that $k=2$ and $\E$ is a semi-stable vector bundle of slope $\lambda'_1 > 0$. By the property of the Harder-Narasimhan filtration, we see that $\nu_{\E_b}$ is bounded below by the concatenation of $\nu_{\E}$ and $\nu_{\E'}$ that we denote by $\nu_{\E} \oplus \nu_{\E'} $. Moreover, by \cite[lemma 2.9]{MC}, $\nu_{\E_b}$ is bounded above by $\nu_{\E \oplus \E'}$.
  
  If $\lambda < \lambda'_1 $ then from the inequalities $ \nu_{\E} \oplus \nu_{\E'} \le \nu_{\E_b} \le \nu_{\E \oplus \E'} $, we deduce that $ \E_b = \E \oplus \E' $. In this case we can apply lemma \ref{itm : second computation} to the triple $\E_b$, $\E_{b_{\chi'}}$ and $C_{\chi_{j^*}}$ to deduce that $ i^*_b C_{\chi_{j^*}} \star \mathcal{F}_{\chi'}$ vanishes if $ b \neq b_{\chi} $. 

   Suppose that $ \lambda \ge \lambda'_1 $. Notice that by \cite[Corollary 2.9]{MC} we see that if $h$ is a slope of $\nu_{\E_b}$ then $ \lambda \ge h > 0 $ since all the slopes of $\nu_{\E}$ and of $\nu_{\E'}$ are strictly positive and smaller than $\lambda$.

  Note that we have $ |\I(2)| \ge 2 $ since $ r > 2 $. By using proposition \ref{itm : shape of strata}, we deduce that $\E_b$ has a slope not smaller than $ 1/n_i $ for every $ i \in \I(2) $. However since $|\I(2)| \ge 2$, we deduce that $ m = \rank \OO(\lambda_2)^{m_2} > n_i $. It contradicts the fact that $ \lambda = 1/m \ge 1/n_i $ and allows us to conclude that 
  \[
    \sum_{j=1}^{k-1} |\I(j)| > 1
  \]  
\end{proof}

     Now by the lemma, we know that $ \displaystyle \sum_{j=1}^{k-1} |\I(j)| > 1 $, then the biggest slope of $\nu_{\E}$ is $\lambda_1$. Hence if $ h_{\rm max}$ is the biggest slope of $\nu_{\E_b}$ then $ h_{\rm max} \ge \lambda_1$.

    If $h_{\rm max} = \lambda_1$ then again, by applying lemma \ref{itm : second computation} to $\E_{b}$, $\E_{\chi'}$ and $\mu = (1, 0^{(n-1)})$ we can conclude that $ i^*_b C_{\chi_{j^*}} \star \mathcal{F}_{\chi'}$ vanishes if $ b \neq b_{\chi} $ and it gives us a contradiction. Thus we deduce that $h_{\rm max} > \lambda_1$ and that $b$ is not smaller than $b_{\chi}$ with respect to the usual order in $B(\GL_n)$.
    
    By proposition \ref{itm : shape of strata}, there exists $\xi = (a_1,\dotsc, a_r) \in \Irr(S_{\phi})^+$ such that $ b = b_{\xi}$. Since $\E_b$ is a direct sum of semi-stable vector bundles with strictly positive slopes, by the previous paragraph we deduce that
     \[
     C_{\xi} \star \mathcal{F}_{\Id} \simeq \mathcal{F}_{\xi}.
     \]

     We know that the spectral action preserves compact object then $ i^*_b C_{\chi} \star \mathcal{F}_{\Id} $ is a compact object and moreover all of its Schur-irreducible constituents have $L$-parameter given by $\phi$. Therefore, up to replacing $\xi$ by another $\xi'$ such that $b_{\xi} = b_{\xi'} = b$ and up to some shift, we can suppose as before that there is a non-trivial morphism 
     \[
     g_1 : i_{b !}i^*_b C_{\chi} \star \mathcal{F}_{\Id} \longrightarrow \mathcal{F}_{\xi}.
     \]

Since the set $S_{\rm supp} \subset B(\GL_n) $ of strata where the restriction of $C_{\chi} \star \mathcal{F}_{\Id}$ is non trivial is finite, there exists maximal elements in $S_{\rm supp}$ with respect to the restriction of the usual order in $B(\GL_n)$. Since $b$ is not smaller than $b_{\chi}$ with respect to that partial order, we can suppose that $b$ is a maximal element in $S_{\rm supp}$ as before. Hence, by applying the excision exact triangle to the closed embedding $  \Bun_n^{ \ge b}  \longrightarrow \Bun_n $, we deduce that there is a non trivial morphism
     \[
     g_2 : C_{\chi} \star \mathcal{F}_{\Id} \longrightarrow i_{b !}i^*_b C_{\chi} \star \mathcal{F}_{\Id} \xrightarrow{\ g_1 \ } \mathcal{F}_{\xi} \simeq C_{\xi} \star \mathcal{F}_{\Id}.
     \]

     We define $\I := \{ i \ | \ \tfrac{a_i}{n_i} = h_{\rm max} \}$ and since $ \displaystyle \sum_{j=1}^{k-1} |\I(j)| > 1 $, we can choose an index $\widetilde{j} \in \displaystyle \bigcup_{j = 1}^{k-1} \I(j) $ such that $| \I \setminus \{ \widetilde{j} \} | \ge 1$.
     
     By applying the auto-equivalence $ C_{\chi^{-1}_{\widetilde{j}}} $ to $g_2$ we get a non trivial morphism  
     \[
     g : C_{\chi \otimes \chi^{-1}_{\widetilde{j}}} \star \mathcal{F}_{\Id} \longrightarrow C_{\xi \otimes \chi^{-1}_{\widetilde{j}}} \star \mathcal{F}_{\Id}.
     \]

     Note that by the choice of $\widetilde{j}$, we see that $ \chi \otimes \chi^{-1}_{\widetilde{j}}, \xi \otimes \chi^{-1}_{\widetilde{j}} \in \Irr(S_{\phi})^{+} $ then by the induction hypothesis on $D = s-1$, we deduce that $C_{\chi \otimes \chi^{-1}_{\widetilde{j}}} \star \mathcal{F}_{\Id} $ is supported on $ b_{\chi \otimes \chi^{-1}_{\widetilde{j}}} $ and $C_{\xi \otimes \chi^{-1}_{\widetilde{j}}} \star \mathcal{F}_{\Id}$ is supported on $ b_{\xi \otimes \chi^{-1}_{\widetilde{j}}} $.
     
     Again, by the choice of $\widetilde{j}$, the biggest slope of $\E_{b_{\xi \otimes \chi^{-1}_{\widetilde{j}}}}$ is still $h_{\rm max}$ and the biggest slope of $\E_{b_{\chi \otimes \chi^{-1}_{\widetilde{j}}}}$ is not bigger than $\lambda_1$. In particular, $b_{\xi \otimes \chi^{-1}_{\widetilde{j}}}$ is not smaller than $b_{\chi \otimes \chi^{-1}_{\widetilde{j}}}$ with respect to the usual partial order of $B(\GL_n)$. Thus by lemma \ref{itm : morphism between strata} there is no non-trivial morphism from $\mathcal{F}_{\chi \otimes \chi^{-1}_{\widetilde{j}}}$ to $\mathcal{F}_{\xi \otimes \chi^{-1}_{\widetilde{j}}}$. Hence a contradiction. Therefore, the restriction $i^*_b C_{\chi} \star \mathcal{F}_{\Id}$ must be trivial.

\subsubsection{The case where $\lambda_k = 0$ and $r = 2$} \textbf{} \\

At this point we know that $ k = r = 2 $ hence we can suppose $\I(1) = \{ 1 \}$ and $\I(2) = \{ 2 \}$. Thus $\E_{\chi} \simeq \OO(\lambda_1)^{m_1} \oplus \OO^m$ with $\chi = (x, 0)$ for some $x \in \N$. Then we choose $\chi' = (x-1, 0)$ and denote $b \neq b_{\chi} \in B(\GL_n)$ an element such that $i^*_bC_{\chi} \star \mathcal{F}_{\Id}$ is non trivial. Hence, there is a modification of type $\mu = (1, 0^{(n-1)})$
\[
f : \E_{b_{\chi'}} \longrightarrow \E_b.
\]

All the arguments in the last paragraph (the case $ r > 2 $) before lemma \ref{itm : first auxiliary lemma} still work. Therefore we deduce that $\E_b \simeq \E_{\xi}$ where $\xi = (x_1, x_2) \in \Irr(S_{\phi})^{> 0}$ and where 
$\rank \OO(\lambda_1)^{m_1} = \dim \phi_1 =: n_1$; $m = \dim \phi_2 =: n_2$ and $ x_1 + x_2 = x $. Moreover we can also deduce that 
there is an injection $\OO(\lambda_1') \hookrightarrow \E_b $ and the quotient is isomorphic to $ \OO(\lambda) $ where $ \lambda'_1 = \tfrac{x_1+x_2-1}{n_1} $, $ \lambda = \tfrac{1}{n_2}$. Remark that we also know (by the case $\lambda_k > 0$) that
\[
C_{\xi} \star \mathcal{F}_{\Id} \simeq \mathcal{F}_{\xi}.
\]

By using lemma \ref{itm : second computation} to the triple $\E_b$, $\E_{b_{\chi'}}$ and $C_{\chi_{j^*}}$, we can deduce that $\lambda \ge \lambda'_1$ as before. In particular we see that $ \tfrac{x_2}{n_2} > \tfrac{x_1}{n_1} $. We consider the following cases

\textbf{Case 1 :} $\tfrac{x_2}{n_2} > \lambda_1$.

Thus $\nu_{\E_{b}}$ is not smaller than $ \nu_{\E_{b_{\chi}}} $ with respect to the usual partial order in $X_{*}(\T)_{\Q}$. By choosing an element $b$ maximal with these properties and using the excision exact triangles argument with respect to the closed embedding $ \Bun_n^{\ge b} \hookrightarrow \Bun_n$ as above, we can obtain a contradiction.

\textbf{Case 2 :} $ \lambda_1 \ge \tfrac{x_2}{n_2} $.

In this case $\nu_{\E_{b}}$ is smaller than $ \nu_{\E_{b_{\chi}}} $ with respect to the usual partial order in $X_{*}(\T)_{\Q}$. Note that the spectral action preserves ULA objects then $C_{\chi} \star \mathcal{F}_{\Id}$ is ULA and then $ i^*_{b_{\xi}} (C_{\chi} \star \mathcal{F}_{\Id}) $ is also ULA. Hence there is a non-trivial morphism
\[
g : C_{\xi} \star \mathcal{F}_{\Id} \simeq \mathcal{F}_{\xi} \longrightarrow i^*_{b_{\xi}} (C_{\chi} \star \mathcal{F}_{\Id}),
\]
more precisely, we can construct $g$ from a sub-module of the non-trivial cohomology group of $i^*_{b_{\xi}} (C_{\chi} \star \mathcal{F}_{\Id})$ which has lowest degree. 

 Now we can suppose that $b$ is minimal with respect to the partial order among the strata where the restriction of $C_{\chi} \star \mathcal{F}_{\Id}$ is non-trivial. In particular $b$ is smaller than and not equal to $b_{\chi}$ in $B(\GL_n)$. By similar excision exact triangles argument for the open embedding $ \Bun_n^{b \ge} \hookrightarrow \Bun_n $, we can suppose that up to some shift, there is a non-trivial morphism
\[
g_1 : C_{\xi} \star \mathcal{F}_{\Id} \simeq \mathcal{F}_{\xi} \longrightarrow C_{\chi} \star \mathcal{F}_{\Id}.
\]

Now we apply the auto-equivalence $C_{\chi^{-1}_1}$ on $g_1$ to get a non-trivial morphism (where $\chi_1 = (1, 0)$)
\[
g_2 : C_{\xi \otimes \chi^{-1}_1} \star \mathcal{\Id} \simeq \mathcal{F}_{\xi \otimes \chi^{-1}_1 } \longrightarrow C_{\chi \otimes \chi^{-1}_1} \star \mathcal{F}_{\Id} \simeq \mathcal{F}_{\chi \otimes \chi^{-1}_1}.
\]

If $x_1 + x_2 > 2$ then we see that $b_{\chi \otimes \chi^{-1}_1}$ is not smaller than $b_{\xi \otimes \chi^{-1}_1}$ and lemma \ref{itm : morphism between strata} gives rise to a contradiction.

If $x_1 + x_2 = 2$ then $x_1 = x_2 = 1$ and $\xi \otimes \chi^{-1}_1 = (0, 1)$ and $\chi \otimes \chi^{-1}_1 = (1, 0)$. Applying again the auto-equivalence $\chi_2^{-1}$ on $g_2$ to get a non-trivial morphism
\[
g_3 : C_{\Id} \star \mathcal{F}_{\Id} \longrightarrow C_{\xi'} \star \mathcal{F}_{\Id},
\]
where $\xi' = (1, -1)$. By lemma \ref{itm : morphism between strata} the restriction $i_{1}^*C_{\xi'} \star \mathcal{F}_{\Id}$ is non trivial. However it contradicts equation (\ref{itm : restriction to 1}).

It allows us to finally conclude that 
\[
C_{\chi} \star \mathcal{F}_{\Id} \simeq \mathcal{F}_{\chi}.
\]

\section{Hecke-eigensheaves for $\GL_n$} \label{itm : Hecke-eigensheaves}

The main goals of this subsection is to use theorem \ref{itm : main theorem} to give an explicit description of the Hecke eigensheaves associated to some $L$-parameters $\phi$. \\ 

Throughout this paragraph, we fix an $L$-parameter $\phi = \phi_1 \oplus \dotsc \oplus \phi_r$ satisfying the conditions of theorem \ref{itm : main theorem}. In particular the connected component $[C_{\phi}]$ of $[Z^1(W_{\Q_p}, \widehat{\GL}_n)/\widehat{\GL}_n]$ containing $\phi$ is isomorphic to $[\bb G_m^r/ \bb G_m^r ]$ with the trivial action of $\bb G_m^r$. Then $S_{\phi} = \displaystyle \prod_{i = 1}^r \Gm $, $ \Irr(S_{\phi}) \displaystyle \simeq \prod_{i = 1}^r \Z $ and all the algebraic irreducible representations of $ S_{\phi} $ are of dimension $1$. Moreover, the category $\Rep_{\Q_{\ell}} (S_{\phi}) $ is semi-simple, we see that the regular representation of $S_{\phi}$ has the following description
\[
V_{\text{reg}} = \bigoplus_{\chi \in \Irr(S_{\phi})} \chi.
\] 
We deduce an explicit description of a Hecke eigensheaf associated to $\phi$.
\begin{theorem} \label{itm : Hecke eigensheaf}
The sheaf
\[
\displaystyle \mathcal{G}_{\phi} := \bigoplus_{\chi \in \Irr(S_{\phi})} \mathcal{F}_{\chi}
\]
is a non trivial Hecke eigensheaf corresponding to the $L$-parameter $\phi$.
\end{theorem}
\begin{proof}
    By the description of the $\mathcal{F}_{\chi}$'s, it is clear that the stalk of $\mathcal{G}_{\phi}$ at the stratum $ \Bun^1_n $ is isomorphic to $ \mathcal{F}_{\Id} $. Thus the sheaf $\mathcal{G}_{\phi}$ is non-zero. 

    Remark that for every $\chi' \in \Irr(S_{\phi})$ we have $ \displaystyle \chi' \otimes \bigoplus_{\chi \in \Irr(S_{\phi})} \chi = \bigoplus_{\chi \in \Irr(S_{\phi})} \chi $. Thus by theorem \ref{itm : main theorem} and the monoidal property of the spectral action, we deduce that
    \[
    C_{\chi'} \star \mathcal{G}_{\phi} = C_{\chi'} \star \bigoplus_{\chi \in \Irr(S_{\phi})} C_{\chi} \star \mathcal{F}_{\Id} = \bigoplus_{\chi \in \Irr(S_{\phi})} C_{\chi} \star \mathcal{F}_{\Id} = \mathcal{G}_{\phi}.
    \]

    Let $ V $ be an algebraic representation of $ \GL_n(\overline{\Q}_{\ell}) $ and let $\T_V$ be the corresponding Hecke operator. We show that $ \displaystyle \T_V ( \mathcal{G}_{\phi} ) = \mathcal{G}_{\phi} \boxtimes V \circ \phi $ as sheaf on $ \Bunn $ with $W_{\Q_p}$-action. Recall that the restriction of $V$ to $S_{\phi}$ admits a commuting $W_{\Q_p}$-action given by $\phi$. Thus, as $S_{\phi} \times W_{\Q_p}$-representation, we have
    \[
    V \circ \phi \simeq \bigoplus_{\chi \in \Irr(S_{\phi})} \chi \boxtimes \sigma_{\chi},    
    \]
where $\sigma_{\chi} := \Hom_{S_{\phi}}(\chi, V \circ \phi)$ is a $W_{\Q_p}$-representation.

    In particular, since $\dim \chi = 1$ we deduce that $ \displaystyle \bigoplus_{\chi \in \Irr(S_{\phi})} \sigma_{\chi} = V \circ \phi $.
    
    On the other hand, we have
    \[
    \T_V(\mathcal{G}_{\phi}) = \bigoplus_{\chi \in \Irr(S_{\phi})} \big( C_{\chi} \star \mathcal{G}_{\phi} \big) \boxtimes \sigma_{\chi} = \bigoplus_{\chi \in \Irr(S_{\phi})} \mathcal{G}_{\phi} \boxtimes \sigma_{\chi} = \mathcal{G}_{\phi} \boxtimes \bigoplus_{\chi \in \Irr(S_{\phi})} \sigma_{\chi} = \mathcal{G}_{\phi} \boxtimes V \circ \phi.
    \]

    Therefore $\mathcal{G}_{\phi}$ is a non-zero Hecke eigensheaf of the $L$-parameter $\phi$.
\end{proof}
\begin{example}
Let us describe the stalks of $\mathcal{G}_{\phi}$ in some special cases. 

Suppose further that $ \phi = \phi_1 \oplus \dotsc \oplus \phi_n $ is a sum of $n$ characters. Thus the Newton polygon of the vector bundle $\E_{b_{\chi}}$ are of the form $(d_1^{(n_1)}, \dotsc, d_k^{(n_k)})$ for a decreasing chain $ d_1 > d_2 > \dotsc > d_k $ of integers (where $d_i^{(n_i)}$ indicates that $d_i$ appears with multiplicity $n_i$).


Let $ b = b_{\xi} $ for $ \xi = (d_1^{(n_1)}, \dotsc, d_k^{(n_k)}) \in \Irr(S_{\phi}) $. We see that $\G_b = \GL_{n_1} \times \dotsc \times \GL_{n_k}$ and for a character $ \chi = (t_1, \dotsc, t_n) \in \Irr(S_{\phi})$ we have $ b_{\chi} = b $ if and only if $(t_1, \dotsc, t_n)$ is a permutation of $(d_1^{(n_1)}, \dotsc, d_k^{(n_k)})$. The stabilizer of $(t_1, \dotsc, t_n)$ under the action of the permutation group $S_n$ is a subgroup isomorphic to $ S_{n_1} \times S_{n_2} \times \dotsc \times S_{n_k} $. Therefore the stalk of $\mathcal{G}_{\phi}$ at $\Bun^b_n$ is given by
\[
\bigoplus_{w \in W_{\GL_n} / W_{\G_b}} \mathcal{F}_{\xi^w}
\]
where $W_{\G}$ denotes the Weyl group of $\G$. This description is compatible with the one given by L. Hamann in \cite[Theorem 1.14]{Ham1}.
\end{example}

 One expects that the restriction of $ i^*_b \mathcal{G}_{\phi}$ to the stratum $\Bun_n^b$ is closely related to the functor $\Red_b$ defined by Shin \cite[\S 6.2]{SWS} and Bertoloni-Meli \cite[Definition 5.6]{ABM}. It is first observed by Hamann in \cite[\S 11]{Ham1}.

\section{Harris-Viehmann's conjecture} \label{itm : Harris-Viehmann}

The main goals of this section is to compute part of the cohomology of the moduli spaces $\Sht(\GL_n, b, b', \mu)$ and deduce new cases of the Harris-Viehmann's conjecture for $\GL_n$. 

Throughout this section, we fix an $L$-parameter $\phi = \phi_1 \oplus \dotsc \oplus \phi_r$ satisfying the conditions of theorem \ref{itm : main theorem}. In particular we have $S_{\phi} = \displaystyle \prod_{i = 1}^r \Gm $, $ \Irr(S_{\phi}) \displaystyle \simeq \prod_{i = 1}^r \Z $ and $[C_{\phi}] \simeq [\bb G_m^r / \bb G_m^r]$, the quotient of $\bb G_m^r$ by the trivial action of $\bb G_m^r$.

Let $b, b' \in B(\GL_n)$ be such that $\phi$ is $\G_{b}(\Q_p)$-relevant (see the paragraph before theorem \ref{itm : relevant}) and and let $\pi_{b}$ be the irreducible representation of $\G_{b}(\Q_p)$ corresponding to $\phi$ under the local Langlands correspondence. Let $\lambda \in \Irr(S_{\phi})$ be the character corresponding to the pair $(b, \pi_{b})$. Then by lemmas \ref{shimhecke} and \ref{itm : kappa and modulus} we have
\[
 R\Gamma_{c}(\GL_n,b',b,\mu)[\delta^{1/2}_b \otimes \pi_{b}][d_{b}] \simeq  i_{b'}^{*}\T_{\mu^{-1}}\mathcal{F_{\lambda}}, 
\]
where $d_b = \langle 2\rho, \nu_b \rangle $.

As $S_{\phi} \times W_{\Q_p}$-representations, we have an identification
\[
r_{\mu^{-1}} \circ \phi = \bigoplus_{\chi \in \Irr(S_{\phi})} \chi \boxtimes \sigma_{\chi}, 
\]
where $\sigma_{\chi}$ is the $W_{\Q_p}$-representations $\Hom_{S_{\phi}}(\chi, r_{\mu^{-1}} \circ \phi)$, then we have 
\[
\T_{\mu^{-1}} \mathcal{F}_{\lambda} = \bigoplus_{\chi \in \Irr(S_{\phi})} C_{\chi} \star \mathcal{F}_{\lambda} \boxtimes \sigma_{\chi}. 
\]

Hence we can use the explicit description of the action of $C_{\chi}$ to compute $R\Gamma_{c}(\GL_n,b',b,\mu)) [\delta^{1/2}_b \otimes \pi_{b}]$.

\begin{theorem} \label{itm : Harris-Viehmann conjecture}
    Let $ b \in B(\GL_n) $ be an element such that $\phi$ is $ \G_b(\Q_p) $-relevant and let $\M$ be the standard Levi subgroup of $\GL_n$ that is the split inner form of $\G_b$ and $\rP$ is the standard parabolic subgroup of $\GL_n$ whose Levi factor is $\M$. Let $\mu$ be an arbitrary cocharacter of $\GL_n$ and $ \pi_b $ be an irreducible representation of $\G_b(\Q_p)$ such that its corresponding $L$-parameter $\phi^b$ post-composed with the natural embedding $ ^{L}\G_b(\overline{\mathbb{Q}}_{\ell}) \longrightarrow ^{L}\GL_n(\overline{\mathbb{Q}}_{\ell}) $ is $\phi$. Thus we have 
    \begin{equation} \phantomsection \label{itm : cohomology of RZ spaces}
      R\Gamma_{c}(\GL_n,b,\mu) [\delta^{1/2}_b \otimes \pi_b] = \pi_1 \boxtimes \Hom_{S_{\phi}} (\chi^{-1}_b ,r_{\mu^{-1}} \circ \phi)[-d]  
    \end{equation}
    where $\pi_1$ is the irreducibe representation of $\GL_n(\Q_p)$ whose $L$-parameter is $\phi$ and $\chi_b$ is the unique character of $S_{\phi}$ corresponding to the couple $(b, \pi_b)$ and where $d = \langle 2\rho, \nu_b \rangle$. In particular, the generalized Harris-Viehmann's conjecture is true in this case. More precisely, suppose that as $ ^{L}\M(\overline{\mathbb{Q}}_{\ell})$-representation, we have 
    \begin{equation} 
   r_{\mu^{-1}} = \bigoplus_{\mu_{\M} \in X_{*}(\M) } (r_{\mu^{-1}_{\M}})^{m_{\mu^{-1}_{\M}}} 
\end{equation}
for some multiplicity $ m_{\mu^{-1}_{\M}} = \dim \Hom ( r_{\mu^{-1}_{\M}} ,r_{\mu^{-1} | \M} ) $. Then we have

\begin{align*}
R\Gamma_{c}(\GL_n, b,\mu) [ \delta^{1/2}_b \otimes \pi_{b}] [d - d_{\M}]  &\simeq \bigoplus_{\mu_{\M} \in X_{*}(\M) } \big( \Ind_{\rP}^{\GL_n} R\Gamma_{c}(\M, b_{\M},\mu_{\M})  [\delta^{1/2}_{b_{\M}} \otimes \pi_{b}]\big)^{m_{\mu^{-1}_{\M}}}.
\end{align*}
where $d_{\M} = \langle 2\rho_{\M}, \nu_{b_{\M}} \rangle$, $ d = \langle 2\rho, \nu_b \rangle $ and $b_{\M}$ is the reduction of $b$ to $\M$.
\end{theorem}
\begin{proof}
We recall that $\Sht(\GL_n,b,\mu)$ is the space parametrizing modifications $f : \E_1 \longrightarrow \E_b$ of type $\mu$.
As before, we have
\[
 R\Gamma_{c}(\GL_n, b,\mu) [\delta^{1/2}_b \otimes \pi_{b}] \simeq  i_{1}^{*}\T_{\mu^{-1}} \mathcal{F}_{\chi_b}[-d].
\]

Then we consider the $S_{\phi} \times W_{\Q_p} $-representation $ \displaystyle r_{\mu^{-1}} \circ \phi = \bigoplus_{\chi \in \Irr(S_{\phi})} \chi \boxtimes \sigma_{\chi}$ where $\sigma_{\chi}$ is the $W_{\Q_p}$-representation $\Hom_{S_{\phi}}(\chi, r_{\mu^{-1}} \circ \phi)$. Since the couple $(b, \pi_b)$ corresponds to the character $\chi_b$, we deduce that $ i_{1}^{*} C_{\chi} \star \mathcal{F}_{\chi_b} \simeq 0 $ if $\chi \neq \chi^{-1}_b$. Thus
\begin{align*}
 R\Gamma_{c}(\GL_n, b,\mu) [\delta^{1/2}_b \otimes \pi_{b}] &\simeq \bigoplus_{\chi \in \Irr(S_{\phi})} i^{*}_1C_{\chi} \star \mathcal{F}_{\chi_b} \boxtimes \sigma_{\chi}[-d] \\
 &\simeq \quad i^{*}_1C_{\chi^{-1}_b} \star \mathcal{F}_{\chi_b} \boxtimes \sigma_{\chi^{-1}_b}[-d] \\
 &\simeq \pi_1 \boxtimes \Hom_{S_{\phi}} (\chi^{-1}_b ,r_{\mu^{-1}} \circ \phi)[-d].   
\end{align*}

Suppose that $\E_b \simeq \E(\lambda_1) \oplus \dotsc \oplus \E(\lambda_k)$ where $\lambda_1 > \dotsc > \lambda_k$ and for each $j$, $\E(\lambda_j)$ is a semi-stable vector bundle of slope $\lambda_j$. Then $\M \simeq \GL_{m_1} \times \dotsc \times \GL_{m_k}$ where $m_j := \rank\E(\lambda_j)$ and the natural embedding $ ^{L}\G_b(\overline{\mathbb{Q}}_{\ell}) \longrightarrow \ ^{L}\GL_n(\overline{\mathbb{Q}}_{\ell}) $ induces a morphism of $L$-groups $ \eta : \ ^{L}\M(\overline{\mathbb{Q}}_{\ell}) \longrightarrow \ ^{L}\GL_n(\overline{\mathbb{Q}}_{\ell}) $ such that $ \phi = \eta \circ \phi^b $. In particular we have $S_{\phi^b} \simeq S_{\phi} $ and from now on, for each $ \chi \in \Irr(S_{\phi}) $, we denote by $ \chi_{\M} $ its corresponding character in $ \Irr(S_{\phi^b}) $. Now we consider $ r_{\mu^{-1}} $ as an $ \M $-representation by $\eta$. Since the category $\Rep_{\overline{\Q}_{\ell}}\M$ is semi-simple, there is a decomposition
\begin{equation} \phantomsection \label{itm : restriction to M}
   r_{\mu^{-1}} = \bigoplus_{\mu_{\M} \in X_{*}(\M) } (r_{\mu^{-1}_{\M}})^{m_{\mu^{-1}_{\M}}} 
\end{equation}
for some multiplicity $ m_{\mu^{-1}_{\M}} = \dim \Hom ( r_{\mu^{-1}_{\M}} ,r_{\mu^{-1} | \M} ) $. Moreover the decomposition (\ref{itm : restriction to M}) is compatible with $W_{\Q_p}$-action since $ \phi = \eta \circ \phi^b $. Thus
\[
 \Hom_{S_{\phi}} (\chi^{-1}_b ,r_{\mu^{-1}} \circ \phi) = \bigoplus_{\mu_{\M} \in X_{*}(\M) } \Hom_{S_{\phi^b}} (\chi^{-1}_{b, \M} ,r_{\mu^{-1}_{\M}} \circ \phi^b)^{m_{\mu^{-1}_{\M}}}.
\]

Denote by $b_{\M}$ the reduction of $b$ to $\M$ and $\pi^{\M}_1$ the unique irreducible representation of $\M(\Q_p)$ whose $L$-parameter is $\phi^b$ then $ \pi_1 = \Ind_{\rP}^{\GL_n} \pi^{\M}_1 $ (normalized parabolic induction). By using the formula (\ref{itm : cohomology of RZ spaces}) for $\M$ we deduce that  
\begin{align*}
R\Gamma_{c}(\GL_n, b,\mu) [ \delta^{1/2}_b \otimes \pi_{b}] [d - d_{\M}]  &\simeq \pi_1 \boxtimes \Hom_{S_{\phi}} (\chi^{-1}_b ,r_{\mu^{-1}} \circ \phi)[-d_{\M}]  \\
&\simeq \pi_1 \boxtimes \bigoplus_{\mu_M \in X_{*}(M) } \Hom_{S_{\phi^b}} (\chi^{-1}_{b, \M} ,r_{\mu^{-1}_M} \circ \phi^b)^{m_{\mu^{-1}_M}}[-d_{\M}]  \\
&\simeq \bigoplus_{\mu_{\M} \in X_{*}(\M) } \big( \Ind_{\rP}^{\GL_n} R\Gamma_{c}(\M, b_{\M},\mu_{\M})  [\delta^{1/2}_{b_M} \otimes \pi_{b}]\big)^{m_{\mu^{-1}_{\M}}}.
\end{align*}
where $d_{\M} = \langle 2\rho_{\M}, \nu_{b_{\M}} \rangle$ and $ d = \langle 2\rho, \nu_b \rangle $.


If $\mu = (1^{(a)}, 0^{(n - a)})$ is a minuscule cocharacter then $ r_{\mu^{-1}} = (\Lambda^{a} \Std)^{\vee} $. Hence if $\chi_b = (d_1, \dotsc, d_r)$ for $ d_1, \dotsc, d_r \ge 0 $ then 
\[
\Hom_{S_{\phi}} (\chi^{-1}_b ,r_{\mu^{-1}} \circ \phi) = \bigotimes^r_{i = 1} (\Lambda^{d_i} \phi_i)^{\vee},
\]
thus
\[
R\Gamma_{c}(\GL_n, b,\mu) [ \delta^{1/2}_b \otimes \pi_{b}] [d - d_{\M}] \simeq \Ind_{\rP}^{\GL_n} R\Gamma_{c}(\M, b_{\M},\mu_{\M})  [\delta^{1/2}_{b_M} \otimes \pi_{b}]
\]
where $\mu_{\M} = \mu_1 \times \dotsc \times \mu_k$ and $\mu_i = (1^{(\deg\E(\lambda_i))}, 0^{(\rank\E(\lambda_i) - \deg\E(\lambda_i))}) $. Hence we obtain the Harris-Viehmann's conjecture in this particular case \cite[Conjecture 5.2]{Har} \cite[Conjecture 8.5]{RV} \cite[Conjecture 3.2.1]{Ber22}.
  
\end{proof}
\begin{example}
We suppose that $ b \in B(\GL_2) $ is an element such that $\mathcal{E}_b = \mathcal{O}(-1) \oplus \mathcal{O}(-2)$; $\mu^{-1} = (3, 0)$ and $ \phi = \phi_1 \oplus \phi_2 $ is the sum of $2$ characters. We see that $ S_{\phi} \simeq \Gm \times \Gm $ and $ \Irr(S_{\phi}) \simeq \Z \times \Z $. We denote by $\chi_{(x, y)}$ the character corresponding to $(x,y)$. Since $ r_{\mu^{-1}} = \Sym^3 \Std \GL_2 $, we have the following identification of $ S_{\phi} \times W_{\Q_p} $-representations
\[
r_{\mu^{-1}} \circ \phi \simeq \chi_{(3,0)} \boxtimes \phi_1^3 \bigoplus \chi_{(2,1)} \boxtimes \phi_1^2 \otimes \phi_2 \bigoplus \chi_{(1,2)} \boxtimes \phi_1 \otimes \phi_2^2 \bigoplus \chi_{(0,3)} \boxtimes \phi^3_2.
\]

Let $\pi_b$ be the representation such that the couple $(b, \pi_b)$ corresponds to the character $\chi_{(-1, -2)}$. Thus theorem \ref{itm : main theorem} implies that $ i^*_1C_{\chi_{(3, 0)}} \star \mathcal{F}_{\chi_{(-1, -2)}} = i^*_1C_{\chi_{(2, 1)}} \star \mathcal{F}_{\chi_{(-1, -2)}} = i^*_1C_{\chi_{(0, 3)}} \star \mathcal{F}_{\chi_{(-1, -2)}} = 0 $ and $i^*_1C_{\chi_{(1, 2)}} \star \mathcal{F}_{\chi_{(-1, -2)}} = \pi_1$. Hence
    \[
    R\Gamma_{c}(\GL_n,b,\mu)) [ \delta^{1/2}_{b} \otimes \pi_b] = \pi_1 \boxtimes \phi_1 \otimes \phi^2_2 [-d],
    \]
where $d = \langle 2\rho, \nu_b \rangle = 1 $.\footnote{See also the computation in \cite[Example 8.10]{Nao}.}

\end{example}
\begin{example}
    We suppose that $ \phi = \phi_1 \oplus \phi_2 $ where $\phi_1$ and $\phi_2$ are irreducible representations whose dimension are given by natural numbers $n_1 > n_2$. Let $ b \in B(\GL_n) $ be the element such that $ \mathcal{E}_{b} = \OO(-1/n_1) \oplus \O(-1/n_2) $. We are going to compute the cohomology of Rapoport-Zink spaces/local Shimura varieties $ R\Gamma_{c}(\GL_n,b,\mu_i))  [\pi_b] $ ($i = 1, 2$) where $ \mu^{-1}_1 = (1^{(2)}, 0^{(n-2)}) $; $ \mu^{-1}_2 = (2, 0^{(n-1)}) $ and $\pi_b$ is the irreducible representation of $ \G_b(\Q_p) $ whose $L$-parameter is $\phi$. Thus the couple $(b, \pi_b)$ corresponds to the character $\chi_{(-1, -1)}$.

    We see that $ S_{\phi} \simeq \Gm \times \Gm $ and $ \Irr(S_{\phi}) \simeq \Z \times \Z $. We denote by $\chi_{(x, y)}$ the character corresponding to $(x,y)$. Since $ \displaystyle r_{\mu^{-1}_1} = \Lambda^2 \Std$, we deduce the following identification of $ S_{\phi} \times W_{\Q_p} $-representations
    \[
    r_{\mu^{-1}_1} \circ \phi \simeq \chi_{(-2, 0)} \boxtimes \Lambda^2 \phi_1 \bigoplus \chi_{(-1, -1)} \boxtimes \phi_1 \otimes \phi_2 \bigoplus \chi_{(0, -2)} \boxtimes \Lambda^2 \phi_2.
    \]

    By theorem \ref{itm : main theorem}, we have $ i^*_1C_{\chi_{(2, 0)}} \star \mathcal{F}_{\chi_{(1,1)}} = i^*_1C_{\chi_{(0, 2)}} \star \mathcal{F}_{\chi_{(1,1)}} = 0 $ and $i^*_1C_{\chi_{(1, 1)}} \star \mathcal{F}_{\chi_{(1,1)}} = \pi_1$ where $\pi_1$ is the irreducible representation of $\GL_n$ whose $L$-parameter is $\phi$. Hence
    \[
    R\Gamma_{c}(\GL_n,b,\mu_1)[\delta^{1/2}_b \otimes \pi_b] = \pi_1 \boxtimes \phi_1 \otimes \phi_2[-d],
    \]
where $d = \langle 2\rho, \nu_b \rangle = n_2 - n_1 $.

    Similarly, we have $ \displaystyle r_{\mu^{-1}_2} = \Sym^2 \Std $, we deduce the following identification of $ S_{\phi} \times W_{\Q_p} $-representations
    \[
    r_{\mu^{-1}_2} = \chi_{(2, 0)} \boxtimes \Sym^2 \phi_1 \bigoplus \chi_{(1, 1)} \boxtimes \phi_1 \otimes \phi_2 \bigoplus \chi_{(0, 2)} \boxtimes \Sym^2 \phi_2.
    \]

    Therefore
    \[
    R\Gamma_{c}(\GL_n,b,\mu_2)) [\delta^{1/2}_b \otimes \pi_b] = \pi_1 \boxtimes \phi_1 \otimes \phi_2[-d],
    \]
where $d = \langle 2\rho, \nu_b \rangle = n_2 - n_1$.    
\end{example}

\section{On the categorical form of Fargues' conjecture for $\GL_n$} \phantomsection \label{itm : categorical Langlands}
In this section, we describe the map $\Psi_{\GL_n}$ from the spectral Bernstein center of $\GL_n$ to its Bernstein center by using the compatibility between Fargues-Scholze $L$-parameters and the usual $L$-parameters for $\GL_n$. Then we combine this description with theorem \ref{itm : main theorem} to describe the action of $\Perf([C_{\phi}])$ on $\Dlis(\Bun_n, \ov \Q_{\ell})^{\omega}$.
\subsection{Bernstein centers}\textbf{}

Recall that the spectral Bernstein center of $\GL_n$ is $ \mc Z^{\rm spec}(\GL_n, \overline{\Q}_{\ell}) := \mathcal{O}(Z^1(W_{\Q_p}, \GL_n))^{\GL_n} $, the ring of global functions on the stack/the coarse moduli space of $L$-parameters. The geometric Bernstein center of $\GL_n$ is $ \mc Z^{\rm geom}( \GL_n, \overline{\Q}_{\ell} ) := \pi_0 (\End (\id_{\Dlis(\Bun_n, \overline{\Q}_{\ell})})) $. There is a natural identification between $ \mc Z^{\rm spec}(\GL_n, \overline{\Q}_{\ell}) $ and the algebra of excursion operators (\cite[theorem VIII.5.1]{FS}). By \cite[theo. VIII.4.1]{FS}, there is a map $ \mc Z^{\rm spec}(\GL_n, \overline{\Q}_{\ell}) \longrightarrow \mc Z^{\rm geom}(\GL_n, \overline{\Q}_{\ell})$ which induces a map
\[
\Psi_{\GL_n} : \mc Z^{\rm spec}(\GL_n, \overline{\Q}_{\ell}) \longrightarrow \mc Z (\GL_n, \overline{\Q}_{\ell})
\]
by the inclusion $\Dlis(\GL_n(\Q_p), \overline{\Q}_{\ell}) \hookrightarrow \Dlis(\Bun_n, \overline{\Q}_{\ell})$ and where $\mc Z (\GL_n, \overline{\Q}_{\ell})$ is the Bernstein center of $\GL_n$.

Let $f \in \mathcal{O}(Z^1(W_{\Q_p}, \GL_n))^{\GL_n}$ be an element and denote by $\mathcal{O}$ the structural sheaf of $[Z^1(W_{\Q_p}, \GL_n) / \GL_n ]$. In particular the multiplication by $ f $ gives rise to a map of $\mathcal{O}(Z^1(W_{\Q_p}, \GL_n))^{\GL_n}$-modules $ \mc O \xrightarrow{ \ \cdot f \ } \mc O  $. Note that $\mathcal{O}$ is the vector bundle $C_{\rm tri}$ corresponding to the Hecke operator of the trivial representation of $\GL_n$. Hence via the spectral action, the map $ \mc O \xrightarrow{ \ \cdot f \ } \mc O  $ induce an element in $\End(\mc O \star A) \simeq \End(A)$ for every $A \in \Dlis(\Bun_n, \ov \Q_{\ell})$ and therefore it induces an element $\widetilde{f}$ in $ \End(\id_{\Dlis(\Bun_n, \overline{\Q}_{\ell})^{\omega}}) $. Now the excursion operator corresponding to $f$ also gives rise to an element $\overline{f}$ in $ \End(\id_{\Dlis(\Bun_n, \overline{\Q}_{\ell})^{\omega}}) $. By the compatibility of the spectral action and the excursion operators \cite[Theorem 5.2.1]{Zou}, we have $ \widetilde{f} = \overline{f} $. More concretely, let $A$ be a Schur-irreducible element in $ \Dlis(\Bun_n, \overline{\Q}_{\ell})^{\omega} $ whose Fargues-Scholze $L$-parameter is given by $\phi$. Then $ \mc O \xrightarrow{ \ \cdot f \ } \mc O  $ acting on $A$ gives us an endomorphism
\[
\{ \mc O \star A = A \longrightarrow \mc O \star A = A \} \in \End(A) = \overline{\Q}_{\ell},
\]
which will be precisely the scalar $\alpha$ given by evaluating $A$ on the excursion datum corresponding to $f$ (see also \cite[page 24] {Ham}). By \cite[Proposition VIII.3.8]{FS}, $\phi$ corresponds to a surjective map 
\[
\rm Ev_{\phi} : \mathcal{O}(Z^1(W_{\Q_p}, \GL_n))^{\GL_n} \longrightarrow \overline{\Q}_{\ell}
\]
and the scalar $\alpha$ above obtained by evaluating $\phi$ on the excursion datum corresponding to $f$ is exactly $ \mathrm{Ev}_{\phi} (f) $.\\

Now let $\phi = \phi_1 \oplus \dotsc \oplus \phi_r$ be an $L$-parameter satisfying the condition of theorem \ref{itm : main theorem} and let $[C_{\phi}]$ be the connected component of $[Z^1(W_{\Q_p}, \GL_n)/\GL_n]$ containing $\phi$. By proposition \ref{itm : simple connected components}, $[C_{\phi}]$ is isomorphic to $[\bb G_m^r / \bb G_m^r]$ where $ \bb G_m^r $ acts trivially. Therefore, the ring of global functions of $[C_{\phi}]$ is given by $ \overline{\Q}_{\ell}[X_1, \dotsc, X_r, X_1^{-1}, \dotsc, X_r^{-1}] $ and it is a direct factor of $\mathcal{O}(Z^1(W_{\Q_p}, \GL_n))^{\GL_n}$. Now we want to give an explicit description of the restriction of the map $\Psi_{\GL_n}$ to the factor $ \mathcal{O}([C_{\phi}]) \simeq \overline{\Q}_{\ell}[X_1, \dotsc, X_r, X_1^{-1}, \dotsc, X_r^{-1}] $. As we will see, this description ultimately comes from the fact that the Fargues-Scholze $L$-parameter and the usual $L$-parameter are compatible up to semi-simplification for irreducible representations of $\GL_n(\Q_p)$.

Let $ \rB \subset \GL_n $ be the standard upper triangular Borel subgroup and let $\psi$ be a generic character of the unipotent radical $ \U \subset \rB$. This Whittaker datum yields the Whittaker representation $ \cInd_{\U(\Q_p)}^{\GL_n(\Q_p)} \psi $. The so-called Whittaker sheaf is the sheaf $\mathcal{W}_{\psi}$ concentrated on $\Bun^1_n$ corresponding to $\cInd_{\U(\Q_p)}^{\GL_n(\Q_p)} \psi $.

Let $ \M := \GL_{n_1} \times \dotsc \times \GL_{n_r} $ where $n_i = \dim \phi_i \ (1 \le i \le r)$ be the Levi subgroup of $\GL_n$ corresponding to $\phi$. Let $\pi^{\M}$ be the supercuspidal representation of $\M$ whose $L$-parameter is given by $\phi_1 \times \dotsc \times \phi_r$. Thus $\mathfrak{s}_{\phi} := (\M, \pi^{\M})$ is a cuspidal pair of $\GL_n$ and we denote by $ \Rep( \mathfrak{s}_{\phi}) $ the corresponding Bernstein block of the category $\Rep_{\overline{\Q}_{\ell}} \GL_n(\Q_p)$ of smooth representations of $\GL_n(\Q_p)$. By \cite{Bern}, $\Rep_{\overline{\Q}_{\ell}} \GL_n(\Q_p)$ decomposes into a product of indecomposable Bernstein blocks. Let $\mathcal{W}_{\mathfrak{s}_{\phi}}$ be the Bernstein component of $ \cInd_{\U(\Q_p)}^{\GL_n(\Q_p)} \psi $ in $\mathfrak{s}_{\phi}$ and let $\mathcal{Z}_{\mathfrak{s}_{\phi}}$ be the center of the block $ \Rep( \mathfrak{s}_{\phi}) $. Let $\M^{\rm un}$ be the variety  of unramified characters of $\M$, thus $\M^{\rm un} \simeq \bb G_m^r $. Let $\M^{\rm un}_{\phi}$ be the quotient of $\M^{\rm un}$ by the (finite) subgroup of unramified characters fixing $\phi_1 \times \dotsc \times \phi_r$. Thus $\M^{\rm un}_{\phi}$ is also isomorphic to $\bb G_m^r$. It is known by \cite{Bern}, \cite{BH} that $\mathcal{Z}_{\mathfrak{s}_{\phi}}$ is isomorphic to the global functions of the variety $ \M^{\rm un}_{\phi}$. 
Thus
\[
\mathcal{Z}_{\mathfrak{s}_{\phi}} \simeq \End(\mathcal{W}_{\mathfrak{s}_{\phi}}) \simeq \overline{\Q}_{\ell}[Y_1, \dotsc, Y_r, Y_1^{-1}, \dotsc, Y_r^{-1}].
\]

It is known that $\mathcal{W}_{\mathfrak{s}_{\phi}}$ is a projective object of the block $ \Rep( \mathfrak{s}_{\phi}) $ \cite[Corollary 8.6]{CS19} and we are going to prove that $\mathcal{W}_{\mathfrak{s}_{\phi}}$ is in fact a pro-generator of this block.

Indeed, if $\phi$ is a supercuspidal $L$-parameter then $\pi := \pi^{\GL_n}$ is supercuspidal. Moreover by \cite[\S 9.2]{BH}, we know that $\Ws = \cInd^{\GL_n}_{\GL_n^o} \pi' $ where $\GL_n^o := \{ g \in \GL_n | \det g \in \Z_p^* \} $ is an open, dense, normal subgroup of $\GL_n$ and $\pi'$ is the unique direct factor of the restriction $ \pi_{| \GL_n^o} $ such that $\Hom_{\GL_n^o}(\cInd_{\U}^{\GL_n^o} \psi, \pi')$ is non trivial. In particular, $\Ws$ is a pro-generator of the block.

In general, $\mathfrak{s}_{\phi} = [\M, \pi^{\M}]$ is a cuspidal pair for $\GL_n$ and we also consider $\mathfrak{t}_{\phi} = [\M, \pi^{\M}]$ as a cuspidal pair of $\M$. Hence, by \cite[\S 4]{Soll} the representation $\Ind_{\rP}^{\GL_n}(\mathcal{W}_{{\mathfrak{t}_{\phi}}})$ is a pro-generator of $\Rep(\mathfrak{s}_{\phi})$ where $\rP$ is the standard parabolic subgroup of $\GL_n$ whose Levi factor is $\M$. By \cite[\S 9.3]{BH} and \cite[Proposition 4.1]{Soll} we deduce that $ \Ws \simeq \Ind_{\rP}^{\GL_n}(\mathcal{W}_{{\mathfrak{t}_{\phi}}}) $ is a pro-generator of $\Rep(\mathfrak{s}_{\phi})$ (note that \cite[Proposition 2 \S 10.1]{BH} implies that the algebras $\End_{\G}(\I_{\rP}^{\G}(\rm E_{\rm B}))$ and $\End_{\rm L}(\I^{\rm L}_{\rP \cap \rm L}(\rm E_{\rm B}))$ in \cite[Proposition 4.1]{Soll} are isomorphic for the Bernstein blocks we consider).

Moreover, the block $ \Rep( \mathfrak{s}_{\phi}) $ is equivalent to the category of modules over $\mathcal{Z}_{\mathfrak{s}_{\phi}}$ via the functor
\begin{align*}
   \rF : \Rep( \mathfrak{s}_{\phi}) &\longrightarrow \mathcal{Z}_{\mathfrak{s}_{\phi}}-\text{Mod} \\
    \pi &\longmapsto \Hom_{\Rep( \mathfrak{s}_{\phi})}(\mathcal{W}_{\mathfrak{s}_{\phi}}, \pi).
\end{align*}

Let $\pi_{\phi}$ be the irreducible representation of $\GL_n(\overline{\Q}_{\ell})$ whose $L$-parameter is $\phi$. Without loss of generality, by replacing $\phi$ by an appropriate $L$-parameter in $[C_{\phi}]$, we can suppose that $ \rF (\pi_{\phi}) $ is isomorphic to $\overline{\Q}_{\ell}$ as $ \mathcal{Z}_{\mathfrak{s}_{\phi}} $-module  where $\mathcal{Z}_{\mathfrak{s}_{\phi}}$ acts via the character
\[
\overline{\Q}_{\ell}[Y_1, \dotsc, Y_r, Y_1^{-1}, \dotsc, Y_r^{-1}] \longrightarrow \overline{\Q}_{\ell}, \quad Y_i \longmapsto 1.
\]

More generally, let $\xi = (\xi_1, \dotsc, \xi_r)$ be an unramified character of $\M$, we denote by $\phi \otimes \xi$ the $L$-parameter $ \displaystyle \bigoplus_{i=1}^r \phi_i \otimes \xi_i $ and by $ \pi_{\phi \otimes \xi} $ the corresponding irreducible representation of $\GL_n(\Q_p)$. Then $\rF( \pi_{\phi \otimes \xi} )$ is isomorphic to $\overline{\Q}_{\ell}$ as $ \mathcal{Z}_{\mathfrak{s}_{\phi}} $-module  where $\mathcal{Z}_{\mathfrak{s}_{\phi}}$ acts via the reduction
\[
\overline{\Q}_{\ell}[Y_1, \dotsc, Y_r, Y_1^{-1}, \dotsc, Y_r^{-1}] \longrightarrow  \overline{\Q}_{\ell}[Y_1, \dotsc, Y_r, Y_1^{-1}, \dotsc, Y_r^{-1}] / \mathfrak{m}_{\xi} \simeq \overline{\Q}_{\ell}
\]
where $\mathfrak{m}_{\xi}$ is the maximal ideal corresponding to the closed point $ \xi $ in $\M^{\rm un}_{\phi}$. More precisely, $\xi$ corresponds to the points $(\xi_1^{k_1}, \dotsc, \xi_r^{k_r})$ in $ \M^{\rm un}_{\phi} $ where $k_i$ is the torsion number of  $\phi_i$ for $ 1 \leq i \leq r $.  \\

Recall that we associate a pair $(b_{\chi}, \pi_{\chi})$ to each character $\chi$ of $S_{\phi} \simeq \bb G_m^r$ where $b_{\chi} \in B(\GL_n)$ and $\pi_{\chi}$ is an irreducible representation of $\G_{b_{\chi}}(\Q_p)$. Let $\mf{s}_{\phi}(\chi) = [\tau_{\chi}, \M_{b_{\chi}}] $, resp. $\mf{t}_{\phi}(\chi) = [\tau_{\chi}, \M_{b_{\chi}}] $  be the cuspidal pair of $\G_{b_{\chi}}(\Q_p)$, resp. of $\M_{b_{\chi}}(\Q_p)$ corresponding to $\pi_{\chi}$, resp. $\tau_{\chi}$ where $(\tau_{\chi}, \M_{b_{\chi}})$ is the cuspidal support of $\pi_{\chi}$. Let $\Rep(\mf{s}_{\phi}(\chi))$, resp. $\Rep(\mf{t}_{\phi}(\chi))$ be the corresponding Bernstein block of $\G_{b_{\chi}}(\Q_p)$, resp. $\M_{b_{\chi}}(\Q_p)$. As above we can describe a pro-generator of this category. Indeed, $\mathcal{W}_{\mf{t}_{\phi}(\chi)} := \cInd^{\M_{b_{\chi}}}_{\M_{b_{\chi}}^o}(\tau_{\chi})'$ is a pro-generator of $\Rep(\mf{t}_{\phi}(\chi))$ where $\tau(\chi)'$ is a direct factor of the restriction $ \tau_{ \chi | \M_{b_{\chi}}^o} $. Similarly $\Wschi := \Ind_{\rP_{b_{\chi}}}^{\G_{b_{\chi}}} (\mathcal{W}_{\mf{t}_{\phi}(\chi)}) $ is a pro-generator of $\Rep(\mf{s}_{\phi}(\chi))$. It is also known that $\mathcal{Z}_{\mf{s}_{\phi}(\chi)} = \End(\Wschi) \simeq \overline{\Q}_{\ell}[X_1, \dotsc, X_r, X_r^{-1}, \dotsc, X_r^{-1}]$ and the block $ \Rep( \mathfrak{s}_{\phi}(\chi)) $ is equivalent to the category of modules over $\mathcal{Z}_{\mathfrak{s}_{\phi}(\chi)}$ via the functor
\begin{align*}
   \rF : \Rep( \mathfrak{s}_{\phi}(\chi)) &\longrightarrow \mathcal{Z}_{\mathfrak{s}_{\phi}(\chi)}-\text{Mod} \\
    \pi &\longmapsto \Hom_{\Rep( \mathfrak{s}_{\phi}(\chi))}(\mathcal{W}_{\mathfrak{s}_{\phi}(\chi)}, \pi).
\end{align*}

We also denote by $\mathcal{F}_{\Wchi}$ the sheaf $ i_{b_{\chi} !}(\delta^{-1/2}_{b_{\chi}} \otimes \Wschi)[-d_{\chi}] $ supported on $b_{\chi}$ where $d_{\chi} = \langle 2\rho, \nu_{b_{\chi}} \rangle$ and where $i_{b_{\chi}} : \Bun_n^{b_{\chi}} \longrightarrow \Bun_n$ is the canonical immersion. By abuse of notations, we sometimes use $\Ws$ instead of $\mathcal{F}_{\mathcal{W}(\Id)}$ to denote $ i_{1 !}(\Ws) $. \\
  
We denote by $R$ the ring $\OO([C_{\phi}]) \simeq \overline{\Q}_{\ell}[X_1, \dotsc, X_r, X_1^{-1}, \dotsc, X_r^{-1}] $ and by $R_{\chi}$ the ring $\mathcal{Z}_{\mathfrak{s}_{\phi}(\chi)} \simeq \overline{\Q}_{\ell}[Y_1, \dotsc, Y_r, Y_1^{-1}, \dotsc, Y_r^{-1}] $ to lighten the notations. We have a natural identification between $R$ and $R_{\chi}$ by the map that sends $X_i$ to $Y_i$ for $1 \le i \le r$. If $\chi = \Id$, we even denote $R_{\chi}$ by $R'$. Note that for each Bernstein block $\mathfrak{s}$ of $\Rep_{\overline{\Q}_{\ell}}(\GL_n(\overline{\Q}_p))$, we have a canonical projection map $ \Pr_{\mathfrak{s}} : \mathcal{Z}(\GL_n, \overline{\Q}_{\ell}) \longrightarrow \mathcal{Z}_{\mathfrak{s}} $.

The following lemma is well known for the experts.

\begin{lemma} \phantomsection \label{itm : explicit Bernstein morphism}
    If $\mathfrak{s} \neq \mathfrak{s}_{\phi}$ then the map $ \Pr_{\mathfrak{s}} \circ \Psi_{\GL_n | \mathcal{O}([C_{\phi}])} : \mathcal{O}([C_{\phi}]) \longrightarrow \mathcal{Z}(\GL_n, \overline{\Q}_{\ell}) \longrightarrow \mathcal{Z}_{\mathfrak{s}} $ is the zero map and if $\mathfrak{s} = \mathfrak{s}_{\phi}$ then we have the following description
    \begin{align*}
       \mathrm{Pr}_{\mathfrak{s}_{\phi}} \circ \Psi_{\GL_n | \mathcal{O}([C_{\phi}])} : \mathcal{O}([C_{\phi}]) & \longrightarrow \mathcal{Z}(\GL_n, \overline{\Q}_{\ell}) \longrightarrow \mathcal{Z}_{\mathfrak{s}_{\phi}} \\
        X_i &\longmapsto Y_i.
    \end{align*}
\end{lemma}
\begin{proof}
    Let $\pi$ be any irreducible representation of $\GL_n(\Q_p)$ that does not belong to the Bernstein block $\Rep( \mathfrak{s}_{\phi})$. By the compatibility of the Fargues-Scholze L-parameter and the usual L-parameter, we see that the $L$-parameter $\varphi$ of $\pi$ does not belong to the connected component $[C_{\phi}]$. Thus $\Psi_{\GL_n | \mathcal{O}([C_{\phi}])}(X_i) \in \mathcal{Z}(\GL_n, \overline{\Q}_{\ell}) $ acts trivially on $\pi$ since the valuation of $X_i$ at the closed point corresponding to $\varphi^{\rm ss}$ vanishes. Hence $\Pr_{\mathfrak{s}} \circ \Psi_{\GL_n | \mathcal{O}([C_{\phi}])}$ is the zero map. 

    Now let $\xi = (\xi_1, \dotsc, \xi_r)$ be an unramified character of $\M^{\rm un}$ and $\pi$ be an irreducible representation whose $L$-parameter is given by $\phi \otimes \xi$. We denote by $\mathfrak{m}_{\xi}$ (resp. $\mathfrak{m}'_{\xi}$) the maximal ideal of $R$ (resp. $R'$) corresponding to the point $\xi$. Then by the above description, if we let $\alpha$ be the image of $X_i$ in the quotient $R / \mathfrak{m}_{\xi}$ then the excursion operator corresponding to $X_i$ acts on $\pi$ by $\alpha \in \ov \Q_{\ell} \simeq \End(\pi)$. That means the element $ Y'_i := \mathrm{Pr}_{\mathfrak{s}_{\phi}} \circ \Psi_{\GL_n | \mathcal{O}([C_{\phi}])}(X_i)$ also acts on $\pi$ by $\alpha \in \overline{\Q}_{\ell} \simeq \End(\pi) $. Since the usual $L$-parameter of $\pi$ is the same as its Fargues-Scholze $L$-parameter, we deduce that the image of $Y'_i$ in $R' / \mathfrak{m}'_{\xi}$ is also given by $\alpha$. In other words, $Y'_i - Y_i $ is in the maximal ideal $ \mathfrak{m}'_{\xi} $. This holds for an arbitrary point $\xi$, hence $ Y'_i = Y_i $ since the Jacobson radical of $\mathcal{O}([C_{\phi}])$ is trivial. To see this, we just note that the units in $\mathcal{O}([C_{\phi}])$ are the monomials and for a ring $A$, un element $u$ belongs to its Jacobson radical if and only if $ 1 + ua $ is a unit for all $a \in A$.     
\end{proof}

In more general situation where we consider a character $\chi \in \Irr(S_{\phi})$, we can embed the derived category of $\Rep(\mathfrak{s}_{\phi}(\chi))$ into $\Dlis(\Bun_n, \overline{\Q}_{\ell})$ by the functor
\begin{align*}
  i_{\chi} : \mathrm{D}(\Rep(\mathfrak{s}_{\phi}(\chi))) &\longrightarrow \quad \Dlis(\Bun_n, \overline{\Q}_{\ell}) \\
  \pi \ \quad \quad \quad &\longmapsto \quad i_{b_{\chi}!}( \delta^{-1/2}_{b_{\chi}} \otimes \pi)[-d_{\chi}]
\end{align*}
where $d_{\chi} = \langle 2\rho, \nu_{b_{\chi}} \rangle$. Therefore we have an induced map
\[
\Psi^{\chi}_{\GL_n} : \mc Z^{\rm spec}(\GL_n, \overline{\Q}_{\ell}) \longrightarrow \mc Z_{\mathfrak{s}_{\phi}(\chi)}.
\]

We know that Fargues-Scholze $L$-parameter is also compatible with the usual one for inner forms of $\GL_n$ \cite[Theorem 6.6.1]{HKW}. Moreover, the Fargues-Scholze $L$-parameter associated to the sheaf $i_{b_{\chi}!}( \delta^{-1/2}_{b_{\chi}} \otimes \pi)[-d_{\chi}]$ is the same as the usual $L$-parameter of $\pi$ post-composed with the natural embedding $ \widehat{\G^*_{b_{\chi}}}(\overline{\Q}_{\ell}) \hookrightarrow \widehat{\GL}_n(\overline{\Q}_{\ell}) $. Then by the same arguments as in lemma \ref{itm : explicit Bernstein morphism}, we can show the following result.
\begin{lemma} \phantomsection \label{itm : general explicit Bernstein morphism}
The restriction of the map $\Psi^{\chi}_{\GL_n}$ to $\mathcal{O}([C_{\phi}])$ is given explicitly by:
    \begin{align*}
        \Psi^{\chi}_{\GL_n | \mathcal{O}([C_{\phi}])} : \mathcal{O}([C_{\phi}]) & \longrightarrow \mathcal{Z}_{\mathfrak{s}_{\phi}(\chi)} \\
        X_i &\longmapsto Y_i.
    \end{align*}
\end{lemma}
\subsection{On spectral action on $\Bun_n$} \textbf{}

The main goal of this paragraph is to study the full sub-category $ \mathcal{C}:= \Dlis^{[C_{\phi}]}(\Bunn,\ol{\mathbb{Q}}_{\ell})^{\omega} \subset \Dlis(\Bunn,\ol{\mathbb{Q}}_{\ell})^{\omega}$ generated by compact objects whose Schur-irreducible constituents have $L$-parameter in the connected component $[C_{\phi}]$ and also the action of $\Perf([C_{\phi}])$ on $\mathcal{C}$. First, we are going to show the identity $C_{\chi} \star \Ws \simeq i_{b_{\chi} !} \big( \delta^{-1/2}_{b_{\chi}} \otimes \Wschi \big) [-d_{\chi}] $ for $\chi \neq \Id$ (the case $\chi = \Id$ is trivial since $C_{\Id}$ is the identity functor of $\mathcal{C}$). 

\begin{theorem} \phantomsection \label{itm : spectral action - basic}
    Let $\phi = \phi_1 \oplus \dotsc \oplus \phi_r$ be an $L$-parameter satisfying the conditions of theorem \ref{itm : main theorem}. With the above notations, for each $\chi \in \Irr(S_{\phi})$ we have
    \[
    C_{\chi} \star \Ws \simeq \mathcal{F}_{\Wchi}. 
    \]
\end{theorem}
\begin{proof}
We will show that the spectral action by $C_{\chi}$ gives an equivalence between the block $\Rep(\mathfrak{s}_{\phi})$ and the block $\Rep(\mathfrak{s}_{\phi}(\chi))$ of the categories $\Rep_{\ol\Q_{\ell}}(\GL_n(\Q_p))$, respectively $\Rep_{\ol\Q_{\ell}}(\G_{b_{\chi}}(\Q_p))$. 

Indeed, let $\mathcal{F}$ be the sheaf supported on $\Bun^{1}_n$ corresponding to a finitely generated representation $\pi$ in the block $\Rep(\mathfrak{s}_{\phi})$ and consider $C_{\chi} \star \mathcal{F}$. We show that the support of this sheaf is the stratum $b_{\chi}$. We know that the sheaf $C_{\chi} \star \mathcal{F}$ is compact, thus its restriction $i_b^{*} C_{\chi} \star \mathcal{F}$ to a stratum $b$ different from $b_{\chi}$ is compact. Thus, there exists an integer $k$ and an irreducible representation $\rho$ of $\G_{b}(\Q_p)$ such that the Fargues-Scholze parameter $\phi_{\rho}$ of $i_{b !}^{\mathrm{ren}} (\rho)$ belongs to the connected component of $\phi$ and
\[
\Hom_{\Dlis(\Bunn, \ol{\Q}_{\ell})}( i_{b!}i_b^{*} C_{\chi} \star \mathcal{F} , i_{b !}^{\mathrm{ren}} (\rho)[k] ) \neq 0.
\]

The parameter $\phi_{\rho}$ satisfies the conditions in the beginning of section \ref{itm : condition A1}. We can then write $i_{b !}^{\mathrm{ren}} (\rho)[k]$ in the form $\mathcal{F}_{\chi'}[h]$ corresponding to the parameter $\phi_{\rho}$ where $\chi'$ is a character such that $b_{\chi'} = b$. Then for $b' \neq b$ in $B(\GL_n)$, by corollary \ref{itm : useful corollary} we see that 
\begin{align*}
\Hom_{\Dlis(\Bunn, \ol{\Q}_{\ell})}( i_{b'!}i_{b'}^{*} C_{\chi} \star \mathcal{F} , i_{b !}^{\mathrm{ren}} (\rho)[k] ) &= \Hom_{\Dlis(\Bunn, \ol{\Q}_{\ell})}( i_{b'!}i_{b'}^{*} C_{\chi} \star \mathcal{F} , \mathcal{F}_{\chi'}[h] ) \\
&= \Hom_{D(\G_{b'}(\Q_p), \ov \Q_{\ell})}( i_{b'}^{*} C_{\chi} \star \mathcal{F} , i_{b'}^{!}\mathcal{F}_{\chi'}[h] ) \\
&= 0.   
\end{align*}

Therefore,
\[
\Hom_{\Dlis(\Bunn, \ol{\Q}_{\ell})}(C_{\chi} \star \mathcal{F} , i_{b !}^{\mathrm{ren}} (\rho)[k] ) = \Hom_{\Dlis(\Bunn, \ol{\Q}_{\ell})}( i_{b!}i_b^{*} C_{\chi} \star \mathcal{F} , i_{b !}^{\mathrm{ren}} (\rho)[k] ) \neq 0.
\]
By applying the spectral action $C_{\chi^{-1}}$, we have 
\[
\Hom_{\Dlis(\Bunn, \ol{\Q}_{\ell})}(\mathcal{F} , C_{\chi^{-1}} i_{b !}^{\mathrm{ren}} (\rho)[k] ) \neq 0,
\] 
however, theorem \ref{itm : main theorem} implies that the restriction of the sheaf $C_{\chi^{-1}} i_{b !}^{\mathrm{ren}} (\rho)[k]$ to the identity stratum vanishes. It yields a contradiction since $\mathcal{F}$ is supported on the identity stratum. Therefore $C_{\chi} \star \mathcal{F}$ is supported on the stratum corresponding to $b_{\chi}$ and moreover, it belongs to the direct sum of blocks 
\[
\bigoplus_{b_{\xi} = b_{\chi}} i^{\mathrm{ren}}_{b_{\xi}!} D(\Rep(\mathfrak{s}_{\phi}(\xi))).
\]

However, by the same argument, one can show that if $\pi'$ is finitely generated smooth representation in $D(\Rep(\mathfrak{s}_{\phi}(\xi)))$ for then $C_{\chi^{-1}} \star i_{b_{\xi}!}(\pi')$ is supported on the identity stratum of $\Bunn$ only if $\chi = \xi$. Thus 
$C_{\chi} \star \mathcal{F}$ belongs to $i^{\mathrm{ren}}_{b_{\chi}!} D(\Rep(\mathfrak{s}_{\phi}(\chi)))$ and inversely if $\pi'$ is finitely generated smooth representation in $D(\Rep(\mathfrak{s}_{\phi}(\chi)))$ then $C_{\chi^{-1}} \star i_{b_{\chi}!}(\pi') $ belongs to the category $D(\Rep(\mathfrak{s}_{\phi})$.

Now the sheaf $\mathcal{W}_{\mathfrak{s}_{\phi}}$ is a projective generator of $\Rep(\mathfrak{s}_{\phi})$, we see that $C_{\chi} \star \mathcal{W}_{\mathfrak{s}_{\phi}} $ is a projective generator of $\Rep(\mathfrak{s}_{\phi}(\chi))$. Note that $\mathcal{F}_{\Wchi}$ is also a projective generator of $i^{\mathrm{ren}}_{b_{\chi}!}\Rep(\mathfrak{s}_{\phi}(\chi))$. For every Schur irreducible sheaf $\mathcal{G}$ in $i^{\mathrm{ren}}_{b_{\Id}!} D(\Rep(\mathfrak{s}_{\phi}))$ we have 
\[
\Hom_{\Dlis(\Bunn, \ol{\Q}_{\ell})}(C_{\chi} \star \mathcal{W}_{\mathfrak{s}_{\phi}} , C_{\chi} \star \mathcal{G}) = \Hom_{\Dlis(\Bunn, \ol{\Q}_{\ell})}(\mathcal{W}_{\mathfrak{s}_{\phi}} , \mathcal{G} ) =\Hom_{\Dlis(\Bunn, \ol{\Q}_{\ell})}(\mathcal{F}_{\Wchi} , C_{\chi} \star \mathcal{G} ),
\]
where the second equality follows from theorem \ref{itm : main theorem} and the fact that parabolic induction functor is exact. It implies that 
\[
    C_{\chi} \star \Ws \simeq \mathcal{F}_{\Wchi}. 
\]
\end{proof}

Let $\Id$ be the trivial character of $S_{\phi}$ and let $C_{\rm tri}$ be the structural sheaf of the stack of $L$-parameter. We know by construction that $C_{\Id}$ is the restriction of $C_{\rm tri}$ on the connected component $[C_{\phi}]$. By \cite[Lemma 3.8]{Ham} we have the identity $ C_{\Id} \star \mathcal{W} = C_{\Id} \star \Ws = \Ws $. Remark that we can identify $R$ and $R'$ by the map $\theta$ that sends $X_i$ to $Y_i$. The next goal is to describe $ \bL \star \Ws $ where $\bL$ is a perfect complex on $[C_{\phi}]$.

\begin{theorem} \phantomsection \label{itm : spectral action - general}
    Let $\bL$ be a perfect complex of $R$-module. For each $\chi \in \Irr(S_{\phi})$, let $\bL(\chi)$ be the perfect complex $\bL$ together with the action of $\bb G_m^r$ acting by $\chi$ so that it gives rise to a perfect complex on $[C_{\phi}]$. We denote by $\pi_{\bL}(\chi)$ the complex of smooth $\G_{b_{\chi}}(\Q_p)$-representations in the derived category of the Bernstein block $\Rep(\mathfrak{s}_{\phi}(\chi))$ corresponding to the $R_{\chi}$-perfect complex $\bL$ (by the natural identification between $R$ and $R_{\chi}$). Then
\[
\bL(\chi) \star \Ws \simeq i_{b_{\chi} !} ( \delta^{-1/2}_{b_{\chi}} \otimes \pi_{\bL}(\chi))[-d_{\chi}].
\]
\end{theorem}
\begin{proof}
In general, one uses cones and retracts to construct the spectral action of a perfect complex from the Hecke operators. In the case we consider the perfect complex $\bL(\chi)$ is generated by cones and retracts from the vector bundle $C_{\chi}$. Hence one can compute $\bL(\chi) \star \Ws$ by using the identity $C_{\chi} \star \Ws = i_{b_{\chi} !} ( \delta^{-1/2}_{b_{\chi}} \otimes \Wschi)[-d_{\chi}]$ and by tracing back the process of constructing $\bL(\chi)$ from $C_{\chi}$.

For simplicity we suppose $\chi = \Id$, the general case is similar. In the following, every perfect complex of $R$-modules is equipped with the trivial $\bb G_m^r$-action so that it gives rise to a perfect complex on $[C_{\phi}]$. In order to lighten the notation, we simply write $\bL$ for $\bL(\Id)$. Since $\bL$ is a perfect complex of $R$-modules, it has a finite resolution by finitely generated projective $R$-modules
\[
0 \longrightarrow P_m \xrightarrow{ \ f_m \ } P_{m-1} \xrightarrow{ \ f_{m-1} \ } \dotsc \xrightarrow{ \ f_1 \ } P_0 \xrightarrow{ \ f_0 \ } \bL.
\]

Now, we trace back the construction of the spectral action of $\bL$. We consider the cone $ \textrm{Cone}(f_{m})$ of $f_m$. There is an exact triangle
\[
P_{m} \xrightarrow{ \ f_m \ } P_{m-1} \xrightarrow{ \ g_m \ } \textrm{Cone}(f_{m})
\]
and there is an isomorphism $\textrm{Cone}(f_m) \simeq P_{m-1}/f_m(P_m)$ in the derived category of the category of $R$-modules. By Quillen-Suslin's theorem, the $R$-modules $P_i$'s are free. Thus $P_i \simeq R^{t_i}$ for $ 1 \le i \le m $ and $f_m \in \Hom(R^{t_m}, R^{t_{m-1}}) \simeq \displaystyle \End(R, R)^{t_mt_{m-1}} $. By applying the above exact triangle on $\Ws$, we get an exact triangle
\[
\Ws^{t_m} \xrightarrow{ \ \widetilde{f_m} \ } \Ws^{t_{m-1}} \xrightarrow{ \ \widetilde{g_m} \ } \textrm{Cone}(f_{m}) \star \Ws,
\]
and $ \textrm{Cone}(f_{m}) \star \Ws $ is the cone of $\widetilde{f_m}$. By the equivalence of $\Rep(\mathfrak{s}_{\phi})$ and the category of $R'$-modules via the functor $F$ explained above, we get an exact triangle in the derived category of $\Rep(\mathfrak{s}_{\phi})$
\[
(R')^{t_m} \xrightarrow{ \ F(\widetilde{f_m}) \ } (R')^{t_{m-1}} \xrightarrow{ \ F(\widetilde{g_m}) \ } F(\textrm{Cone}(f_{m}) \star \Ws),
\]
and therefore $F(\textrm{Cone}(f_{m}) \star \Ws)$ is the cone of $F(\widetilde{f_m})$. However, we can see that an element in $\End(R, R)$ corresponds to the multiplying by a function in $R$. We can identify $R$ and $R'$ via the natural map $\theta$ that sends $X_i$ to $Y_i$ for $1 \le i \le r$ and then we can extend $\theta$ to identify the category of $R$-modules with that of $R'$-modules. Thus by lemma \ref{itm : explicit Bernstein morphism}, $F(\widetilde{f_m})$ is given by the image of $\theta(f_m)$. Therefore, by uniqueness of cones, $F(\textrm{Cone}(f_{m}) \star \Ws) \simeq \theta(\mathrm{Cone}(f_m)) $. 

The morphism $ P_{m-1} \xrightarrow{ \ f_{m-1} \ } P_{m-2} $ induces a map $g$ from $\mathrm{Cone}(f_m)$ to $P_{m-2}$. Similarly, the map $\theta(f_{m-1}) : \theta(P_{m-1}) \longrightarrow \theta(P_{m-2}) $ induces a map $ \theta(g) $ from $ \theta(\mathrm{Cone}(f_m))$ to $\theta(P_{m-2})$ and we can show that $ F( \widetilde{g} ) = \theta(g) $ (up to some quasi-isomorphism from $\theta(\mathrm{Cone}(f_m))$ to $\theta(\mathrm{Cone}(f_m))$) where $\widetilde{g}$ is the morphism from $\mathrm{Cone}(f_m) \star \Ws $ to $ P_{m-2} \star \Ws $ induced by the spectral action.

Now we can apply the above process to the chain 
\[
\mathrm{Cone}(f_m) \xrightarrow{ \ g \ } P_{m-2} \xrightarrow{ \ f_{m-2} \ } \dotsc \xrightarrow{ \ f_1 \ } P_0 \xrightarrow{ \ f_0 \ } \bL.
\]

Therefore by induction argument we get the identity 
\[
\bL \star \Ws \simeq i_{1!} \pi_{\bL}.
\]
\end{proof}

Recall that $ S_{\phi} \simeq \displaystyle \prod_{i=1}^r \bb G_m $ and $\Irr(S_{\phi}) \simeq \displaystyle \prod_{i=1}^r \Z $. As before, for each $1 \le j \le r$ we denote by $\chi_j$ the element of the form $(0, \dotsc, 0, 1, 0, \dotsc, 0) $ where $1$ is in the $j^{th}$-position. Denote by $\mathfrak{m}_{\phi}$ the maximal ideal in $R$ corresponding to $\phi$. We can define an $R$-module $\bL$ whose underlying $\overline{\Q}_{\ell}$-vector space is $\overline{\Q}_{\ell}$ and where $R$ acts via the map $R \longrightarrow R/\mathfrak{m}_{\phi} \simeq \overline{\Q}_{\ell} $. For each character $\chi \in \Irr(S_{\phi})$, we can define a perfect sheaf $A(\chi)$ on $[C_{\phi}]$ by letting $\bb G^r_m$ acts on the fiber of $\bL$ via the character $\chi$. It is a skyscraper sheaf on $[C_{\phi}]$ supported on the closed point defined by $\phi$. Note that in general the derived tensor product $A(\chi) \otimes^{\bL} A(\chi')$ is not isomorphic to $ A(\chi \otimes \chi') $. Recall that for each $\phi$ and $\chi \in \Irr(S_{\phi})$, we defined a sheaf $\mathcal{F}_{\chi}$ as in theorem \ref{itm : main theorem}. We have the following corollary of theorem \ref{itm : spectral action - general}.

\begin{corollary}
 For each $\chi \in \Irr(S_{\phi})$ we have $ A(\chi) \star \Ws \simeq \mathcal{F}_{\chi}.$  
\end{corollary}

We now want to study the full sub-category $ \mathcal{C}:= \Dlis^{[C_{\phi}]}(\Bunn,\ol{\mathbb{Q}}_{\ell})^{\omega} \subset \Dlis(\Bunn,\ol{\mathbb{Q}}_{\ell})^{\omega}$ consisting of compact objects whose Schur-irreducible constituents have $L$-parameter in the connected component $[C_{\phi}]$. For each $\chi$, we regard the derived category $\rm D(\Rep(\mathfrak{s}_{\phi}(\chi)))^{\omega}$ as a full sub-category of $\Dlis^{[C_{\phi}]}(\Bunn,\ol{\mathbb{Q}}_{\ell})^{\omega}$ via the functor $i_{\chi}$ as before.   
\begin{align*}
  i_{\chi} : \mathrm{D}(\Rep(\mathfrak{s}_{\phi}(\chi)))^{\omega} &\longrightarrow \quad \Dlis(\Bun_n, \overline{\Q}_{\ell})^{\omega} \\
  \pi \ \quad \quad \quad &\longmapsto \quad i_{b_{\chi}!}( \delta^{-1/2}_{b_{\chi}} \otimes \pi)[-d_{\chi}]
\end{align*}
\begin{theorem} \phantomsection \label{itm : orthogonal decomposition}
    We have an orthogonal decomposition
    \[
    \Dlis^{[C_{\phi}]}(\Bun_n, \overline{\Q}_{\ell})^{\omega} \simeq \bigoplus_{\chi \in \Irr(S_{\chi})}\mathrm{D}(\Rep(\mathfrak{s}_{\phi}(\chi)))^{\omega}
    \]
\end{theorem}
\begin{proof}

Let $\mc H$ be a complex in $\Dlis^{[C_{\phi}]}(\Bun_n, \overline{\Q}_{\ell})^{\omega}$ then by proposition \ref{itm : shape of strata}, the support of $\mc H$ is a finite union of strata of the form $b_{\chi}$ where $\chi \in \Irr(S_{\phi})$. Consider a restriction $i^*_{b_{\chi}} \mc H $ of $\mc H$ to such a stratum. Then all the irreducible constituents of the cohomology groups of $i^*_{b_{\chi}} \mc H$ have Fargues-Scholze $L$-parameters in the connected component $[C_{\phi}]$. Therefore, as smooth representations of $\G_{b_{\chi}}(\Q_p)$, these irreducible constituents belong to the direct sum
\[
\bigoplus_{b_{\chi'} = b_{\chi}}  \Rep(\mathfrak{s}_{\phi}(\chi'))^{\omega}
\]
of Bernstein blocks of $\Rep_{\ov \Q_{\ell}} (\G_{b_{\chi}}(\Q_p))$. Hence $\Dlis^{[C_{\phi}]}(\Bun_n, \overline{\Q}_{\ell})^{\omega}$ is generated by the full sub-categories $ \mathrm{D}(\Rep(\mathfrak{s}_{\phi}(\chi)))^{\omega} $'s where $\chi \in \Irr(S_{\phi})$.

To prove the direct decomposition now, it suffices to prove that if $\mathcal{F} \in \rm D(\Rep(\mathfrak{s}_{\phi}(\chi)))^{\omega} $ and $\mathcal{G} \in \rm D(\Rep(\mathfrak{s}_{\phi}(\chi')))^{\omega} $ for $\chi \neq \chi'$ then $\Hom_{\mc C}(\mathcal{F}, \mathcal{G}) = 0$. 

If $b_{\chi} = b_{\chi'} = b$ then it is clear that $\Hom_{\mc C}(\mathcal{F}, \mathcal{G}) = 0$ since $\Rep(\mathfrak{s}_{\phi}(\chi))$ and $\Rep(\mathfrak{s}_{\phi}(\chi'))$ are two distinct Bernstein blocks of $\Rep_{\overline{\Q}_{\ell}}(\G_{b}(\Q_p))$.

We suppose that $ b_{\chi} \neq b_{\chi'} $. Let $\bL$ and $\bL'$ be the complexes in $\rm D(\mc Z_{\mathfrak{s}_{\phi}(\chi)} \mhyphen Mod)$ and in $\rm D(\mc Z_{\mathfrak{s}_{\phi}(\chi')} \mhyphen Mod)$ corresponding respectively to $\mathcal{F}$ and $\mathcal{G}$. Since they are compact, the complexes $\bL$ and $\bL'$ are perfect. Then by theorems \ref{itm : spectral action - basic} and \ref{itm : spectral action - general} we see that $ \bL(\chi) \star \Ws \simeq \mathcal{F} $ and $ \bL'(\chi') \star \Ws \simeq \mathcal{G} $. Since $C_{\chi^{-1}} \star (-)$ is an auto-equivalence of $ \Dlis^{[C_{\phi}]}(\Bun_n, \overline{\Q}_{\ell})^{\omega} $, we deduce that
\begin{align*}
    \Hom_{\mc C}(\mathcal{F}, \mathcal{G}) &=  \Hom_{\mc C}(\bL(\chi) \star \Ws, \bL'(\chi') \star \Ws ) \\
    &= \Hom_{\mc C}(\bL(\Id) \star \Ws, \bL'(\chi' \otimes \chi^{-1}) \star \Ws ).
\end{align*}

Again by theorem $\ref{itm : spectral action - general}$ the complex $\bL(\Id) \star \Ws$ is supported on $\Bun^1_n$ and $\bL'(\chi' \otimes \chi^{-1}) \star \Ws$ is supported on the complement of $\Bun^1_n$. Therefore lemma \ref{itm : morphism between strata} implies that 
\[
\Hom_{\mc C}(\bL(\Id) \star \Ws, \bL'(\chi' \otimes \chi^{-1}) \star \Ws ) = 0
\]
and it allows us to conclude. 
\end{proof}

\section{Local-Global compatibility} \label{itm : Local-Global compatibility} \textbf{}

In this section we prove some particular cases of Fargues' and Caraiani-Scholze's local-global conjecture \cite[\S 7]{Fa}, \cite[1.18]{CS}.

\subsection{Cohomology of Igusa varieties} \textbf{}

\subsubsection{Shimura varieties and Mantovan's formula} \textbf{}

We consider some simple Shimura varieties of type $A$ as in \cite{Kot92}. Assume that $F$ is an imaginary quadratic field, $p$ is split in $F$ and that we have an unramified integral PEL datum of the form $(\rm B, \mc O_{\rm B} , *, V, \Lambda_0, \langle, \rangle, h )$, where
\begin{enumerate}
    \item[$\bullet$] $\rm B$ is a division algebra with center $F$,
    \item[$\bullet$] $*$ is an involution of the second kind,
    \item[$\bullet$] $\rm B, \mc O_{\rm B}$ is a $\Z_{(p)}$ maximal order in $\rm B$ that is preserved by $*$ such that $ \rm B, \mc O_{\rm B} \otimes_{\Z} \Z_p $ is a maximal order in $\rm B_{\Q_p}$, 
    \item[$\bullet$] $V$ is a simple $ \rm B$-module,
    \item[$\bullet$] $\langle, \rangle : V \times V \longrightarrow \Q $ is a $*$-Hermitian pairing with respect to the $\rm B$-action,
    \item[$\bullet$] $\Lambda_0$ is a $\Z_p$-lattice in $V_{\Q_p}$ that is preserved by $\mc O_{\rm B}$ and self-dual for $\langle, \rangle$.
    \item[$\bullet$] $h : \C \longrightarrow \End_{\R}(V_{\R})$ is an $\R$-algebra homomorphism satisfying the equality $\langle h(z), w \rangle = \langle v , h(z^c)w \rangle, \forall v,w \in V_{\R} $ and $z \in \C$ and such that the bilinear pairing $(v, w) \longmapsto \langle v , h(\sqrt{-1})w \rangle $ is symmetric and positive definite.
\end{enumerate}

We can define a similitude unitary group $ \rm \textbf{G} $ over $\Q$ associated to this PEL datum and denote by $\rm \textbf{G}'$ the unitary group that is the kernel of the similitude factor, thus $\rm \textbf{G}'(\R)$ $\simeq \U (q, n - q)$ for some $q \in \N$. We assume further that the $p$-adic group $\rm \textbf{G}'_{\Q_p}$ is isomorphic to $\GL_{n, \Q_p}$ and hence $\mathrm{\textbf{G}}_{\Q_p} \simeq \GL_{n, \Q_p} \times \bb G_{m, \Q_p} $.

Let $\mu : \Gm \longrightarrow \rm \textbf{G} $ be the group homomorphism over $\C$ corresponding to $h$. Associated to the above PEL datum is a system of \textit{projective} Shimura varieties $\{\Sh_K\}$ defined over the reflex field $E$ and where $K$ runs over the set of sufficiently small open compact subgroups of $\rm \textbf{G}(\A^{\infty})$. Moreover, the PEL datum also gives rise to a system of integral models $\mc S_p = \{ \mc S_{K^p} \}$ defined over $ \mc O_{E, (p)} $ such that the generic fiber of each $\mc S_{K^p}$ is naturally identified with $\Sh_{K^pK_p^{\rm hs}}$ where $K^p$ runs over the set of sufficiently small open compact subgroups of $\rm \textbf{G}(\A^{\infty, p})$ and $K_p^{\rm hs} \subset \GL_n(\Q_p)$ is a hyperspecial subgroup.

On the other hand, for each $b \in B(\rm \textbf{G}'_{\Q_p}, -\mu)$, we can define an Igusa variety $\Ig_b$ which is closely related to the special fiber $\overline{\mc S}_p$ of $\mc S_p$. The Igusa variety is a projective system of smooth varieties $ \Ig_b = \{ \Ig_{b, m, K^p} \}$ where $m \in \N$ and $K^p$ are sufficiently small open compact subgroups of $\rm \textbf{G}(\A^{\infty, p})$ as before. Let $\mc L$ be an algebraic representation of $\rm \textbf{G}$ over $\overline{\Q}_{\ell}$. Then it gives rises to $\ell$-adic sheaves on $\Sh$ and $\Ig_b$, which will be denoted by $\mc L$ (by abuse of notation). Define
\[
R\Gamma_c(\Sh, \mc L) := \mathrm{colim}_{\overrightarrow{\ K \ }} R\Gamma_c(\Sh_K, \mc L),
\]
\[
R\Gamma_c(\Ig_b, \mc L) := \mathrm{colim}_{\overrightarrow{K^p, m} } R\Gamma_c(\Ig_{b, m, K^p}, \mc L).
\]

These complexes are endowed with an action of $\rm \textbf{G}(\A^{\infty}) \times \Gal(\overline{E}/E)$ and of $\mathrm{\textbf{G}}(\A^{\infty, p}) \times \big( \G_b(\Q_p) \times \Q_p^{\times} \big)$ respectively.

For each $K$, the special fiber $\overline{\mc S}_{K}$ has a stratification by the Kottwitz set $B(\GL_n, -\mu)$. Thus we can compute the cohomology of Shimura varieties by the cohomology of its strata. We recall a formula of Mantovan that expresses a cohomological relation between $R\Gamma_c(\Sh, \mc L)$, $R\Gamma_c(\Ig_b, \mc L)$'s and the Hecke operator $\T_{-\mu}$. Denote by $\pi$ an irreducible representation of $\mathrm{\textbf{G}}(\Q_p)$ and let $\Pi$ be any $\mc L$-cohomological automorphic representation of $\mathrm{\textbf{G}}(\A)$ globalizing $\pi$. We denote by $R\Gamma_c(\Sh, \mc L)[\Pi^{\infty, p}]$ and $R\Gamma_c(\Ig_b, \mc L)[\Pi^{\infty, p}]$ the $[\Pi^{\infty, p}]$-isotypic part of the cohomology of the Shimura variety, respectively the $[\Pi^{\infty, p}]$-isotypic part of the cohomology of the Igusa varieties.

\begin{proposition} (\cite[Theorem 22]{Man2}, \cite[Theorems. 6.26, 6.32]{LS2018})
    We set $d := \dim \Sh = \langle 2\rho_{\GL_n}, \mu \rangle$ and for $b \in B(\GL_n, -\mu)$, we set $d_b := \dim \Ig_b = \langle 2\rho_{\GL_n} , \nu_b \rangle$. Then the complex $R\Gamma_c(\Sh, \mc L)[\Pi^{\infty, p}]$ has a filtration as a complex of $ \mathrm{\textbf{G}}(\Q_p) \times W_{\Q_p} $ representations with graded pieces isomorphic to $ i_1^* \T_{-\mu}(i_{b!} \delta^{-1}_b \otimes R\Gamma_c(\Ig_b, \mc L)[\Pi^{\infty, p}] )[-d](-\tfrac{d}{2})$ where $b$ runs in the Kottwitz set $B(\GL_n, -\mu)$. More precisely, the graded pieces are isomorphic to 
    \[
    R\Gamma_c(\GL_n, b, \mu) \otimes^{\bL}_{\mathcal{H}(\G_b)}R\Gamma_c(\Ig_b, \mc L)[\Pi^{\infty, p}][2d_b-d](-\tfrac{d}{2}) 
    \]
    as complexes of $ \textbf{G}(\Q_p) \times W_{\Q_p} $-modules.
\end{proposition}

\begin{remark}\begin{enumerate}
    \item In fact, Mantovan and Lan-Stroh proved the proposition with $R\Hom_{\mathcal{H}(\G_b)}$ instead of $\otimes^{\bL}_{\mathcal{H}(\G_b)}$. However the cohomology groups of the complex $R\Gamma_c(\Sh, \mc L)[\Pi^{\infty, p}]$, respectively of the complex $R\Gamma_c(\Ig_b, \mc L)[\Pi^{\infty, p}]$ are admissible as $ \textbf{G}(\Q_p)$-representations, respectively as $ \G_b(\Q_p)$-representations. Therefore the Hom-Tensor duality gives us the dual statement with $\otimes^{\bL}_{\mathcal{H}(\G_b)}$ in place of $R\Hom_{\mathcal{H}(\G_b)}$. By using lemma \ref{shimhecke}, we deduce the expression in term of $\T_{-\mu}$.
    \item Koshikawa and Hamann-Lee proved \cite[Theorem. 7.1]{Ko1}, \cite[Theorem. 1.13]{HL} some analogue of this proposition for torsion coefficients and for more general Shimura varieties of type A and type C.   
\end{enumerate}
 \end{remark}

\subsubsection{Hodge-Tate period map and local-global compatibility} \textbf{}

Recall that by \cite[Theorem. 1.10]{CS}, for each sufficiently small compact open subgroup $K^p \subset \mathrm{\textbf{G}(\A^{\infty, p})}$ there is a perfectoid space $\Sh_{K^p}$ over $E_p$ such that
\[
\Sh_{K^p} \sim \lim_{\overleftarrow{K_p}} \Sh_{K_pK^p} \otimes_E E_{p}
\]
and there is also a Hodge Tate period map 
\[
\pi_{\rm HT} : \Sh_{K^p} \longrightarrow \mathcal{F}\ell_{\GL_n, \mu}.
\]

Let $\phi$ be any local $L$-parameter of $\GL_n$. Fargues and Caraiani-Scholze have conjectured that the perverse sheave $R\pi_{\rm HT *} \overline{\Q}_{\ell}$ on $\mathcal{F}\ell_{\GL_n, \mu}$ is related to the conjectural Hecke eigensheaf on $\Bun_n$ associated with $\phi$ via the natural map $ \overleftarrow{h} : \mathcal{F}\ell_{\GL_n, \mu} \longrightarrow \Bun_n $, by some form of local-global compatibility. We remark that Caraiani-Scholze also compared the fibers of the Hodge-Tate period map with Igusa varieties. Hence we will first compute the cohomology of Igusa varieties and then deduce some form of the local-global compatibility. 

Let $\overline{\phi} = \overline{\phi}_1 \oplus \dotsc \oplus \overline{\phi}_r$ be an $L$-parameter of $\GL_n$ satisfying condition (A1). Then $S_{\overline{\phi}} \simeq \bb G_m^r $ and the connected component of $\overline{\phi}$ in $[Z^1(W_{\Q_p}, \widehat{\GL}_n)/\widehat{\GL}_n]$ is isomorphic to $[\bb G_m^r / \bb G_m^r]$ as before. We consider an $L$-parameter $\phi = \phi_1 \otimes \chi_{t_1} \oplus \dotsc \oplus \phi_r \otimes \chi_{t_r} $ where $ t = (t_1, \dotsc, t_r) $ is a closed point in $ \bb G_m^r $ and where $\chi_{t_i}$ is the unramified character of $W_{\Q_p}$ corresponding to $t_i$. Let $r_{-\mu}$ be the highest weight representation associated with $-\mu$. For $\chi \in \Irr(S_{\phi})$ we denote by $ \sigma_{\chi}$ the $W_{\Q_p}$-representation 
\[
\Hom_{S_{\phi}}(\chi, r_{-\mu} \circ \phi),
\]
and we suppose the following condition :
\begin{enumerate}
    \item[(A2)] $\Hom_{W_{\Q_p}}(\sigma_{\chi}, \sigma_{\chi'}) = 0$ if $\chi \neq \chi'$.
\end{enumerate}



Note that the locus of the points $t$ such that $\phi$ satisfies condition (A2) is a dense subset of $\bb G_m^r$. 

Denote by $\pi$ the irreducible representation of $\GL_n(\Q_p)$ whose $L$-parameter is given by $\phi$ and fix a character $\omega$ of $\Q_p^{\times}$. Let $\Pi$ be any $\mc L$-cohomological automorphic representation of $\mathrm{\textbf{G}}(\A)$ globalizing $\pi \times \omega $, in particular $\Pi$ is cuspidal. For each $b \in B(\GL_n, -\mu)$, we are going to compute the $\Pi^{\infty, p}$-isotypic part
\[
R\Gamma_c(\Ig_b, \mc L)[\Pi^{\infty, p}]
\]
of the cohomology of the variety $\Ig_b$ as $\G_b(\Q_p) \times \Q_p^{\times}$ representation. The rough idea is that Mantovan's formula gives us a relation between $R\Gamma_c(\Sh, \mc L)$, $R\Gamma_c(\Ig_b, \mc L)$'s and the Hecke operator $\T_{-\mu}$. However, in our case, we fully described the Hecke operator $\T_{-\mu}$ and it leads to a transparent relation between $R\Gamma_c(\Sh, \mc L)$ and $R\Gamma_c(\Ig_b, \mc L)$'s. Finally, the knowledge of $R\Gamma_c(\Sh, \mc L)$ allows us to describe $R\Gamma_c(\Ig_b, \mc L)$'s.

\begin{theorem} \label{itm : cohomology of Igusa varieties}
    Suppose that $\phi$ satisfies conditions $(A1)$ and $(A2)$. Then there exists a multiplicity $m \in \N$ such that for all $b \in B(\GL_n, -\mu)$, the complex $R\Gamma_c(\Ig_b, \mc L)[\Pi^{\infty, p}]$ concentrates in middle degree and we have an isomorphism of $\G_b(\Q_p) \times \Q_p^{\times}$-representations
    \[
    H^{d_b}_c(\Ig_b, \mc L)[\Pi^{\infty, p}] = m \bigoplus_{\substack{ \chi \in \Irr(S_{\phi}) \\ [b_{\chi}] = [b]}} \delta^{1/2}_b \otimes \big( \pi_{\chi} \times \omega \big).
    \]
    where $b_{\chi}$ and $\pi_{\chi}$ are constructed from $\chi$ as in section \ref{itm : combinatoric description of Hecke operators}.
\end{theorem}
\begin{proof}
    The first step is to compute the $[\Pi^{\infty, p}]$-isotypic component $R\Gamma_c(\Sh, \mc L)[\Pi^{\infty, p}]$ as complex of $\mathrm{\textbf{G}}(\Q_p) \times W_{\Q_p}$-modules. By the hypothesis imposed on $\Pi$, the complex $R\Gamma_c(\Sh, \mc L)[\Pi^{\infty, p}]$ is concentrated in degree $d$. Therefore, by the results of Kottwitz and Harris-Taylor we deduce that there exists a multiplicity $m \in \N$ such that
    \[
    H^d_c(\Sh, \mc L)[\Pi^{\infty, p}] = m (\pi \times \omega) \boxtimes r_{-\mu} \circ \phi (-\tfrac{d}{2})
    \]
    as $\mathrm{\textbf{G}}(\Q_p) \times W_{\Q_p}$-representations. 

    By Mantovan's product formula, there is a filtration of the $\mathrm{\textbf{G}}(\Q_p) \times W_{\Q_p}$-modules $R\Gamma_c(\Sh, \mc L)[\Pi^{\infty, p}]$ whose graded pieces are isomorphic to $ i_1^* \T_{-\mu}(i_{b!} \delta^{-1}_b \otimes R\Gamma_c(\Ig_b, \mc L)[\Pi^{\infty, p}] )[-d](-\tfrac{d}{2})$ for $b \in B(\GL_n, -\mu)$. Thus, it induces a spectral sequence converging to the cohomology groups of $R\Gamma_c(\Sh, \mc L)[\Pi^{\infty, p}]$. 

    Since the Fargues-Scholze $L$-parameter of $\pi$ is given by $\phi$, we deduce that every irreducible subquotient of the $\G_b(\Q_p)$-modules $ \delta^{-1}_b \otimes H_c^k( \Ig_b, \mc L)[\Pi^{\infty, p}]$ also has Fargues-Scholze $L$-parameter (constructed using Hecke operators on $\Bun_n$) given by $\phi$. However, the complex $R\Gamma_c(\Sh, \mc L)[\Pi^{\infty, p}]$ of $\GL_n(\Q_p)$-modules is admissible and the complexes $R\Gamma_c(\Ig_b, \mc L)[\Pi^{\infty, p}]$ of $\G_b(\Q_p)$-modules are also admissible. Therefore, by theorem \ref{itm : orthogonal decomposition}, we deduce that for each $b \in B(\GL_n, -\mu)$, there is an orthogonal decomposition
    \[
    i_{b !} \big( \delta^{-1}_b \otimes R\Gamma_c(\Ig_b, \mc L)[\Pi^{\infty, p}] \big) = \bigoplus_{\substack{ \chi \in \Irr(S_{\phi}) \\ [b_{\chi}] = [b]}} i_{b !}\big( \delta^{-1}_b \otimes R\Gamma_c(\Ig_b, \mc L)[\Pi^{\infty, p}]\big)_{\chi}.
    \]
    where $i_b : \Bun_b \longrightarrow \Bun_n$ is the usual embedding and $i_{b !}\big( \delta^{-1}_b \otimes R\Gamma_c(\Ig_b, \mc L)[\Pi^{\infty, p}]\big)_{\chi}$ belongs to the full subcategory of $\Dlis(\Bun_n, \overline{\Q}_{\ell})$ corresponding to $\chi$ and the connected component $[C_{\phi}]$ as in the previous sections. Thus it suffices to compute $i_{b !}\big( \delta^{-1}_b \otimes R\Gamma_c(\Ig_b, \mc L)[\Pi^{\infty, p}]\big)_{\chi}$. Moreover we have
    \begin{align*}
        & \quad \ i_1^* \T_{-\mu}(i_{b!} \delta^{-1}_b \otimes R\Gamma_c(\Ig_b, \mc L)[\Pi^{\infty, p}] )[-d](-\tfrac{d}{2}) \\
        &\simeq i_1^* \T_{-\mu} \bigoplus_{\substack{ \chi \in \Irr(S_{\phi}) \\ [b_{\chi}] = [b]}} i_{b !}\big( \delta^{-1}_b \otimes R\Gamma_c(\Ig_b, \mc L)[\Pi^{\infty, p}]\big)_{\chi}[-d](-\tfrac{d}{2}) \\
        &\simeq i_1^* \bigoplus_{\substack{\chi' \in \Irr(S_{\phi})}} \Big( C_{\chi'} \star \bigoplus_{\substack{ \chi \in \Irr(S_{\phi}) \\ [b_{\chi}] = [b]}} i_{b !}\big( \delta^{-1}_b \otimes R\Gamma_c(\Ig_b, \mc L)[\Pi^{\infty, p}]\big)_{\chi} \Big) \boxtimes \sigma_{\chi'} [-d](-\tfrac{d}{2}) \\
        &\simeq i_1^* \bigoplus_{\substack{ \chi \in \Irr(S_{\phi}) \\ [b_{\chi}] = [b]}} \Big(  C_{\chi^{-1}} \star i_{b !}\big( \delta^{-1}_b \otimes R\Gamma_c(\Ig_b, \mc L)[\Pi^{\infty, p}]\big)_{\chi} \Big) \boxtimes \sigma_{\chi^{-1}} [-d](-\tfrac{d}{2}), \numberthis \label{itm : galois action of each term}
    \end{align*}
    where the last isomorphism comes from the fact that $[b_{\xi}] \neq [1]$ if $\xi \neq \Id$ and the complex $C_{\chi'} \star \Big( i_{b !}\big( \delta^{-1}_b \otimes R\Gamma_c(\Ig_b, \mc L)[\Pi^{\infty, p}]\big)_{\chi} \Big)$ is supported on $[b_{\chi' \otimes \chi}]$.

    Recall that there is a filtration of the $\mathrm{\textbf{G}}(\Q_p) \times W_{\Q_p}$-modules $R\Gamma_c(\Sh, \mc L)[\Pi^{\infty, p}]$ whose graded pieces are isomorphic to $ i_1^* \T_{-\mu}(i_{b!} \delta^{-1}_b \otimes R\Gamma_c(\Ig_b, \mc L)[\Pi^{\infty, p}] )[-d](-\tfrac{d}{2})$ for $b \in B(\GL_n, -\mu)$ and then it induces a spectral sequence converging to the cohomology groups of $R\Gamma_c(\Sh, \mc L)[\Pi^{\infty, p}]$. However, by combining equation $(\ref{itm : galois action of each term})$ and the fact that $\Hom_{W_{\Q_p}}(\sigma_{\chi}, \sigma_{\chi'}) = 0$ for $\chi \neq \chi'$, we deduce that the spectral sequence degenerates. Thus we have an isomorphism of complexes of $\mathrm{\textbf{G}}(\Q_p) \times W_{\Q_p}$-modules
    \begin{align*}
      i_1^* \bigoplus_{\substack{ \chi \in \Irr(S_{\phi})}} \Big( C_{\chi^{-1}} \star i_{b_{\chi} !}\big( \delta^{-1}_{b_{\chi}} \otimes R\Gamma_c(\Ig_{b_{\chi}}, \mc L)[\Pi^{\infty, p}]\big)_{\chi} \Big) \boxtimes \sigma_{\chi^{-1}} [-d](-\tfrac{d}{2}) &\simeq R\Gamma_c(\Sh, \mc L)[\Pi^{\infty, p}]  \\
      &\simeq m \big( \pi \times \omega \big) \boxtimes r_{-\mu} \circ \phi [-d](-\tfrac{d}{2}).
    \end{align*}

 By using the fact that $\Hom_{W_{\Q_p}}(\sigma_{\chi}, \sigma_{\chi'}) = 0$ for $\chi \neq \chi'$ and identifying the $W_{\Q_p}$-action in both sides we deduce that
\[
i_1^* C_{\chi^{-1}} \star \Big( i_{b_{\chi} !}\big( \delta^{-1}_{b_{\chi}} \otimes R\Gamma_c(\Ig_{b_{\chi}}, \mc L)[\Pi^{\infty, p}]\big)_{\chi} \Big) \simeq m \cdot m_{\chi^{-1}} \big( \pi \times \omega \big),
\]
where $m_{\chi^{-1}} = 1$ if $ \sigma_{\chi^{-1}} \neq 0 $ and  $m_{\chi^{-1}} = 0$ otherwise. Since $\mu$ is minuscule, we can check that $ \sigma_{\chi^{-1}} \neq 0 $ if and only if $ b_{\chi} \in B(\GL_n, -\mu) $. Now by applying $C_{\chi} \star i_{1 !} $ on both sides and by applying theorem \ref{itm : main theorem}, we get
\[
 i_{b_{\chi} !}\big(\delta^{-1}_{b_{\chi}} \otimes R\Gamma_c(\Ig_{b_{\chi}}, \mc L)[\Pi^{\infty, p}]\big)_{\chi}  \simeq  
i_{b_{\chi} !} \Big( m \cdot m_{\chi^{-1}} \big(\delta^{-1/2}_{b_{\chi}} \otimes \pi_{\chi} \times \omega \big)[-d_b] \Big).
\]

Thus for each $b \in B(\GL_n, -\mu)$, by taking the sum over the $\chi$'s such that $[b_{\chi}] = [b]$, we obtain the following formula  
\[
 R\Gamma_c(\Ig_b, \mc L)[\Pi^{\infty, p}] = m \bigoplus_{\substack{ \chi \in \Irr(S_{\phi}) \\ [b_{\chi}] = [b]}} \delta^{1/2}_{b} \otimes \big( \pi_{\chi} \times \omega \big) [-d_b].
\]
\end{proof}

We have completely described the Hecke eigensheaves associated with the $L$-parameters $\phi$ satisfying conditions (A1), (A2) and we have also computed some part of the cohomology of the Igusa varieties related to $\phi$. Hence we can deduce some weak form of the local-global compatibility.

\begin{corollary}
    Let $\phi$ be an $L$-parameter satisfying conditions $(A1)$ and $(A2)$. Denote by $\mathcal{G}_{\phi}$ the Hecke eigensheaf on $\Bun_n$ corresponding to $\phi$ by theorem \ref{itm : Hecke eigensheaf}. Let $x$ be a point of $\mathcal{F}\ell_{\GL_n, \mu}$ in the stratum corresponding to $b$. Then there exists a multiplicity $m \in \N$ such that we have an isomorphism
    \[
    \delta_b \otimes \big(R\pi_{\rm HT *} \overline{\Q}_{\ell}\big)_x[\Pi^{\infty, p}] \simeq \overleftarrow{h}^*\big(i^*_b(\mathcal{G}_{\phi})^{\oplus m}\big),
    \]
    where $ \overleftarrow{h} : \mathcal{F}\ell_{\GL_n, \mu} \longrightarrow \Bun_n $ is the natural map.
\end{corollary}
\begin{proof}
    This is a consequence of the description of the stalks of $\mathcal{G}_{\phi}$ by theorem \ref{itm : Hecke eigensheaf}, the computation of the cohomology of Igusa varieties $\Ig_b$ by theorem \ref{itm : cohomology of Igusa varieties} and the comparision of the fibers of $\pi_{HT}$ with Igusa varieties \cite[Theorem. 1.15]{CS}, \cite[Corollary 3.13]{HL}. 
\end{proof}

\begin{remark}
\begin{enumerate}
\item It is now clear that Caraiani-Scholze's comparison theorem \cite[Theorem 1.15]{CS} and the local-global compatibility conjecture imply a strong relation between the cohomology of Igusa varieties and the stalks of the conjectural Hecke eigensheaves. 
\item Theorem \ref{itm : cohomology of Igusa varieties} generalizes the $[\Pi^{\infty, p}]$-isotypic part of \cite[Theorem 6.7]{SWS} and gives more precise information. However, the results in \cite{SWS} and \cite{BMS} could give information on more general part of the cohomology of the Igusa varieties. Hence as mentioned above, it could reveal important information on the conjectural Hecke eigensheaves associated to more general $L$-parameters such as the ones corresponding to the Steinberg representations. 
\end{enumerate}
\end{remark}

\subsection{Scope of generalization} \textbf{}

We conclude with comments on the possibility of generalizing our work.

\subsubsection{Torsion coefficients} \textbf{}

We only consider the category $\Dlis(\Bunn, \Lambda)^{\omega}$ where $\Lambda = \ol{\mathbb{Q}}_{\ell}$ in this paper but with extra efforts, one could run the same arguments in the case $\Lambda = \ol{\mathbb{F}}_{\ell}$, at least when $\ell$ is a very good prime. These questions will be studied in an upcoming project. Combining the expected results in the case $\Lambda = \ol{\mathbb{F}}_{\ell}$ with the method developed by Linus Hamann and Yi-Sing Lee \cite{HL} (which is based on the works \cite{CS, CS1, Ko1, Zha}), one can obtain many new cases of the vanishing of the cohomology of Shimura varieties with torsion coefficients.

\subsubsection{Other groups} \textbf{}

The ideas and arguments in this paper are primarily geometric and many of them could work in more general contexts other than $\GL_n$. However, the paper still depends on a number of special aspects of $\GL_n$ such as the compatibility of the usual local $L$-parameter with the one constructed by Fargues-Scholze and the explicit combinatoric nature of the Kottwitz set $B(\GL_n)$. In particular, there is a cocharacter $\mu := (1, 0^{(n-1)})$ whose highest weight representation is simple and convenient for our analysis.   

We expect that one can extend the main results of this paper to the groups $\SL_n$, $\Res_{\Q_p}(\GL_{n, F})$, $(\G)\U_n$ or even $(\G)\SO_{2n+1}$ if the compatibility of $L$-parameters is available. The $\SL_n$ and $(\G)\U_{2n}$ cases are particularly interesting since it could shed light to the functoriality between $\GL_n$ and $\SL_n$ as well as between $\GU_{2n}$ and $\U_{2n}$ in the context of categorical local Langlands program.

It is not clear whether one can apply the method to the groups $(\G)\Sp_{2n}$ and $(\G)\SO_{2n}$. In these cases, the compatibility of $L$-parameters seems to be out of reach and more importantly, the highest weight representations associated with minuscule cocharacters are the spin representations which are harder to analyse\footnote{I thank Peter Scholze for pointing out the possible difficulties concerning complicated highest weight representations associated with minuscule cocharacters.}. However, thanks to \cite{Ham}, one can apply the method to the group $(\G)\Sp_{4}$.  

For a general reductive group $\G$, it is reasonable to speculate that some form of the theorems \ref{itm : main theorem}, \ref{itm : spectral action - general} and \ref{itm : orthogonal decomposition} are true, at least for $L$-parameters $\phi$ such that the corresponding connected component $[C_{\phi}]$ in the stack of $L$-parameters is isormorphic to $[\bb G_m^r / \bb G_m^r \times S ]$ or $[\bb G_m^{r-1} / \bb G_m^r \times S ]$ where $r$ is a natural number and $S$ is an abelian group and $\bb G^r_m \times S$ acts trivially. The point is that in this case we have an orthogonal decomposition

\[
\Perf([C_{\phi}]) \simeq \bigoplus_{\chi \in \Irr( \bb G_m^r \times S )} \Perf(\bb G_m^r),
\]
and the categorical local Langlands program would implied a similar orthogonal decomposition of the category $\Dlis^{[C_{\phi}]}(\Bun_{\G},\ol{\mathbb{Q}}_{\ell})^{\omega}$. 

It is interesting to use the method in this paper to study the case $\Lambda = \ol{\mathbb{F}}_{\ell}$ or $ \Lambda = \ol{\mathbb{Z}}_{\ell}$. It is also reasonable to speculate that some versions of the main theorems still hold in this context. 

\subsubsection{The equal characteristic case} \textbf{}

In \cite{FS}, Fargues and Scholze treated the mixed characteristic and equal characteristic cases uniformly. One might ask to what extend the method in this paper could apply to the equal characteristic case. Some important results such as the classification of $\G$-bundles over the "equal characteristic Fargues-Fontaine curve" and the compatibility between Fargues-Scholze's construction of $L$-parameters with that of Genestier-Lafforgues for $\GL_n$ were obtained in \cite{Ans, L-H}. However, one might need to establish many other important results in order to apply the arguments in this paper. 

\bibliographystyle{amsalpha}
\bibliography{publications}

\providecommand{\bysame}{\leavevmode\hbox to3em{\hrulefill}\thinspace}
\providecommand{\MR}{\relax\ifhmode\unskip\space\fi MR }
\providecommand{\MRhref}[2]{%
  \href{http://www.ams.org/mathscinet-getitem?mr=#1}{#2}
}
\providecommand{\href}[2]{#2}
\begin{thebibliography}{KMSW14}

\bibitem[ALB21]{AL}
Johannes Ansch\"{u}tz and Arthur-C\'esar Le-Bras, \emph{Averaging functors in
  {F}argues' program for {$GL_{n}$}}, Preprint (2021), arXiv:2104.04701.

\bibitem[Ans19]{Ans}
Johannes Ansch\"{u}tz, \emph{Reductive group schemes over the
  {F}argues-{F}ontaine curve}, Math. Ann. \textbf{374} (2019), no.~3-4,
  1219--1260. \MR{3985110}

\bibitem[Bad08]{Badu-Alex}
Alexandru~Ioan Badulescu, \emph{Global {J}acquet-{L}anglands correspondence,
  multiplicity one and classification of automorphic representations}, Invent.
  Math. \textbf{172} (2008), no.~2, 383--438, With an appendix by Neven Grbac.
  \MR{2390289}

\bibitem[Ber84]{Bern}
J.~N. Bernstein, \emph{Le ``centre'' de {B}ernstein}, Representations of
  reductive groups over a local field, Travaux en Cours, Hermann, Paris, 1984,
  Edited by P. Deligne, pp.~1--32. \MR{771671}

\bibitem[BH03]{BH}
Colin~J. Bushnell and Guy Henniart, \emph{Generalized {W}hittaker models and
  the {B}ernstein center}, Amer. J. Math. \textbf{125} (2003), no.~3, 513--547.
  \MR{1981032}

\bibitem[BL95]{BL}
Arnaud Beauville and Yves Laszlo, \emph{Un lemme de descente}, C. R. Acad. Sci.
  Paris S\'{e}r. I Math. \textbf{320} (1995), no.~3, 335--340. \MR{1320381}

\bibitem[BM]{ABM}
Alexander Bertoloni~Meli, \emph{An averaging formula for the cohomology of
  {P}el-type {R}apoport--{Z}ink spaces}.

\bibitem[BM22]{Ber22}
\bysame, \emph{The cohomology of unramified {R}apoport-{Z}ink spaces of
  {EL}-type and {H}arris's conjecture}, J. Inst. Math. Jussieu \textbf{21}
  (2022), no.~4, 1163--1218. \MR{4454341}

\bibitem[BMHN22]{BHN}
Alexander Bertoloni~Meli, Linus Hamann, and Kieu~Hieu Nguyen,
  \emph{Compatibility of the {F}argues--{S}cholze correspondence for unitary
  groups}, https://arxiv.org/abs/2207.13193.

\bibitem[BMN21]{BMN}
Alexander Bertoloni~Meli and Kieu~Hieu Nguyen, \emph{The {K}ottwitz conjecture
  for unitary {PEL}-type {R}apoport--{Z}ink spaces}, 2021.

\bibitem[BMO22]{BMO}
Alexander Bertoloni~Meli and Masao Oi, \emph{The ${B}({G})$-parametrization of
  the local {L}anglands correspondence}, https://arxiv.org/pdf/2211.13864.pdf.

\bibitem[BMS]{BMS}
Alexander Bertoloni~Meli and Sug~Woo Shin, \emph{The stable trace formula for
  {I}gusa varieties, ii}.

\bibitem[Boy99]{Boy}
P.~Boyer, \emph{Mauvaise r\'{e}duction des vari\'{e}t\'{e}s de {D}rinfeld et
  correspondance de {L}anglands locale}, Invent. Math. \textbf{138} (1999),
  no.~3, 573--629. \MR{1719811}

\bibitem[BR10]{Badu-Renard}
A.~I. Badulescu and D.~Renard, \emph{Unitary dual of {${\rm GL}(n)$} at
  {A}rchimedean places and global {J}acquet-{L}anglands correspondence},
  Compos. Math. \textbf{146} (2010), no.~5, 1115--1164. \MR{2684298}

\bibitem[CFS21]{CFS}
Miaofen Chen, Laurent Fargues, and Xu~Shen, \emph{On the structure of some
  {$p$}-adic period domains}, Camb. J. Math. \textbf{9} (2021), no.~1,
  213--267. \MR{4325262}

\bibitem[Che21]{MC}
Miaofen Chen, \emph{Fargues-rapoport conjecture for p-adic period domains in
  the non-basic case}, J. Eur. Math. Soc. (2021).

\bibitem[CS17]{CS}
Ana Caraiani and Peter Scholze, \emph{On the generic part of the cohomology of
  compact unitary {S}himura varieties}, Ann. of Math. (2) \textbf{186} (2017),
  no.~3, 649--766. \MR{3702677}

\bibitem[CS19]{CS19}
Kei~Yuen Chan and Gordan Savin, \emph{Bernstein-{Z}elevinsky derivatives: a
  {H}ecke algebra approach}, Int. Math. Res. Not. IMRN (2019), no.~3, 731--760.
  \MR{3910471}

\bibitem[CS24]{CS1}
Ana Caraiani and Peter Scholze, \emph{On the generic part of the cohomology of
  non-compact unitary {S}himura varieties}, Ann. of Math. (2) \textbf{199}
  (2024), no.~2, 483--590. \MR{4713019}

\bibitem[DHKM25]{DH}
Jean-Fran\c{c}ois Dat, David Helm, Robert Kurinczuk, and Gilbert Moss,
  \emph{Moduli of {L}anglands parameters}, J. Eur. Math. Soc. (JEMS)
  \textbf{27} (2025), no.~5, 1827--1927, https://arxiv.org/abs/2009.06708v3.
  \MR{4889237}

\bibitem[Far04]{Far04}
Laurent Fargues, \emph{Cohomologie des espaces de modules de groupes
  {$p$}-divisibles et correspondances de {L}anglands locales}, no. 291, 2004,
  Vari\'{e}t\'{e}s de Shimura, espaces de Rapoport-Zink et correspondances de
  Langlands locales, pp.~1--199. \MR{2074714}

\bibitem[Far20]{Far20}
\bysame, \emph{Simple connexit\'{e} des fibres d'une application
  d'{A}bel-{J}acobi et corps de classes local}, Ann. Sci. \'{E}c. Norm.
  Sup\'{e}r. (4) \textbf{53} (2020), no.~1, 89--124. \MR{4093441}

\bibitem[Far25]{Fa}
\bysame, \emph{Geometrization of the local {L}anglands correspondence: an
  overview}, J. Math. Sci. Univ. Tokyo \textbf{32} (2025), no.~2, 157--240.
  \MR{4973974}

\bibitem[FF18]{FF}
Laurent Fargues and Jean-Marc Fontaine, \emph{Courbes et fibr\'{e}s vectoriels
  en th\'{e}orie de {H}odge {$p$}-adique}, Ast\'{e}risque (2018), no.~406,
  xiii+382, With a preface by Pierre Colmez. \MR{3917141}

\bibitem[FS24]{FS}
L.~Fargues and P.~Scholze, \emph{Geometrization of the local {L}anglands
  correspondence}, https://arxiv.org/pdf/2102.13459v4.

\bibitem[GI19]{GI}
Ildar Gaisin and Naoki Imai, \emph{Non-semi-stable loci in {H}ecke stacks and
  {F}argues' conjecture}, Preprint (2019),
  https://www.ms.u-tokyo.ac.jp/~naoki/file/conv.pdf.

\bibitem[Ham25]{Ham}
L.~Hamann, \emph{Compatibility of the {G}an--{T}akeda and {F}argues--{S}cholze
  local {L}anglands}, Preprint (2025), https://arxiv.org/abs/2109.01210v3.

\bibitem[Ham26]{Ham1}
L.~Hamann, \emph{Geometric {E}isenstein series, intertwining operators, and
  {S}hin's averaging formula, with an appendix by {Alexander Bertoloni Meli}},
  https://arxiv.org/abs/2209.08175v7.

\bibitem[Han17]{Han2}
David. Hansen, \emph{Degenerating vector bundles in $p$-adic {H}odge theory},
  to appear in J. Inst. Math. Jussieu (2017).

\bibitem[Han20]{Han}
D.~Hansen, \emph{On the supercuspidal cohomology of basic local {S}himura
  varieties}, Available at home page of first named author.

\bibitem[Han21]{Han1}
\bysame, \emph{Moduli of local shtukas and harris's conjecture}, Tunisian J.
  Math. (2021), no.~4, 749--799.

\bibitem[Han23]{Hansen}
D.~Hansen, \emph{Beijing notes on the categorical local langlands conjecture},
  https://arxiv.org/abs/2310.04533.

\bibitem[Har01]{Har}
Michael Harris, \emph{Local {L}anglands correspondences and vanishing cycles on
  {S}himura varieties}, European {C}ongress of {M}athematics, {V}ol. {I}
  ({B}arcelona, 2000), Progr. Math., vol. 201, Birkh\"{a}user, Basel, 2001,
  pp.~407--427. \MR{1905332}

\bibitem[Hel20]{Helm}
David Helm, \emph{Curtis homomorphisms and the integral {B}ernstein center for
  {${\rm GL}_n$}}, Algebra Number Theory \textbf{14} (2020), no.~10,
  2607--2645. \MR{4190413}

\bibitem[Hel23]{Hell}
Eugen Hellmann, \emph{On the derived category of the {I}wahori-{H}ecke
  algebra}, Compos. Math. \textbf{159} (2023), no.~5, 1042--1110. \MR{4586564}

\bibitem[HHS]{HHS}
L.~Hamann, D.~Hansen, and P.~Scholze, \emph{Untitled project on geometric
  {E}isenstein series}.

\bibitem[HHS24]{HHS1}
Linus Hamann, David Hansen, and Peter Scholze, \emph{Geometric eisenstein
  series i: Finiteness theorems}, To Appear in Inventiones (2024),
  https://arxiv.org/abs/2409.07363.

\bibitem[HI25]{HI}
Linus Hamann and Naoki Imai, \emph{Dualizing complexes on the moduli of
  parabolic bundles}, J. Reine Angew. Math. \textbf{825} (2025), 139--183.
  \MR{4939941}

\bibitem[HKW21]{HKW}
David Hansen, Tasho Kaletha, and Jared Weinstein, \emph{On the {K}ottwitz
  conjecture for local {S}himura varieties}, Preprint (2021), arXiv:1709.06651.

\bibitem[HL23]{HL}
Linus Hamann and Si-Ying Lee, \emph{Torsion vanishing for some {S}himura
  varieties}.

\bibitem[Hon18]{Ho2}
Serin Hong, \emph{Harris-{V}iehmann conjecture for {H}odge-{N}ewton reducible
  {R}apoport-{Z}ink spaces}, J. Lond. Math. Soc. (2) \textbf{98} (2018), no.~3,
  733--752. \MR{3893199}

\bibitem[HT01]{HT}
Michael Harris and Richard Taylor, \emph{The geometry and cohomology of some
  simple {S}himura varieties}, Annals of Mathematics Studies, vol. 151,
  Princeton University Press, Princeton, NJ, 2001, With an appendix by Vladimir
  G. Berkovich. \MR{1876802}

\bibitem[Ima19]{Nao}
Naoki Imai, \emph{Convolution morphisms and {K}ottwitz conjecture}, Preprint
  (2019), https://www.ms.u-tokyo.ac.jp/~naoki/file/conv.pdf.

\bibitem[Ima22]{Nao1}
\bysame, \emph{Local {L}anglands correspondences in {$\ell$}-adic
  coefficients}, Preprint (2022),
  https://www.ms.u-tokyo.ac.jp/~naoki/file/llcl.pdf.

\bibitem[Ked08]{Ked}
Kiran~S. Kedlaya, \emph{Slope filtrations for relative {F}robenius}, no. 319,
  2008, Repr\'{e}sentations $p$-adiques de groupes $p$-adiques. I.
  Repr\'{e}sentations galoisiennes et $(\phi,\Gamma)$-modules, pp.~259--301.
  \MR{2493220}

\bibitem[KL15]{KL}
Kiran~S. Kedlaya and Ruochuan Liu, \emph{Relative {$p$}-adic {H}odge theory:
  foundations}, Ast\'{e}risque (2015), no.~371, 239. \MR{3379653}

\bibitem[KMSW14]{KMSW}
Tasho Kaletha, Alberto Minguez, Sug~Woo Shin, and Paul-James White,
  \emph{Endoscopic classification of representations: Inner forms of unitary
  groups}, 2014.

\bibitem[Kos21]{Ko1}
Teruhisa Koshikawa, \emph{On the generic part of the cohomology of local and
  global {S}himura varieties}, Preprint (2021), arXiv:2106.10602.

\bibitem[Kot92]{Kot92}
Robert~E. Kottwitz, \emph{On the {$\lambda$}-adic representations associated to
  some simple {S}himura varieties}, Invent. Math. \textbf{108} (1992), no.~3,
  653--665. \MR{1163241}

\bibitem[Kot97]{KottwitzIsocrystals2}
\bysame, \emph{Isocrystals with additional structure. {II}}, Compositio Math.
  \textbf{109} (1997), no.~3, 255--339. \MR{1485921}

\bibitem[Laf18]{VL}
Vincent Lafforgue, \emph{Chtoucas pour les groupes r\'{e}ductifs et
  param\'{e}trisation de {L}anglands globale}, J. Amer. Math. Soc. \textbf{31}
  (2018), no.~3, 719--891. \MR{3787407}

\bibitem[LH23]{L-H}
Siyan~Daniel Li-Huerta, \emph{Local-global compatibility over function fields},
  https://arxiv.org/abs/2301.09711.

\bibitem[LS18]{LS2018}
Kai-Wen Lan and Beno\^{\i}t Stroh, \emph{Nearby cycles of automorphic \'{e}tale
  sheaves}, Compos. Math. \textbf{154} (2018), no.~1, 80--119. \MR{3719245}

\bibitem[Man05]{Man2}
Elena Mantovan, \emph{On the cohomology of certain {PEL}-type {S}himura
  varieties}, Duke Math. J. \textbf{129} (2005), no.~3, 573--610. \MR{2169874}

\bibitem[Man08]{Man08}
\bysame, \emph{On non-basic {R}apoport-{Z}ink spaces}, Ann. Sci. \'{E}c. Norm.
  Sup\'{e}r. (4) \textbf{41} (2008), no.~5, 671--716. \MR{2504431}

\bibitem[Ngu20]{KH20}
Kieu~Hieu Nguyen, \emph{Shtukas adiques, modifications et applications}, Bull.
  Soc. Math. France \textbf{148} (2020), no.~4, 623--650. \MR{4221800}

\bibitem[Ngu23]{KH1}
\bysame, \emph{Un cas {PEL} de la conjecture de {K}ottwitz}, 2023,
  pp.~2239--2304. \MR{4661455}

\bibitem[NV]{NV}
Kieu~Hieu Nguyen and Eva Viehmann, \emph{A harder-narasimhan stratification of
  the $b^+_{dR}$-grassmannian}.

\bibitem[RV14]{RV}
Michael Rapoport and Eva Viehmann, \emph{Towards a theory of local {S}himura
  varieties}, M\"{u}nster J. Math. \textbf{7} (2014), no.~1, 273--326.
  \MR{3271247}

\bibitem[RZ96]{RZ}
M.~Rapoport and Th. Zink, \emph{Period spaces for {$p$}-divisible groups},
  Annals of Mathematics Studies, vol. 141, Princeton University Press,
  Princeton, NJ, 1996. \MR{1393439}

\bibitem[Sec09]{Secherre}
Vincent Secherre, \emph{Proof of the {T}adi\'{c} conjecture ({U}0) on the
  unitary dual of {${\rm GL}_m(D)$}}, J. Reine Angew. Math. \textbf{626}
  (2009), 187--203. \MR{2492994}

\bibitem[She14]{Shen14}
Xu~Shen, \emph{On the {H}odge-{N}ewton filtration for {$p$}-divisible groups
  with additional structures}, Int. Math. Res. Not. IMRN (2014), no.~13,
  3582--3631. \MR{3229763}

\bibitem[Shi12]{SWS}
Sug~Woo Shin, \emph{On the cohomology of {R}apoport-{Z}ink spaces of
  {EL}-type}, Amer. J. Math. \textbf{134} (2012), no.~2, 407--452. \MR{2905002}

\bibitem[Sol22]{Soll}
Maarten Solleveld, \emph{Endomorphism algebras and {H}ecke algebras for
  reductive {$p$}-adic groups}, J. Algebra \textbf{606} (2022), 371--470.
  \MR{4432237}

\bibitem[SW20]{SW}
P.~Scholze and J.~Weinstein, \emph{Berkeley lectures on $p$-adic geometry},
  Annals of Mathematics Studies, vol. 389, Princeton University Press, 2020.

\bibitem[Vie24]{Vie}
Eva Viehmann, \emph{On {N}ewton strata in the {$B_{\rm dR}^+$}-{G}rassmannian},
  Duke Math. J. \textbf{173} (2024), no.~1, 177--225. \MR{4728690}

\bibitem[Zha23]{Zha}
Mingjia Zhang, \emph{A {PEL}-type {I}gusa stack and the p-adic geometry of
  {S}himura varieties}, https://arxiv.org/pdf/2309.05152.pdf.

\bibitem[Zhu20]{Zhu1}
X.~Zhu, \emph{Coherent sheaves on the stack of {L}anglands parameters},
  Preprint (2020), arXiv:2008.02998.

\bibitem[Zou22]{Zou}
Konrad Zou, \emph{The categorical form of {F}argues’ conjecture for tori},
  Preprint (2022), arXiv:2202.13238.

\bibitem[Zou25]{Zou25}
\bysame, \emph{Categorical local langlands for gln for parameters of
  langlands-shahidi type with integral coefficients}, Preprint (2025), arXiv:
  2504.06499v1.

\end{thebibliography}
\end{document}